\documentclass[10pt,a4paper]{article}
\ifdefined\pdfoutput
    \pdfoutput=1 
\fi

\usepackage{amsfonts,amsmath,amsbsy,amsthm,cases,listings
}
\def\thickhline{\noalign{\hrule height1.2pt}}
\usepackage{array}
\usepackage{color}
\usepackage{multirow}
\usepackage{epstopdf}
\usepackage{graphicx}
\usepackage{pgf,tikz}
\usepackage{mathrsfs}
\usepackage[marginal]{footmisc}
\usepackage{enumitem}
\usepackage[titletoc]{appendix}
\usepackage{booktabs}
\usepackage{url}
\usepackage{amssymb}
\usepackage{subfigure}
\usepackage{authblk}
\usepackage{pgfplots}
\usepackage{csquotes}
\usepackage{mathtools}
\pgfplotsset{compat=1.5}
\usepackage{xcolor}
\usetikzlibrary{arrows}

\allowdisplaybreaks[1]
\usepackage{float}

\usepackage[backref=page]{hyperref}
\usepackage[a4paper,margin=2.5cm]{geometry}
\definecolor{uuuuuu}{rgb}{0.27,0.27,0.27}
\definecolor{sqsqsq}{rgb}{0.1255,0.1255,0.1255}

\theoremstyle{plain}
\newtheorem{theorem}{Theorem}[section]
\newtheorem{lemma}[theorem]{Lemma}
\newtheorem{proposition}[theorem]{Proposition}

\newtheorem{claim}[theorem]{Claim}

\theoremstyle{definition}
\newtheorem{definition}[theorem]{Definition}
\newtheorem{conjecture}[theorem]{Conjecture}

\theoremstyle{remark}

\newtheorem{fact}[theorem]{Fact}

\makeatletter
\newcommand{\RomanNumeralCaps}[1]{\@Roman{#1}}
\makeatother

\begin{document}
\title{\bf\Large  Density Hajnal--Szemer\'{e}di theorem for cliques of size four}
\date{\today}
\author[1]{Jianfeng~Hou}
\author[1]{Caiyun~Hu}
\author[2]{Xizhi~Liu}
\author[1]{Yixiao~Zhang}
\affil[1]{{\small Center for Discrete Mathematics,
            Fuzhou University, Fujian, China}}
\affil[2]{{\small School of Mathematical Sciences, University of Science and Technology of China, Hefei, China}}
\maketitle
\begin{abstract}
    The exact minimum-degree thresholds that guarantee the existence of $k$ vertex-disjoint copies of the complete graph $K_r$ in an $n$-vertex graph were determined in the celebrated theorems of Dirac (1952) for $r=2$, Corr\'adi--Hajnal (1963) for $r=3$, and Hajnal--Szemer\'edi (1970) for general $r \ge 4$.

    A natural and fundamental problem is to determine the corresponding average-degree (or density) thresholds. This problem turns out to be substantially more difficult. The seminal theorem of Erd\H{o}s--Gallai (1959) resolved the case $r=2$. The case $r=3$ was settled for large $n$ by Allen, B\"ottcher, Hladk\'y, and Piguet~\cite{ABHP15}, who remarkably identified four distinct classes of extremal constructions arising in different ranges of $k$. They further posed the problem of establishing a density analogue of the Hajnal--Szemer\'edi theorem, noting that ``We are not even sure what the complete family of extremal graphs should be.''

    In this paper, we establish an asymptotic density Hajnal-Szemer\H{e}di theorem for \(K_4\) by determining \(\mathrm{ex}(n,(k+1)K_4)\) up to an additive \(O(n)\) term for all integers \(n\ge 4k\ge 0\).
    Five families of constructions show that this bound is sharp up to an additive \(O(n)\) term.
    For general \(r\ge 5\), we propose a candidate family consisting of \(r+1\) extremal construction classes.

\end{abstract}
\tableofcontents
\section{Introduction}\label{SEC:Introduction}
\subsection{Motivation and main results}
We identify a graph $G$ with its edge set; in particular, $|G|$ represents the number of edges in $G$. 
The number of vertices in $G$ is denoted by $v(G)$. 
The \textbf{minimum degree}, \textbf{average degree}, and \textbf{maximum degree} of $G$ are denoted by $\delta(G)$, $d(G)$, and $\Delta(G)$, respectively. 

Given two graphs $F$ and $G$, a collection $\mathcal{C}$ of vertex-disjoint copies of $F$ in $G$ is called an \textbf{$F$-tiling} in $G$. We denote by $V(\mathcal{C})$ the set of vertices covered by the members of $\mathcal{C}$. 
The \textbf{$F$-matching number} $\nu(F, G)$ of $G$ is the maximum size of an $F$-tiling in $G$. 
The case $F = K_2$ corresponds to the well-known \textbf{matching number} of $G$, which is denoted simply as $\nu(G)$. 
We say a graph $G$ is \textbf{$F$-free} if $\nu(F, G) = 0$. 

The study of the following problem  encompasses several central topics in Extremal Graph Theory.
Given a graph $F$ and integers $n, k\in \mathbb{N}$ satisfying $n \ge k \cdot v(F) \ge 0$:
\begin{align}\label{prob:main}
    \text{What constraints on an $n$-vertex graph $G$ force it to satisfy $\nu(F, G) \ge k+1$?} \tag{$\ast$}
\end{align}
From the perspective of the minimum degree, the celebrated Corr\'{a}di--Hajnal Theorem~\cite{CH63} and Hajnal--Szemer\'{e}di Theorem~\cite{HS70} determined the precise minimum degree threshold that ensures a graph $G$ satisfies $\nu(K_{r}, G) \ge k+1$ for $r = 3$ and $r \ge 4$, respectively. 
These two foundational results have inspired numerous subsequent works (see e.g.~\cite{Sey74,FH94,FK95,KSS95,FK96,FK96b,AY96,KSS98a,KSS98b,Kom00,KSS01,Kaw02,SZ03,SZ04,KO06,CKO07,KK08,KO09,BST09,Zhao09,HS10,ABH11,KM15,Tre16,MS17,HT20,LS23}) and have been extended to general graphs.
For additional related results, we refer the reader to surveys~\cite{KS96,KSSS02,KO09Survery}.
\begin{theorem}[Corr\'{a}di--Hajnal~\cite{CH63}, Hajnal--Szemer\'{e}di~\cite{HS70}]\label{THM:HS-mindeg}
    Let $r \ge 3$ and $n \ge r (k+1) \ge 0$ be integers. 
    Suppose that 
    \begin{align*}
        \delta(G) >  k + \left\lfloor \frac{r-2}{r-1}(n-k) \right\rfloor. 
    \end{align*}
    Then $\nu(K_{r}, G) \ge k + 1$.
\end{theorem}
\textbf{Remark}. The tightness of the bound $\delta(G) >  k + \left\lfloor \frac{r-2}{r-1}(n-k) \right\rfloor$ is demonstrated by the construction $E_{1}(n,k,r)$ defined in Section~\ref{SEC:Remarks}. 

From the perspective of the average degree, the well-studied Tur\'{a}n problem (see e.g.~\cite{Mantel07,TU41,ES46,FS13}), which asks for the maximum number $\mathrm{ex}(n,F)$ of edges in an $F$-free graph on $n$ vertices, is equivalent to determining the minimum edge density of $G$ that ensures $\nu(F, G) \ge 1$.
The celebrated Erd\H{o}s--Gallai Theorem~\cite{EG59} determined for all integers $n \ge 2t \ge 0$, the minimum edge density of $G$ that ensures $\nu(G) \ge t$. 
For general graphs, the answer is less complete. 
We use $k$ to denote $\nu(F,G)\le k$.
Thus, $\mathrm{ex}(n,(k+1)F)$ is the maximum number of edges in an $n$-vertex graph $G$ with $\nu(F,G) < k+1$.
An old theorem by Erd\H{o}s~\cite{Erdos62} determined $\mathrm{ex}\left(n, (k+1)K_3\right)$ for $k \le \frac{\sqrt{n}}{20}$. 
Later, Moon~\cite{Moon68} improved and extended the theorem of Erd\H{o}s to general $r \ge 4$ and determined $\mathrm{ex}\left(n, (k+1)K_{r}\right)$ for $k \le \frac{2 n-3 r^2+8 r-5}{r^3-r^2+1}$.
Akiyama--Frankl~\cite{AF85} determined $\mathrm{ex}\left(n, (k+1)K_{r}\right)$ for the case when $k = \lfloor n/r \rfloor - 1$ (see also~\cite{BE78}), which was later extended by~\cite{BKT13}.   
For bipartite graphs $F$, Grosu--Hladk{\'y}~\cite{GH12} established a general upper bound for $\mathrm{ex}\left(n, (k+1)F\right)$, which is asymptotically tight for certain classes of bipartite graphs. 
Extending the work of Erd\H{o}s and Moon, we established general upper bounds, which are tight in many cases, for $\mathrm{ex}\left(n, (k+1)F\right)$ in a sequence of studies~\cite{HLLYZ23,HHLLYZ23a,HHLLYZ23b,HHLLYZ23c}, when $k$ lies in the interval $\left[0, \varepsilon_{F}\frac{\mathrm{ex}(n,F)}{n}\right]$, where $\varepsilon_F > 0$ is a small constant.

Determining $\mathrm{ex}\left(n, (k+1)F\right)$, even asymptotically, for all $k$ in the interval $\left[0, \frac{n}{v(F)}\right]$ is extremely challenging in general. 
Indeed, even in the case where $F$ is a triangle, $\mathrm{ex}\left(n, (k+1)F\right)$ was completely determined only a decade ago by Allen--B\"{o}ttcher--Hladk\'{y}--Piguet~\cite{ABHP15} for large $n$. 
Their results, which we find remarkable, show that there are four different classes of extremal constructions corresponding to four different regimes of $k$. 
Below, we present only the asymptotic version of their result and refer the reader to~\cite{ABHP15} for additional details on the extremal constructions.
\begin{theorem}[Allen--B\"{o}ttcher--Hladk\'{y}--Piguet~\cite{ABHP15}]\label{THM:ABHP-density-CH}
    Suppose that $n$ is sufficiently large. Then 
    \begin{align*}
        \mathrm{ex}(n,(k+1)K_{3})
        = O(n) + 
        \begin{cases}
            \frac{n^2}{4} + \frac{k n}{2} - \frac{k^2}{4}, &\quad\text{if}\quad k \in \left[0, \frac{2n}{9}\right], \\[0.5em]
            \frac{n^2}{4} + 2k^2, &\quad\text{if}\quad k \in \left[\frac{2n}{9}, \frac{n}{4}\right], \\[0.5em]
            2kn- 2k^2, &\quad\text{if}\quad k \in \left[\frac{n}{4}, \frac{(5+\sqrt{3})n}{22}\right], \\[0.5em]
            \frac{n^2}{2} - 3kn + 9k^2, &\quad\text{if}\quad k \in \left[\frac{(5+\sqrt{3})n}{22}, \frac{n}{3}\right].
        \end{cases}
    \end{align*}
\end{theorem}
In~\cite{ABHP15}, Allen--B\"{o}ttcher--Hladk\'{y}--Piguet raised the problem of establishing a density version of the Hajnal--Szemer\'{e}di Theorem, i.e. determining $\mathrm{ex}(n,(k+1)K_{r})$ for $r \ge 4$ and $k \in \left[0, \frac{n}{r}\right]$.
They remarked: 
\begin{displayquote}
    Some parts of such an argument can be made to work, but there are some additional difficulties for $r\ge 4$ that do not appear for $r=3$. We are not even sure what the complete family of extremal graphs should be.
\end{displayquote}
We contribute to this problem by proposing a candidate set of extremal constructions for general $r \ge 4$ (see Section~\ref{SEC:Remarks}), and determining asymptotically the value of $\mathrm{ex}(n,(k+1)K_{4})$ 
for all integers $k$ satisfying $ 0\le k\le\lfloor n/4\rfloor-1$.
The case $k=\lfloor n/4\rfloor$ is trivial, since $\mathrm{ex}(n,(k+1)K_4)=\binom{n}{2}$.
%

\begin{figure}[H]
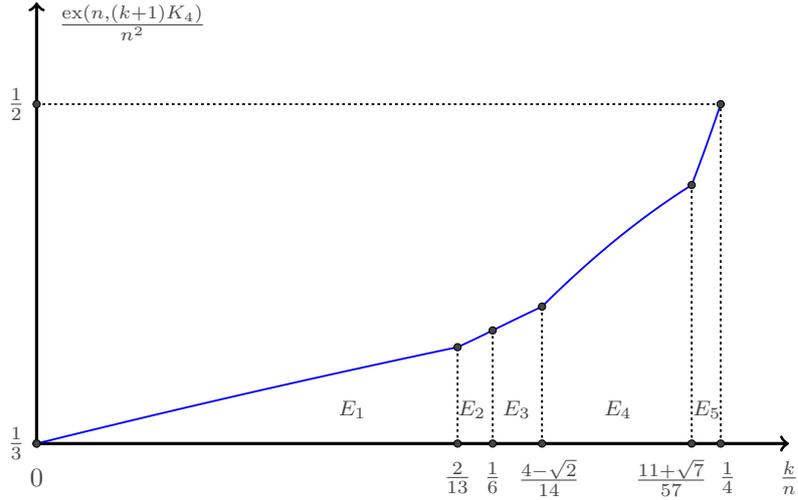

\centering

\caption{The asymptotic behavior of $\frac{\mathrm{ex}(n,(k+1)K_4)}{n^2}$ as a function of $\frac{k}{n}$.}
\label{Fig:functions}
\end{figure}

%
\begin{theorem}\label{THM:Mian-HS-density-K4}
    Suppose that $n$ and $k$ are integers satisfying $n \ge 4k \ge 0$. 
    Then 
    \begin{align*}
        \mathrm{ex}(n,(k+1)K_4)
        = \Xi(n,k) + O(n),  
    \end{align*}
    where 
    \begin{align*}
        \Xi(n,k)
        \coloneqq 
        \begin{cases}
            \frac{n^2}{3} + \frac{k n}{3} - \frac{k^2}{6}, &\quad\text{if}\quad k \in \left[0, \frac{2n}{13}\right],\\[0.5em]
            \frac{n^2}{3} + 2k^2, &\quad\text{if}\quad k \in \left[\frac{2n}{13}, \frac{n}{6}\right],\\[0.5em]
            \frac{n^2}{4} + k n - k^2, &\quad\text{if}\quad k \in \left[\frac{n}{6}, \frac{(4-\sqrt{2})n}{14}\right],\\[0.5em]
            3kn - \frac{9k^2}{2}, &\quad\text{if}\quad k \in \left[\frac{(4-\sqrt{2})n}{14}, \frac{(11+\sqrt{7})n}{57}\right],\\[0.5em]
            n^2 - 8kn +24k^2, &\quad\text{if}\quad k \in \left[\frac{(11+\sqrt{7})n}{57}, \frac{n}{4}\right].
        \end{cases}
    \end{align*}
\end{theorem}
\textbf{Remarks.}
\begin{enumerate}[label=(\roman*)]
    \item The function $\Xi(n,k)$ arises from the edge densities of five classes of constructions $E_{1}(n,k)$, $\ldots$, $E_{5}(n,k)$, which we will define in the next subsection.
    We conjecture that these constructions are extremal, i.e. $\mathrm{ex}(n,(k+1)K_{4}) = \max_{i\in [5]}|E_{i}(n,k)|$, when $n$ is large. 
    %
    \item Another potential approach that we believe could be useful (although it may also require significant technical work) for determining the exact value of $\mathrm{ex}(n,(k+1)K_{4})$, is the stability method, first introduced by Simonovits~\cite{SI68}.  
\end{enumerate}

\medskip

\textbf{Other related work}:
Note that the edge density version of Question~\eqref{prob:main} can be rephrased as follows:  
\begin{align}\label{question:edge-density-matching}
    \text{Given edge density, minimize the number of vertex-disjoint copies of $F$ in a graph.} \tag{$\star$}
\end{align}
One could replace ``the number of vertex-disjoint copies of $F$'' with other parameters, such as ``the number of copies of $F$'', which corresponds to the classical Erd{\H o}s--Rademacher Problem (see ~\cite{Erdos55,Goo59,Erd62,MM62,NS63,Bol76,LS76,Nik76,NK81,LS83,Fis89,Raz08,Nik11,Rei16,Mub10,Mub13,PR17,KLPS20,LPS20,LM22,BC23,MY23,LP25} for related results); or with ``the number of edge-disjoint copies of $F$'' (see e.g.~\cite{EGP66,Erd69,Gyo88,Gyo91} for related results). 
These results, combined with theorems from Extremal Set Theory (see the survey~\cite{FT16}), can provide some bounds for~\eqref{question:edge-density-matching}. However, these bounds are usually non-tight (even asymptotically), as the extremal constructions for these problems are usually quite different. 

\subsection{Extremal constructions}
In this subsection, we describe the structures of the five classes $E_{1}(n,k), \ldots, E_{5}(n,k)$ of extremal constructions.
For convenience, we slightly abuse notation by referring to $E_{i}(n,k)$ as a typical member in this class.

Given $t$ pairwise disjoint sets $V_1, \ldots, V_t$, we use $K_{t}[V_1, \ldots, V_t]$ to denote \textbf{the complete $t$-partite graph} with parts $V_1, \ldots, V_t$. We will omit the subscript $t$ if it is clear from the context. 

Let $n \ge 4k \ge 0$ be integers. Define the following five classes of  graphs on $n$ vertices (see Figure~\ref{Fig:extremal-graphs})$\colon$
\begin{itemize}
    \item The vertex set of $E_{1}(n,k)$ consists of four parts $X, Y_1, Y_2, Y_3$, with sizes given by 
    \begin{align*}
        |X| = k, \quad 
        |Y_1| = \left\lfloor \frac{n-k}{3} \right\rfloor, \quad 
        |Y_2| = \left\lfloor \frac{n-k+1}{3} \right\rfloor, \quad\text{and}\quad 
        |Y_3| = \left\lfloor \frac{n-k+2}{3} \right\rfloor.   
    \end{align*}
    The edge set of $E_{1}(n,k)$ is defined as
    \begin{align*}
        E_{1}(n,k)
        = \binom{X}{2} \cup K[X, Y_1, Y_2, Y_3]. 
    \end{align*}
    \item The vertex set of $E_{2}(n,k)$ consists of four parts $X, Y_1, Y_2, Y_3$, with sizes satisfying 
    \begin{align*}
        |X| = 2k+1, \quad 
        |Y_1| = \left\lfloor \frac{n+i_{1}}{3} \right\rfloor, \quad 
        |Y_2| = \left\lfloor \frac{n+i_{2}}{3} \right\rfloor, \quad\text{and}\quad 
        |Y_3| = \left\lfloor \frac{n+i_{3}}{3} \right\rfloor  - |X|,  
    \end{align*}
    where $\{i_1, i_2, i_3\} = \{0,1,2\}$.  
    The edge set of $E_{2}(n,k)$ is defined as
    \begin{align*}
        E_{2}(n,k)
        = \binom{X}{2} \cup K[X \cup Y_3, Y_1, Y_2]. 
    \end{align*}
    Note that $E_{2}(n,k)$ is defined only for $k \in \left[0, \frac{n-1}{6}\right]$. 
    \item The vertex set of $E_{3}(n,k)$ consists of three parts $X, Y_1, Y_2$, with sizes given by 
    \begin{align*}
        |X| = 2k+1, \quad 
        |Y_1| = \left\lfloor \frac{n-|X|}{2} \right\rfloor, \quad\text{and}\quad
        |Y_2| = \left\lceil \frac{n-|X|}{2} \right\rceil. 
    \end{align*}
    The edge set of $E_{3}(n,k)$ is defined as
    \begin{align*}
        E_{3}(n,k)
        = \binom{X}{2} \cup K[X, Y_1, Y_2]. 
    \end{align*}
    \item The vertex set of $E_{4}(n,k)$ consists of two parts $X, Y_1$, with sizes given by 
    \begin{align*}
        |X| = 3k+2 \quad\text{and}\quad
        |Y_1| = n-|X|.
    \end{align*}
    The edge set of $E_{4}(n,k)$ is defined as
    \begin{align*}
        E_{4}(n,k)
        = \binom{X}{2} \cup K[X, Y_1]. 
    \end{align*}
    
\begin{figure}[H]
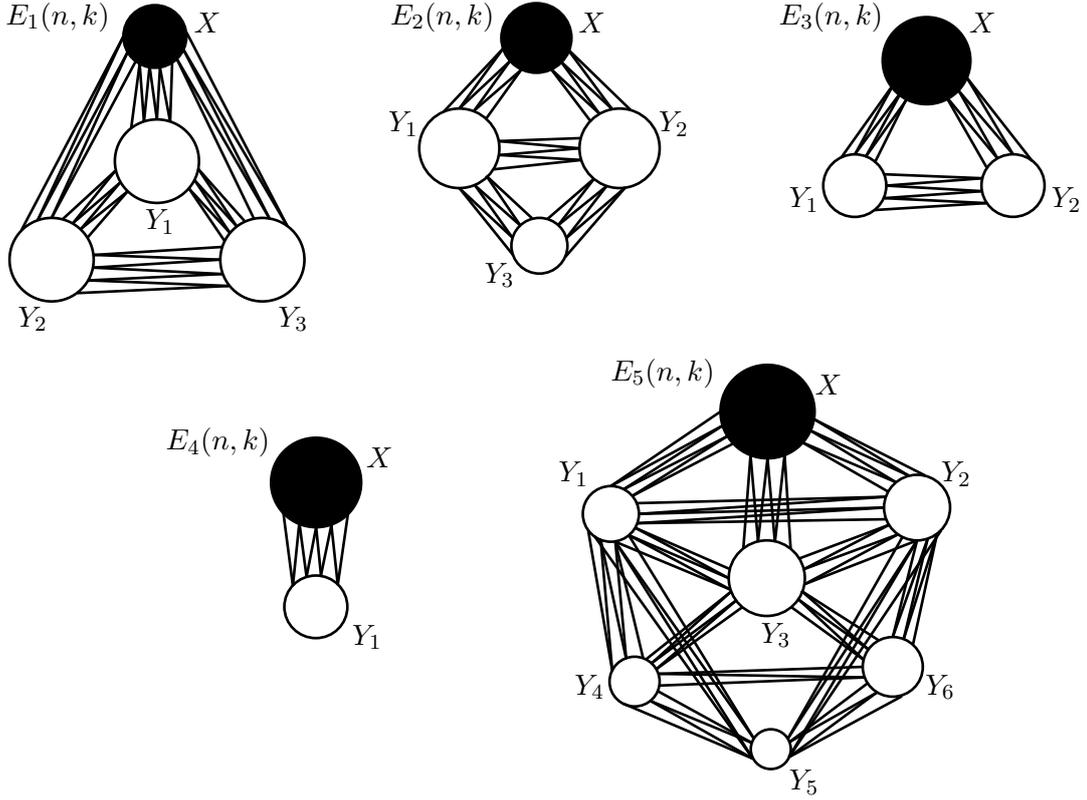

\centering
\tikzset{every picture/.style={line width=1pt}} 


\caption{Structures of $E_{1}(n,k), \ldots, E_{5}(n,k)$.}
\label{Fig:extremal-graphs}
\end{figure}
    
    %
    \item The vertex set of $E_{5}(n,k)$ consists of seven parts $X, Y_1, \ldots, Y_6$, with sizes satisfying 
    \begin{align*}
        & |X| = 12k-2n+9-j, \quad
        |Y_1| + |Y_2| + |Y_3| = 4k+3 - |X|,  \quad \\[0.5em]
        & |Y_4| + |Y_5| + |Y_6| = n-4k-3,  \quad |Y_1|+|Y_6|=n-4k-3+ \left\lfloor \frac{j}{3} \right\rfloor, \\[0.5em]
        &|Y_2|+|Y_4|=n-4k-3+ \left\lfloor \frac{j+1}{3} \right\rfloor,  \quad \text{and} \quad |Y_3|+|Y_5|=n-4k-3+ \left\lfloor \frac{j+2}{3} \right\rfloor,
    \end{align*}
    where $j \in \{0,1,2,3\}$. 
    The edge set of $E_{5}(n,k)$ is defined as
    \begin{align*}
        E_{5}(n,k)
        = \binom{X}{2} \cup K[X, Y_1\cup Y_2 \cup Y_3] \cup K[Y_1 \cup Y_6, Y_2 \cup Y_4, Y_3 \cup Y_5]. 
    \end{align*}
   %
   Note that $E_{5}(n,k)$ is defined only for $\frac{2n-9}{12}\le k\le\frac{n-3}{4}$ and $j\in\{0,1,2,3\}$ satisfying $j\le12k-2n+9$.
\end{itemize}

Simple calculations show that 
\begin{align*}
    |E_{1}(n,k)|
    & = \binom{k}{2}+ k(n-k) + \left\lfloor\frac{(n-k)^2}{3} \right\rfloor 
    \approx \frac{n^2}{3}+ \frac{k n}{3} -\frac{k^2}{6},  \\[0.5em]
    |E_{2}(n,k)|
    & = \binom{2k+1}{2} + \left\lfloor\frac{n^2}{3} \right\rfloor 
    \approx \frac{n^2}{3} + 2k^2, \\[0.5em]
    |E_{3}(n,k)|
    & =  \binom{2k+1}{2} + (2k+1)(n-2k-1) + \left\lfloor\frac{(n-2k-1)^2}{4} \right\rfloor 
    \approx \frac{n^2}{4} + k n- k^2,\\[0.5em]
    |E_{4}(n,k)|
    & = \binom{3k+2}{2}+ (3k+2)(n-3k-2) 
    \approx 3k n-\frac{9 k^2}{2}, \\[0.5em]
    |E_{5}(n,k)|
    & = \binom{12k-2n+9}{2}+(12k-2n+9)(2n-8k-6)+3\left( n-4k-3\right)^2 \\[0.5em]
    & \approx n^2-8 k n + 24 k^2.
\end{align*}
We conjecture that $E_{1}(n,k), \ldots, E_{5}(n,k)$ are the extremal constructions for $\mathrm{ex}(n,(k+1)K_4)$ when $n$ is large. The following table outlines the extremal range for each class of constructions.

\begin{table}[H]
\centering
\linespread{2} \selectfont
\begin{tabular}{c|l}
    \thickhline
    \textbf{Construction}  & \textbf{Extremal Range} \\
    \thickhline
    $E_1(n,k)$ & $I_1(n) \coloneqq \left[0, \frac{2n-9}{13} \right]$  \\
    \hline
    $E_2(n,k)$ & $I_2(n) \coloneqq  \left[\frac{2n-9}{13}, \frac{n-1}{6} \right]$ \\
    \hline
    $E_3(n,k)$ & $I_3(n) \coloneqq  \left[ \frac{n-1}{6}, \frac{4n-11-\sqrt{2n^2-4n-5}}{14} \right]$ \\
    \hline
    $E_4(n,k)$ & $I_4(n) \coloneqq \left[ \frac{4n-11-\sqrt{2n^2-4n-5}}{14}, \frac{22n-75+\sqrt{28n^2-108n+153}}{114} \right]$ \\
    \hline
    $E_5(n,k)$ & $I_5(n) \coloneqq  \left[ \frac{22n-75+\sqrt{28n^2-108n+153}}{114},\frac{n-3}{4} \right]$  \\
    \thickhline
\end{tabular}
\caption{Extremal range for each class of constructions.} \label{Tab:extremal-range}
\end{table}

\subsection{Setup}\label{SUBSEC:setup}
In this subsection, we present the notations and assumptions that will be used throughout the remainder of this paper, unless otherwise stated.

Let $n \ge 4k \ge 0$ be integers and assume that $n$ is sufficiently large. 
Let $G$ be an $n$-vertex graph with $\nu(K_4, G) = k$. 
A $4$-tuple of families $(\mathcal{A}, \mathcal{B}, \mathcal{C}, \mathcal{D})$ is a \textbf{rank-$4$-packing} of $G$ if 
\begin{enumerate}[label=(\roman*)]
    \item $\mathcal{A}$ is a $K_4$-tiling in $G$, 
    \item $\mathcal{B}$ is a $K_3$-tiling in $G$,
    \item $\mathcal{C}$ is a $K_2$-tiling in $G$,  
    \item $\mathcal{D}$ is a $K_1$-tiling in $G$ (i.e. a collection of vertices), and
    \item $V(\mathcal{A}) \cup V(\mathcal{B}) \cup V(\mathcal{C}) \cup V(\mathcal{D}) = V(G)$ is a partition of $V(G)$. 
\end{enumerate}
As suggested in~\cite{ABHP15}, for the remainder of the paper, we additionally assume that $(|\mathcal{A}|, |\mathcal{B}|, |\mathcal{C}|, |\mathcal{D}|)$ is \textbf{maximized in lexicographic order} among all rank-$4$-packings of $G$. 
Note that, based on the assumption on $G$, we have 
\begin{align}\label{equ:assumption-ABCD}
    |\mathcal{A}| = \nu(K_4, G) = k
    \quad\text{and}\quad 
    4 |\mathcal{A}| + 3|\mathcal{B}| + 2|\mathcal{C}|+|\mathcal{D}| = |V(G)| = n.
\end{align}

Below, we partition $\mathcal{A}$ further into six subfamilies. 

Let $A, B, C,D \subseteq  V(G)$ be four disjoint sets such that $G[A] \cong K_4$, $G[B] \cong K_3$, $G[C] \cong K_2$, and $G[D] \cong K_1$. We say 
\begin{itemize}
    \item $B$ \textbf{$3$-sees} $A$ if there exists a vertex $v\in A$ such that $G[B\cup \{v\}] \cong K_4$, 
    \item $B$ \textbf{$2$-sees} $A$ if there exist two vertices $v_1, v_2\in A$ and two vertices $u_1, u_2 \in B$ such that $G[\{v_1, v_2, u_1, u_2\}] \cong K_4$,
    \item $C$ \textbf{sees} $A$ if there exist two vertices $v_1, v_2\in A$ such that $G[C\cup \{v_1, v_2\}] \cong K_4$, 
    \item $D$ \textbf{sees} $A$ if there exist three vertices $v_1, v_2, v_3\in A$ such that $G[D\cup \{v_1, v_2,v_3\}] \cong K_4$,
\end{itemize}

\begin{figure}[H]
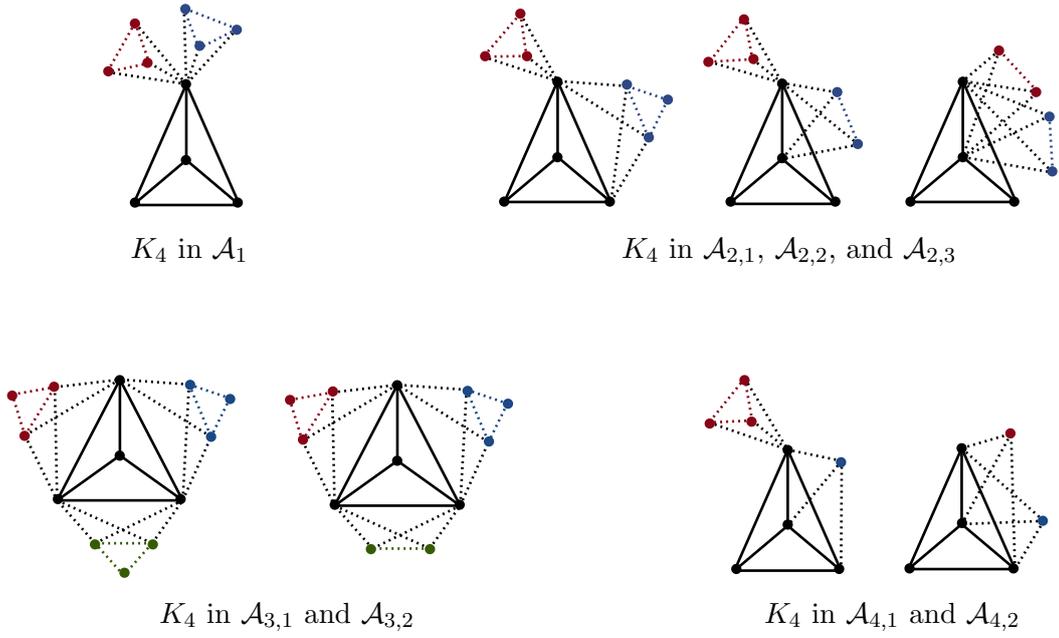

\centering

\tikzset{every picture/.style={line width=1pt}} 


\caption{Members in $\mathcal{A}_{1}, \ldots, \mathcal{A}_{4}$.}
\label{Fig:def-A1A2A3A4}
\end{figure}

The family $\mathcal{A}$ is partitioned into six subfamilies, $\mathcal{A}_1, \ldots, \mathcal{A}_6$,  as follows (see Figure~\ref{Fig:def-A1A2A3A4})$\colon$
\begin{itemize}
    \item $\mathcal{A}_1$ consists of members in $\mathcal{A}$ that are $3$-seen by at least two members in $\mathcal{B}$, 
    \item  $\mathcal{A}_2$ is the union of three families $\mathcal{A}_{2,1}$,  $\mathcal{A}_{2,2}$, and $\mathcal{A}_{2,3}$, where 
    \begin{itemize}
        \item $\mathcal{A}_{2,1}$ consists of members in $\mathcal{A}\setminus \mathcal{A}_1$ that are $3$-seen by one member in $\mathcal{B}$ and $2$-seen by at least one other member in $\mathcal{B}$,
        \item $\mathcal{A}_{2,2}$ consists of members in $\mathcal{A}\setminus \mathcal{A}_1$ that are $3$-seen by one member in $\mathcal{B}$ and seen by at least one member in $\mathcal{C}$,
        \item $\mathcal{A}_{2,3}$ consists of members in $\mathcal{A}\setminus \mathcal{A}_1$ that are seen by at least two members in $\mathcal{C}$, 
    \end{itemize}
    \item $\mathcal{A}_3$ is the union of two families $\mathcal{A}_{3,1}$ and $\mathcal{A}_{3,2}$, where
    \begin{itemize}
        \item $\mathcal{A}_{3,1}$ consists of members in $\mathcal{A}\setminus (\mathcal{A}_1 \cup \mathcal{A}_2)$ that are $2$-seen by at least three members in $\mathcal{B}$, 
        \item $\mathcal{A}_{3,2}$ consists of members in $\mathcal{A}\setminus (\mathcal{A}_1 \cup \mathcal{A}_2)$ that are $2$-seen by at least two members in $\mathcal{B}$ and seen by at least one member in $\mathcal{C}$, 
    \end{itemize}
    \item $\mathcal{A}_4$ is the union of two families $\mathcal{A}_{4,1}$ and $\mathcal{A}_{4,2}$, where
    \begin{itemize}
        \item $\mathcal{A}_{4,1}$ consists of members in $\mathcal{A}\setminus (\mathcal{A}_1 \cup \mathcal{A}_2 \cup \mathcal{A}_3)$ that are $3$-seen by one member in $\mathcal{B}$ and seen by at least one member in $\mathcal{D}$, 
        \item $\mathcal{A}_{4,2}$ consists of members in $\mathcal{A}\setminus (\mathcal{A}_1 \cup \mathcal{A}_2 \cup \mathcal{A}_3)$ that are seen by at least two members in $\mathcal{D}$, 
    \end{itemize}
    \item $\mathcal{A}_5$ and $\mathcal{A}_6$ are obtained through the following process$\colon$ 
    Initially, we set $\mathcal{A}_5 \coloneqq \emptyset$ and $\mathcal{A}_6 \coloneqq \mathcal{A}\setminus (\mathcal{A}_1 \cup \mathcal{A}_2 \cup \mathcal{A}_3 \cup \mathcal{A}_4)$. If there exists a member, say $Q$, in $\mathcal{A}_6$ that sends at most $15\left(|\mathcal{A}_{6}|-1\right)$ edges to other members in $\mathcal{A}_6$, then we update $\mathcal{A}_5$ and $\mathcal{A}_6$ by letting $\mathcal{A}_5 \coloneqq \mathcal{A}_5 \cup \{Q\}$ and $\mathcal{A}_6 \coloneqq \mathcal{A}_6 \setminus \{Q\}$. This process is repeated until every member in $\mathcal{A}_6$ sends more than $15\left(|\mathcal{A}_6|-1\right)$ edges to other members in $\mathcal{A}_6$.
\end{itemize}
Define 
\begin{align*}
    (a_1, \ldots, a_6, a, b, c, d) 
    \coloneqq \left(|\mathcal{A}_1|, \ldots, |\mathcal{A}_{6}|, |\mathcal{A}|,|\mathcal{B}|, |\mathcal{C}|,|\mathcal{D}|\right). 
\end{align*}
Note from~\eqref{equ:assumption-ABCD} that 
\begin{align*}
    a_1 + \cdots + a_6 = a = k
    \quad\text{and}\quad 
    4a + 3b + 2c + d = n.
\end{align*}
%

\begin{table}[H]
  \centering
    \linespread{1.4} \selectfont
    \begin{tabular}{c!{\vrule width 1.2pt}c|c|c|c|c|c|c|c|c}
    \thickhline
    \empty & $a_1$ & $a_2$ & $a_3$ & $a_4$ & $a_5$ & $a_6$ & $b$ & $c$ & $d$\\
    \thickhline
    $E_{1}(n,k)$ & $k$ & $0$ & $0$ & $0$ & $0$ & $0$ & $\frac{n -4 k}{3}$ & $0$ & $0$\\
    \hline
    $E_{2}(n,k)$ & $0$ & $k$ & $0$ & $0$ & $0$ & $0$ & $\frac{n - 6k}{3}$ & $k$ & $0$\\ 
    \hline
    $E_{3}(n,k)$ & $0$ & $k$ & $0$ & $0$ & $0$ & $0$ & $1$ & $\frac{n-4k-3}{2}$ & $0$\\
    \hline
    $E_{4}(n,k)$ & $0$ & $0$ & $0$ & $k$ & $0$ & $0$ & $1$ & $0$ & $n-4k-3$\\
    \hline
    $E_{5}(n,k)$ & $0$ & $0$ & $0$ & $0$ & $0$ & $k$ & $/$ & $/$ & $/$\\
    \thickhline
    \end{tabular}
    \caption{The value of $(a_1, \ldots, a_6, b,c,d)$ in the extremal constructions.} \label{Tab:extremal-vector}
\end{table}

\textbf{Outline of the proof}:
We begin by recalling the proof outline of Allen--B{\" o}ttcher--Hladk\'{y}--Piguet~\cite{ABHP15} for Theorem~\ref{THM:ABHP-density-CH}.
Their approach starts with taking a lexicographically maximized rank-$3$-packing $(\mathcal{T}, \mathcal{M}, \mathcal{J})$ of $G$, where $\mathcal{T}$ is a $K_3$-tiling, $\mathcal{M}$ is $K_2$-tiling, and $\mathcal{J}$ is a $K_1$-tiling.  The family $\mathcal{T}$ is further partitioned into four subfamilies $\mathcal{T}_1, \mathcal{T}_2, \mathcal{T}_3, \mathcal{T}_4$, based on the connections of its members to those in $\mathcal{M}$, $\mathcal{J}$, and other members in $\mathcal{T}$. 
Using the maximality of $(|\mathcal{T}|, |\mathcal{M}|, |\mathcal{J}|)$ and clever combinatorial arguments, they derived an upper bound $f(|\mathcal{T}_1|, \ldots, |\mathcal{T}_4|, |\mathcal{M}|, |\mathcal{J}|)$ for the number of edges in terms of six variables $(|\mathcal{T}_1|, \ldots, |\mathcal{T}_4|, |\mathcal{M}|, |\mathcal{J}|)$. 
Finally, they proved that the maximum value of $f(t_1, t_2, t_3, t_4, m,j)$ over the region 
\begin{align*}
    \left\{(t_1, t_2, t_3, t_4, m,j) \in \mathbb{N}^{6} \colon t_1+ t_2 + t_3 + t_4 = k ~\text{and}~ 3k+2m+j = n\right\}. 
\end{align*}
is given by the number of edges of members in the four classes of extremal constructions. 

Our proof builds upon the framework developed by Allen--B{\" o}ttcher--Hladk\'{y}--Piguet. However, as they noted ``there are some additional difficulties for $r\ge 4$ that do not appear for $r=3$". Indeed, our initial attempt to naively replicate their argument can only yield the desired upper bound (i.e. $\Xi(n,k)$, as defined in Theorem~\ref{THM:Mian-HS-density-K4}) on $\mathrm{ex}(n,(k+1)K_4)$ for $k \in \left[0, \frac{n}{8}\right]$. This limitation arises because the maximum value of the upper bound we obtained for $|G|$ (derived from the combination of bounds in Section~\ref{SEC:Local-estimate} and Lemma~\ref{LEMMA:A6BCD-upper-bound-a}) agree with $\Xi(n,k)$ only for $k \in \left[0, \frac{n}{8}\right]$. 

To address this issue, we introduce several key innovations:
We employ the more delicate partition $\mathcal{A} = \mathcal{A}_1 \cup \cdots \cup \mathcal{A}_6$ as defined earlier, where, notably (and somewhat unexpectedly), members of $\mathcal{A}_3$ do not appear in any extremal construction, yet their inclusion is essential for achieving the desired upper bound. 
The definition of $\mathcal{A}_3$ allows us to choose three disjoint edges from members of $\mathcal{B}\cup\mathcal{C}$ for the rotations in Section~\ref{SEC:A3}, giving the improved bound $e(\mathcal{A}_3)\le7a_3^2$.
Significant effort is also required to establish useful upper bounds for the number of edges in $\mathcal{A}_3$ (see Section~\ref{SEC:A3}), $\mathcal{A}_4 \cup \mathcal{B} \cup \mathcal{C}$ (see Section~\ref{SEC:A4BC}), and $\mathcal{A}_6 \cup \mathcal{B} \cup \mathcal{C} \cup \mathcal{D}$ (see Section~\ref{SEC:A6BCD}), respectively. 
In fact, we derive multiple upper bounds for these parts, each tailored to specific cases. 
Similarly, we employ several distinct $9$-variable functions to upper bound $|G|$ (see Section~\ref{SEC:Global-estimate}), allowing us to achieve the desired final bound in different cases.
Additionally, to make our calculations less complicated, we ignore the linear terms in the upper bounds we derived and focus only on the quadratic terms. This simplification is one of the reasons for the $O(n)$ error term in Theorem~\ref{THM:Mian-HS-density-K4}. 

Our approach for the case $r = 4$ suggests that, to determine $\mathrm{ex}(n,(k+1)K_r)$ for general $r \ge 5$, a carefully designed partition of the $K_{r}$-tiling is essential. 

\medskip

\textbf{Organization of the paper}: 
We present some preliminary results in the next section.
In Section~\ref{SEC:Local-estimate}, we provide basic upper bounds for the number of edges crossing the parts $\{\mathcal{A}_1, \ldots, \mathcal{A}_6, \mathcal{B}, \mathcal{C}, \mathcal{D}\}$ and inside each part. 
In Section~\ref{SEC:A3}, we present improved upper bounds for $e(\mathcal{A}_3)$. 
In Section~\ref{SEC:A4BC}, we present improved upper bounds for $e(\mathcal{A}_4) + e(\mathcal{A}_4, \mathcal{B} \cup \mathcal{C})$. 
In Section~\ref{SEC:A6BCD}, we present improved upper bounds for $e(\mathcal{A}_6 \cup \mathcal{B} \cup \mathcal{C} \cup \mathcal{D})$. 
In Section~\ref{SEC:Global-estimate}, we combine these upper bounds to prove Theorem~\ref{THM:Mian-HS-density-K4}, after solving certain quadratic (or convex) optimization problems (Sections~\ref{SUBSEC:inequality-Phi1},~\ref{SUBSEC:inequality-Phi2}, and~\ref{SUBSEC:inequality-Phi3}).
Section~\ref{SEC:Remarks} contains an application of Theorem~\ref{THM:Mian-HS-density-K4} and a conjecture for general $r \ge 5$.

\medskip 

\begin{figure}[H]
\centering
\tikzset{every picture/.style={line width=1pt}} 


\caption{Structure of the proof for Theorem~\ref{THM:Mian-HS-density-K4}. Here, $\alpha \coloneqq \frac{20+\sqrt{10}}{130} = 0.178171...$.} 
\label{Fig:MindMap}
\end{figure}

\section{Preliminaries}\label{SEC:Prelim}
For every vertex set $S \subseteq V(G)$, let $G[S]$ denote the \textbf{induced subgraph} of $G$ on $S$. The number of edges in $G[S]$ is denoted by $e_{G}(S)$ for simplicity. 
We use $G - S$ to denote the induced subgraph of $G$ on $V(G) \setminus S$. 
Given another vertex set $T \subseteq V(G)$ that is disjoint from $S$, let $G[S,T]$ denote the \textbf{induced bipartite subgraph} of $G$ on $S$ and $T$, i.e. $G[S,T]$ consists of all edges in $G$ that have nonempty intersection with both $S$ and $T$. The number of edges in $G[S,T]$ is denoted by $e_{G}(S,T)$ for simplicity. 
Let $\overline{G}$ denote the \textbf{complement} of $G$. 

Let $\mathcal{F}_1$ be an $F_1$-tiling in $G$ and $\mathcal{F}_2$ be an $F_2$-tiling in $G$, and assume that $V(\mathcal{F}_1) \cap V(\mathcal{F}_2) = \emptyset$. 
We use $e_{G}(\mathcal{F}_1, \mathcal{F}_2)$ to denote the number of edges in the induced bipartite subgraph $G[V(\mathcal{F}_1), V(\mathcal{F}_2)]$. 
In particular, if $F_1$ and $F_2$ are vertex-disjoint subgraphs of $G$, then $e_{G}(F_1, F_2)$ is the number of edges in the induced bipartite subgraph $G[V(F_1), V(F_2)]$. 
Similarly, we use $e_{G}(\mathcal{F}_1)$ and $e_{G}(F_1)$ to denote the number of edges in the induced subgraphs $G[V(\mathcal{F}_1)]$ and $G[V(F_1)]$, respectively. 
We will omit the subscript $G$ from the notations defined above if it is clear from the context. 

The following classical theorems by Erd\H{o}s--Gallai will be useful for us. 
\begin{theorem}[Erd\H{o}s--Gallai~\cite{EG59}]\label{THM:Erdos-Gallai-matching}
    Let $n \ge t \ge 1$ be integers satisfying $n \ge 2t+1$. 
    Suppose that $G$ is a graph on $n$ vertices with $\nu(G) \le t$. 
    Then 
    \begin{align*}
        |G|
        \le \max\left\{t(n-t) + \binom{t}{2},~ \binom{2t+1}{2}\right\}
        \le t n.
    \end{align*}
\end{theorem}

For every integer $t \ge 1$, let $P_t$ denote the path on $t$ vertices. 
\begin{theorem}[Erd\H{o}s--Gallai~\cite{EG59}]\label{THM:Erdos-Gallai-path-Turan}
    Let $n \ge t \ge 2$ be integers. Then 
    \begin{align*}
        \mathrm{ex}(n,P_t) \leq \frac{(t-2)n}{2}.
    \end{align*}
\end{theorem}

%
%
For every integer $t \ge 1$, let $L_{t}$ denote the graph obtained from the complete graph $K_t$ by removing a Hamiltonian cycle. 
Observe that $L_{t}$ is contained in $K_{\lceil t/2 \rceil}[2, 2, \ldots, 2]$ (the complete $\lceil t/2 \rceil$-partite graph with each part has exactly two vertices), and hence, the chromatic number $\chi(L_{t})$ of $L_t$ satisfies 
\begin{align*}
    \chi(L_{t})
    = t - \lfloor t/2 \rfloor
    = \lceil t/2 \rceil.
\end{align*}
Given two graphs $G_1$ and $G_2$, the \textbf{join} $G_1  \times G_2$ is the graph obtained from the vertex-disjoint union of $G_1$ and $G_2$ by adding all edges that have nonempty intersection with both $V(G_1)$ and $V(G_2)$, i.e.
\begin{align*}
    G_1  \times G_2
    \coloneqq G_1 \cup G_2 \cup \left\{\{u,v\} \colon (u,v) \in V(G_1) \times V(G_2) \right\}.
\end{align*}
An additional observation is that, for $t \ge 4$, the graph $L_t$ is contained in $P_4 \times K_{\lceil t/2 \rceil -2}[2, 2, \ldots, 2]$. 
So by a classical theorem of Simonovits~{\cite[Theorem~1]{Sim74}} (see also~{\cite[Theorem~1.2]{MQ20}}), we derive the following upper bound for $\mathrm{ex}(n,L_{t})$.  
\begin{theorem}[Simonovits~\cite{Sim74}]\label{THM:Simonovits74}
    Let $t \ge 5$ be an integer. Then 
    \begin{align*}
        \mathrm{ex}(n,L_{t}) 
        = \left(1 - \frac{1}{\lceil t/2 \rceil - 1}\right) \frac{n^2}{2} + O(n). 
    \end{align*}
\end{theorem}
The following property of $L_{t}$ will be used in Section~\ref{SEC:A4BC}.
\begin{fact}\label{FACT:Lt-4-vtx}
    Let $t \ge 5$ be an integer. The induced subgraph of $L_{t}$ on every subset of size $4$ contains a copy of $P_4$. 
\end{fact}
Given integers $m, n \ge 0$, the \textbf{bipartite Tur\'{a}n number} $\mathrm{ex}(m,n,F)$ of a graph $F$ is defined as the maximum number of edges in an $F$-free $m$ by $n$ bipartite graph.
We will use a theorem by Gy{\'a}rf{\'a}s--Rousseau--Schelp~\cite{GRS84} on $\mathrm{ex}(m,n,P_{t})$. Below, we present a weaker version of their result.
\begin{theorem}[Gy{\'a}rf{\'a}s--Rousseau--Schelp~\cite{GRS84}]\label{THM:GRS84-bipartite-path}
    Let $n, m \ge 1$ and $t \ge 3$ be integers.
    Then 
    \begin{align*}
        \mathrm{ex}(m,n,P_{t})
        \le \frac{t-2}{2} (m+n).
    \end{align*}
\end{theorem}
The \textbf{Zarankiewicz number} $Z(m,n,s,t)$ is the maximum number of edges in an $m$ by $n$ bipartite graph $G$ such that $G$ contains no copy of the complete bipartite graph $K_{s,t}$, where the $s$-vertex part of $K_{s,t}$ is contained in the $m$-vertex part of $G$ and the $t$-vertex part of $K_{s,t}$ is contained in the $n$-vertex part of $G$. 
We will use the following upper bound on $Z(m,n,a,b)$ by \v{C}ul\'{i}k~\cite{Cu56} (see also~{\cite[Theorem~3.28]{FS13}}). 
\begin{theorem}[\v{C}ul\'{i}k~\cite{Cu56}]\label{THM:Zarankiewicz-number}
    Suppose that $n,m \ge 0$ are integers satisfying $n \geq (b-1)\binom{m}{a}$. 
    Then 
    \begin{align*}
        Z(m,n,a,b) = (a-1)n + (b-1)\binom{m}{a}.
    \end{align*}
\end{theorem}

The following simple facts are useful for finding edges with certain properties between members of $\mathcal{A} \cup \mathcal{B} \cup \mathcal{C} \cup \mathcal{D}$. 
\begin{fact}\label{FACT:3-by-4-bipartite-graph}
    Let $H = H[V_1, V_2]$ be a bipartite graph. 
    \begin{enumerate}[label=(\roman*)]
        \item\label{FACT:3-by-4-bipartite-graph-1} Suppose that $(|V_1|, |V_2|) = (2,3)$ and $|H| \ge 5$. Then there exists a vertex in $V_1$ that is adjacent to all vertices in $V_2$. There also exists a pair of vertices in $V_2$ such that both vertices in $V_1$ are adjacent to this pair. 
        \item\label{FACT:3-by-4-bipartite-graph--2} Suppose that $(|V_1|, |V_2|) = (2,4)$ and $|H| \ge 5$. Then there exists a vertex in $V_1$ that is adjacent to at least three vertices in $V_2$. 
        \item\label{FACT:3-by-4-bipartite-graph--1} Suppose that $(|V_1|, |V_2|) = (2,4)$ and $|H| \ge 6$. Then there exists a pair of vertices in $V_2$ that are adjacent to both vertices in $V_1$. 
        \item\label{FACT:3-by-4-bipartite-graph-2} Suppose that $(|V_1|, |V_2|) = (3,3)$ and $|H| \ge 7$. Then there exists a vertex in $V_1$ that is adjacent to all vertices in $V_2$.
        \item\label{FACT:3-by-4-bipartite-graph-3} Suppose that $(|V_1|, |V_2|) = (4,3)$ and $|H| \ge 9$. Then there exists a vertex in $V_1$ that is adjacent to all vertices in $V_2$.
        \item\label{FACT:3-by-4-bipartite-graph-4} Suppose that $(|V_1|, |V_2|) = (4,3)$ and $|H| \ge 8$. Then there exist $\{u_1, u_2\} \subseteq  V_1$ and $\{v_1, v_2\} \subseteq  V_2$ such that $\{u_i, v_j\} \in H$ for every $(i,j) \in [2] \times [2]$. 
        \item\label{FACT:3-by-4-bipartite-graph-5} Suppose that $(|V_1|, |V_2|) = (4,3)$ and $|H| \ge 7$. Then there exists a vertex in $V_2$ that is adjacent to at least three vertices in $V_1$. 
    \end{enumerate}
\end{fact}
\section{Local estimation \RomanNumeralCaps{1}: basic upper bounds}\label{SEC:Local-estimate}
In this subsection, we establish basic upper bounds for the number of edges crossing the parts in $\{\mathcal{A}_1, \ldots, \mathcal{A}_6, \mathcal{B}, \mathcal{C}, \mathcal{D}\}$ and for the edges within each part. 

The following simple fact follows directly from the definitions of $\mathcal{A}, \mathcal{B}, \mathcal{C}, \mathcal{D}$ and the maximality of $(|\mathcal{A}|, |\mathcal{B}|, |\mathcal{C}|, |\mathcal{D}|)$.
\begin{fact}\label{FACT:edges-ABCD}
    The following statements hold. 
    \begin{enumerate}[label=(\roman*)]
        \item\label{FACT:edges-ABCD-1} $e(\mathcal{B}) \le 3b + 6 \binom{b}{2} =3b^2$,   
        \item\label{FACT:edges-ABCD-2} $e(\mathcal{B}, \mathcal{C}) \le 4bc$,  
        \item\label{FACT:edges-ABCD-3} $e(\mathcal{B}, \mathcal{D}) \le 2bd$,  
        \item\label{FACT:edges-ABCD-4} $e(\mathcal{C}) \le c^2$, 
        \item\label{FACT:edges-ABCD-5} $e(\mathcal{C}, \mathcal{D}) \le c d$, 
        \item\label{FACT:edges-ABCD-6} $e(\mathcal{D}) = 0$. 
    \end{enumerate}
    Consequently, $e(\mathcal{B} \cup \mathcal{C} \cup \mathcal{D}) \le \Psi(b,c,d)$, where
    \begin{align}\label{equ:def-Psi-bcd}
        \Psi(b,c,d) 
        \coloneqq  3b^2 + 4bc + 2bd + c^2 + cd. 
    \end{align}
\end{fact}
\subsection{Upper bounds for $e(\mathcal{A}_i, \mathcal{A}_j)$}
In this subsection, we establish an upper bound for $e(\mathcal{A}_i, \mathcal{A}_j)$ for each $(i,j) \in [6] \times [6]$. For convenience, we set $e(\mathcal{A}_i, \mathcal{A}_i) \coloneqq e(\mathcal{A}_i)$ for every $i \in [6]$. 

The following lemma provides an upper bound for $e(\mathcal{A}_1, \mathcal{A}_i)$ for every $i \in [6]$.
\begin{lemma}\label{LEMMA:Q1Qi}
    The following statements hold for every member $Q = \{p, q, s, t\} \in \mathcal{A}_1$. 
    \begin{enumerate}[label=(\roman*)]
        \item \label{LEMMA:Q1Qi-1} Let $i \in [4]$ and $Q_i\in\mathcal{A}_i\setminus\{Q\}$. Then $e(Q, Q_i) \le 15$. 
        \item\label{LEMMA:Q1Qi-2} Let $i \in [6]$ and $Q_i \in \mathcal{A}_i$. Suppose that $e(Q, Q_i) \ge 14$. Then $Q_i$ cannot be $3$-seen by any member of $\mathcal{B}$. In particular, $e(Q,Q_1) \le 13$ for every $Q_1 \in \mathcal{A}_1 \setminus \{Q\}$, and hence, 
        \begin{align*}
            e(\mathcal{A}_1) 
            \le 13\binom{a_1}{2} + 6a_1
            \le \frac{13 a_1^2}{2}. 
        \end{align*}
        \item\label{LEMMA:Q1Qi-3} Let $i \in [2,6]$ and $Q_i \in \mathcal{A}_i$. Suppose that $e(Q, Q_i) \ge 15$. Then $Q_i$ cannot be $2$-seen by any member of $\mathcal{B}$ and cannot be seen by any member of $\mathcal{C}$. 
        \item\label{LEMMA:Q1Qi-4}  Let $i \in [2,6]$ and $Q_i \in \mathcal{A}_i$. If $|\mathcal{A}_1| \ge 2$, then $e(\mathcal{A}_1, Q_i) \le 13 a_1$. Consequently, for every $i \in [2,6]$, 
        \begin{align*}
            e(\mathcal{A}_1, \mathcal{A}_i)
            & \le 13a_1 a_i + 3 a_i.  
        \end{align*}
    \end{enumerate}
\end{lemma}
\begin{proof}[Proof of Lemma~\ref{LEMMA:Q1Qi}]

\begin{figure}[H]
\centering

\tikzset{every picture/.style={line width=1pt}} 



\caption{Left: $Q$ is $3$-seen by a triangle (blue) in $\mathcal{B}$, and $Q_i$ is seen by a vertex (red) in $\mathcal{D}$. Right: after rotation, the number of vertex-disjoint copies of $K_4$ increases by one.} 
\label{Fig:Q1Ai-a}
\end{figure}

First, we prove Lemma~\ref{LEMMA:Q1Qi}~\ref{LEMMA:Q1Qi-1}. Fix $i\in[4]$ and $Q_i\in\mathcal{A}_i\setminus\{Q\}$. Suppose to the contrary that $e(Q,Q_i)=16$.
By the definition of $\mathcal{A}_i$, there exist a nonempty set $W$ contained in a member of $\mathcal{B}\cup\mathcal{C}\cup\mathcal{D}$ and a set $R\subseteq Q_i$ such that $|W|+|R|=4$ and $G[W\cup R]\cong K_4$.
By the definition of $\mathcal{A}_1$, $Q$ is $3$-seen by at least two members of $\mathcal{B}$.
Since $W$ meets at most one of these members, there exists a member $T\in\mathcal{B}$ disjoint from $W$ that $3$-sees $Q$.
Without loss of generality, we assume that $p$ is adjacent to all three vertices in $T$. 
Choose a vertex $x\in Q_i\setminus R$, which exists since $|R|\le3$.
Since $e(Q,Q_i)=16$, the vertex $x$ is adjacent to all three vertices in $Q\setminus\{p\}$. Therefore, the sets
\[
T\cup\{p\},\qquad
W\cup R,\qquad
(Q\setminus\{p\})\cup\{x\}
\]
induce three vertex-disjoint copies of $K_4$ in $G$.
Replacing $Q$ and $Q_i$ with these three copies of $K_4$ (see Figure~\ref{Fig:Q1Ai-a} for an example) would increase the size of $\mathcal{A}$, contradicting the maximality of $(|\mathcal{A}|,|\mathcal{B}|,|\mathcal{C}|,|\mathcal{D}|)$.

\begin{figure}[H]
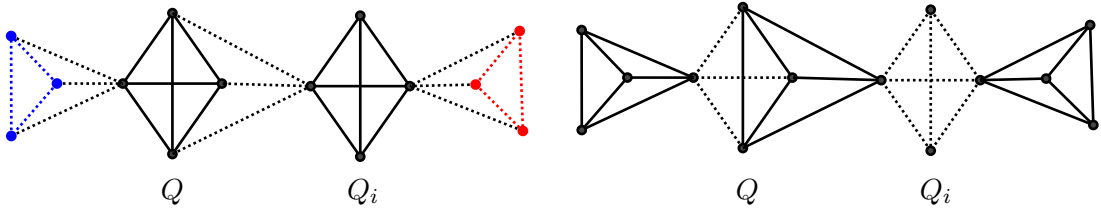

\centering

\tikzset{every picture/.style={line width=1pt}} 


\caption{Left: $Q$ is $3$-seen by a triangle (blue) in $\mathcal{B}$, and $Q_i$ is $3$-seen by a different triangle (red) in $\mathcal{B}$. Right: after rotation, the number of vertex-disjoint copies of $K_4$  increases one.} 
\label{Fig:Q1Qi-b}
\end{figure}
Next, we prove Lemma~\ref{LEMMA:Q1Qi}~\ref{LEMMA:Q1Qi-2}. 
Fix $i \in [6]$ and $Q_i = \{p_i, q_i, s_i, t_i\} \in \mathcal{A}_i$, assuming that $e(Q, Q_i) \ge 14$. Suppose to the contrary that $Q_i$ is $3$-seen by a member $T_i \in \mathcal{B}$. By symmetry, we may assume that $p_i$ is adjacent to all three vertices in $T_i$. 
By the definition of $\mathcal{A}_1$, there exists a member $T \in \mathcal{B}\setminus \{T_i\}$ that $3$-sees $Q$, and by symmetry, we may assume that $p$ is adjacent to all three vertices in $T$. 
Since $e(Q, Q_i) \ge 14 = 16-2$, there exists a vertex in $\{q_i, s_i, t_i\}$ that is adjacent to all vertices in $\{q, s, t\}$. However, the rotation shown in Figure~\ref{Fig:Q1Qi-b} would increase the size of $|\mathcal{A}|$, contradicting the maximality of $(|\mathcal{A}|, |\mathcal{B}|, |\mathcal{C}|, |\mathcal{D}|)$.

Next, we prove Lemma~\ref{LEMMA:Q1Qi}~\ref{LEMMA:Q1Qi-3}. 
Fix $i \in [2,6]$ and $Q_i = \{p_i, q_i, s_i, t_i\} \in \mathcal{A}_i$, assuming that $e(Q, Q_i) \ge  15$. Suppose to the contrary that $Q_i$ is seen by a member $\{u,v\} \in \mathcal{C}$ (the proof for the case where $Q_i$ is $2$-seen by a member in $\mathcal{B}$ is similar, so we omit it here). By symmetry, we may assume that $p_i$ and $q_i$ are adjacent to both vertices in $\{u,v\}$. 
By the definition of $\mathcal{A}_1$, there exists a member $T \in \mathcal{B}$ that $3$-sees $Q$, and by symmetry, we may assume that $p$ is adjacent to all three vertices in $T$. Since $e(Q, Q_i) \ge  15 = 16-1$, there exists a vertex in $\{s_i, t_i\}$ that is adjacent to all vertices in $\{q, s, t\}$. However, the rotation shown in Figure~\ref{Fig:Q1Qi-c} would increase the size of $|\mathcal{A}|$, contradicting the maximality of $(|\mathcal{A}|, |\mathcal{B}|, |\mathcal{C}|, |\mathcal{D}|)$.

\begin{figure}[H]
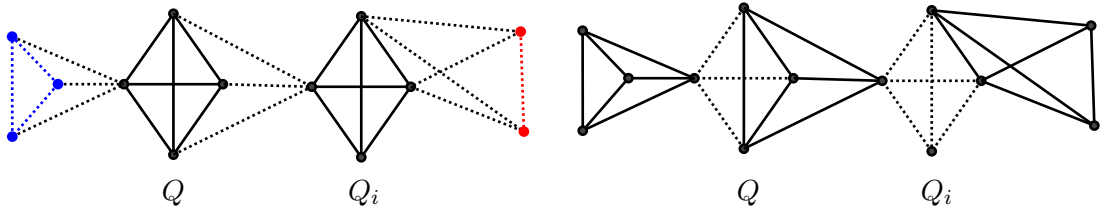

\centering

\tikzset{every picture/.style={line width=1pt}} 



\caption{Left: $Q$ is $3$-seen by a triangle (blue) in $\mathcal{B}$, and $Q_i$ is seen by an edge (red) in $\mathcal{C}$. Right: after rotation, the number of vertex-disjoint copies of $K_4$  increases one.} 
\label{Fig:Q1Qi-c}
\end{figure}

Next, we prove Lemma~\ref{LEMMA:Q1Qi}~\ref{LEMMA:Q1Qi-4}. Fix $i \in [2,6]$ and $Q_i = \{p_i, q_i, s_i, t_i\} \in \mathcal{A}_i$. 
Suppose to the contrary that $e(\mathcal{A}_1, Q_i) \ge 13a_1 + 1$. 
Then it follows from the Pigeonhole Principle and the assumption $|\mathcal{A}_1| \ge 2$ that there exist two distinct $Q_1 = \{p_1, q_1, s_1, t_1\} \in \mathcal{A}_1$ and $Q_1' = \{p_1', q_1', s_1', t_1'\} \in \mathcal{A}_1$ such that 
\begin{align*}
    e(Q_1, Q_i) + e(Q_1', Q_i) 
    \ge 27. 
\end{align*}
By the definition of $\mathcal{A}_1$, there exist two distinct members $T_1, T_1' \in \mathcal{B}$ such that, by symmetry, we may assume $p_1$ is adjacent to all vertices in $T_1$ and $p_1'$ is adjacent to all vertices in $T_1'$. 

\textbf{Case 1}$\colon$ $e(Q_1, Q_i) \ge 14$ and $e(Q_1', Q_i) \ge 13$. 

Note that the bipartite graph $G[Q_i, \{q_1', s_1', t_1'\}]$ satisfies 
\begin{align*}
    |G[Q_i, \{q_1', s_1', t_1'\}]|
    \ge 13 - 4 
    = 9.
\end{align*}
So it follows from Fact~\ref{FACT:3-by-4-bipartite-graph}~\ref{FACT:3-by-4-bipartite-graph-3} that there exists a vertex in $Q_i$ that is adjacent to all vertices in $\{q_1', s_1', t_1'\}$, and by symmetry, we may assume that $p_i$ is such a vertex in $Q_i$. 

Note that the bipartite graph $G[\{q_i, s_i, t_i\}, \{q_1, s_1, t_1\}]$ satisfies 
\begin{align*}
    |G[\{q_i, s_i, t_i\}, \{q_1, s_1, t_1\}]|
    \ge 14 - (1+3+3)
    = 7. 
\end{align*}
So it follows from Fact~\ref{FACT:3-by-4-bipartite-graph}~\ref{FACT:3-by-4-bipartite-graph-2} that there exists a vertex in $\{q_i, s_i, t_i\}$ that is adjacent to all vertices in $\{q_1, s_1, t_1\}$. However, the rotation shown in Figure~\ref{Fig:Q1Qi-d} would increase the size of $|\mathcal{A}|$, contradicting the maximality of $(|\mathcal{A}|, |\mathcal{B}|, |\mathcal{C}|, |\mathcal{D}|)$. 

\medskip 

\begin{figure}[H]
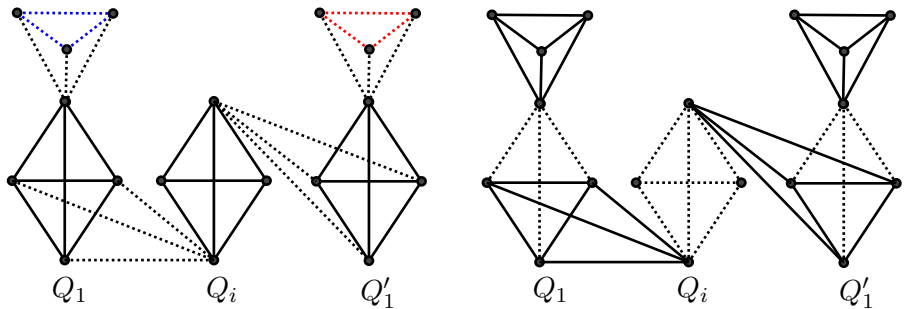

\centering

\tikzset{every picture/.style={line width=1pt}} 



\caption{Left: $Q_1$ is $3$-seen by a triangle (blue) in $\mathcal{B}$, and $Q_1'$ is $3$-seen by another triangle (red) in $\mathcal{B}$. Right: after rotation, the number of vertex-disjoint copies of $K_4$ increases by one.} 
\label{Fig:Q1Qi-d}
\end{figure}

\textbf{Case 2}$\colon$ $e(Q_1, Q_i) \ge 15$ and $e(Q_1', Q_i) \ge 12$.

Note that the bipartite graph $G[Q_i, \{q_1', s_1', t_1'\}]$ satisfies 
\begin{align*}
    |G[Q_i, \{q_1', s_1', t_1'\}]|
    \ge 12 - 4 
    = 8. 
\end{align*}
So it follows from Fact~\ref{FACT:3-by-4-bipartite-graph}~\ref{FACT:3-by-4-bipartite-graph-4} that there exist two vertices in $Q_i$ that are both adjacent to a certain pair of vertices in $\{q_1', s_1', t_1'\}$, and by symmetry, we may assume that $\{p_i, q_i\}$ are the two such vertices in $Q_i$. 

Note that the bipartite graph $G[\{s_i, t_i\}, \{q_1, s_1, t_1\}]$ satisfies 
\begin{align*}
    |G[\{s_i, t_i\}, \{q_1, s_1, t_1\}]|
    \ge 15 - (4+6)
    = 5. 
\end{align*}
So it follows from Fact~\ref{FACT:3-by-4-bipartite-graph}~\ref{FACT:3-by-4-bipartite-graph-1} that there exists a vertex in $\{s_i, t_i\}$ that is adjacent to all vertices in $\{q_1, s_1, t_1\}$. However, the rotation shown in Figure~\ref{Fig:Q1Qi-e} would increase the size of $|\mathcal{A}|$, contradicting the maximality of $(|\mathcal{A}|, |\mathcal{B}|, |\mathcal{C}|, |\mathcal{D}|)$.

\medskip 

\begin{figure}[H]
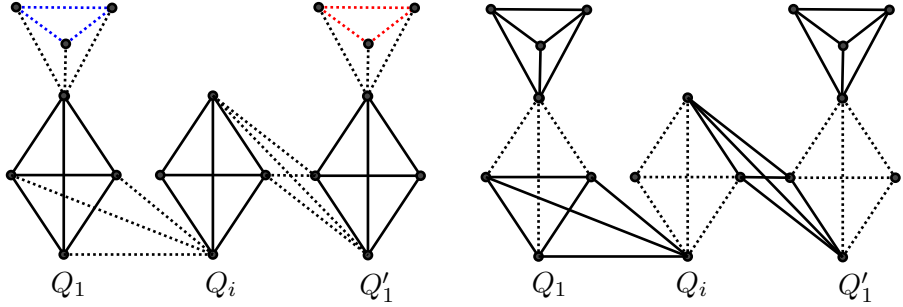

\centering

\tikzset{every picture/.style={line width=1pt}} 



\caption{Left: $Q_1$ is $3$-seen by a triangle (blue) in $\mathcal{B}$, and $Q_1'$ is $3$-seen by another triangle (red) in $\mathcal{B}$. Right: after rotation, the number of vertex-disjoint copies of $K_4$ increases by one.} 
\label{Fig:Q1Qi-e}
\end{figure}

\textbf{Case 3}$\colon$ $e(Q_1, Q_i) = 16$ and $e(Q_1', Q_i) \ge 11$.

Note that the bipartite graph $G[Q_i, \{q_1', s_1', t_1'\}]$ satisfies 
\begin{align*}
    |G[Q_i, \{q_1', s_1', t_1'\}]|
    \ge 11 - 4 
    = 7.
\end{align*}
So it follows from Fact~\ref{FACT:3-by-4-bipartite-graph}~\ref{FACT:3-by-4-bipartite-graph-5} that there exists a vertex in $\{q_1', s_1', t_1'\}$ that is adjacent to three vertices in $Q_i$, and by symmetry, we may assume that $\{p_i, q_i, s_i\}$ are the three such vertices in $Q_i$. 
Note that the vertex $t_i$ is adjacent to all vertices in $\{q_1, s_1, t_1\}$. However, the rotation shown in Figure~\ref{Fig:Q1Qi-f} would increase the size of $|\mathcal{A}|$, contradicting the maximality of $(|\mathcal{A}|, |\mathcal{B}|, |\mathcal{C}|, |\mathcal{D}|)$.
This proves Lemma~\ref{LEMMA:Q1Qi}~\ref{LEMMA:Q1Qi-4}. 
\medskip 

\begin{figure}[H]
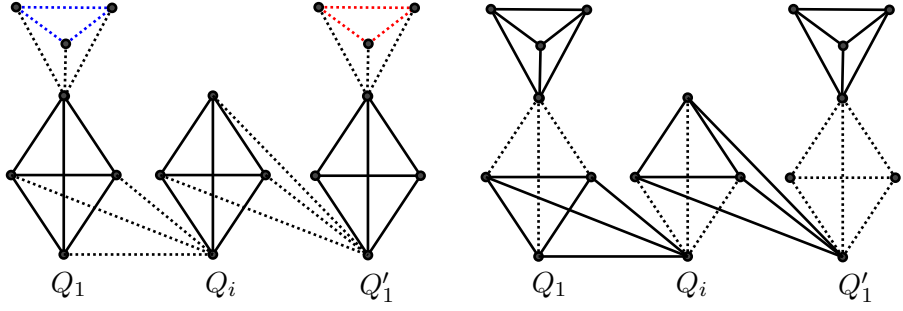

\centering

\tikzset{every picture/.style={line width=1pt}} 



\caption{Left: $Q_1$ is $3$-seen by a triangle (blue) in $\mathcal{B}$, and $Q_1'$ is $3$-seen by another triangle (red) in $\mathcal{B}$. Right: after rotation, the number of vertex-disjoint copies of $K_4$ increases by one.} 
\label{Fig:Q1Qi-f}
\end{figure}

It follows from Lemma~\ref{LEMMA:Q1Qi}~\ref{LEMMA:Q1Qi-4} that, for every $i \in [2,6]$,
\begin{align*}
    e(\mathcal{A}_1, \mathcal{A}_i)
    & \le 
    \begin{cases}
        16 a_i, &\quad\text{if}\quad a_1 \le 1, \\[0.5em]
        13a_1 a_i, &\quad\text{if}\quad a_1 \ge 2, 
    \end{cases}\\[0.5em]
    & \le 13a_1 a_i + 3 a_i.  
\end{align*}
This completes the proof of Lemma~\ref{LEMMA:Q1Qi}.
\end{proof}
The following lemma provides an upper bound for $e(\mathcal{A}_2, \mathcal{A}_i)$ for every $i \in [2, 6]$.
\begin{lemma}\label{LEMMA:Q2Qi}
    The following statements hold for every $Q = \{p,q,s,t\} \in \mathcal{A}_2$. 
    \begin{enumerate}[label=(\roman*)]
        \item\label{LEMMA:Q2Qi-1}  Let $i \in \{3,4\}$ and $Q_i \in \mathcal{A}_i$. Then $e(Q, Q_i) \le 15$. 
        \item\label{LEMMA:Q2Qi-2} Let $i \in [2,6]$ and $Q_i \in \mathcal{A}_i$. Suppose that $e(Q, Q_i) \ge 15$. 
        Then $Q_i$ cannot be $3$-seen by any member of $\mathcal{B}$, and cannot be seen by any member of $\mathcal{C}$. 
        In particular, $e(Q,Q_2) \le 14$ for every $Q_2 \in \mathcal{A}_2 \setminus \{Q\}$, and hence, 
        \begin{align*}
            e(\mathcal{A}_2) 
            \le 14\binom{a_2}{2} + 6a_2
            \le 7 a_2^2. 
        \end{align*}
        \item\label{LEMMA:Q2Qi-3} Let $i \in \{5,6\}$ and $Q_i \in \mathcal{A}_i$. Suppose that $e(Q, Q_i) = 16$. Then $Q_i$ cannot be $2$-seen by any member of $\mathcal{B}$. 
        \item\label{LEMMA:Q2Qi-4} Let $i \in [3,6]$ and $Q_i \in \mathcal{A}_i$.  If $|\mathcal{A}_2| \ge 2$, then $e(\mathcal{A}_2, Q_i) \le 14 a_2$. Consequently, for every $i \in [3,6]$, 
        \begin{align*}
            e(\mathcal{A}_2, \mathcal{A}_i)
            & \le 14a_2 a_i + 2a_i. 
        \end{align*}
    \end{enumerate}
\end{lemma}
\begin{proof}[Proof of Lemma~\ref{LEMMA:Q2Qi}]
The proof of Lemma~\ref{LEMMA:Q2Qi}~\ref{LEMMA:Q2Qi-1} is similar to that of Lemma~\ref{LEMMA:Q1Qi}~\ref{LEMMA:Q1Qi-1} (see e.g. Figure~\ref{Fig:Q1Ai-a} for the case $i = 4$), so we omit it here. 

\medskip 

\begin{figure}[H]
\centering

\tikzset{every picture/.style={line width=1pt}} 



\caption{Left: $Q$ is seen by an edge (blue) in $\mathcal{C}$, and $Q_i$ is $3$-seen by a triangle (red) in $\mathcal{B}$. Right: after rotation, the number of vertex-disjoint copies of $K_4$ increases one.} 
\label{Fig:Q2Qi-a}
\end{figure}

Next, we prove Lemma~\ref{LEMMA:Q2Qi}~\ref{LEMMA:Q2Qi-2}.  Fix $i \in [2,6]$ and $Q_i = \{p_i, q_i, s_i, t_i\} \in \mathcal{A}_i$, assuming that $e(Q, Q_i) \ge 15$.
Suppose first that $Q_i$ is $3$-seen by a member $T_i\in\mathcal{B}$. By symmetry, we may assume that $p_i$ is adjacent to all three vertices in $T_i$.
By the definition of $\mathcal{A}_2$, there exists a copy of $K_4$ using at most two vertices of $Q$ and vertices from a member of $(\mathcal{B}\cup\mathcal{C})\setminus\{T_i\}$.
After selecting this copy and \(T_i \cup \{p_i\}\), there are exactly three unused vertices in \(Q_i\) and at least two unused vertices in \(Q\). 
Since $e(Q,Q_i)\ge15$, there is at most one missing edge between these unused vertices.
The remaining vertices therefore contain a copy of $K_4$ disjoint from the two previously constructed copies.
This increases the size of $\mathcal{A}$ (see Figure~\ref{Fig:Q2Qi-a} for an example), contradicting the maximality of $(|\mathcal{A}|,|\mathcal{B}|,|\mathcal{C}|,|\mathcal{D}|)$.
If $Q$ is $3$-seen by a member of $\mathcal{B}$ and $Q_i$ is seen by a member of $\mathcal{C}$, then the same argument, with $Q$ and $Q_i$ interchanged, gives a contradiction.

\medskip

\begin{figure}[H]
\centering

\tikzset{every picture/.style={line width=1pt}} 



\caption{Left: $Q$ is $2$-seen by a different edge (blue) in $\mathcal{C}$, and $Q_i$ is seen by a different edge (red) in $\mathcal{C}$. Right: after rotation, the number of vertex-disjoint copies of $K_4$ remains the same, while the number of vertex-disjoint copies of $K_3$ increases one.} 
\label{Fig:Q2Qi-b2}
\end{figure}

Therefore, to show that $Q_i$ cannot be seen by a member of $\mathcal{C}$, we only need to assume that $Q$ is not $3$-seen by any member of $\mathcal{B}$.
In particular, we may assume that $Q\in\mathcal{A}_{2,3}$.
Suppose to the contrary that $Q_i$ is seen by a member $\{u_i,v_i\} \in \mathcal{C}$, and by symmetry, we may assume that $u_i$ and $v_i$ are both adjacent to $p_i$ and $q_i$. By the definition of $\mathcal{A}_{2,3}$, there exists a member $\{u, v\} \in \mathcal{C} \setminus \{\{u_i, v_i\}\}$ such that $u$ and $v$ are both adjacent to the same pair of vertices in $Q$, and by symmetry, we may assume that $p$ and $q$ are two such vertices in $Q$. Since 
\begin{align*}
    |G[\{s, t\}, \{s_i, t_i\}]|
    \ge 15 - (8 + 4) = 3, 
\end{align*}
there exists a vertex in $\{s_i, t_i\}$ that is adjacent to both $s$ and $t$. However, the rotation shown in Figure~\ref{Fig:Q2Qi-b2} would increase the value of $|\mathcal{B}|$ (while keeping $|\mathcal{A}|$ unchanged), contradicting the maximality of $(|\mathcal{A}|, |\mathcal{B}|, |\mathcal{C}|, |\mathcal{D}|)$.

For Lemma~\ref{LEMMA:Q2Qi}~\ref{LEMMA:Q2Qi-3}, if $Q_i$ were $2$-seen by a member of $\mathcal{B}$, choose two disjoint copies of $K_4$ as above, using vertices of distinct members of $\mathcal{B}\cup\mathcal{C}$.
Since $e(Q,Q_i)=16$, the two remaining vertices of $Q$, together with the two remaining vertices of $Q_i$, form a copy of $K_4$, contradicting the maximality of $|\mathcal{A}|$ (see Figure~\ref{Fig:Q2Qi-b3}).

\medskip

\begin{figure}[H]
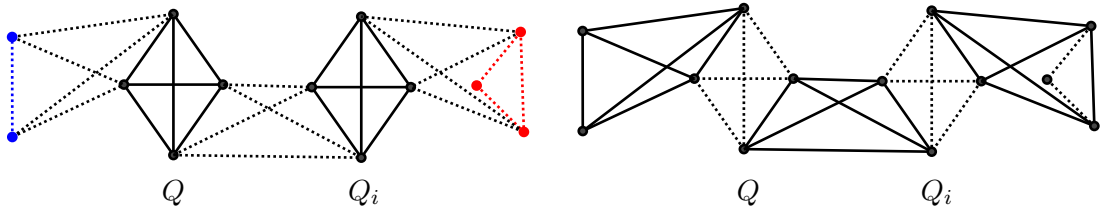

\centering

\tikzset{every picture/.style={line width=1pt}} 



\caption{Left: $Q$ is seen by an edge (blue) in $\mathcal{C}$, and $Q_i$ is $2$-seen by a triangle (red) in $\mathcal{B}$. Right: after rotation, the number of vertex-disjoint copies of $K_4$ increases one.} 
\label{Fig:Q2Qi-b3}
\end{figure}

It remains to show Lemma~\ref{LEMMA:Q2Qi}~\ref{LEMMA:Q2Qi-4}. 
Fix $i \in [3,6]$ and $Q_i = \{p_i, q_i, s_i, t_i\} \in \mathcal{A}_i$.  Suppose to the contrary that $e(\mathcal{A}_2, Q_i) \ge 14 a_2 + 1$.
It follows from the Pigeonhole Principle and the assumption $|\mathcal{A}_2| \ge 2$ that there exist two distinct members $Q_2 = \{p_2, q_2, s_2, t_2\} \in \mathcal{A}_2$ and $Q_2' = \{p_2', q_2', s_2', t_2'\} \in \mathcal{A}_2$ such that 
\begin{align*}
    e(Q_2, Q_i) + e(Q_2', Q_i) 
    \ge 29. 
\end{align*}
Suppose that at least one of $Q_2,Q'_2$ belongs to $\mathcal{A}_{2,1}\cup\mathcal{A}_{2,2}$. By the definition of $\mathcal{A}_2$, we can choose two disjoint copies of $K_4$ using distinct members of $\mathcal{B}\cup\mathcal{C}$, one of which uses a triangle in $\mathcal{B}$.
Without loss of generality, leaving unused subsets $X\subseteq Q_2$ and $Y\subseteq Q_2'$ with $|X|=3$ and $|Y|=2$.
Their common neighborhoods $U,W\subseteq Q_i$ satisfy $|U|+|W|\ge5$,
since $e(Q_2, Q_i) + e(Q_2', Q_i) \ge 29$.
If $|W|\ge2$, choose distinct vertices $x\in U$ and $y,z\in W$.
Then $X\cup\{x\}$ and $Y\cup\{y,z\}$ induce two vertex-disjoint copies of $K_4$.
Otherwise, $|W|\le1$ gives $e(Y,Q_i)\le5$, while
$e(Q_2,Q_i)+e(Q_2',Q_i)\ge29$ forces $e(Y,Q_i)\ge5$.
Thus, $e(Y,Q_i)=5$ and $e(X,Q_i)=12$.
Choose $u\in Y$ with three neighbors $Z\subseteq Q_i$.
Then $\{u\}\cup Z$ and $X\cup(Q_i\setminus Z)$ induce two vertex-disjoint copies of $K_4$.
These give four disjoint $K_4$'s replacing $Q_2,Q_2',Q_i$, a contradiction.
Thus, we may assume that both $Q_2$ and $Q'_2$ belong to $\mathcal{A}_{2,3}$. 
By the definition of $\mathcal{A}_{2,3}$, there exist distinct members $\{u_2, v_2\}, \{u_2',v_2'\} \in \mathcal{C}$ that see $Q_2$ and $Q_2'$, respectively. By symmetry, we may assume that $u_2$ and $v_2$ are both adjacent to $\{p_2, q_2\}$, and $u_2'$ and $v_2'$ are both adjacent to $\{p_2', q_2'\}$. 

\medskip

\begin{figure}[H]
\centering

\tikzset{every picture/.style={line width=1pt}} 



\caption{Left: $Q_2$ is seen by an edge (blue) in $\mathcal{C}$, and $Q_2'$ is seen by another edge (red) in $\mathcal{C}$. Right: after rotation, the number of vertex-disjoint copies of $K_4$ remains the same, while the number of vertex-disjoint copies of $K_3$ increases by one.} 
\label{Fig:Q2Qi-c}
\end{figure}

\textbf{Case 1}$\colon$ $e(Q_2, Q_i) \ge 15$ and $e(Q_2', Q_i) \ge 14$. 

Note that 
\begin{align*}
    |G[\{s_2', t_2'\}, Q_i]| 
    \ge 14 - 2\cdot 4
    = 6. 
\end{align*}
So it follows from Fact~\ref{FACT:3-by-4-bipartite-graph}~\ref{FACT:3-by-4-bipartite-graph--1} that there exists a pair of vertices in $Q_i$ that are adjacent to both $s_2'$ and $t_2'$, and by symmetry, we may assume that $\{p_i, q_i\}$ are two such vertices in $Q_i$. 
Since 
\begin{align*}
    |G[\{s_2, t_2\}, \{s_i, t_i\}]|
    \ge 15 - (8 + 4)
    = 3, 
\end{align*}
there exists a vertex in $\{s_i, t_i\}$ that is adjacent to both $s_2$ and $t_2$. However, the rotation shown in Figure~\ref{Fig:Q2Qi-c} would increase the value of $|\mathcal{B}|$ (while keeping $|\mathcal{A}|$ unchanged), contradicting the maximality of $(|\mathcal{A}|, |\mathcal{B}|, |\mathcal{C}|, |\mathcal{D}|)$.

\textbf{Case 2}$\colon$ $e(Q_2, Q_i) = 16$ and $e(Q_2', Q_i) \ge 13$.

Note that 
\begin{align*}
    |G[\{s_2', t_2'\}, Q_i]| 
    \ge 13 - 2 \cdot 4
    = 5.
\end{align*}
So it follows from Fact~\ref{FACT:3-by-4-bipartite-graph}~\ref{FACT:3-by-4-bipartite-graph--2} that there exists a vertex in $Q_2'$ that is adjacent to at least three vertices in $Q_i$, and by symmetry, we may assume that $\{p_i, q_i, s_i\}$ are three such vertices in $Q_i$. Since $e(Q_2, Q_i) = 16$, $t_i$ is adjacent to both $s_2$ and $t_2$. However, the rotation shown in Figure~\ref{Fig:Q2Qi-d} would increase the value of $|\mathcal{B}|$ (while keeping $|\mathcal{A}|$ unchanged), contradicting the maximality of $(|\mathcal{A}|, |\mathcal{B}|, |\mathcal{C}|, |\mathcal{D}|)$. 
This proves Lemma~\ref{LEMMA:Q2Qi}~\ref{LEMMA:Q2Qi-4}.

\medskip 

\begin{figure}[H]
\centering

\tikzset{every picture/.style={line width=1pt}} 



\caption{Left: $Q_2$ is seen by an edge (blue) in $\mathcal{C}$, and $Q_2'$ is seen by another edge (red) in $\mathcal{C}$. Right: after rotation, the number of vertex-disjoint copies of $K_4$ remains the same, while the number of vertex-disjoint copies of $K_3$ increases by one.} 
\label{Fig:Q2Qi-d}
\end{figure}

It follows from Lemma~\ref{LEMMA:Q2Qi}~\ref{LEMMA:Q2Qi-4} that 
\begin{align*}
    e(\mathcal{A}_2, \mathcal{A}_i)
    & \le 
    \begin{cases}
        16 a_i, &\quad\text{if}\quad a_2 \le 1, \\[0.5em]
        14a_2 a_i, &\quad\text{if}\quad a_2 \ge 2, 
    \end{cases} \\[0.5em]
    & \le 14a_2 a_i + 2a_i. 
\end{align*}
This completes the proof of Lemma~\ref{LEMMA:Q2Qi}. 
\end{proof}
The following lemma provides an upper bound for $e(\mathcal{A}_3, \mathcal{A}_i)$ for $i \in [3,6]$.
\begin{lemma}\label{LEMMA:Q3Qi}
    The following statements hold for every $Q = \{p,q,s,t\} \in \mathcal{A}_3$.
    \begin{enumerate}[label=(\roman*)]
        \item\label{LEMMA:Q3Qi-1} 
        Let $i\in[3,6]$ and $Q_i\in\mathcal{A}_i\setminus\{Q\}$.
        Suppose that $e(Q, Q_i) = 16$. Then $Q_i$ cannot be $3$-seen or $2$-seen by any member of $\mathcal{B}$, and cannot be seen by any member of $\mathcal{C}$. In particular, $e(Q, Q_3) \le 15$ for every $Q_3 \in \mathcal{A}_{3}\setminus \{Q\}$, and hence, 
        \begin{align*}
            e(\mathcal{A}_3)
            \le 15 \binom{a_3}{2} + 6a_3
            \le \frac{15 a_3^2}{2}. 
        \end{align*}
        \item\label{LEMMA:Q3Qi-2} Let $i\in\{4,5,6\}$ and $Q_i\in\mathcal{A}_i$. Suppose that $|\mathcal{A}_3| \ge 2$. Then $e(\mathcal{A}_3, Q_i) \le 15 a_3$ for every $i \in \{4,5,6\}$. Consequently, for every $i \in \{4,5,6\}$, 
        \begin{align*}
            e(\mathcal{A}_3, \mathcal{A}_i) 
            & \le 15 a_3 a_i + a_i. 
        \end{align*}
    \end{enumerate}
\end{lemma}
\begin{proof}[Proof of Lemma~\ref{LEMMA:Q3Qi}]
    %
    By the definition of $\mathcal{A}_3$, we can choose the required edges from distinct members of $\mathcal{B}\cup\mathcal{C}$.
    If \ref{LEMMA:Q3Qi-1} fails, the rotation in Lemma~\ref{LEMMA:Q2Qi}~\ref{LEMMA:Q2Qi-2}, with $e(Q,Q_i)=16$, gives three copies of $K_4$ replacing two.
    If \ref{LEMMA:Q3Qi-2} fails, there are distinct $Q,Q'\in\mathcal{A}_3$ with $e(Q,Q_i)=16$ and $e(Q',Q_i)\ge15$, and the same argument of Lemma~\ref{LEMMA:Q2Qi}~\ref{LEMMA:Q2Qi-4} gives four copies replacing three.
    Both contradict the maximality of $|\mathcal{A}|$. 
    It follows from Lemma~\ref{LEMMA:Q3Qi}~\ref{LEMMA:Q3Qi-1} that 
    \[
    e(\mathcal{A}_3)\le15\binom{a_3}{2}+6a_3\le\frac{15a_3^2}{2}.
    \]
    By Lemma~\ref{LEMMA:Q3Qi}~\ref{LEMMA:Q3Qi-2} and the trivial bound, for every $i\in\{4,5,6\}$,
    \[
    e(\mathcal{A}_3,\mathcal{A}_i)
    =\sum_{Q_i\in\mathcal{A}_i}e(\mathcal{A}_3,Q_i)
    \le
    \begin{cases}
        16a_3a_i, & \text{if }a_3\le1,\\
        15a_3a_i, & \text{if }a_3\ge2,
    \end{cases}
    \le15a_3a_i+a_i.
    \]
\end{proof}

The following lemma provides an upper bound for $e(\mathcal{A}_4, \mathcal{A}_i)$ for $i \in \{5,6\}$.
\begin{lemma}\label{LEMMA:Q4Qi}
    The following statements hold. 
    \begin{enumerate}[label=(\roman*)]
        \item\label{LEMMA:Q4Qi-1} $e(Q_4, Q_4') \le 15$ for all distinct members $Q_4, Q_4' \in \mathcal{A}_4$. In particular, 
        \begin{align*}
            e(\mathcal{A}_4)
            \le 15\binom{a_4}{2} + 6a_4
            \le \frac{15 a_4^2}{2}. 
        \end{align*}
        \item\label{LEMMA:Q4Qi-2} Suppose that $|\mathcal{A}_4| \ge 2$. Then $e(\mathcal{A}_4, Q) \le 15 a_4$ for every $Q\in \mathcal{A}_5 \cup \mathcal{A}_6$. Consequently, for $i \in \{5,6\}$, 
        \begin{align*}
            e(\mathcal{A}_4, \mathcal{A}_i)
            & \le 15 a_4 a_i + a_i. 
        \end{align*}
    \end{enumerate}
\end{lemma}

\begin{figure}[H]
\centering

\tikzset{every picture/.style={line width=1pt}} 



\caption{Left: $Q_4$ is seen by a vertex (blue) in $\mathcal{D}$, and $Q_4'$ is seen by another vertex (red) in $\mathcal{D}$. Right: after rotation, the numbers of vertex-disjoint copies of $K_4$ and $K_3$ remain the same, while the number of vertex-disjoint copies of $K_2$ increases by one.} 
\label{Fig:Q4Qi-a}
\end{figure}

\begin{figure}[H]
\centering

\tikzset{every picture/.style={line width=1pt}} 

%
\caption{Left: $Q_4$ is seen by a vertex (blue) in $\mathcal{D}$, and $Q_4'$ is seen by another vertex (red) in $\mathcal{D}$. Right: after rotation, the numbers of vertex-disjoint copies of $K_4$ and $K_3$ remain the same, while the number of vertex-disjoint copies of $K_2$ increases by one.} 
\label{Fig:Q4Qi-b}
\end{figure}

\begin{proof}[Proof of Lemma~\ref{LEMMA:Q4Qi}]
    The proofs are similar to those of Lemma~\ref{LEMMA:Q1Qi}~\ref{LEMMA:Q1Qi-2} and~\ref{LEMMA:Q1Qi-4}.

    For Lemma~\ref{LEMMA:Q4Qi}~\ref{LEMMA:Q4Qi-1}, suppose that $e(Q_4,Q'_4)=16$ for two distinct members $Q_4,Q'_4\in\mathcal A_4$.
    If one of $Q_4,Q'_4$ is $3$-seen, use its triangle and a vertex of $\mathcal D$ seeing the other to increase $|\mathcal A|$.
    Otherwise, choose distinct vertices of $\mathcal D$ seeing them. The rotation in Figure~\ref{Fig:Q4Qi-a} then increases $|\mathcal C|$ while keeping $|\mathcal A|,|\mathcal B|$ unchanged, a contradiction.
    It follows from Lemma~\ref{LEMMA:Q4Qi}~\ref{LEMMA:Q4Qi-1} that
    \[
    e(\mathcal{A}_4)\le15\binom{a_4}{2}+6a_4
    \le\frac{15a_4^2}{2}.
    \]
    For Lemma~\ref{LEMMA:Q4Qi}~\ref{LEMMA:Q4Qi-2}, suppose that $e(\mathcal A_4,Q)\ge15a_4+1$ for some $Q\in\mathcal A_5\cup\mathcal A_6$.
    Then there exist distinct members $Q_4,Q'_4\in\mathcal A_4$ such that $e(Q_4,Q)=16$ and $e(Q'_4,Q)\ge15$.
    If one of $Q_4,Q'_4$ is $3$-seen by a member of $\mathcal B$, choose its triangle and a vertex of $\mathcal D$ seeing the other. 
    These give two copies of $K_4$, and the remaining vertices of $Q_4,Q'_4$, together with $Q$, give two further disjoint copies of $K_4$, since $e(Q_4,Q)=16$ and $e(Q'_4,Q)\ge15$. 
    This increases $|\mathcal A|$, a contradiction.
    Otherwise, choose distinct vertices of $\mathcal D$ seeing $Q_4$ and $Q'_4$. These give two copies of $K_4$, while the remaining vertices together with $Q$ give another copy of $K_4$ and an additional edge (see Figure~\ref{Fig:Q4Qi-b}). 
    This increases $|\mathcal C|$ while keeping $|\mathcal A|,|\mathcal B|$ unchanged, a contradiction.

    By Lemma~\ref{LEMMA:Q4Qi}~\ref{LEMMA:Q4Qi-2} and the trivial bound, for every $i\in\{5,6\}$,
    \begin{align*}
        e(\mathcal{A}_4,\mathcal{A}_i)
        &=\sum_{Q\in\mathcal{A}_i}e(\mathcal{A}_4,Q)\le
        \begin{cases}
            16a_4a_i, & \text{if }a_4\le1,\\
            15a_4a_i, & \text{if }a_4\ge2,
        \end{cases}        
        \le15a_4a_i+a_i.
    \end{align*}
\end{proof}
The following lemma provides an upper bound for $e(\mathcal{A}_5 \cup \mathcal{A}_6)$.
\begin{lemma}\label{LEMMA:A5A6-edges}
    We have 
    \begin{align*}
        e(\mathcal{A}_5 \cup \mathcal{A}_6)
        \le 15\binom{a_5}{2} + 6 a_5 +  15a_5a_6 + 16 \binom{a_6}{2} + 6 a_6 
        \le \frac{15 a_5^2}{2}  + 15 a_5 a_6 + 8 a_6^2.
    \end{align*}
\end{lemma}
\begin{proof}[Proof of Lemma~\ref{LEMMA:A5A6-edges}]
    We have the trivial upper bound $e(\mathcal{A}_6) \le 16 \binom{a_6}{2} + 6 a_6$, so it suffices to show that 
    \begin{align*}
        e(\mathcal{A}_5) + e(\mathcal{A}_5, \mathcal{A}_6) 
        \le 15\binom{a_5}{2} + 6 a_5 +  15a_5a_6.
    \end{align*}
    %
    Recall from the definition of $\mathcal{A}_5$ that there exists an ordering $Q_1, \ldots, Q_{a_5}$ of elements in $\mathcal{A}_5$ such that for every $i \in [a_5]$,  
    \begin{align*}
        e(Q_i, Q_{i+1} \cup \cdots \cup Q_{a_5} \cup V(\mathcal{A}_6))
        \le 15 \left(a_5 - i + |\mathcal{A}_6|\right)
        = 15(a_5 + a_6 - i). 
    \end{align*}
    It follows that 
    \begin{align}\label{equ:A5-A5A6}
        e(\mathcal{A}_5) + e(\mathcal{A}_5,  \mathcal{A}_6)
        & \le 6a_5 +  \sum_{i=1}^{a_5} e(Q_i, Q_{i+1} \cup \cdots \cup Q_{a_5} \cup V(\mathcal{A}_6)) \notag \\[0.5em]
        & \le 6 a_5 + \sum_{i=1}^{a_5} 15(a_5 + a_6 - i)
        = 6a_5 + 15 \binom{a_5}{2} + 15a_5 a_6, 
    \end{align}
    as desired. 
\end{proof}

\subsection{Upper bounds for $e(\mathcal{A}_i, \mathcal{B}\cup \mathcal{C} \cup \mathcal{D})$}
In this subsection, we establish an upper bound on $e(\mathcal{A}_i, \mathcal{B}\cup \mathcal{C} \cup \mathcal{D})$ for each $i \in [6]$.

The following lemma provides an upper bound for $e(\mathcal{A}_1, \mathcal{B} \cup \mathcal{C} \cup \mathcal{D})$.
\begin{lemma}\label{LEMMA:Q1BCD}
    The following statements hold for every $Q = \{p,q,s,t\} \in \mathcal{A}_1$. 
    \begin{enumerate}[label=(\roman*)]
        \item\label{LEMMA:Q1BCD-1} $e(Q, B) \le 9$ for every $B \in \mathcal{B}$. 
        \item\label{LEMMA:Q1BCD-2} $e(Q, C) \le 6$ for every $C \in \mathcal{C}$.
        \item\label{LEMMA:Q1BCD-3} $e(Q, D) \le 3$ for every $D \in \mathcal{D}$. 
    \end{enumerate}
    Consequently, 
    \begin{align*}
        e(\mathcal{A}_1, \mathcal{B} \cup \mathcal{C} \cup \mathcal{D})
        \le (9b + 6c + 3d) a_1. 
    \end{align*}
\end{lemma}
\begin{proof}[Proof of Lemma~\ref{LEMMA:Q1BCD}]
    Fix $B \in \mathcal{B}$. Suppose to the contrary that $e(Q, B) \ge 10 = 12-2$. Then there exist two vertices in $Q$ that are both adjacent to all three vertices in $B$, and by symmetry, we may assume that $\{p,q\}$ are two such vertices in $Q$. 
    By the definition of $\mathcal{A}_1$, there exists a member $B' \in \mathcal{B}\setminus \{B\}$ and a vertex in $Q$ that is adjacent to all vertices in $B'$. However, the rotation shown in Figure~\ref{Fig:Q1B-a} would increase the value of $|\mathcal{A}|$, contradicting the maximality of $(|\mathcal{A}|, |\mathcal{B}|, |\mathcal{C}|, |\mathcal{D}|)$.  This proves Lemma~\ref{LEMMA:Q1BCD}~\ref{LEMMA:Q1BCD-1}. The proofs of Lemma~\ref{LEMMA:Q1BCD}~\ref{LEMMA:Q1BCD-2} and~\ref{LEMMA:Q1BCD-3} are similar (see Figure~\ref{Fig:Q1C-a} and Figure~\ref{Fig:Q1D-a}), so we omit them here. 
\end{proof}

\medskip

\begin{figure}[H]
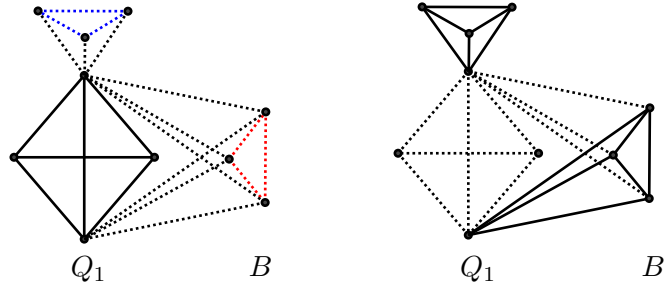

\centering

\tikzset{every picture/.style={line width=1pt}} 



\caption{Left: $Q_1$ is $3$-seen by a triangle (blue) in $\mathcal{B}$ and has two vertices that are adjacent to all vertices of another triangle (red). Right: after rotation, the number of vertex-disjoint copies of $K_4$ increases by one.} 
\label{Fig:Q1B-a}
\end{figure}

\begin{figure}[H]
\centering
\begin{minipage}[t]{0.45\textwidth}
\centering
\tikzset{every picture/.style={line width=1pt}} 
%
\caption{Left: $Q_1$ is $3$-seen by a triangle (blue) in $\mathcal{B}$ and has three vertices that are adjacent to both vertices of an edge (red). Right: after rotation, the number of vertex-disjoint copies of $K_4$ increases by one.} 
\label{Fig:Q1C-a}
\end{minipage}
\hfill
\begin{minipage}[t]{0.45\textwidth}
\centering
\tikzset{every picture/.style={line width=1pt}} 
%
\caption{Left: $Q_1$ is $3$-seen by a triangle (blue) in $\mathcal{B}$ and all four vertices are adjacent to a vertex in $\mathcal{D}$ (red). Right: after rotation, the number of vertex-disjoint copies of $K_4$ increases by one.} 
\label{Fig:Q1D-a}
\end{minipage}
\end{figure}

The following lemma provides an upper bound for $e(\mathcal{A}_2, \mathcal{B} \cup \mathcal{C} \cup \mathcal{D})$.
\begin{lemma}\label{LEMMA:Q2BCD}
    The following statements hold for every $Q = \{p,q,s,t\} \in \mathcal{A}_2$. 
    \begin{enumerate}[label=(\roman*)]
        \item\label{LEMMA:Q2BCD-1} $e(Q, B) \le 8$ for all but at most one $B \in \mathcal{B}$. 
        \item\label{LEMMA:Q2BCD-2} $e(Q, C) \le 6$ for every $C \in \mathcal{C}$.
        \item\label{LEMMA:Q2BCD-3} $e(Q, D) \le 3$ for every $D \in \mathcal{D}$. 
    \end{enumerate}
    Consequently, 
    \begin{align*}
        e(\mathcal{A}_2, \mathcal{B} \cup \mathcal{C} \cup \mathcal{D})
        \le (8b + 6c + 3d) a_2 + 4a_2. 
    \end{align*}
\end{lemma}

\begin{proof}[Proof of Lemma~\ref{LEMMA:Q2BCD}]
    By the definition of $\mathcal{A}_2$, there is at most one member, say $B'$, in $\mathcal{B}$ that can $3$-see $Q$; otherwise, $Q$ would belong to $\mathcal{A}_1$. It follows that $e(Q,B) \le 4 \cdot 2 = 8$ for every $B\in \mathcal{B}\setminus \{B'\}$. This proves Lemma~\ref{LEMMA:Q2BCD}~\ref{LEMMA:Q2BCD-1}. In particular, it follows that 
    \begin{align*}
        e(\mathcal{A}_2, \mathcal{B})
        = \sum_{Q_2 \in \mathcal{A}_2} e(Q_2, \mathcal{B})
        \le a_2\left(8 (|\mathcal{B}|-1) + 3\cdot 4\right)
        = 8 a_2 b + 4a_2. 
    \end{align*}
    Next, we prove Lemma~\ref{LEMMA:Q2BCD}~\ref{LEMMA:Q2BCD-2}. 
    If $Q\in\mathcal{A}_{2,1}\cup\mathcal{A}_{2,2}$, then some member of $\mathcal{B}$ $3$-sees $Q$.
    If $e(Q,C)\ge7$ for some $C\in\mathcal{C}$, the argument in the proof of Lemma~\ref{LEMMA:Q1BCD}~\ref{LEMMA:Q1BCD-2} gives two disjoint copies of $K_4$ using this triangle, $C$, and vertices of $Q$. Replacing $Q$ by these copies contradicts the maximality of $|\mathcal{A}|$.
    So it suffices to consider the case $Q\in \mathcal{A}_{2,3}$. 
    Fix $C\in \mathcal{C}$. Suppose to the contrary that $e(Q,C) \ge 7 = 8-1$. 
    Then there exist three vertices in $Q$ that are all adjacent to both vertices in $C$. By the definition of $\mathcal{A}_{2,3}$, there exists a member $C' \in \mathcal{C} \setminus \{C\}$ such that both vertices of $C'$ are adjacent to the same pair of vertices in $Q$. However, the rotation shown in Figure~\ref{Fig:Q2C-a} would increase the value of $|\mathcal{B}|$ (while keeping $|\mathcal{A}|$ unchanged), contradicting the maximality of $|\mathcal{B}|$. 

    For Lemma~\ref{LEMMA:Q2BCD}~\ref{LEMMA:Q2BCD-3}, suppose that $e(Q,D)=4$. 
    If $Q$ is $3$-seen, the rotation in Lemma~\ref{LEMMA:Q1BCD}~\ref{LEMMA:Q1BCD-3} increases $|\mathcal{A}|$.
    Otherwise, $Q\in\mathcal{A}_{2,3}$, and two vertices of $Q$ form a copy of $K_4$ with some $C\in\mathcal{C}$, while the other two form a triangle with $D$. This increases $|\mathcal{B}|$ while keeping $|\mathcal{A}|$ unchanged, a contradiction (see Figure~\ref{Fig:Q2D-a}).   
\end{proof}

\medskip

\begin{figure}[H]
\centering
\begin{minipage}[t]{0.45\textwidth}
\centering
\tikzset{every picture/.style={line width=1pt}} 

\caption{Left: $Q$ is seen by an edge (blue) in $\mathcal{C}$ and has three vertices that are adjacent to both vertices of an edge (red). Right: after rotation, the number of vertex-disjoint copies of $K_4$ remains the same, while the number of vertex-disjoint copies of $K_3$ increases by one.} 
\label{Fig:Q2C-a}
\end{minipage}
\hfill
\begin{minipage}[t]{0.45\textwidth}
\centering
\tikzset{every picture/.style={line width=1pt}} 
%
\caption{Left: $Q$ is seen by an edge (blue) in $\mathcal{C}$ and all four vertices are adjacent to a vertex in $\mathcal{D}$ (red). Right: after rotation, the number of vertex-disjoint copies of $K_4$ remains the same, while the number of vertex-disjoint copies of $K_3$ increases by one.} 
\label{Fig:Q2D-a}
\end{minipage}
\end{figure}
%
The following lemma provides an upper bound for $e(\mathcal{A}_3, \mathcal{B} \cup \mathcal{C} \cup \mathcal{D})$.
\begin{lemma}\label{LEMMA:Q3BCD}
    The following statements hold for every $Q = \{p,q,s,t\} \in \mathcal{A}_3$. 
    \begin{enumerate}[label=(\roman*)]
        \item\label{LEMMA:Q3BCD-1} $e(Q, B) \le 8$ for every $B \in \mathcal{B}$. 
        \item\label{LEMMA:Q3BCD-2} $e(Q, C) \le 5$ for all but at most one $C \in \mathcal{C}$.
        \item\label{LEMMA:Q3BCD-3} $e(Q, D) \le 3$ for all but at most one $D \in \mathcal{D}$. 
    \end{enumerate}
    Consequently, 
    \begin{align*}
        e(\mathcal{A}_3, \mathcal{B} \cup \mathcal{C} \cup \mathcal{D})
        \le (8b + 5c + 3d) a_3 + 4a_3. 
    \end{align*}
\end{lemma}

\begin{proof}[Proof of Lemma~\ref{LEMMA:Q3BCD}]
    Fix $B \in \mathcal{B}$. It follows from the definition of $\mathcal{A}_{3}$ that there is no vertex in $Q$ that is adjacent to all three vertices of $B$; otherwise, $Q$ would belong to $\mathcal{A}_{2}$. Therefore, $e(Q,B) \le 4\cdot 2 = 8$. This proves Lemma~\ref{LEMMA:Q3BCD}~\ref{LEMMA:Q3BCD-1}. 

    By the definition of $\mathcal{A}_{3}$, there is at most one member in $\mathcal{C}$, say $C'$, that can see $Q$; otherwise, $Q$ would belong to $\mathcal{A}_{2}$. It follows that $e(Q,C) \le 3+2=5$ for every $C \in \mathcal{C}\setminus\{C'\}$. This proves Lemma~\ref{LEMMA:Q3BCD}~\ref{LEMMA:Q3BCD-2}. In particular, we have 
    \begin{align*}
        e(\mathcal{A}_{3}, \mathcal{C})
        =\sum_{Q_3 \in \mathcal{A}_{3}}e(Q_{3}, \mathcal{C})
        \le a_3 \left(5(|\mathcal{C}| - 1) + 8\right)
        = 5 a_3 c + 3a_3. 
    \end{align*}
    Next, we prove Lemma~\ref{LEMMA:Q3BCD}~\ref{LEMMA:Q3BCD-3}. Suppose to the contrary that there exist two members $u,v \in \mathcal{D}$, each adjacent to all four vertices in $Q$. 
    Replacing one vertex of $Q$ by one of these two vertices gives a new copy of $K_4$, leaving a disjoint edge. 
    This increases $|\mathcal{C}|$ while keeping $|\mathcal{A}|,|\mathcal{B}|$ unchanged, a contradiction (see Figure~\ref{Fig:Q3D}).
\end{proof}

\medskip

\begin{figure}[H]
\centering
\begin{minipage}[t]{0.45\textwidth}
\centering
\tikzset{every picture/.style={line width=1pt}} 


\caption{Left: two vertices are adjacent to all vertices in $Q$. Right: after rotation, the numbers of vertex-disjoint copies of $K_4$ and $K_3$ remain the same, while the number of vertex-disjoint copies of $K_2$ increases by one.} 
\label{Fig:Q3D}
\end{minipage}
\hfill
\begin{minipage}[t]{0.45\textwidth}
\centering
\tikzset{every picture/.style={line width=1pt}} 

%
\caption{Left: $Q$ is seen by a vertex (blue) in $\mathcal{D}$ and there is another vertex adjacent to all vertices in $Q$. Right: after rotation, the numbers of vertex-disjoint copies of $K_4$ and $K_3$ remain the same, while the number of vertex-disjoint copies of $K_2$ increases by one.} 
\label{Fig:Q4D}
\end{minipage}
\end{figure}
The following lemma provides an upper bound for $e(\mathcal{A}_4, \mathcal{B} \cup \mathcal{C} \cup \mathcal{D})$.
\begin{lemma}\label{LEMMA:Q4BCD}
    The following statements hold for every $Q = \{p,q,s,t\} \in \mathcal{A}_4$. 
    \begin{enumerate}[label=(\roman*)]
        \item\label{LEMMA:Q4BCD-1} $e(Q, B) \le 7$ for all but at most two $B \in \mathcal{B}$. In particular,
        \[
        e(\mathcal A_4,\mathcal B)\le 7a_4b+10a_4.
        \]
        \item\label{LEMMA:Q4BCD-2} $e(Q, C) \le 5$ for all but at most one $C \in \mathcal{C}$. In particular,
        \[
        e(\mathcal A_4,\mathcal C)\le 5a_4c+3a_4.
        \]
        \item\label{LEMMA:Q4BCD-3} $e(Q, D) \le 3$ for every $D \in \mathcal{D}$. 
    \end{enumerate}
    Consequently, 
    \begin{align*}
        e(\mathcal{A}_4, \mathcal{B} \cup \mathcal{C} \cup \mathcal{D})
        \le (7b+5c+3d)a_4 + 13a_4. 
    \end{align*}
\end{lemma}
\begin{proof}[Proof of Lemma~\ref{LEMMA:Q4BCD}]
    By the definition of $\mathcal{A}_1, \ldots, \mathcal{A}_4$, $Q$ is $3$-seen and $2$-seen by at most two members, say $B_1, B_2$, in $\mathcal{B}$. Otherwise, it would be contained in $\mathcal{A}_1 \cup \mathcal{A}_2 \cup \mathcal{A}_3$. 
    It follows that $e(Q,B) \le 2+2+2+1 = 7$ for every $B \in \mathcal{B}\setminus \{B_1, B_2\}$. This proves Lemma~\ref{LEMMA:Q4BCD}~\ref{LEMMA:Q4BCD-1}. In particular, we have 
    \begin{align*}
        e(\mathcal{A}_{4}, \mathcal{B})
        = \sum_{Q_4 \in \mathcal{A}_{4}} e(Q_{4}, \mathcal{B})
        \le a_4 \left(7(|\mathcal{B}| - 2) + 12\cdot 2\right) 
        = 7a_4 b + 10 a_4. 
    \end{align*}
    %
    
    As in Lemma~\ref{LEMMA:Q3BCD}~\ref{LEMMA:Q3BCD-2}, at most one member of $\mathcal{C}$ sees $Q$. 
    Every $C\in\mathcal{C}$ with $e(Q,C)\ge 6$ sees $Q$ by Fact~\ref{FACT:3-by-4-bipartite-graph}~\ref{FACT:3-by-4-bipartite-graph--1}, proving Lemma~\ref{LEMMA:Q4BCD}~\ref{LEMMA:Q4BCD-2}.

    Next, we prove Lemma~\ref{LEMMA:Q4BCD}~\ref{LEMMA:Q4BCD-3}. The proof for the case $Q\in \mathcal{A}_{4,1}$ is similar to that of Lemma~\ref{LEMMA:Q1BCD}~\ref{LEMMA:Q1BCD-3} (see Figure~\ref{Fig:Q1D-a}) and is therefore omitted. So it suffices to consider the case $Q\in \mathcal{A}_{4,2}$. 
    Suppose there is $D\in \mathcal{D}$ such that $e(Q,D)=4$.  It follows from the definition of $\mathcal{A}_{4,2}$ that there exists a member $D' \in \mathcal{D}\setminus \{D\}$ that is adjacent to at least three vertices in $Q$. However, the rotation shown in Fig~\ref{Fig:Q4D} would increase the value of $|\mathcal{C}|$ (while keeping both $|\mathcal{A}|$ and $|\mathcal{B}|$ unchanged), contradicting the maximality of $(|\mathcal{A}|, |\mathcal{B}|, |\mathcal{C}|, |\mathcal{D}|)$. 
\end{proof}
The following lemma provides an upper bound for $e(\mathcal{A}_i, \mathcal{B} \cup \mathcal{C} \cup \mathcal{D})$ for $i \in \{5,6\}$.
\begin{lemma}\label{LEMMA:Q5Q6BCD}
    The following statements hold for every $Q = \{p,q,s,t\} \in \mathcal{A}_5 \cup \mathcal{A}_6$. 
    \begin{enumerate}[label=(\roman*)]
        \item\label{LEMMA:Q5Q6BCD-1} $e(Q, B) \le 7$ for all but at most two $B \in \mathcal{B}$. 
        \item\label{LEMMA:Q5Q6BCD-2} $e(Q, C) \le 5$ for all but at most one $C \in \mathcal{C}$.
        \item\label{LEMMA:Q5Q6BCD-3} $e(Q, D) \le 2$ for all but at most one $D \in \mathcal{D}$. 
    \end{enumerate}
    Consequently, for $i \in \{5,6\}$, 
    \begin{align*}
        e(\mathcal{A}_i, \mathcal{B} \cup \mathcal{C} \cup \mathcal{D})
        \le (7b + 5c + 2d) a_i + 15a_i. 
    \end{align*}
\end{lemma}
\begin{proof}[Proof of Lemma~\ref{LEMMA:Q5Q6BCD}]
    Proofs of Lemma~\ref{LEMMA:Q5Q6BCD}~\ref{LEMMA:Q5Q6BCD-1} and~\ref{LEMMA:Q5Q6BCD-2} are similar to those of Lemma~\ref{LEMMA:Q4BCD}~\ref{LEMMA:Q4BCD-1} and~\ref{LEMMA:Q4BCD-2}.  Lemma~\ref{LEMMA:Q5Q6BCD}~\ref{LEMMA:Q5Q6BCD-3} follows easily from the definition of $\mathcal{A}_{5} \cup \mathcal{A}_6$. So we omit them here. 
\end{proof}
%

\section{Local estimation \RomanNumeralCaps{2}$\colon$ $e(\mathcal{A}_3)$}\label{SEC:A3}
In this section, we prove the following  improved upper bound for $e(\mathcal{A}_3)$, which refines the bound $e(\mathcal{A}_3) \le \frac{15 a_3^2}{2}$ given in Lemma~\ref{LEMMA:Q3Qi}~\ref{LEMMA:Q3Qi-1}.

\begin{lemma}\label{LEMMA:A3-upper-bound}
    We have 
    \begin{align*}
        e(\mathcal{A}_3) 
        \le 7 a_3^2. 
    \end{align*}
\end{lemma}
\begin{proof}[Proof of Lemma~\ref{LEMMA:A3-upper-bound}]
Let us define an auxiliary graph $H$ whose vertex set is $\mathcal{A}_3$, and two members $Q, Q' \in \mathcal{A}_3$ are adjacent in $H$ iff $e(Q,Q') = 15$ (recall from Lemma~\ref{LEMMA:Q3Qi}~\ref{LEMMA:Q3Qi-1} that $e(Q_3,Q_3') \le 15$ for every pair $Q_3,Q_3' \in \mathcal{A}_3$). 
For each vertex $u\in V(H)$, let $Q_u\in\mathcal{A}_3$ denote the corresponding copy of $K_4$ in $G$, and let $V(Q_u)$ denote its vertex set.

Let $(\mathcal{T}, \mathcal{M}, \mathcal{I})$ be a rank-$3$-packing of $H$, defined as follows$\colon$
\begin{itemize}
    \item $\mathcal{T}$ is a $K_3$-tiling in $H$, 
    \item $\mathcal{M}$ is a $K_2$-tiling in $H$,
    \item $\mathcal{I}$ is a $K_1$-tiling in $H$, and 
    \item $V(\mathcal{T}) \cup V(\mathcal{M}) \cup V(\mathcal{I}) = V(H)$ is a partition of $V(H)$, 
\end{itemize}
with the condition that $(|\mathcal{T}|, |\mathcal{M}|, |\mathcal{I}|)$ is maximum in the lexicographical order. Note that 
\begin{align}\label{equ:A3-aux-gp-partition}
    3|\mathcal{T}| + 2|\mathcal{M}| + |\mathcal{I}| 
    = |V(H)| 
    = |\mathcal{A}_3| 
    = a_3.
\end{align}
Let 
\begin{align*}
    U_1 
    \coloneqq \bigcup_{u \in V(\mathcal{T})} V(Q_{u}),
    \quad 
    U_2 
    \coloneqq \bigcup_{u \in V(\mathcal{M})} V(Q_{u}), 
    \quad\text{and}\quad 
    U_3 
    \coloneqq \bigcup_{u \in V(\mathcal{I})} V(Q_{u}).
\end{align*}
Note that 
\begin{align}\label{equ:A3-U1U2U3}
    e(\mathcal{A}_{3})
    & = \sum_{i\in [3]}e_{G}(U_i) + \sum_{1\le i< j \le 3}e_{G}(U_i, U_j),  
\end{align}
and note from the definition of $H$ that 
\begin{align}\label{equ:A3-U3}
    e_{G}(U_3)
    \le 14 \binom{|\mathcal{I}|}{2} + \binom{4}{2} \cdot |\mathcal{I}|
    = 7 |\mathcal{I}|^2 - |\mathcal{I}|.
\end{align}

\medskip

%
\begin{figure}[H]
\centering
\tikzset{every picture/.style={line width=1pt}} 

\caption{A rotation that increases the number of vertex-disjoint copies of $K_4$.}
\label{Fig:A3QQ-T-M}
\end{figure}

\begin{claim}\label{CLAIM:A3QQ-T-M-12}
    Let $T=\{u_1,u_2,u_3\} \in \mathcal{T}$ and $M=\{v_1,v_2\} \in \mathcal{M}$. For every $(i,j) \in [3] \times [2]$,  we have 
    \begin{align*}
        e_{G}(Q_{u_i}, Q_{v_j}) 
        \le 12.
    \end{align*}
    Consequently, 
    \begin{align*}
        e_{G}(U_1, U_2)
        \le 12 \cdot 3 |\mathcal{T}| \cdot  2|\mathcal{M}|
        = 72 |\mathcal{T}| |\mathcal{M}|. 
    \end{align*}
\end{claim}
\begin{proof}[Proof of Claim~\ref{CLAIM:A3QQ-T-M-12}]
    Fix $T=\{u_1,u_2,u_3\} \in \mathcal{T}$ and $M=\{v_1,v_2\} \in \mathcal{M}$. Assume that $Q_{u_i} = \{p_i, q_i, s_i, t_i\}$ for $i \in [3]$ and $Q_{v_i} = \{p_i', q_i', s_i', t_i'\}$ for $i \in [2]$. Suppose to the contrary that this claim is not true. By symmetry, we may assume that $e_{G}(Q_{u_3}, Q_{v_1})\ge 13$.

    By the definition of $\mathcal{A}_{3}$ (see Figure~\ref{Fig:def-A1A2A3A4}), there exist three pairwise disjoint edges (which belong to members of $\mathcal{B} \cup \mathcal{C}$) $\{x_1, y_1\}, \{x_2, y_2\}, \{x_2', y_2'\} \in G-V(\mathcal{A})$ such that
    \begin{itemize}
        \item both vertices in $\{x_1, y_1\}$ are adjacent to a certain pair of vertices in $Q_{u_1}$,
        \item both vertices in $\{x_2, y_2\}$ are adjacent to a certain pair of vertices in $Q_{u_2}$,
        \item and both vertices in $\{x_2', y_2'\}$ are adjacent to a certain pair of vertices in $Q_{v_2}$.
    \end{itemize}   
    By symmetry, we may assume that $\{p_1, q_1\} \subseteq  Q_{u_1}$, $\{p_2, q_2\} \subseteq  Q_{u_2}$, and $\{p_2', q_2'\} \subseteq  Q_{v_2}$ are such vertices. 

    Since $e_{G}(Q_{u_3}, Q_{v_1})\ge 13 = 16 - 3$, there exists a vertex in $Q_{v_1}$ that is adjacent to all four vertices in $Q_{u_3}$, and by symmetry, we may assume that $p_1'$ is  such vertex in $Q_{v_1}$. Notice that 
    \begin{align*}
        e_G(\{s_2', t_2'\}, \{q_1', s_1', t_1'\})
        \ge 15 - (4+4+2)
        = 5.
    \end{align*}
    By Fact~\ref{FACT:3-by-4-bipartite-graph}~\ref{FACT:3-by-4-bipartite-graph-1}, there exists a pair of vertices in $\{q_1', s_1', t_1'\}$ that are adjacent to both vertices in $\{s_2', t_2'\}$. 
    
    Note that 
    \begin{align*}
        e_G(\{s_1, t_1\}, \{s_2, t_2\})
        \ge 15 - (4+4+2+2)
        = 3. 
    \end{align*}
    So there exists a vertex in $\{s_1, t_1\}$ that is adjacent to both vertices in $\{s_2, t_2\}$, and by symmetry, we may assume that $s_1$ is such a vertex. 
    Since $e(Q_{u_1}, Q_{u_3})= e(Q_{u_2}, Q_{u_3}) = 15 = 16 - 1$, we have 
    \begin{align*}
        e_G(\{s_1, s_2, t_2\}, Q_{u_3})
        \ge 3 \cdot 4- (1+1)
        = 10. 
    \end{align*}
    It follows that there exists a vertex in $Q_{u_3}$ that is adjacent to all vertices in $\{s_1, s_2, t_2\}$. Recall that the vertex  $p_1' \in Q_{v_1}$ is adjacent to all four vertices in $Q_{u_3}$. 
    However, the rotation shown in Figure~\ref{Fig:A3QQ-T-M} would increase the value of $|\mathcal{A}|$, contradicting the maximality of $(|\mathcal{A}|, |\mathcal{B}|, |\mathcal{C}|, |\mathcal{D}|)$. 
\end{proof}

\begin{claim}\label{CLAIM:A3QQ-T-T-12}
    Let $T=\{u_1,u_2,u_3\}, T' =\{v_1, v_2, v_3\} \in \mathcal{T}$ be two distinct members. For every $(i,j) \in [3] \times [3]$, we have 
    \begin{align*}
        e_{G}(Q_{u_i}, Q_{v_j}) 
        \le 12.
    \end{align*}
    Consequently, 
    \begin{align*}
        e_{G}(U_1)
        \le 12 \cdot 3 \cdot 3 \cdot \binom{|\mathcal{T}|}{2} + \binom{4}{2} \cdot 3|\mathcal{T}| + 3 \cdot 15 \cdot |\mathcal{T}|
        = 54 |\mathcal{T}|^2 + 9|\mathcal{T}|. 
    \end{align*}
\end{claim}
\begin{proof}[Proof of Claim~\ref{CLAIM:A3QQ-T-T-12}]
    Suppose to the contrary that this claim is not true. By symmetry, we may assume that $e_G(Q_{u_3},Q_{v_1})\ge13$.
    Since $e_G(Q_{v_1},Q_{v_2})=15$, the rotation in the proof of Claim~\ref{CLAIM:A3QQ-T-M-12} applies to $Q_{u_1},Q_{u_2},Q_{u_3},Q_{v_1},Q_{v_2}$.
    It replaces these five members by six vertex-disjoint copies of $K_4$, leaving $Q_{v_3}$ unchanged.
    This contradicts the maximality of $|\mathcal{A}|$.
\end{proof}

\medskip

\begin{figure}[H]
\centering
\tikzset{every picture/.style={line width=1pt}} 


\caption{A rotation that increases the number of vertex-disjoint copies of $K_4$.}
\label{Fig:A3QQ-T-I}
\end{figure}

\begin{claim}\label{CLAIM:A3QQ-T-I-14}
    Let $T=\{u_1,u_2,u_3\} \in \mathcal{T}$ and $I=\{v\} \in \mathcal{I}$. For every $i \in [3]$, we have 
    \begin{align*}
        e_{G}(Q_{u_i}, Q_{v}) 
        \le 14. 
    \end{align*}
    Consequently, 
    \begin{align*}
        e_{G}(U_1,U_3)
        \le 14 \cdot 3|\mathcal{T}| \cdot |\mathcal{I}| 
        = 42 |\mathcal{T}| |\mathcal{I}|. 
    \end{align*}
\end{claim}
\begin{proof}[Proof of Claim~\ref{CLAIM:A3QQ-T-I-14}]
    Suppose to the contrary that this claim is not true. By symmetry, we may assume that $e_G(Q_{u_3},Q_v)\ge15$. As in the proof of Claim~\ref{CLAIM:A3QQ-T-M-12}, choose three pairwise disjoint edges from members of $\mathcal{B}\cup\mathcal{C}$ and form three disjoint copies of $K_4$ using two vertices from each of $Q_{u_1},Q_{u_2},Q_v$, respectively.
    The remaining vertices in $Q_{u_1}\cup Q_{u_2}$ contain a triangle, which has at least two common neighbors in $Q_{u_3}$. 
    The remaining two vertices of $Q_v$ have at least three common neighbors in $Q_{u_3}$. Choose one vertex of $Q_{u_3}$ to complete the triangle, and two other vertices to form a copy of $K_4$ with the remaining pair of $Q_v$.
    Together with the first three copies, these replace four members of $\mathcal{A}$ by five, contradicting the maximality of $|\mathcal{A}|$ (see Figure~\ref{Fig:A3QQ-T-I}).
\end{proof}

\begin{claim}\label{CLAIM:A3QQ-MM-56}
    Let $M=\{u_1, u_2\}, M'=\{v_1, v_2\} \in \mathcal{M}$ be two distinct members. We have 
    \begin{align}\label{equ:CLAIM:A3QQ-MM-56}
        e_{G}\left(Q_{u_1}\cup Q_{u_2}, Q_{v_1}\cup Q_{v_2}\right)\le 56.
    \end{align}
    Consequently, if $|\mathcal{M}| \ge 2$, then 
    \begin{align*}
        e_{G}(U_2)
        \le 56 \cdot \binom{|\mathcal{M}|}{2} + \binom{4}{2} \cdot 2|\mathcal{M}| + 15 \cdot |\mathcal{M}|
        = 28 |\mathcal{M}|^2 - |\mathcal{M}|.
    \end{align*}
\end{claim}

\medskip
\begin{figure}[H]
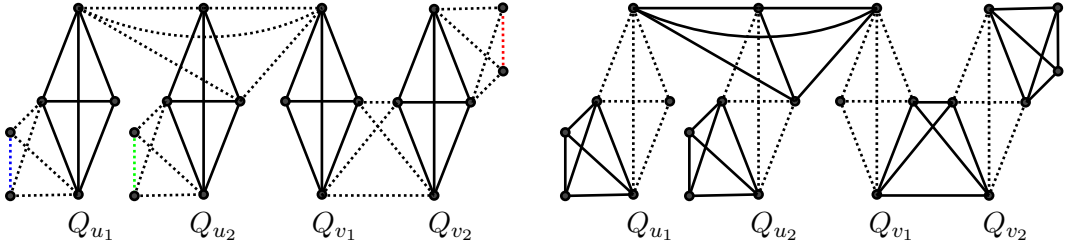

\centering

\tikzset{every picture/.style={line width=1pt}} 


\caption{A rotation that increases the number of vertex-disjoint copies of $K_4$.} 
\label{Fig:A3QQ-MM}
\end{figure}

\begin{proof}[Proof of Claim~\ref{CLAIM:A3QQ-MM-56}]
    Fix $M=\{u_1, u_2\}, M'=\{v_1, v_2\} \in \mathcal{M}$. Assume that $Q_{u_i} = \{p_i, q_i, s_i, t_i\}$ and $Q_{v_i} = \{p_i', q_i', s_i', t_i'\}$ for $i \in [2]$. Suppose to the contrary that~\eqref{equ:CLAIM:A3QQ-MM-56} fails. 
    Then by the Pigeonhole Principle, there exists a pair $(i,j) \in [2] \times [2]$ such that $e(Q_{u_i}, Q_{v_j}) \ge  \lceil 57/4 \rceil = 15$. By symmetry, we may assume that $(i,j) = (2,1)$. 
    Recall from Lemma~\ref{LEMMA:Q3Qi}~\ref{LEMMA:Q3Qi-1} that $e_G(Q_{u_i},Q_{v_j})\le 15$ for every $(i,j)\in[2]\times[2]$, so $e_G(Q_{u_2},Q_{v_1})=15$.
    
    It follows from the maximality of $(|\mathcal{T}|, |\mathcal{M}|, |\mathcal{I}|)$ (in particular, the maximality of $|\mathcal{T}|$) that
    \begin{itemize}
        \item $e(Q_{u_1}, Q_{v_1}) \le 14$, since otherwise, $\{u_1, u_2, v_1\}$ would be in $\mathcal{T}$, 
        \item and $e(Q_{u_2}, Q_{v_2}) \le 14$, since otherwise, $\{v_1, v_2, u_2\}$ would be in $\mathcal{T}$.
    \end{itemize}
    Combining this with the counterassumption on \eqref{equ:CLAIM:A3QQ-MM-56}, we obtain
    \begin{align*}
        e_{G}(Q_{u_1}, Q_{v_2})
        & = e_{G}\left(Q_{u_1}\cup Q_{u_2}, Q_{v_1}\cup Q_{v_2}\right)
        - e_{G}(Q_{u_1}, Q_{v_1}) - e_{G}(Q_{u_2}, Q_{v_1}) - e_{G}(Q_{u_2}, Q_{v_2}) \\[0.5em]
        & \ge 57 - 14 - 15- 14
        = 14. 
    \end{align*}
    By interchanging $(u_1,u_2)$ and $(v_2,v_1)$ if necessary, we may further assume that $e_G(Q_{u_1},Q_{v_1})\ge e_G(Q_{u_2},Q_{v_2})$. Then
    \[
    2e_G(Q_{u_1},Q_{v_1})
    \ge e_G(Q_{u_1},Q_{v_1})+e_G(Q_{u_2},Q_{v_2})
    \ge57-15-15=27,
    \]
    and hence $e_G(Q_{u_1},Q_{v_1})\ge14$.

    Similar to the proof of Claim~\ref{CLAIM:A3QQ-T-M-12},  by the definition of $\mathcal{A}_{3}$, there exist three pairwise disjoint edges (which belong to members in $\mathcal{B} \cup \mathcal{C}$) $\{x_1, y_1\}, \{x_2, y_2\}, \{x_2', y_2'\} \in G-V(\mathcal{A})$ such that 
    \begin{itemize}
        \item both vertices in $\{x_1, y_1\}$ are adjacent to a certain pair of vertices in $Q_{u_1}$, 
        \item both vertices in $\{x_2, y_2\}$ are adjacent to a certain pair of vertices in $Q_{u_2}$, 
        \item and both vertices in $\{x_2', y_2'\}$ are adjacent to a certain pair of vertices in $Q_{v_2}$.
    \end{itemize}
     By symmetry, we may assume that $\{p_1, q_1\} \subseteq  Q_{u_1}$, $\{p_2, q_2\} \subseteq  Q_{u_2}$, and $\{p_2', q_2'\} \subseteq  Q_{v_2}$ are such vertices. Also, similar to the proof of Claim~\ref{CLAIM:A3QQ-T-M-12}, we may assume that $s_1$ is adjacent to both vertices in $\{s_2, t_2\}$. 

    Since $e_G(Q_{v_1}, Q_{u_1}) \ge 14 = 16 -2$, there exist two vertices in $Q_{v_1}$ that are adjacent to $s_{1}$, and by symmetry, we may assume that $\{p_1', q_1'\}$ are such two vertices. 
    Since $e_G(Q_{u_2}, Q_{v_1}) = 15$, one of the vertices in $\{p_1', q_1'\}$ must be adjacent to both $\{s_2, t_2\}$, and by symmetry, we may assume that $p_1'$ is adjacent to both $\{s_2, t_2\}$. Note that 
    \begin{align*}
        e_G(\{q_1', s_1', t_1'\}, \{s_2', t_2'\})
        \ge 15 - (4+2\cdot 3) 
        = 5.
    \end{align*}
    So it follows from Fact~\ref{FACT:3-by-4-bipartite-graph}~\ref{FACT:3-by-4-bipartite-graph-1} that both vertices in $\{s_2', t_2'\}$ are adjacent a certain pair of vertices in $\{q_1', s_1', t_1'\}$. 
    However, the rotation shown in Figure~\ref{Fig:A3QQ-MM} would increase the size of $\mathcal{A}$, contradicting the maximality of $(|\mathcal{A}|, |\mathcal{B}|, |\mathcal{C}|, |\mathcal{D}|)$. 
\end{proof}

\begin{claim}\label{CLAIM:A3QQ-M-I}
    For every $\{u_1, u_2\} \in \mathcal{M}$ and $w \in \mathcal{I}$, we have 
    \begin{align}\label{equ:CLAIM:A3QQ-M-I-a}
        e(Q_{u_1} \cup Q_{u_2}, Q_{w}) \le 29. 
    \end{align}
    Moreover, if equality holds and $|\mathcal{M}| \ge 2$, then for every $\{v_1, v_2\} \in \mathcal{M} 
    \setminus \{\{u_1, u_2\}\}$, we have 
    \begin{align}\label{equ:CLAIM:A3QQ-M-I-b}
        e(Q_{v_1} \cup Q_{v_2}, Q_{w}) \le 24. 
    \end{align}
    In particular, for every pair $\{u_1, u_2\}, \{v_1, v_2\} \in \mathcal{M}$ and for every $w \in \mathcal{I}$,
    \begin{align*}
        \frac{1}{2} \left(e(Q_{u_1} \cup Q_{u_2}, Q_{w}) + e(Q_{v_1} \cup Q_{v_2}, Q_{w}) \right) 
        \le 28. 
    \end{align*}
    Consequently, 
    \begin{align*}
        e(U_2, U_3)
        \le 
        \begin{cases}
            29|\mathcal{I}|, &\quad\text{if}\quad |\mathcal{M}| = 1, \\[0.5em]
            28 |\mathcal{M}||\mathcal{I}|, &\quad\text{if}\quad |\mathcal{M}| \ge 2. 
        \end{cases}
    \end{align*}
\end{claim}

\begin{proof}[Proof of Claim~\ref{CLAIM:A3QQ-M-I}]
    Fix $\{u_1, u_2\} \in \mathcal{M}$ and $w \in \mathcal{I}$. Assume that $Q_{u_i} = \{p_i, q_i, s_i, t_i\}$ for $i \in [2]$ and $Q_{w} = \{p,q,s,t\}$. 
    Inequality~\eqref{equ:CLAIM:A3QQ-M-I-a} follows easily from Lemma~\ref{LEMMA:Q3Qi}~\ref{LEMMA:Q3Qi-1} and the maximality of $(|\mathcal{T}|, |\mathcal{M}|, |\mathcal{I}|)$, since otherwise, $\{u_1, u_2, w\}$ would be in $\mathcal{T}$, a contradiction. 
    So it remains to prove~\eqref{equ:CLAIM:A3QQ-M-I-b}. 
    
    Assume that~\eqref{equ:CLAIM:A3QQ-M-I-a} holds with equality and $|\mathcal{M}| \ge 2$. Fix $\{v_1, v_2\} \in \mathcal{M} 
    \setminus \{\{u_1, u_2\}\}$. Assume that $Q_{v_i} = \{p_i', q_i', s_i', t_i'\}$ for $i \in [2]$. Suppose to the contrary that $e(Q_{v_1} \cup Q_{v_2}, Q_{w}) \ge 25$. 
    By symmetry and the Pigeonhole Principle, we may assume that $e(Q_{v_1}, Q_{w}) \ge 13$. 
    
    Similar to the proof of Claim~\ref{CLAIM:A3QQ-T-M-12},  by the definition of $\mathcal{A}_{3}$, there exist three pairwise disjoint edges $\{x_1, y_1\}, \{x_2, y_2\}, \{x_2', y_2'\} \in G-V(\mathcal{A})$ such that 
    \begin{itemize}
        \item both vertices in $\{x_1, y_1\}$ are adjacent to a certain pair of vertices in $Q_{u_1}$,
        \item  both vertices in $\{x_2, y_2\}$ are adjacent to a certain pair of vertices in $Q_{u_2}$, 
        \item and both vertices in $\{x_2', y_2'\}$ are adjacent to a certain pair of vertices in $Q_{v_2}$.
    \end{itemize}
    By symmetry, we may assume that $\{p_1, q_1\} \subseteq  Q_{u_1}$, $\{p_2, q_2\} \subseteq  Q_{u_2}$, and $\{p_2', q_2'\} \subseteq  Q_{v_2}$ are such vertices. Also, similar to the proof of Claim~\ref{CLAIM:A3QQ-T-M-12}, we may assume that $s_1$ is adjacent to both vertices in $\{s_2, t_2\}$. 
    
     Since $e(Q_{u_1} \cup Q_{u_2}, Q_{w}) = 29 = 32-3$, there exists a vertex in $Q_{w}$ that is adjacent to all three vertices in $\{s_1, s_2, t_2\}$, and by symmetry, we may assume that $p$ is such a vertex. 
    Note that 
    \begin{align*}
        e_G(\{q,s,t\}, Q_{v_1})
        \ge 13 - 4
        = 9. 
    \end{align*}
    So it follows from Fact~\ref{FACT:3-by-4-bipartite-graph}~\ref{FACT:3-by-4-bipartite-graph-3} that there exists a vertex in $Q_{v_1}$ that is adjacent to all three vertices in $\{q,s,t\}$, and by symmetry, we may assume that $p_{1}'$ is such a vertex. 
    Additionally, since 
    \begin{align*}
        e_G(\{q_1', s_1', t_1'\}, \{s_2', t_2'\})
        \ge 15 - (4+4+2)
        = 5,
    \end{align*}
    by Fact~\ref{FACT:3-by-4-bipartite-graph}~\ref{FACT:3-by-4-bipartite-graph-1}, 
    there exist two vertices in $\{q_1', s_1', t_1'\}$ that are adjacent to both vertices in $\{s_2', t_2'\}$. 
    However, the rotation shown in Figure~\ref{Fig:A3QQ-MM-b} would increase the value of $|\mathcal{A}|$, contradicting the maximality of $(|\mathcal{A}|, |\mathcal{B}|, |\mathcal{C}|, |\mathcal{D}|)$. 
\end{proof}

\medskip
\begin{figure}[H]
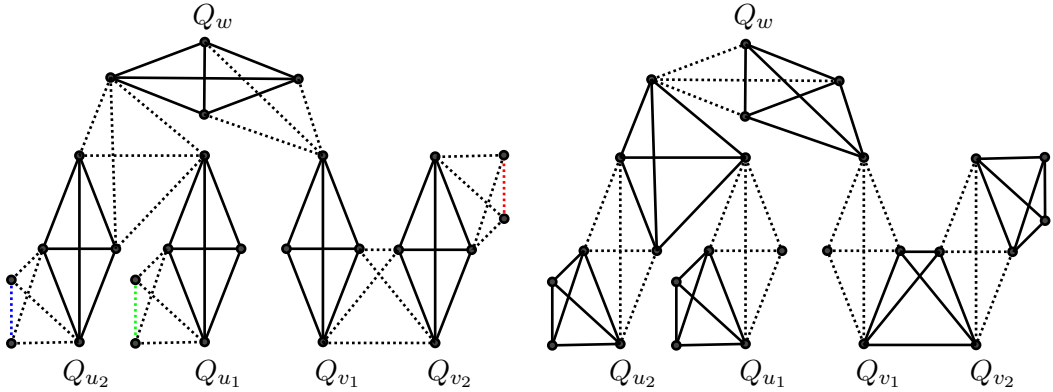

\centering

\tikzset{every picture/.style={line width=1pt}} 


\caption{A rotation that increases the number of vertex-disjoint copies of $K_4$.} 
\label{Fig:A3QQ-MM-b}
\end{figure}

%
Suppose that $|\mathcal{M}| \ge 2$. 
Then it follows from~\eqref{equ:A3-aux-gp-partition},~\eqref{equ:A3-U1U2U3},~\eqref{equ:A3-U3}, Claims~\ref{CLAIM:A3QQ-T-M-12},~\ref{CLAIM:A3QQ-T-T-12},~\ref{CLAIM:A3QQ-T-I-14},~\ref{CLAIM:A3QQ-MM-56}, and~\ref{CLAIM:A3QQ-M-I} that 
\begin{align*}
    e(\mathcal{A}_{3})
    & \le 54 |\mathcal{T}|^2 + 9|\mathcal{T}| + 28 |\mathcal{M}|^2 - |\mathcal{M}| +  72|\mathcal{T}||\mathcal{M}| + 42|\mathcal{T}||\mathcal{I}| + 28 |\mathcal{M}||\mathcal{I}| + 7|\mathcal{I}|^2 - |\mathcal{I}| \\[0.5em]
    & \le 63 |\mathcal{T}|^2 + 28 |\mathcal{M}|^2  +84|\mathcal{T}||\mathcal{M}| +42 |\mathcal{T}||\mathcal{I}|+28|\mathcal{M}||\mathcal{I}| +7|\mathcal{I}|^2 \\[0.5em]
    & = 7(3|\mathcal{T}| + 2|\mathcal{M}| + |\mathcal{I}|)^2
    = 7 a_3^2,
\end{align*}
as desired. 

Suppose that $|\mathcal{M}| = 1$. Then similar to the argument above, and using the trivial equality $29 |\mathcal{M}||\mathcal{I}| - |\mathcal{I}| = 28 |\mathcal{M}||\mathcal{I}|$,
we have 
\begin{align*}
    e(\mathcal{A}_{3})
    & \le 54 |\mathcal{T}|^2 + 9|\mathcal{T}| + 27 |\mathcal{M}| +  72|\mathcal{T}||\mathcal{M}| + 42|\mathcal{T}||\mathcal{I}| + 29 |\mathcal{M}||\mathcal{I}| + 7|\mathcal{I}|^2 - |\mathcal{I}| \\[0.5em]
    & \le 63 |\mathcal{T}|^2 + 28 |\mathcal{M}|^2  +84|\mathcal{T}||\mathcal{M}| +42 |\mathcal{T}||\mathcal{I}|+28|\mathcal{M}||\mathcal{I}| +7|\mathcal{I}|^2 \\[0.5em]
    & = 7(3|\mathcal{T}| + 2|\mathcal{M}| + |\mathcal{I}|)^2
    = 7 a_3^2,
\end{align*}
also as desired. 
This completes the proof of Lemma~\ref{LEMMA:A3-upper-bound}. 
\end{proof}

\section{Local estimation \RomanNumeralCaps{3}$\colon$ $e(\mathcal{A}_4) + e(\mathcal{A}_4, \mathcal{B} \cup \mathcal{C})$}\label{SEC:A4BC}
In this section, we establish an improved upper bound for $e(\mathcal{A}_4) + e(\mathcal{A}_4, \mathcal{B} \cup \mathcal{C})$, 
refining the following bounds given in Lemma~\ref{LEMMA:Q4Qi}~\ref{LEMMA:Q4Qi-1} ($e(\mathcal{A}_4)\le 15a_4^2/2 $) and Lemma~\ref{LEMMA:Q4BCD}~\ref{LEMMA:Q4BCD-1}($e(\mathcal{A}_{4}, \mathcal{B})\le 7a_4 b + 10 a_4$),
~\ref{LEMMA:Q4BCD-2}($e(\mathcal{A}_{4}, \mathcal{C})
    \le 5a_4 c + 3 a_4$): 
\begin{align*}
    e(\mathcal{A}_{4}) 
    \le \frac{15 a_4^2}{2}
    \quad\text{and}\quad 
    e(\mathcal{A}_{4}, \mathcal{B} \cup \mathcal{C})
    \le (7b + 5c)a_4 + 13 a_4. 
\end{align*}
%

Let $\alpha \in (0,1)$ and $x,y \ge 0$ be real numbers. Define  
\begin{align*}
    g_{\alpha}(x,y)
    & \coloneqq 
    \begin{cases}
        7 \alpha x^2 + 7 xy, &\quad\text{if}\quad x \le \frac{y}{2\alpha}, \\[0.5em]
        \frac{15 \alpha x^2}{2} + \frac{13 xy}{2} + \frac{y^2}{8\alpha}, &\quad\text{if}\quad x > \frac{y}{2\alpha}, 
    \end{cases}  \\[0.5em]
    \quad\text{and}\quad
    h_{\alpha}(x, y)
    & \coloneqq 
    \begin{cases}
        7 \alpha x^2 + 5 xy, &\quad\text{if}\quad x \le \frac{y}{2\alpha}, \\[0.5em]
        \frac{15 \alpha x^2}{2} + \frac{9 xy}{2} + \frac{y^2}{8\alpha}, &\quad\text{if}\quad x > \frac{y}{2\alpha}. 
    \end{cases}
\end{align*}
Additionally, define 
\begin{align*}
    \hat{g}(x,y)
    \coloneqq 
    \begin{cases}
        \frac{21 x^2}{4} + 7xy, &\quad\text{if}\quad x \le \frac{6 y}{5}, \\[0.5em]
        \frac{45 x^2}{8}+\frac{61 x y}{10}+\frac{27 y^2}{50}, &\quad\text{if}\quad x \in \left[\frac{6y}{5}, \frac{(86 - 4\sqrt{210}) y}{15}\right], \\[0.5em]
        \frac{28 x^2}{5}+\frac{479 x y}{75}+\frac{103 y^2}{1125}, &\quad\text{if}\quad x \in \left[\frac{(86 - 4\sqrt{210}) y}{15}, \frac{56y}{15}\right], \\[0.5em]
        \frac{651 x^2}{116}+\frac{913 x y}{145}+\frac{113 y^2}{435}, &\quad\text{if}\quad x\in \left[\frac{56y}{15}, \frac{38 y}{5}\right] \\[0.5em]
        \frac{45 x^2}{8}+\frac{61 x y}{10}+\frac{151 y^2}{150}, &\quad\text{if}\quad x\ge \frac{38 y}{5}. 
    \end{cases}
\end{align*}
\begin{lemma}\label{LEMMA:A4BC-together}
    There exists an absolute constant $C_{\ref{LEMMA:A4BC-together}} > 0$ such that 
    \begin{align*}
        e(\mathcal{A}_{4}) + e(\mathcal{A}_{4}, \mathcal{B} \cup \mathcal{C})
        \le C_{\ref{LEMMA:A4BC-together}} n +
        \begin{cases}
            g_{1/2}(a_4, b) + h_{1/2}(a_4, c), 
            &\quad\text{if}\quad d < \frac{b+c}{2}, \\[0.5em]
            \hat{g}(a_4, b) + h_{1/4}(a_4, c), &\quad\text{if}\quad d \ge \frac{b+c}{2}.
        \end{cases}
    \end{align*}
\end{lemma}
Lemma~\ref{LEMMA:A4BC-together} will follow from Lemmas~\ref{LEMMA:A4-B-bipartite-bound},~\ref{LEMMA:A4-B-10-partite-bound}, and~\ref{LEMMA:A4-C-bipartite-bound}, which will be proved in the following two subsections.  

\subsection{$\mathcal{A}_{4}$ and $\mathcal{B}$}\label{SUBSEC:A4B}
In this subsection, we prove the following two upper bounds for $\alpha \cdot e(\mathcal{A}_4) + e(\mathcal{A}_4, \mathcal{B})$.
\begin{lemma}\label{LEMMA:A4-B-bipartite-bound}
    For every $\alpha \in (0,1)$, there exists a constant $C_{\ref{LEMMA:A4-B-bipartite-bound}} = C_{\ref{LEMMA:A4-B-bipartite-bound}}(\alpha) > 0$ such that 
    \begin{align*}
        \alpha \cdot e(\mathcal{A}_4) + e(\mathcal{A}_4, \mathcal{B})
        \le g_{\alpha}(a_4, b) + C_{\ref{LEMMA:A4-B-bipartite-bound}} n. 
    \end{align*}
    %
\end{lemma}
\begin{lemma}\label{LEMMA:A4-B-10-partite-bound}
    There exists an absolute constant $C_{\ref{LEMMA:A4-B-10-partite-bound}} > 0$ such that 
    \begin{align*}
        \frac{3}{4}\cdot e(\mathcal{A}_4)+e(\mathcal{A}_4,\mathcal{B}) 
        \le \hat{g}(a_4,b) + C_{\ref{LEMMA:A4-B-10-partite-bound}} n.
    \end{align*}
\end{lemma}

Given $Q \in \mathcal{A}$ and $T = \{u,v,w\} \in \mathcal{B}$, we say $T$ \textbf{$1$-sees} $Q$ (see Figure~\ref{Fig:Def-1-see}) if there exist three distinct vertices $\{p,q,s\} \subseteq  Q$ such that $p$ is adjacent to both $u$ and $v$, $q$ is adjacent to both $v$ and $w$, and $s$ is adjacent to both $u$ and $w$. 

The following fact follows easily from the definitions (refer to the definitions of $3$-see and $2$-see in Section~\ref{SUBSEC:setup}). 
\begin{fact}\label{FACT:e-Q-T-7-edges}
    Let $Q \in \mathcal{A}$ and $T \in \mathcal{B}$. 
    Suppose that $e(Q, T) \ge 7$. 
    Then $T$ either $3$-sees, $2$-sees, or $1$-sees $Q$.
\end{fact}

\begin{figure}[H]
\centering
\tikzset{every picture/.style={line width=1pt}} 

\begin{tikzpicture}[x=0.75pt,y=0.75pt,yscale=-1,xscale=1]

\draw    (271.51,82.95) -- (220.25,115.73) ;
\draw    (271.51,82.95) -- (326.89,116.28) ;
\draw    (326.89,116.28) -- (220.25,115.73) ;
\draw    (220.25,115.73) -- (272.15,26.61) ;
\draw    (326.89,116.28) -- (272.15,26.61) ;
\draw    (271.51,82.95) -- (272.15,26.61) ;
\draw [fill=uuuuuu] (271.51,82.95) circle (1.5pt);
\draw [fill=uuuuuu] (220.25,115.73) circle (1.5pt);
\draw [fill=uuuuuu] (326.89,116.28) circle (1.5pt);
\draw [fill=uuuuuu] (272.15,26.61) circle (1.5pt);
\draw    (371,74.16) -- (419.19,105.97) ;
\draw    (371,74.16) -- (419.09,42.67) ;
\draw    (419.09,42.67) -- (419.19,105.97) ;
\draw [fill=uuuuuu] (371,74.16) circle (1.5pt);
\draw [fill=uuuuuu] (419.19,105.97) circle (1.5pt);
\draw [fill=uuuuuu] (419.09,42.67) circle (1.5pt);
\draw  [dash pattern=on 1pt off 1.2pt]  (272.15,26.61) -- (419.09,42.67) ;
\draw  [dash pattern=on 1pt off 1.2pt]  (272.15,26.61) -- (371,74.16) ;
\draw  [dash pattern=on 1pt off 1.2pt]  (271.51,82.95) -- (371,74.16) ;
\draw  [dash pattern=on 1pt off 1.2pt]  (271.51,82.95) -- (419.19,105.97) ;
\draw  [dash pattern=on 1pt off 1.2pt]  (326.89,116.28) -- (419.09,42.67) ;
\draw  [dash pattern=on 1pt off 1.2pt]  (326.89,116.28) -- (419.19,105.97) ;
\end{tikzpicture}
\caption{Auxiliary figure illustrating the definition of $1$-see.} 
\label{Fig:Def-1-see}
\end{figure}
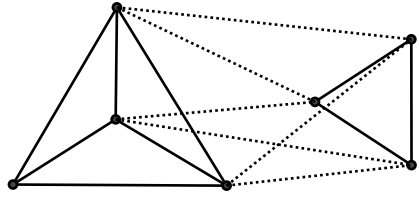

The key ingredient in proofs of Lemmas~\ref{LEMMA:A4-B-bipartite-bound} and~\ref{LEMMA:A4-B-10-partite-bound} is as follows. 
\begin{lemma}\label{LEMMA:A4-B-P4-free}
    There do not exist four distinct members $Q_1,Q_2,Q_3,Q_4 \in \mathcal{A}_4$ and three distinct members $T_1,T_2,T_3 \in \mathcal{B}$ such that, for every $i \in [3]$, 
    \begin{enumerate}[label=(\roman*)]
        \item\label{LEMMA:A4-B-P4-free-1} $e(Q_i,Q_{i+1})=15$, and  
        \item\label{LEMMA:A4-B-P4-free-2} $T_i$ $1$-sees both $Q_i$ and $Q_{i+1}$. 
    \end{enumerate}
\end{lemma}
\textbf{Remark.} The constraint $e(Q_2, Q_3) = 15$ can be replaced by $e(Q_2, Q_3) \ge 14$. 
\begin{proof}[Proof of Lemma~\ref{LEMMA:A4-B-P4-free}]
    Suppose to the contrary that this lemma fails. Fix distinct members $Q_1,Q_2,Q_3,Q_4 \in \mathcal{A}_4$ and distinct members $T_1,T_2,T_3 \in \mathcal{B}$ such that, 
    for every $i \in [3]$, $e(Q_i,Q_{i+1})=15$ and $T_i$ $1$-sees both $Q_i$ and $Q_{i+1}$.

First, let us consider $T_2$, $Q_2$, and $Q_3$. 
Assume that $T_2 = \{u_2, v_2, w_2\}$. 
\begin{itemize}
    \item Let $p_2$ and $p_3$ denote the common neighbors of $u_2, v_2$ in $Q_2$ and $Q_3$, respectively. 
    \item Let $q_2$ and $q_3$ denote the common neighbors of $v_2, w_2$ in $Q_2$ and $Q_3$, respectively. 
    \item Let $s_2$ and $s_3$ denote the common neighbors of $u_2, w_2$ in $Q_2$ and $Q_3$, respectively. 
\end{itemize}
Note that the existence of $\{p_2, p_3, q_2, q_3, s_2, s_3\}$ follows from Assumption~\ref{LEMMA:A4-B-P4-free-2} and the definition of $1$-see. 
Since $e(Q_2, Q_3) = 15 \ge 16 - 2$, at least one pair in $\{p_2p_3, q_2q_3, s_2s_3\}$ is an edge in $G$, and by symmetry, we may assume that $p_2p_3 \in G$. Note that $\{p_2, p_3, u_2, v_2\}$ induces a copy of $K_4$ in $G$. 

For a similar reason, we know that 
\begin{itemize}
    \item there exist $\{u_1, v_1\} \subseteq  T_1$, $p_1 \in Q_1$, and $q_2 \in Q_2 \setminus \{p_2\}$ such that $\{p_1, q_2, u_1, v_1\}$ induces a copy of $K_4$ in $G$, and 
    \item there exist $\{u_3, v_3\} \subseteq  T_3$, $q_3 \in Q_3 \setminus \{p_3\}$, and $p_4 \in Q_4$ such that $\{q_3, p_4, u_3, v_3\}$ induces a copy of $K_4$ in $G$. 
\end{itemize}

Since $e(Q_1 \setminus \{p_1\}, Q_2 \setminus \{p_2, q_2\}) \ge 15 - (2\cdot 4 + 2) = 5 = 6-1$, it follows from Fact~\ref{FACT:3-by-4-bipartite-graph}~\ref{FACT:3-by-4-bipartite-graph-1} that there exists a vertex $s_2 \in Q_2 \setminus \{p_2, q_2\}$ that is adjacent to all three vertices in $Q_1 \setminus \{p_1\}$. 
Similarly, there exists a vertex $s_3 \in Q_3 \setminus \{p_3, q_3\}$ that is adjacent to all three vertices in $Q_4 \setminus \{p_4\}$. 
However, the rotation shown in Figure~\ref{Fig:Q4BCD-a} would increase the size of $\mathcal{A}$, contradicting the maximality of $(|\mathcal{A}|, |\mathcal{B}|, |\mathcal{C}|, |\mathcal{D}|)$. 
\end{proof}

\medskip

\begin{figure}[H]
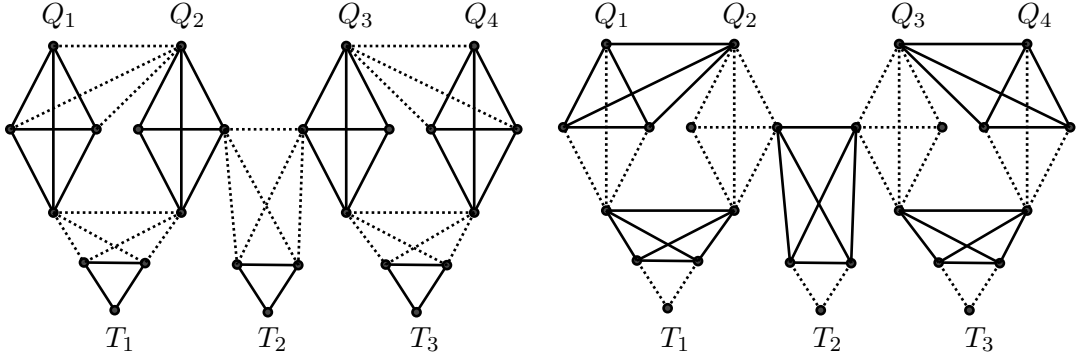

\centering

\tikzset{every picture/.style={line width=1pt}} 


\caption{A rotation that increases the number of vertex-disjoint copies of $K_4$.} 
\label{Fig:Q4BCD-a}
\end{figure}

Let us first present the proof of Lemma~\ref{LEMMA:A4-B-bipartite-bound}. 
\begin{proof}[Proof of Lemma~\ref{LEMMA:A4-B-bipartite-bound}]
    Fix $\alpha \in (0,1)$. We use $O_\alpha(n)$ to denote a quantity whose absolute value is at most $C_\alpha n$ for some constant $C_\alpha>0$ depending only on $\alpha$.
    Recall that $b = |\mathcal{B}|$. 
    Suppose that $b < 7$. 
    The definition of $g_\alpha$ gives $g_\alpha(a_4,b)\ge15\alpha a_4^2/2$.  Since $b\le6$ and $4a_4\le n$, we also have $12a_4b\le18n$.     
    Therefore, it follows from  Lemma~\ref{LEMMA:Q4Qi}~\ref{LEMMA:Q4Qi-1} that 
    \begin{align*}
        \alpha \cdot e(\mathcal{A}_4) + e(\mathcal{A}_4, \mathcal{B})
        \le \alpha \cdot \frac{15 a_4^2}{2} + 3\cdot 4\cdot a_4 b
        \le g_{\alpha}(a_4, b) + O_{\alpha}(n),
    \end{align*}
    which proves Lemma~\ref{LEMMA:A4-B-bipartite-bound}. 
    So it suffices to consider the case $b \ge 7$. 

    Define an auxiliary graph $H$ whose vertex set is $\mathcal{A}_4$ and two members $Q, Q' \in \mathcal{A}_4$ are adjacent in $H$ iff $e(Q, Q') = 15$. 
    It follows from Lemma~\ref{LEMMA:Q4Qi}~\ref{LEMMA:Q4Qi-1} that 
    \begin{align}\label{equ:A4-B-H}
        e(\mathcal{A}_4)
        \le 15 |H| + 14 \left(\binom{a_4}{2} - |H|\right) + 6 a_4 
        = 7 a_4^2  + |H| - a_4.
    \end{align}
    Define an auxiliary bipartite graph $F$ whose vertex set is $\mathcal{A}_4 \cup \mathcal{B}$ and a pair of members $(Q,T) \in \mathcal{A}_4 \times \mathcal{B}$ are adjacent in $F$ iff $e(Q, T) \ge 7$ (recall from Lemma~\ref{LEMMA:Q4BCD}~\ref{LEMMA:Q4BCD-1} that $e(Q, B) \le 7$ for all but at most two members $B \in \mathcal{B}$). 
    Define 
    \begin{align*}
        Y 
        \coloneqq 
        \left\{ Q \in \mathcal{A}_{4} \colon d_{F}(Q) \ge \frac{b+7}{2} \right\}. 
    \end{align*}
    It follows from the definition of $Y$ and Lemma~\ref{LEMMA:Q4BCD}~\ref{LEMMA:Q4BCD-1} that
    \begin{align}\label{equ:A4-B-Y1}
        e(\mathcal{A}_4, \mathcal{B})
        & \le |Y| \cdot \left(7 (b - 2) + 2 \cdot 3\cdot 4\right) \notag  \\[0.5em]
        & \quad + \left(|\mathcal{A}_4| - |Y|\right) \cdot \left(7\cdot \frac{b+7}{2} + 6 \cdot \left(b-2-\frac{b+7}{2}\right) + 2 \cdot 3\cdot 4\right) \notag \\[0.5em]
        & = \frac{13 a_4 b}{2} + \frac{b |Y|}{2} + \frac{31 a_4}{2} - \frac{11 |Y|}{2}
        \le \frac{13 a_4 b}{2} + \frac{b |Y|}{2} + 16 n.
    \end{align}
    
    \begin{claim}\label{CLAIM:A4-B-bipartite-bound-Y1}
        The induced subgraph $H[Y]$ is $P_4$-free. Thus, by Theorem~\ref{THM:Erdos-Gallai-path-Turan}, $|H[Y]| \le |Y|$. 
    \end{claim}
    \begin{proof}[Proof of Claim~\ref{CLAIM:A4-B-bipartite-bound-Y1}]
        Suppose to the contrary that there exist a copy of $P_4$, say $Q_1Q_2Q_3Q_4$, in $H[Y]$. 
        For each $i \in [3]$, let $N_i \coloneqq N_{F}(Q_i) \cap N_{F}(Q_{i+1})$. 
        It follows from the definition of $Y$ that $|N_i| \ge 2 \cdot \frac{b+7}{2} - b = 7$. 
        By the definition of $\mathcal{A}_{4}$, every member in $\mathcal{A}_4$ is $3$-seen or $2$-seen by at most two members in $\mathcal{B}$. Therefore, by Fact~\ref{FACT:e-Q-T-7-edges}, for each $i \in [3]$, there are at least three members in $N_i$ that $1$-see both $Q_i$ and $Q_{i+1}$. 
        So we can greedily select three distinct members $(T_1, T_2, T_3) \in N_1 \times N_2 \times N_3$ such that $T_i$ $1$-sees both $Q_i$ and $Q_{i+1}$ for $i \in [3]$. However, this contradicts Lemma~\ref{LEMMA:A4-B-P4-free}. 
    \end{proof}
    It follows from Claim~\ref{CLAIM:A4-B-bipartite-bound-Y1} that 
    \begin{align*}
        |H|
        \le |Y| + |Y|(a_4-|Y|) + \binom{a_4 - |Y|}{2}
        = \frac{a_{4}^2}{2} - \frac{|Y|^2}{2} + \frac{3|Y|}{2} - \frac{a_4}{2}.
    \end{align*}
    Combining this with~\eqref{equ:A4-B-H}, we obtain
    \begin{align*}
        e(\mathcal{A}_{4})
        \le 7 a_4^2  + \frac{a_{4}^2}{2} - \frac{|Y|^2}{2} + \frac{3|Y|}{2} - \frac{a_4}{2} - a_4
        \le \frac{15 a_{4}^2}{2} - \frac{|Y|^2}{2}. 
    \end{align*}
    Combining this with~\eqref{equ:A4-B-Y1}, we obtain 
    \begin{align*}
        \alpha \cdot e(\mathcal{A}_{4}) + e(\mathcal{A}_{4}, \mathcal{B})
        & \le \alpha\left(\frac{15 a_{4}^2}{2} - \frac{|Y|^2}{2}\right)  + \frac{13 a_4 b}{2} + \frac{b |Y|}{2} + 16 n \\[0.5em]
        & = -\frac{\alpha}{2} \left(|Y| - \frac{b}{2\alpha}\right)^2 + \frac{b^2}{8\alpha} + \frac{15 \alpha a_{4}^2}{2}  + \frac{13 a_4 b}{2} + 16n.
    \end{align*}
    Viewing $-\frac{\alpha}{2} \left(|Y| - \frac{b}{2\alpha}\right)^2 + \frac{b^2}{8\alpha} + \frac{15 \alpha \cdot a_{4}^2}{2}  + \frac{13 a_4 b}{2} + 16n$ as a quadratic polynomial in $|Y|$, and noting that $|Y| \in [0, a_4]$, we obtain 
    \begin{align*}
        \alpha \cdot e(\mathcal{A}_{4}) + e(\mathcal{A}_{4}, \mathcal{B})
        & \le 
        \begin{cases}
            7 \alpha a_4^2 + 7a_4 b +16n, &\quad\text{if}\quad a_4 \le \frac{b}{2\alpha}, \\[0.5em]
            \frac{15 \alpha a_{4}^2}{2}  + \frac{13 a_4 b}{2} + \frac{b^2}{8\alpha} + 16n, &\quad\text{if}\quad a_4 > \frac{b}{2\alpha}.
        \end{cases}
    \end{align*}
    %
    %
    This completes the proof of Lemma~\ref{LEMMA:A4-B-bipartite-bound}. 
\end{proof}

Next, we present the proof of  Lemma~\ref{LEMMA:A4-B-10-partite-bound}. 
\begin{proof}[Proof of Lemma~\ref{LEMMA:A4-B-10-partite-bound}]
    Similar to the proof of Lemma~\ref{LEMMA:A4-B-bipartite-bound}, we may assume that $b \ge 10^6$, since otherwise, by Lemma~\ref{LEMMA:Q4Qi}~\ref{LEMMA:Q4Qi-1}, we are done. 

    Let $H$ and $F$ be the same auxiliary graphs as defined in the proof of Lemma~\ref{LEMMA:A4-B-bipartite-bound}.
    For $i \in [2,10]$, define 
    \begin{align*}
        Z_{i}
        \coloneqq 
        \left\{ Q\in \mathcal{A}_{4} \colon d_{F}(Q) \in \left(\frac{(i-1)b}{10}+10^5,~\frac{i \cdot b}{10}+10^5\right]\right\}. 
    \end{align*}
    Additionally, let $Z_1 \coloneqq \mathcal{A}_{4} \setminus (Z_2 \cup \cdots \cup Z_{10})$. 
    Let $z_i \coloneqq |Z_i|$ for $i \in [10]$.

    \begin{claim}\label{CLAIM:A4-B-Zi-Zj-P4-free}
        The following statements hold. 
        \begin{enumerate}[label=(\roman*)]
            \item\label{CLAIM:A4-B-Zi-Zj-P4-free-1} The induced graph $H[Z_i]$ is $P_4$-free for every $i \in [6,10]$. 
            \item\label{CLAIM:A4-B-Zi-Zj-P4-free-2} The induced bipartite graph $H[Z_i,Z_j]$ is $P_4$-free for every pair $\{i,j\} \subseteq  [10]$ satisfying $i + j \ge 12$.
        \end{enumerate}
    \end{claim}
    \begin{proof}[Proof of Claim~\ref{CLAIM:A4-B-Zi-Zj-P4-free}]
        Suppose to the contrary that one of the graphs in \ref{CLAIM:A4-B-Zi-Zj-P4-free-1} or \ref{CLAIM:A4-B-Zi-Zj-P4-free-2} contains a copy of $P_4$, say $Q_1Q_2Q_3Q_4$. 
        By the definition of $Z_i$, for each $\ell\in[3]$, 
        \[
        |N_F(Q_\ell)\cap N_F(Q_{\ell+1})|\ge d_F(Q_\ell)+d_F(Q_{\ell+1})-b\ge2\cdot10^5.
        \]
        As in the proof of Claim~\ref{CLAIM:A4-B-bipartite-bound-Y1}, we can choose distinct $T_1,T_2,T_3\in\mathcal{B}$ such that $T_\ell$ $1$-sees both $Q_\ell$ and $Q_{\ell+1}$ for each $\ell\in[3]$, contradicting Lemma~\ref{LEMMA:A4-B-P4-free}. 
    \end{proof}
    %
    By Claim~\ref{CLAIM:A4-B-Zi-Zj-P4-free}~\ref{CLAIM:A4-B-Zi-Zj-P4-free-1} and Theorem~\ref{THM:Erdos-Gallai-path-Turan}, for every $i \in [6,10]$, we have 
    \begin{align}\label{equ:A4-B-H-Zi}
        |H[Z_i]|
        \le \mathrm{ex}(z_i, P_4)
        \le z_i.
    \end{align}
    By Claim~\ref{CLAIM:A4-B-Zi-Zj-P4-free}~\ref{CLAIM:A4-B-Zi-Zj-P4-free-2} and Theorem~\ref{THM:GRS84-bipartite-path}, for every pair $\{i,j\} \subseteq  [10]$ satisfying $i + j \ge 12$, we have 
    \begin{align}\label{equ:A4-B-H-Zi-Zj}
        |H[Z_i, Z_j]|
        \le \mathrm{ex}(z_i, z_j, P_4)
        \le z_i + z_j.
    \end{align}
    \begin{claim}\label{CLAIM:A4-B-Zi-L-free}
        The following statements hold. 
        \begin{enumerate}[label=(\roman*)]
            \item\label{CLAIM:A4-B-Z2-L30-free} The induced subgraph $H[Z_2]$ is $L_{30}$-free. 
            \item\label{CLAIM:A4-B-Z3-L15-free} The induced subgraph $H[Z_3]$ is $L_{15}$-free. 
            \item\label{CLAIM:A4-B-Z4-L10-free} The induced subgraph $H[Z_4]$ is $L_{10}$-free. 
            \item\label{CLAIM:A4-B-Z5-L8-free} The induced subgraph $H[Z_5]$ is $L_{8}$-free. 
        \end{enumerate}
        Consequently,  by Theorem~\ref{THM:Simonovits74}, there exists a constant $C_{\ref{CLAIM:A4-B-Zi-L-free}} > 0$ such that 
        \begin{align}
            |H[Z_2]|
            & \le \mathrm{ex}(z_2,L_{30})
                \le \frac{13 z_2^2}{28} + C_{\ref{CLAIM:A4-B-Zi-L-free}} z_2, \label{equ:H-Z2-upper-bound}\\[0.5em]
            |H[Z_3]|
            & \le \mathrm{ex}(z_3,L_{15})
                \le \frac{3 z_3^2}{7} + C_{\ref{CLAIM:A4-B-Zi-L-free}} z_3, \label{equ:H-Z3-upper-bound}\\[0.5em]
            |H[Z_4]|
            & \le \mathrm{ex}(z_4,L_{10})
                \le \frac{3 z_4^2}{8} + C_{\ref{CLAIM:A4-B-Zi-L-free}} z_4, \label{equ:H-Z4-upper-bound}\\[0.5em] 
            |H[Z_5]|
            & \le \mathrm{ex}(z_5,L_{8})
                \le \frac{z_5^2}{3} + C_{\ref{CLAIM:A4-B-Zi-L-free}} z_5. \label{equ:H-Z5-upper-bound}
        \end{align} 
    \end{claim}
    \begin{proof}[Proof of Claim~\ref{CLAIM:A4-B-Zi-L-free}]
        We first prove Claim~\ref{CLAIM:A4-B-Zi-L-free}~\ref{CLAIM:A4-B-Z2-L30-free}.
        Suppose to the contrary that there exists a copy of $L_{30}$ in $H[Z_2]$, say on the vertex set $S \subseteq  Z_2$ of size $30$. 
        Let us consider the induced bipartite subgraph $F[S, \mathcal{B}]$. 
        It follows from the definition of $Z_2$ that 
        \begin{align*}
            |F[S, \mathcal{B}]|
            \ge |S| \cdot \left(\frac{2-1}{10} b + 10^5 \right)
            = 3b + 30 \cdot 10^5
            > Z(|S|, |\mathcal{B}|, 4, 11),
        \end{align*}
        where the last inequality follows from Theorem~\ref{THM:Zarankiewicz-number}. By the definition of the bipartite graph $F$, this means that there exist four distinct members $Q_1, Q_2, Q_3, Q_4 \in Z_2$ and $11$ distinct members $T_1, \ldots, T_{11} \in \mathcal{B}$ such that $e(Q_i, T_j) \ge 7$ for every $(i,j) \in [4] \times [11]$. 
        
        By the definition of $\mathcal{A}_{4}$, each $Q_i$ is $3$-seen or $2$-seen by at most two members in $\mathcal{B}$ (otherwise, $Q_i$ would belong to $\mathcal{A}_1 \cup \mathcal{A}_2 \cup \mathcal{A}_3$). Combining this with Fact~\ref{FACT:e-Q-T-7-edges}, we conclude that there are at least $11 - 4\cdot 2 = 3$ members, say $T_1, T_2, T_3$, such that $T_i$ $1$-sees $Q_j$ for  every $(i,j) \in [3] \times [4]$. 
        Furthermore, it follows from the definition of $L_t$ (see Section~\ref{SEC:Prelim}) and Fact~\ref{FACT:Lt-4-vtx} that $H[\{Q_1, Q_2, Q_3, Q_4\}]$ contains a copy of $P_4$.
        This leads to a contradiction with Lemma~\ref{LEMMA:A4-B-P4-free}. 

        For~\ref{CLAIM:A4-B-Z3-L15-free}--\ref{CLAIM:A4-B-Z5-L8-free}, use the same argument with $(j,t)\in\{(3,15),(4,10),(5,8)\}$. Since $t(j-1)\ge30$, a copy of $L_t$ on $S\subseteq Z_j$ gives 
        \[
        |F[S,\mathcal{B}]|\ge t\left(\frac{(j-1)b}{10}+10^5\right)>3b+10\binom{t}{4}=Z(t,|\mathcal{B}|,4,11).
        \]
        The same choices of four members and three triangles again contradict Lemma~\ref{LEMMA:A4-B-P4-free}.
    \end{proof}
    Recall that $z_1 + \cdots + z_{10} = a_4$. 
    It follows from~\eqref{equ:A4-B-H-Zi},~\eqref{equ:A4-B-H-Zi-Zj},~\eqref{equ:H-Z2-upper-bound},~\eqref{equ:H-Z3-upper-bound},~\eqref{equ:H-Z4-upper-bound},~\eqref{equ:H-Z5-upper-bound} and the trivial upper bound $|H[Z_i, Z_j]| \le z_i z_j$ for $\{i,j\} \subseteq [10]$ with $i + j \le 11$ that 
    \begin{align*}
        |H|
        & = \sum_{i\in [10]}|H[Z_i]| + \sum_{\{i,j\}\subseteq  [10]} |H[Z_i, Z_j]| \\[0.5em]
        & \le \frac{z_1^2}{2} + \frac{13 z_2^2}{28} + \frac{3 z_3^2}{7} + \frac{3 z_4^2}{8} + \frac{z_5^2}{3} + \sum_{i \in [5]} \left(z_i \cdot \sum_{j \in [i+1, 11-i]} z_j\right) + O(n)\\[0.5em]
        & = \frac{a_4^2}{2} + \hat{\eta}(z_2, \ldots, z_{10}) +O(n), 
    \end{align*}
    where 
    \begin{align*}
        \hat{\eta}(z_2, \ldots, z_{10})
        \coloneqq 
        & - \biggl(\frac{z_2^2}{28} +\frac{z_3^2}{14} + \frac{z_4^2}{8} + \frac{z_5^2}{6} + \frac{(z_6+\cdots+z_{10})^2}{2} \\[0.5em]
        & \quad + z_5z_7 + (z_4+z_5)z_8 + (z_3+z_4+z_5)z_9 + (z_2+z_3+z_4+z_5)z_{10} \bigg). 
    \end{align*}
    Combining this with~\eqref{equ:A4-B-H}, we obtain 
    \begin{align}\label{equ:A4-B-e-A4-psi}
        e(\mathcal{A}_{4})
        \le \frac{15 a_{4}^2}{2} + \hat{\eta}(z_2, \ldots, z_{10}) +O(n). 
    \end{align}
    Next, we consider $e(\mathcal{A}_4, \mathcal{B})$.
    For every $i \in [10]$ and $Q \in Z_i$, it follows from the definition of $Z_i$ and Lemma~\ref{LEMMA:Q4BCD}~\ref{LEMMA:Q4BCD-1} that 
    \begin{align*}
        e(Q, \mathcal{B})
        \le 2\cdot 3 \cdot 4 + 7 \left(\frac{i\cdot b}{10} + 10^5 -2\right) + 6 \left(b -\frac{i\cdot b}{10} - 10^5\right)
        = 6b + \frac{i \cdot b}{10} + 10^5 + 10. 
    \end{align*}
    %
    %
    Therefore, 
    \begin{align*}
        e(\mathcal{A}_{4}, \mathcal{B})
         = \sum_{i\in [10]}\sum_{Q\in Z_i} e(Q, \mathcal{B}) 
        & \le \sum_{i\in [10]} z_i \left(6b + \frac{i \cdot b}{10} + 10^5 + 10\right) \\[0.5em]
        & = \frac{61 a_4 b}{10} + \sum_{i\in [2,10]} \frac{(i-1) b z_i}{10} +(10^5 + 10)a_4. 
    \end{align*}
    Combining this with~\eqref{equ:A4-B-e-A4-psi}, we obtain 
    \begin{align*}
        \frac{3}{4} \cdot e(\mathcal{A}_4) + e(\mathcal{A}_{4}, \mathcal{B})
        & \le \frac{3}{4}\left(\frac{15 a_{4}^2}{2} + \hat{\eta}(z_2, \ldots, z_{10})\right) + \frac{61 a_4 b}{10} + \sum_{i\in [2,10]} \frac{(i-1) b z_i}{10}  + O(n) \\[0.5em]
        & = \frac{45 a_4^2}{8} + \frac{61 a_4 b}{10} + \frac{3}{4} \cdot \hat{\eta}(z_2, \ldots, z_{10}) + \sum_{i\in [2,10]} \frac{(i-1) b z_i}{10}  + O(n)  \\[0.5em]
        & = \frac{45 a_4^2}{8} + \frac{61 a_4 b}{10} + \eta(b,z_2, \ldots, z_{10}) + O(n), 
    \end{align*}
    where 
    \begin{align*}
        \eta(b,z_2, \ldots, z_{10})
        & \coloneqq \frac{3}{4} \cdot \hat{\eta}(z_2, \ldots, z_{10}) + \sum_{i\in [2,10]} \frac{(i-1) b z_i}{10} \\[0.5em]
        & = \sum_{i\in [2,10]} \frac{(i-1) b z_i}{10} - \biggl(\frac{3 z_2^2}{112} +\frac{3 z_3^2}{56} + \frac{3 z_4^2}{32} + \frac{z_5^2}{8} + \frac{3 (z_6+\cdots+z_{10})^2}{8} \\[0.5em]
        & \quad + \frac{3}{4} \bigl(z_5z_7 + (z_4+z_5)z_8 + (z_3+z_4+z_5)z_9 + (z_2+z_3+z_4+z_5)z_{10} \bigl) \bigg).
    \end{align*}
    The lemma now follows from some elementary but tedious calculations to establish the desired upper bound for $\eta(b,z_2, \ldots, z_{10})$, which are presented in the Appendix (see Proposition~\ref{PROP:inequality-psi-upper-bound}).
\end{proof}

\subsection{$\mathcal{A}_{4}$ and $\mathcal{C}$}\label{SUBSEC:A4C}
In this subsection, we prove the following upper bound for $\alpha \cdot e(\mathcal{A}_4) + e(\mathcal{A}_4, \mathcal{C})$. 
\begin{lemma}\label{LEMMA:A4-C-bipartite-bound}
    For every $\alpha \in (0,1)$, there exists a constant $C_{\ref{LEMMA:A4-C-bipartite-bound}} = C_{\ref{LEMMA:A4-C-bipartite-bound}}(\alpha) > 0$ such that 
    \begin{align*}
        \alpha \cdot e(\mathcal{A}_4) + e(\mathcal{A}_4, \mathcal{C})
        \le h_{\alpha}(a_4, c) + C_{\ref{LEMMA:A4-C-bipartite-bound}} n. 
    \end{align*}
    %
\end{lemma}
Recall from Section~\ref{SUBSEC:setup} that a member $M = \{u_1, u_2\} \in \mathcal{C}$ sees a member $Q \in \mathcal{A}$ if there exist two vertices $v_1, v_2 \in Q$ such that both $u_1$ and $u_2$ are adjacent to $v_1$ and $v_2$. 
\begin{fact}\label{FACT:AC-not-see-5-edges}
    Let $Q \in \mathcal{A}$ and $M \in \mathcal{C}$. Suppose that $e(Q, M) \ge 5$ but $M$ does not see $Q$. Then there exists a unique vertex in $Q$ that is adjacent to both vertices in $M$, while each of the remaining vertices is adjacent to exactly one vertex in $M$. In particular, $e(Q, M) = 5$.
\end{fact}

The key ingredient in the proof of Lemma~\ref{LEMMA:A4-C-bipartite-bound} is as follows. 
\begin{lemma}\label{LEMMA:A4-C-P3-free}
    There do not exist three distinct members $Q_1,Q_2,Q_3 \in \mathcal{A}_4$ and two distinct members $M_1,M_2 \in \mathcal{C}$ such that, for every $i \in [2]$, 
    \begin{enumerate}[label=(\roman*)]
        \item\label{LEMMA:A4-C-P3-free-1} $e(Q_{i},Q_{i+1})=15$, 
        \item\label{LEMMA:A4-C-P3-free-2} $\min\left\{e(Q_{i},M_{i}),~e(Q_{i+1},M_{i}) \right\} \ge 5$ but $M_{i}$ does not see $Q_{i}$ or $Q_{i+1}$. 
    \end{enumerate}
\end{lemma}
\begin{proof}[Proof of Lemma~\ref{LEMMA:A4-C-P3-free}]
    Suppose to the contrary that this claim fails. Fix three distinct members $Q_1,Q_2,Q_3 \in \mathcal{A}_4$ and two distinct members $M_1,M_2 \in \mathcal{C}$ such that, for $i \in [2]$,~\ref{LEMMA:A4-C-P3-free-1} and~\ref{LEMMA:A4-C-P3-free-2} hold. 
    It follows from~\ref{LEMMA:A4-C-P3-free-2} and Fact~\ref{FACT:AC-not-see-5-edges} that, for each $i \in [2]$, the edges between $M_i$ and $Q_i$ (as well as those between $M_i$ and $Q_{i+1}$) are distributed as follows: there exists a unique vertex in $Q_i$ (resp. $Q_{i+1}$) that is adjacent to both vertices in $M_i$, while each of the remaining three vertices is adjacent to exactly one vertex in $M_i$. 
    For each $i \in [2]$, let $p_i$ denote the vertex in $Q_i$ that is adjacent to both vertices in $M_1$.

\medskip

\begin{figure}[H]
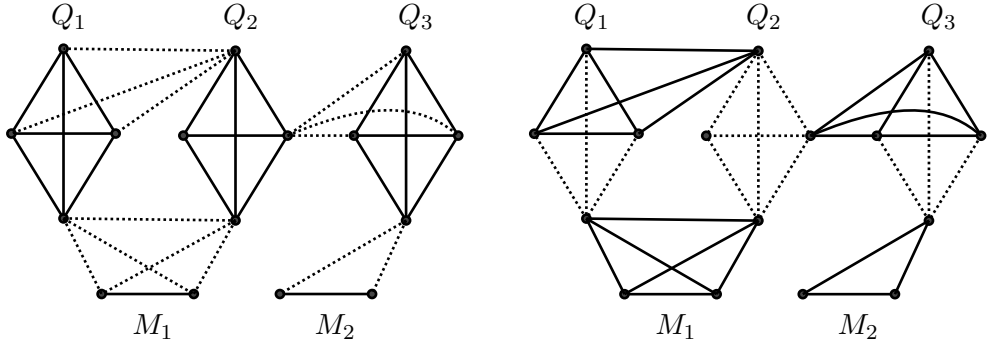

\centering

\tikzset{every picture/.style={line width=1pt}} 


\caption{Auxiliary figure for the proof of Lemma~\ref{LEMMA:A4-C-P3-free} Case 1.} 
\label{Fig:Q4BCD-b1}
\end{figure}

\textbf{Case 1}: $p_1p_2 \in G$. 

    Let $p_3$ denote the vertex in $Q_3$ that is adjacent to both vertices in $M_2$ (the existence of $p_3$ is guaranteed by~\ref{LEMMA:A4-C-P3-free-2} and Fact~\ref{FACT:AC-not-see-5-edges}). Since $e(Q_2, Q_1) = 15$, there exists a vertex $q_2 \in Q_{2} \setminus \{p_2\}$ that is adjacent to all three vertices in $Q_1 \setminus \{p_1\}$. Similarly, there exists a vertex $s_2 \in Q_{2} \setminus \{p_2, q_2\}$ that is adjacent to all three vertices in $Q_{3}\setminus \{p_3\}$. 
    However, the rotation shown in Figure~\ref{Fig:Q4BCD-b1} would increase the size of $\mathcal{B}$ while keeping the size of $\mathcal{A}$ unchanged, contradicting the maximality of $\mathcal{B}$. 

\medskip 

\begin{figure}[H]
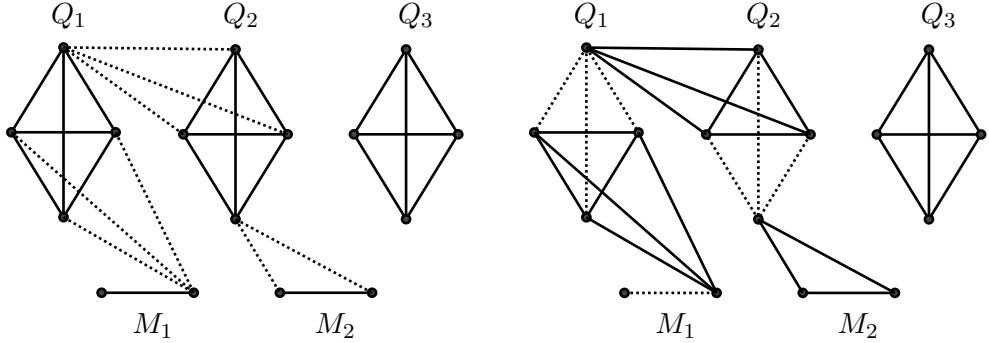

\centering

\tikzset{every picture/.style={line width=1pt}} 


\caption{Auxiliary figure for the proof of Lemma~\ref{LEMMA:A4-C-P3-free} Case 2.} 
\label{Fig:Q4BCD-b2}
\end{figure}

\textbf{Case 2}: $p_1p_2 \not\in G$.

    Similar to the proof of Case 1, let $q_2$ denote the vertex in $Q_2$ that is adjacent to both vertices in $M_2$ (it might hold that $p_2 = q_2$). 
    Note that, by Fact~\ref{FACT:AC-not-see-5-edges}, there exists a vertex $u_1 \in M_1$ that is adjacent to at least three vertices in $Q_1$, including $p_1$. Fix three such vertices (including $p_1$), and let $t_1$ denote the remaining vertex in $Q_1$. Since $p_1p_2 \not\in G$ and $e(Q_1, Q_2) = 15 = 16-1$, the vertex $t_1$ must be adjacent to all three vertices in $Q_{2} \setminus \{q_2\}$. However, the rotation shown in Figure~\ref{Fig:Q4BCD-b2} would increase the size of $\mathcal{B}$ while keeping the size of $\mathcal{A}$ unchanged, contradicting the maximality of $\mathcal{B}$. 

    This completes the proof of Lemma~\ref{LEMMA:A4-C-P3-free}. 
\end{proof}

We are now ready to present the proof of Lemma~\ref{LEMMA:A4-C-bipartite-bound}.
\begin{proof}[Proof of Lemma~\ref{LEMMA:A4-C-bipartite-bound}]
    Fix $\alpha \in (0,1)$. Recall that $c = |\mathcal{C}|$. Similar to the proof of Lemma~\ref{LEMMA:A4-B-bipartite-bound}, we may assume that $c$ is large, say $c \ge 4$. 

Define an auxiliary graph $H$ whose vertex set is $\mathcal{A}_4$ and two members $Q, Q' \in \mathcal{A}_4$ are adjacent in $H$ iff $e(Q,Q') = 15$. 
Recall from Lemma~\ref{LEMMA:Q4Qi}~\ref{LEMMA:Q4Qi-1} that $e(Q,Q') \le 15$ for every pair $Q, Q' \in \mathcal{A}_4$, so the same counting argument
as in~\eqref{equ:A4-B-H} gives
\begin{align}\label{equ:A4C-H-bound}
    e(\mathcal{A}_4) 
    \leq 15 |H| + 14 \left(\binom{a_4}{2} - |H| \right) + 6 a_4 
    = 7a_4^2 +|H| -a_4. 
\end{align}
Define an auxiliary bipartite graph $F$ whose vertex set is $\mathcal{A}_4 \cup  \mathcal{C}$ and a pair of members $(Q, M) \in \mathcal{A}_4 \times \mathcal{C}$ are adjacent in $F$ iff $e(Q,M) \geq 5$ (recall from Lemma~\ref{LEMMA:Q4BCD}~\ref{LEMMA:Q4BCD-2} that $e(Q,C)\leq 5$ for all but at most one member $C \in \mathcal{C}$). 
Define
\begin{align*}
    Y
    \coloneqq  \left\{Q \in \mathcal{A}_4 \colon d_{F}(Q) \geq \frac{c+4}{2} \right\}.
\end{align*}
It follows from the definition of $Y$ and Lemma~\ref{LEMMA:Q4BCD}~\ref{LEMMA:Q4BCD-2} that
\begin{align}\label{EQ:A4-C-eA4-C-bound}
    e(\mathcal{A}_4,\mathcal{C}) 
    & \leq  |Y| \cdot \left(5(c-1)+ 2 \cdot 4\right) \notag \\[0.5em]
    &\quad  + (|\mathcal{A}_4| - |Y|) \cdot \left(4 \cdot \left(c- \frac{c+4}{2} \right) + 5 \cdot \left(\frac{c+4}{2} -1\right)+ 2 \cdot 4  \right) \notag \\[0.5em]
    & = \frac{9 a_4 c}{2} + \frac{c |Y|}{2} + 5a_4 - 2|Y| \leq \frac{9 a_4 c}{2} + \frac{c |Y|}{2} + 5 n.
\end{align}

\begin{claim}\label{CLAIM:HY2-P3-free}
    The induced graph $H[Y]$ is $P_3$-free. Thus, by Theorem~\ref{THM:Erdos-Gallai-path-Turan},
    \begin{align*}
        |H[Y]| 
        \le \frac{|Y|}{2}.
    \end{align*}
\end{claim}
\begin{proof}[Proof of Claim~\ref{CLAIM:HY2-P3-free}]
Suppose to the contrary that there exist a copy of $P_3$, say $Q_1 Q_2 Q_3$, in $H[Y]$. 
For each $i \in [2]$, let $N_i := N_{F}(Q_i) \cap N_{F}(Q_{i+1})$, noting from the definition of $Y$ that $|N_i| \geq 2 \cdot \frac{c+4}{2} -c = 4$. 
By the definition of $\mathcal{A}_4$, every member in $\mathcal{A}_4$ is seen by at most one member in $\mathcal{C}$ (otherwise, it would belong to $\mathcal{A}_{2,3}$). 
Therefore, by the definition of the graph $F$, for each $i \in [2]$, there are at least $|N_i| - 1- 1 \ge 2$ members $M \in \mathcal{C}$ satisfying $\min\left\{e(Q_i,M),~e(Q_{i+1},M)\right\} \ge 5$ and not seeing $Q_i$ or $Q_{i+1}$. 
So we can greedily select two distinct members $(M_1,M_2) \in N_1 \times N_2$ such that, for $i \in [2]$, $\min\left\{e(Q_i,M_i), e(Q_{i+1},M_i)\right\} \ge 5$ and $M_i$ does not see $Q_i$ or $Q_{i+1}$. 
However, this is a contradiction to Lemma~\ref{LEMMA:A4-C-P3-free}.
\end{proof}

It follows from Claim~\ref{CLAIM:HY2-P3-free} that
\begin{align*}
|H| \leq \frac{|Y|}{2}+ |Y|(a_4 - |Y|) + \binom{a_4 - |Y|}{2} = \frac{a_4^2}{2} -\frac{|Y|^2}{2} + |Y| -\frac{a_4}{2}.
\end{align*}
Combining this with~\eqref{equ:A4C-H-bound}, we obtain
\begin{align*}
e(\mathcal{A}_4) \leq 7a_4^2+ \frac{a_4^2}{2} -\frac{|Y|^2}{2} + |Y| -\frac{a_4}{2} -a_4 \leq \frac{15 a_4^2}{2} -\frac{|Y|^2}{2}.
\end{align*}
Combining this with~\eqref{EQ:A4-C-eA4-C-bound}, we obtain 
\begin{align*}
    \alpha \cdot e(\mathcal{A}_4) + e(\mathcal{A}_4,\mathcal{C}) 
    &\leq \alpha \left(\frac{15 a_4^2}{2} -\frac{|Y|^2}{2} \right) + \frac{9 a_4 c}{2} + \frac{c |Y|}{2} + 5 n \\[0.5em]
    &= - \frac{\alpha}{2} \left(|Y|- \frac{c}{2 \alpha} \right)^2 + \frac{c^2}{8 \alpha}+ \frac{15 \alpha  a_4^2}{2} + \frac{9a_4 c}{2} + 5n.
\end{align*}
Viewing $- \frac{\alpha}{2} \left(|Y|- \frac{c}{2 \alpha} \right)^2 + \frac{c^2}{8 \alpha}+ \frac{15 \alpha a_4^2}{2} + \frac{9a_4 c}{2} + 5n$ as a quadratic polynomial in $|Y|$, and noting that $0 \leq |Y| \leq a_4$, we obtain
\begin{align*}
    \alpha \cdot e(\mathcal{A}_4) + e(\mathcal{A}_4,\mathcal{C}) 
    \leq 
    \begin{cases}
    7 \alpha a_4^2 + 5a_4 c + 5n,  &\quad\text{if}\quad a_4 \leq \frac{c}{2 \alpha}, \\[0.5em]
    \frac{15 \alpha a_4^2}{2} + \frac{9a_4 c}{2} + \frac{c^2}{8 \alpha} + 5n, &\quad\text{if}\quad a_4 > \frac{c}{2 \alpha}.
    \end{cases}.
\end{align*}
%
This completes the proof of Lemma~\ref{LEMMA:A4-C-bipartite-bound}. 
\end{proof}

\section{Local estimation \RomanNumeralCaps{4}$\colon$ $e(\mathcal{A}_6 \cup \mathcal{B} \cup \mathcal{C} \cup \mathcal{D})$}\label{SEC:A6BCD}
In this section, we prove improved upper bounds (Lemmas~\ref{LEMMA:A6BCD-upper-bound-b} and~\ref{LEMMA:A6BCD-upper-bound-a}) for $e(\mathcal{A}_6 \cup \mathcal{B} \cup \mathcal{C} \cup \mathcal{D})$, refining the upper bound  
\begin{align}\label{equ:A6BCD-union-trivial-upper-bound}
    e(\mathcal{A}_6 \cup \mathcal{B} \cup \mathcal{C} \cup \mathcal{D})
    & \le 8 a_6^2 + (7b+5c+2d)a_6 + 15 a_6 + \Psi(b,c,d),
\end{align}
which is derived from Lemma~\ref{LEMMA:Q5Q6BCD}, Fact~\ref{FACT:edges-ABCD}, and the trivial upper bound $e(\mathcal{A}_6) \le \binom{4a_6}{2} \le 8 a_6^2$.

\begin{lemma}\label{LEMMA:A6BCD-upper-bound-b}
    There exists an absolute constant $C_{\ref{LEMMA:A6BCD-upper-bound-b}} > 0$ such that 
    \begin{align*}
        e(\mathcal{A}_6 \cup \mathcal{B} \cup \mathcal{C} \cup \mathcal{D})
        \le C_{\ref{LEMMA:A6BCD-upper-bound-b}} n + \Psi(b,c,d) +
        \begin{cases}
            \frac{15 a_6^2}{2} + (7b+5c+2d)a_6, &\quad\text{if}\quad a_6 \le b+c, \\[0.5em]
            8 a_{6}^{2} + \left(\frac{13 b}{2} + \frac{9c}{2}+2d\right) a_6, &\quad\text{if}\quad a_6 \ge b+c.
        \end{cases}
    \end{align*}
\end{lemma}

\begin{lemma}\label{LEMMA:A6BCD-upper-bound-a}
    There exist constants $\mu > 0$ and $C_{\ref{LEMMA:A6BCD-upper-bound-a}} > 0$ such that if 
    \begin{align*}
        e(\mathcal{A}_6) \ge \frac{15 a_6^2}{2} +\mu a_6
        \quad\text{and}\quad 
        a_6
        \ge \frac{3b+2c+d}{2} = \frac{n-4k}{2},
    \end{align*}
    then 
    \begin{align*}
        e(\mathcal{A}_6 \cup \mathcal{B} \cup\mathcal{C} \cup\mathcal{D} )
        \le 8a_6^2 + (3b+2c+d)^2 + C_{\ref{LEMMA:A6BCD-upper-bound-a}} n.
    \end{align*}
\end{lemma}
%
Proof of Lemma~\ref{LEMMA:A6BCD-upper-bound-b} will be presented in Section~\ref{SUBSEC:A6BCD-I}. 
Proof of Lemma~\ref{LEMMA:A6BCD-upper-bound-a} will be presented in Section~\ref{SUBSEC:proof-LEMMA:A6BCD-upper-bound-a} after some necessary preparations.

\subsection{Proof of Lemma~\ref{LEMMA:A6BCD-upper-bound-b}}\label{SUBSEC:A6BCD-I}
In this subsection, we prove Lemma~\ref{LEMMA:A6BCD-upper-bound-b}. 
The key ingredient in the proof is as follows. 
\begin{lemma}\label{LEMMA:A7BCD-atom}
    Suppose that $Q_1, Q_2 \in \mathcal{A}_{6}$ are two distinct members satisfying $e(Q_1, Q_2) = 16$. Then the following statements hold. 
    \begin{enumerate}[label=(\roman*)]
        \item\label{LEMMA:A7BCD-atom-1} There is at most one member $T \in \mathcal{B}$ satisfying 
        \begin{align*}
            e(Q_1, T) \ge 7 \quad\text{and}\quad 
            e(Q_2, T) \ge 7. 
        \end{align*}
        \item\label{LEMMA:A7BCD-atom-2}  There is at most one member $M \in \mathcal{C}$ satisfying 
        \begin{align*}
            e(Q_1, M) \ge 5 \quad\text{and}\quad 
            e(Q_2, M) \ge 5. 
        \end{align*}
    \end{enumerate}
\end{lemma}
\begin{proof}[Proof of Lemma~\ref{LEMMA:A7BCD-atom}]
    Let us first prove~\ref{LEMMA:A7BCD-atom-1}. 
    Suppose to the contrary that there exist two members $T_1, T_2 \in \mathcal{B}$ such that 
    \begin{align*}
        \min\left\{e(Q_i, T_j) \colon (i,j) \in [2] \times [2]\right\} \ge 7.
    \end{align*}
    Assume that $Q_{i} = \{p_i, q_i, s_i, t_i\}$ and $T_i = \{u_i, v_i, w_i\}$ for $i \in [2]$.

\medskip 

\begin{figure}[H]
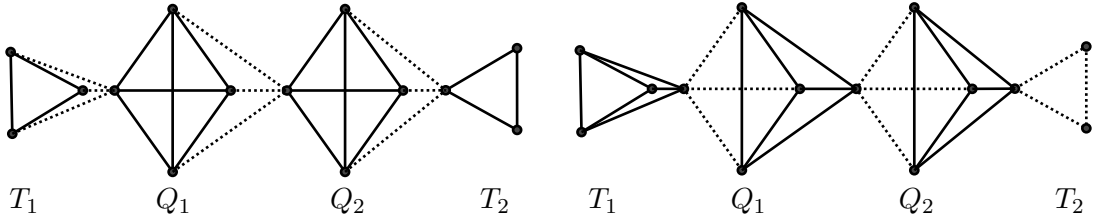

\centering

\tikzset{every picture/.style={line width=1pt}} 


\caption{Left: $Q_1$ is $3$-seen by $T_1 \in \mathcal{B}$. Right: after rotation, the numbers of vertex-disjoint copies of $K_4$ increases by one.} 
\label{Fig:Q6BCD-a}
\end{figure}

\textbf{Case 1}: There exists a pair $(i,j) \in [2] \times [2]$ such that $Q_i$ is $3$-seen by $T_j$. 

By symmetry, we may assume that $Q_1$ is $3$-seen by $T_1$. It follows from the definition that there exists a vertex in $Q_1$ that is adjacent to all vertices in $T_1$, and by symmetry, we may assume that $p_1$ is such a vertex. 
Since $e(Q_2, T_2) \ge 7$, there exists a vertex in $T_2$ that is adjacent to at least $3$ vertices in $Q_2$, and by symmetry, we may assume that $q_2, s_2, t_2$ are these three vertices in $Q_2$.  Since $e(Q_1, Q_2) = 16$, the vertex $p_2$ is adjacent to all three vertices in $\{q_1, s_1, t_1\}$. However, the rotation shown in Figure~\ref{Fig:Q6BCD-a} would increase the size of $\mathcal{A}$, contradicting the maximality of $(|\mathcal{A}|, |\mathcal{B}|, |\mathcal{C}|, |\mathcal{D}|)$. 

\medskip 

\begin{figure}[H]
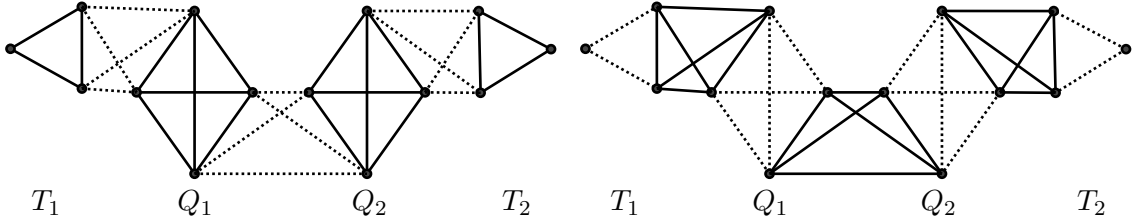

\centering

\tikzset{every picture/.style={line width=1pt}} 


\caption{Left: $Q_1$ is $2$-seen by $T_1 \in \mathcal{B}$ and $Q_2$ is $2$-seen by $T_2 \in \mathcal{B}$. Right: after rotation, the numbers of vertex-disjoint copies of $K_4$ increases by one.} 
\label{Fig:Q6BCD-b}
\end{figure}

\textbf{Case 2}: $Q_1$ is $2$-seen by $T_1$ and $Q_2$ is $2$-seen by $T_2$ (or $Q_1$ is $2$-seen by $T_2$ and $Q_2$ is $2$-seen by $T_1$). 

By symmetry, we may assume that $Q_i$ is $2$-seen by $T_i$ and that $p_i,q_i$ are adjacent to both $u_i,v_i$ for each $i\in[2]$.
Then $\{p_i,q_i,u_i,v_i\}$ induces a copy of $K_4$ for each $i\in[2]$.
Since $e(Q_1,Q_2)=16$, the remaining vertices $\{s_1,t_1,s_2,t_2\}$ induce a third copy of $K_4$.
Thus, the rotation shown in Figure~\ref{Fig:Q6BCD-b} would increase the size
of $\mathcal{A}$, contradicting the maximality of $(|\mathcal{A}|,|\mathcal{B}|,|\mathcal{C}|,|\mathcal{D}|)$.

\textbf{Case 3}: There exists a pair $(i,j) \in [2] \times [2]$ such that $Q_i$ is $2$-seen by $T_j$. 

By symmetry, we may assume that $Q_1$ is $2$-seen by $T_1$. By the definition of $2$-see, there are two vertices in $T_1$, say $\{u_1, v_1\}$, that are both adjacent to two vertices in $Q_1$, say $\{p_1, q_1\}$. 
From Cases 1 and 2, we conclude that $Q_2$ neither $3$-sees nor $2$-sees $T_2$. Combining this with  Fact~\ref{FACT:e-Q-T-7-edges} and the assumption $e(Q_2, T_2) \ge 7$, it follows that $Q_2$ is $1$-seen by $T_2$. 

From Case 1, we conclude that no vertex in $Q_1$ is adjacent to all three vertices in $T_2$. This implies that $e(\{s_1, t_1\}, T_2) \ge 7 - 2\cdot 2 = 3$. So there exists a vertex in $\{s_1, t_1\}$, say $s_1$, that is adjacent to at least two vertices in $T_2$, say $u_2$ and $v_2$. 
Recall that $Q_2$ is $1$-seen by $T_2$. By symmetry, we may assume that $p_2$ is the common neighbor of $u_2$ and $v_2$ in $Q_2$. 
Since $e(Q_1, Q_2) = 16$, the vertex $t_1$ is adjacent to all three vertices in $\{q_2, s_2, t_2\}$. 
However, the rotation shown in Figure~\ref{Fig:Q6BCD-b2} would increase the size of $\mathcal{A}$, contradicting the maximality of $(|\mathcal{A}|, |\mathcal{B}|, |\mathcal{C}|, |\mathcal{D}|)$. 

\medskip

\begin{figure}[H]
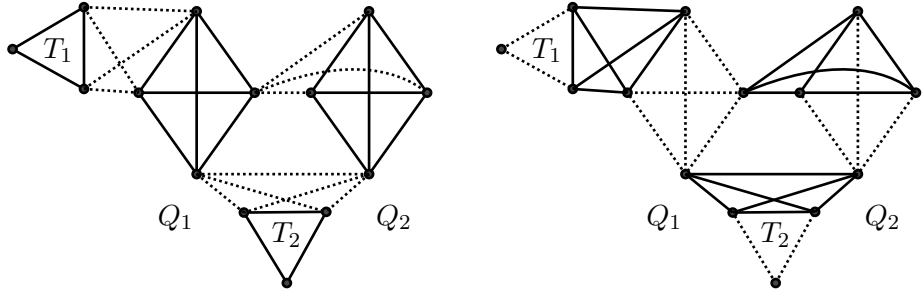

\centering

\tikzset{every picture/.style={line width=1pt}} 


\caption{Left: $Q_1$ is $2$-seen by $T_1 \in \mathcal{B}$. Right: after rotation, the numbers of vertex-disjoint copies of $K_4$ increases by one.} 
\label{Fig:Q6BCD-b2}
\end{figure}

\textbf{Case 4}: For each $i\in [2]$, $Q_{i}$ is not $3$-seen or $2$-seen by any $T_j$. 

It follows from Fact~\ref{FACT:e-Q-T-7-edges} that, for each $(i,j) \in [2] \times [2]$, $Q_i$ is $1$-seen by $T_j$. 
By symmetry, we may assume that $p_1 \in Q_1$ and $p_2 \in Q_2$ are adjacent to both vertices in $\{u_1, v_1\} \subseteq  T_1$. 
Since $T_2$ contains three pairs, at least one pair must have the property that its common neighbor in $Q_1$ is different from $p_1$, and its common neighbor in $Q_2$ is different from $p_2$. 
By symmetry, we may assume $\{u_2, v_2\} \subseteq T_2$ is such a pair in $T_2$, and that $q_1\in Q_1$ and $q_2 \in Q_2$ are the common neighbors of $\{u_2, v_2\}$ in $Q_1$ and $Q_2$, respectively. 
Since $e(Q_1, Q_2) = 16$, both $s_1$ and $t_1$ are adjacent to $s_2$ and $t_2$. 
However, the rotation shown in Figure~\ref{Fig:Q6BCD-c} would increase the size of $\mathcal{A}$, contradicting the maximality of $(|\mathcal{A}|, |\mathcal{B}|, |\mathcal{C}|, |\mathcal{D}|)$. 
This completes the proof of Lemma~\ref{LEMMA:A7BCD-atom}~\ref{LEMMA:A7BCD-atom-1}. 

\medskip 

\begin{figure}[H]
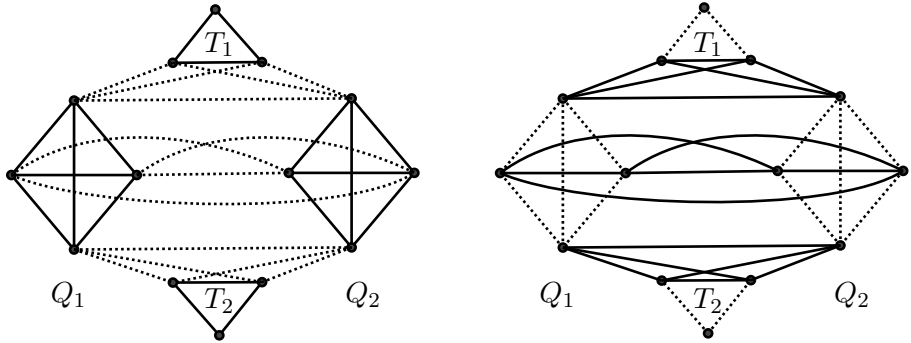

\centering

\tikzset{every picture/.style={line width=1pt}} 


\caption{A rotation that increases the numbers of vertex-disjoint copies of $K_4$.} 
\label{Fig:Q6BCD-c}
\end{figure}

Next, we prove Lemma~\ref{LEMMA:A7BCD-atom}~\ref{LEMMA:A7BCD-atom-2}. Suppose to the contrary that there exist two members $M_1, M_2 \in \mathcal{C}$ such that 
\begin{align*}
    \min\left\{e(Q_i, M_j) \colon (i,j) \in [2] \times [2]\right\} \ge 5.
\end{align*}
Since $e(Q_1, M_1) \ge 5$, there exists a vertex in $M_1$ that has at least three neighbors in $Q_1$. 
Similarly, since $e(Q_2, M_2) \ge 5$, there exists a vertex in $Q_2$ that is adjacent to both vertices in $M_2$.
Since $e(Q_1, Q_2) = 16$, the remaining four vertices in $Q_1 \cup Q_2$ induce a copy of $K_4$ in $G$. 
However, the rotation shown in Figure~\ref{Fig:Q6BCD-d} would increase the size of $\mathcal{B}$ while keeping the size of $\mathcal{A}$ unchanged, contradicting the maximality of $\mathcal{B}$. 
This completes the proof of Lemma~\ref{LEMMA:A7BCD-atom}~\ref{LEMMA:A7BCD-atom-2}. 
\end{proof}

\medskip

\begin{figure}[H]
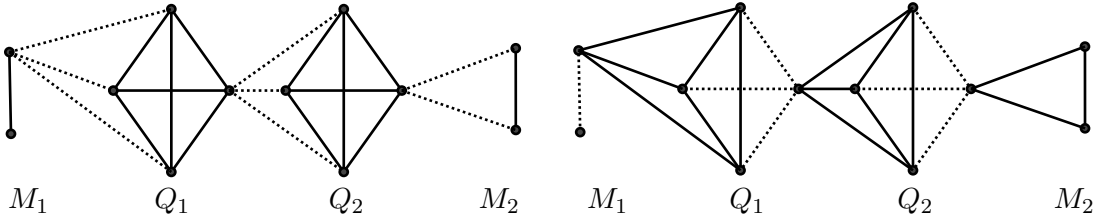

\centering

\tikzset{every picture/.style={line width=1pt}} 


\caption{After rotation, the number of vertex-disjoint copies of $K_4$ remains the same, while the number of vertex-disjoint copies of $K_3$ increases by one.} 
\label{Fig:Q6BCD-d}
\end{figure}

We are now ready to present the proof of Lemma~\ref{LEMMA:A6BCD-upper-bound-b}. 
\begin{proof}[Proof of Lemma~\ref{LEMMA:A6BCD-upper-bound-b}]
    Due to Lemma~\ref{LEMMA:Q5Q6BCD}~\ref{LEMMA:Q5Q6BCD-3} and Fact~\ref{FACT:edges-ABCD}, it suffices to prove that 
     \begin{align*}
        e(\mathcal{A}_6) + e(\mathcal{A}_6, \mathcal{B} \cup \mathcal{C})
        \le O(n) + 
        \begin{cases}
            \frac{15 a_6^2}{2} + (7b+5c)a_6, &\quad\text{if}\quad a_6 \le b+c, \\[0.5em]
            8 a_{6}^{2} + \left(\frac{13 b}{2} + \frac{9c}{2}\right) a_6, &\quad\text{if}\quad a_6 \ge b+c.
        \end{cases}
    \end{align*}
    We may assume that $a_6\ge2$, since otherwise $e(\mathcal{A}_6)+e(\mathcal{A}_6,\mathcal{B}\cup\mathcal{C})=O(n)$.
    Let us define an auxiliary graph $H$ whose vertex set is $\mathcal{A}_{6}$ and two members $Q_{1}, Q_{2} \in \mathcal{A}_{6}$ are adjacent in $H$ iff $e(Q_{1}, Q_{2}) = 16$. 
    Let $\overline{H}$ denote the complement of $H$ on the vertex set $\mathcal{A}_6$.
    Notice from the definition of $H$ that 
    \begin{align}\label{equ:A7BCD-upper-bound-II-A6}
        e(\mathcal{A}_6)
        \le 6 a_6 + 16 |H| + 15 \left(\binom{a_6}{2} - |H|\right)
        =  6 a_6 + 15 \binom{a_6}{2}+ |H|.
    \end{align}
    It follows from Lemma~\ref{LEMMA:Q5Q6BCD} that for every pair $\{Q_1, Q_{2}\} \in \overline{H}$, we have
    \begin{align}\label{equ:A6BCD-Q1Q2-BC-H-bar}
        e(Q_1 \cup Q_{2}, \mathcal{B} \cup \mathcal{C})
        \le 2\cdot 7 \cdot (b-2) + 2 \cdot 3\cdot 8 + 2\cdot 5 \cdot (c-1) +2\cdot 8
        = 14b + 10c + 26.
    \end{align}
    %

Fix a pair $\{Q_1,Q_2\}\in H$.
By the definition of $\mathcal{A}_6$ and
Fact~\ref{FACT:3-by-4-bipartite-graph}~\ref{FACT:3-by-4-bipartite-graph-3},
for each $i\in[2]$, there is at most one member
$B\in\mathcal{B}$ satisfying $e(Q_i,B)\ge 9$.
If such a member exists, then
Fact~\ref{FACT:3-by-4-bipartite-graph}~\ref{FACT:3-by-4-bipartite-graph-4}
and the definition of $\mathcal{A}_6$ imply that
$e(Q_i,B')\le 7$ for every
$B'\in\mathcal{B}\setminus\{B\}$.

By Lemma~\ref{LEMMA:A7BCD-atom}~\ref{LEMMA:A7BCD-atom-1}, at most one member $B\in\mathcal B$ satisfies $e(Q_1,B)\ge 7$ and $e(Q_2,B)\ge 7$.
Otherwise, we have  $e(Q_1\cup Q_2,\mathcal{B})\le 12+6+12+6+(7+6)(b-2)=13b+10$.
We consider the following three cases.

\textbf{Case 1}: $e(Q_1,B)\ge 9$ and $e(Q_2,B)\ge 9$.

For every $B'\in\mathcal{B}\setminus\{B\}$,
we have $e(Q_i,B')\le 7$ for each $i\in[2]$,
and at least one of these two inequalities is strict.
Hence,
\begin{align*}
    e(Q_1\cup Q_2,\mathcal{B})
    \le 12+12+(7+6)(b-1)
    =13b+11.
\end{align*}

\textbf{Case 2}: $e(Q_1,B)\ge 9$ and $7\le e(Q_2,B)\le 8$.

By Lemma~\ref{LEMMA:Q5Q6BCD}~\ref{LEMMA:Q5Q6BCD-1},
there may be an additional exceptional member
$B'\in\mathcal{B}\setminus\{B\}$ with $e(Q_2,B')\le 8$.
Such a member satisfies $e(Q_1,B')\le 6$ by
Lemma~\ref{LEMMA:A7BCD-atom}~\ref{LEMMA:A7BCD-atom-1}.
Hence,
\begin{align*}
    e(Q_1\cup Q_2,\mathcal{B})
    \le 12+8+6+8+(7+6)(b-2)
    =13b+8.
\end{align*}

\textbf{Case 3}: $7\le e(Q_1,B)\le 8$ and $7\le e(Q_2,B)\le 8$.

Similarly, we obtain
\begin{align*}
    e(Q_1\cup Q_2,\mathcal{B})
    \le 8+8+6+8+(7+6)(b-2)
    =13b+4.
\end{align*}

Next, we bound $e(Q_1\cup Q_2,\mathcal{C})$.
By Lemma~\ref{LEMMA:Q5Q6BCD}~\ref{LEMMA:Q5Q6BCD-2}, for each $i\in[2]$, there is at most one member $C\in\mathcal{C}$ satisfying $e(Q_i,C)\ge6$. 
Lemma~\ref{LEMMA:A7BCD-atom}~\ref{LEMMA:A7BCD-atom-2} allows at most one member of $\mathcal{C}$ to send at least five edges to both $Q_1$ and $Q_2$, hence $e(Q_1\cup Q_2,\mathcal{C})\le(5+4)(c-1)+2\cdot8=9c+7$.

Combining these bounds, we obtain
   \begin{align}\label{equ:A6BCD-Q1Q2-BC-H}
        e(Q_1 \cup Q_{2}, \mathcal{B} \cup \mathcal{C}) \le(13b+11)+(9c+7)=13b+9c+18.
    \end{align}

    Combining~\eqref{equ:A6BCD-Q1Q2-BC-H-bar} and~\eqref{equ:A6BCD-Q1Q2-BC-H} and using the fact $|H| + |\overline{H}| = \binom{a_6}{2}$, we obtain 
    \begin{align*}
        e(\mathcal{A}_6, \mathcal{B} \cup \mathcal{C})
        & = \frac{1}{a_6 - 1}\sum_{\{Q_1,Q_2\} \subseteq   \mathcal{A}_6}  e(Q_1 \cup Q_{2}, \mathcal{B} \cup \mathcal{C}) \\[0.5em]
        & = \frac{1}{a_6-1} \left(\sum_{\{Q_1,Q_2\} \in H}  e(Q_1 \cup Q_{2}, \mathcal{B} \cup \mathcal{C}) + \sum_{\{Q_1,Q_2\} \in \overline{H}}  e(Q_1 \cup Q_{2}, \mathcal{B} \cup \mathcal{C}) \right) \\[0.5em]
        & \le \frac{(13 b + 9 c + 18)|H| + (14b + 10c + 26)|\overline{H}|}{a_6-1} 
        \le (7b+5c+13) a_6 - \frac{(b+c)|H|}{a_6-1}. 
    \end{align*}
    Combining this with~\eqref{equ:A7BCD-upper-bound-II-A6}, we obtain 
    \begin{align}\label{equ:A6BCD-A6-BC}
        e(\mathcal{A}_6) + e(\mathcal{A}_6, \mathcal{B} \cup \mathcal{C})
        & \le  6 a_6 + 15 \binom{a_6}{2}+ |H| + (7b+5c+13) a_6 - \frac{(b+c)|H|}{a_6-1} \notag \\[0.5em]
        & = \frac{15 a_6^2}{2} + (7b+5c)a_6 + \frac{23 a_6}{2} + \left(1 - \frac{b+c}{a_6-1}\right) |H| \notag \\[0.5em]
        &\le \frac{15 a_6^2}{2} + (7b+5c)a_6 + \frac{a_6 - (b+c)}{a_6-1}  |H| + \frac{23 a_6}{2}.
    \end{align}
    If $a_6 > b+c$, then $\frac{a_6 - (b+c)}{a_6-1} > 0$, and hence,~\eqref{equ:A6BCD-A6-BC} continues as 
    \begin{align*}
        e(\mathcal{A}_6) + e(\mathcal{A}_6, \mathcal{B} \cup \mathcal{C})
        & \le \frac{15 a_6^2}{2} + (7b+5c)a_6 + \frac{a_6 - (b+c)}{a_6-1}  \binom{a_6}{2} + \frac{23 a_6}{2} \\[0.5em]
        & = 8 a_6^2 + \frac{13b+9 c}{2} a_6 + \frac{23 a_6}{2},
    \end{align*}
    as desired. 
    
    If $a_6 \le b+c$, then $\frac{a_6 - (b+c)}{a_6-1} \le 0$, and hence,~\eqref{equ:A6BCD-A6-BC} continues as 
    \begin{align*}
        e(\mathcal{A}_6) + e(\mathcal{A}_6, \mathcal{B} \cup \mathcal{C})
        & \le \frac{15 a_6^2}{2} + (7b+5c)a_6 + \frac{a_6 - (b+c)}{a_6-1}  \cdot 0 + \frac{23 a_6}{2} \\[0.5em]
        & = \frac{15 a_6^2}{2} + (7b+5c)a_6 + \frac{23 a_6}{2}, 
    \end{align*}
    also as desired. 
    This completes the proof of Lemma~\ref{LEMMA:A6BCD-upper-bound-b}. 
\end{proof}

\subsection{Preparations for Lemma~\ref{LEMMA:A6BCD-upper-bound-a} \RomanNumeralCaps{1}: an extension of Tur\'{a}n}\label{SUBSEC:Preparations-A6BCD}
The aim of this subsection is to establish the following result (Proposition~\ref{PROP:K4-tiling-cover-three-vtx}), which will be crucial for the proof of Lemma~\ref{LEMMA:A6BCD-upper-bound-a}.

For integers $h, \alpha \ge 0$ satisfying $h \ge 3\alpha$, let 
\begin{align}\label{equ:def-p-h-a}
    P(h,\alpha) 
    \coloneqq 
       \begin{cases}
          \binom{h-3\alpha}{2}+\alpha(2h-3\alpha)+12h,  &\quad\text{if}\quad 3\alpha \le h < 16\alpha,  \\[0.5em]
          \binom{h-3\alpha+9}{2}+ (\alpha-3)(2h-3\alpha+9),  &\quad\text{if}\quad h \ge 16\alpha.
       \end{cases}
\end{align}
Simple calculations show that 
\begin{align}\label{equ:P-h-a-upper-bound}
    P(h,\alpha)
    \le \frac{(h-\alpha)^2}{2} + \alpha^2 + \frac{23 h}{2} + \frac{3 \alpha}{2} + 9.
\end{align}
\begin{proposition}\label{PROP:K4-tiling-cover-three-vtx}
    There exists a constant $N_{\ref{PROP:K4-tiling-cover-three-vtx}} > 0$ such that the following holds.
    Let $h,\alpha$ be nonnegative integers satisfying $h\ge\max\{3\alpha,N_{6.4}\}$. 
    Suppose that $H$ is a graph on $h$ vertices and there exists a vertex set $A \subseteq  V(H)$ of size $\alpha$ such that every $K_4$-tiling in $H$ contains at most three vertices in $A$. Then 
    \begin{align*}
        |H|
        \le P(h,\alpha). 
    \end{align*}
\end{proposition}

We will need the following result (which is a special case of~{\cite[Theorem~2]{ABHP14}}) by Allen--B{\" o}ttcher--Hladk\'{y}--Piguet~\cite{ABHP14}. 

Let $\mathcal{T}_{3}(h,\alpha)$ denote the family of all $h$-vertex graphs $H$ satisfying the following properties:
\begin{itemize}
    \item The vertex set of $H$ can be partitioned into seven parts: $X$, $Y_1, Y_2, Y_3$, $Z_1, Z_2, Z_3$, such that 
    \begin{align*}
        \alpha \le |Y_1 \cup Z_1| \le |Y_2 \cup Z_2| \le |Y_3 \cup Z_3| \le \alpha+1
        \quad\text{and}\quad 
        |Z_1|+ |Z_2| + |Z_3| = \alpha. 
    \end{align*}
    \item The edge set of $H$ is given by 
    \begin{align*}
        H 
        = \binom{X}{2} \cup K[X, Y_1 \cup Y_2 \cup Y_3] \cup K[Y_1 \cup Z_1, Y_2\cup Z_2, Y_3 \cup Z_3]. 
    \end{align*}
\end{itemize}

\begin{theorem}[Allen--B{\" o}ttcher--Hladk\'{y}--Piguet~\cite{ABHP14}]\label{THM:ABHP14-exact}
    Let $h, \alpha \ge 0$ be integers satisfying $h \ge 3\alpha$. 
    Suppose that $H$ is a graph on $h$ vertices and $A\subseteq  V(H)$ is a vertex set of size $\alpha$ such that no copy of $K_{4}$ in $H$ has nonempty intersection with $A$. 
    Then 
    \begin{align*}
        |H|
        \le \binom{h-3\alpha}{2}+3\alpha^2+(h-3\alpha)\cdot 2\alpha
        =\binom{h-3\alpha}{2}+ \alpha\cdot(2h-3\alpha).
    \end{align*}
    Furthermore, for every $\varepsilon>0$, there exists $\delta >0$ and $N_0$ such that if $h \ge N_0$ and 
    \begin{align*}
        |H|
        \ge \binom{h-3\alpha}{2}+\alpha(2h-3\alpha)-\delta h^2,
    \end{align*}
    then $H$ is a member in $\mathcal{T}_3(h,\alpha)$ after removing and adding at most $\varepsilon h^2$ edges.
\end{theorem}
Given two graphs $G$ and $H$ with the same number of vertices, the \textbf{edit distance} $\mathrm{edit}(G, H)$ between $G$ and $H$ is the minimum number of edge additions and deletions required to transform $G$ into a copy of $H$.

The following simple but useful property of $\mathcal{T}_3(h,\alpha)$ follows from its definition and elementary calculations (a detailed proof is provided in the Appendix).

\begin{proposition}\label{PROP:T3-h-a-subgraph}
    Let $h, \alpha \ge 0$ be integers satisfying $h \ge 16\alpha $. Suppose that $H \in \mathcal{T}_{3}(h,\alpha)$. Then for every vertex set $S \subseteq  V(H)$ of size size $s\ge \max\{\alpha,h-12\alpha-42\}$, we have 
    \begin{align*}
        |H[S]|
        \ge \binom{s-\alpha}{2}.  
    \end{align*}
\end{proposition}

We are now ready to present the proof of Proposition~\ref{PROP:K4-tiling-cover-three-vtx}.
\begin{proof}[Proof of Proposition~\ref{PROP:K4-tiling-cover-three-vtx}]
    Fix $\varepsilon > 0$. We may assume that $\varepsilon$ is sufficiently small. 
    Let $\delta>0$ be given by the stability part of Theorem~ \ref{THM:ABHP14-exact} with $\varepsilon/2$ in place of $\varepsilon$, and assume that $h$ is sufficiently large.
    Let $H$ and $A$ satisfy the assumptions of Proposition~\ref{PROP:K4-tiling-cover-three-vtx}, so that $|A|=\alpha$ and $h\ge3\alpha$.
    If $\alpha\le2$, then $|H|\le\binom{h}{2}\le P(h,\alpha)$ for sufficiently large $h$. So we may assume that $\alpha\ge3$.   
    Let $V \coloneqq V(H)$.  
    Let $\mathcal{K}$ be a $K_4$-tiling in $H$ such that $|V(\mathcal{K}) \cap A|$ is maximized. Additionally, among all such  $K_4$-tilings, assume that $|\mathcal{K}|$ is minimized. 
    It follows from the minimality of $|\mathcal{K}|$ and the assumption $|V(\mathcal{K}) \cap A| \le 3$ that $|\mathcal{K}| \le 3$. Therefore, $|V(\mathcal{K})| \le 12$. Let $U \subseteq  V$ be a set of size $12$ containing $V(\mathcal{K})$ such that $|U \cap A| = 3$. 
        
    Let $H'$ be the graph obtained from $H$ by removing all edges that have nonempty intersection with $U$, noting that we lose at most $12 h$ edges.  It follows from the maximality of $|V(\mathcal{K}) \cap A|$ that no copy of $K_4$ in $H'$ contains a vertex from $A$. 
    If $h<16\alpha$, then applying Theorem~\ref{THM:ABHP14-exact} to $H'$ and $A$ and adding back the deleted edges gives $|H|\le\binom{h-3\alpha}{2}+\alpha(2h-3\alpha)+12h =P(h,\alpha)$, as desired.
    So we may assume that $h\ge16\alpha$. Since $h-12\ge3(\alpha-3)$, applying Theorem~\ref{THM:ABHP14-exact} to $H[V\setminus U]$ and $A\setminus U$ gives
    \begin{align}\label{equ:3-vtx-H-U-upper-bound}
        |H'|
        = |H[V\setminus U]|
        & \le \binom{h-12-3(\alpha-3)}{2} + (\alpha-3) \left(2(h-12) - 3(\alpha-3)\right) \\[0.5em]
        & \le \binom{h-3\alpha}{2} + \alpha(2h-3\alpha). \notag  
    \end{align}
    %


    We may assume that $|H|>P(h,\alpha)$, since otherwise there is nothing to prove. By the definition of $P(h,\alpha)$,
    \begin{align*}
        |H'|\ge|H|-12h
        & > \binom{h-3\alpha+6}{2}+(2\alpha-3)(h-3\alpha+6)+3(\alpha-2)^2 - 12h \\[0.5em]
        & = \binom{h-3\alpha}{2} + \alpha(2h-3\alpha) - 9(h+\alpha-1) 
        \ge  \binom{h-3\alpha}{2} + \alpha(2h-3\alpha) - \delta h^2.
    \end{align*}
    By the stability part of Theorem~\ref{THM:ABHP14-exact}, there exists a member $\hat{H} \in \mathcal{T}_{3}(h,\alpha)$ such that $\mathrm{edit}(H,\hat{H}) \le \varepsilon h^2$. 
    Let $S \subseteq  V$ have size $s\ge h - 12 \alpha -42$.
    Since $h\ge16\alpha$ and $h$ is sufficiently large, then $\frac h4-42\ge\frac{3h}{16}$ and $\alpha\le\frac{s+42}{4}$.
    By Proposition~\ref{PROP:T3-h-a-subgraph},
    \begin{align*}
        |H[S]|
        &\ge|\hat H[S]|-\operatorname{edit}(H,\hat H)\ge\binom{s-\alpha}{2}-\varepsilon h^2\\
        &\ge\frac{9s^2}{32}-\frac{33s}{4}-\varepsilon h^2\ge\left(\frac9{32}-\frac{256\varepsilon}{9}\right)s^2
        -\frac{33s}{4}\\
        &\ge\left(\frac14+\frac1{100}\right)s^2,
\end{align*}
where the last inequality holds when $\varepsilon$ is sufficiently
small and $h$ is sufficiently large.
    
    Applying Theorem~\ref{THM:ABHP-density-CH} with $n=s$ and $k=12$ gives $
    \mathrm{ex}(s,13K_3) =\frac{s^2}{4}+6s-36+O(s) <\left(\frac14+\frac1{100}\right)s^2$,
    since $s \ge 3h/16$ is sufficiently large and $12\le2s/9$. 
    Combining these inequalities, we obtain
    \begin{align}\label{equ:H-S-nu-K3}
        \nu(K_3, H[S]) \ge 13
        \quad\text{for every subset $S \subseteq  V$ of size at least $h - 12 \alpha -42$.}
    \end{align}
    Let $v_1, v_2, v_{3}$ denote the vertices in $U \cap A$. 
    Suppose that 
    \begin{align*}
        d_{H}(v_1) + d_{H}(v_2) + d_{H}(v_3)
        \le 3h - 12\alpha -42. 
    \end{align*}
    Then it follows from~\eqref{equ:3-vtx-H-U-upper-bound} that 
    \begin{align*}
        |H|
        & \le |H'| + |U\setminus A| \cdot h +  d_{H}(v_1) + d_{H}(v_2) + d_{H}(v_3) \\[0.5em]
        & \le \binom{h-12-3(\alpha-3)}{2} + (\alpha-3) \left(2(h-12) - 3(\alpha-3)\right) + 9h + 3h - 12\alpha -42 \\[0.5em]
        &=\left[\binom{h-3\alpha-3}{2} +12(h-3\alpha-3)+\binom{12}{2}\right]+24(\alpha-3)+(\alpha-3)(2h-3\alpha-15)\\
        &=\binom{h-3\alpha+9}{2}+(\alpha-3)(2h-3\alpha+9)=P(h,\alpha),
    \end{align*}
    as desired. So we may assume that $d_{H}(v_1) + d_{H}(v_2) + d_{H}(v_3) \ge 3h - 12 \alpha -42$. In particular, 
    \begin{align*}
        \min\left\{|N_{H}(v_i)| \colon i \in [3]\right\}
        \ge d_{H}(v_1) + d_{H}(v_2) + d_{H}(v_3) - 2h 
        \ge h - 12 \alpha -42. 
    \end{align*}
    By~\eqref{equ:H-S-nu-K3}, for each $i \in [3]$, there exists a $K_3$-tiling $\mathcal{T}_i$ of size at least $13$ in the induced subgraph $H[N_{H}(v_i)]$. 
    Note that every member in $\mathcal{T}_i$ forms a copy of $K_4$ with $v_i$. 
    
    \begin{claim}\label{CLAIM-3-vtx-A-minus-U}
        There is no copy of $K_4$ in $H$ that contains a vertex from $A\setminus U$.
    \end{claim}
    \begin{proof}[Proof of Claim~\ref{CLAIM-3-vtx-A-minus-U}]
        Suppose to the contrary that there exists a copy of $K_4$ in $H$, denoted by $Q$, that contains at least one vertex from $A\setminus U$. 
        Then a simple greedy argument shows that, for each $i\in[3]$ with $v_i\notin V(Q)$, we can choose $T_i\in\mathcal{T}_i$ such that $Q$ and the copies $Q_i\coloneqq T_i\cup\{v_i\}$ are pairwise vertex-disjoint.
        These copies form a $K_4$-tiling in $H$ covering $v_1,v_2,v_3$ and at least one vertex of $A\setminus U$, contradicting the assumption of Proposition~\ref{PROP:K4-tiling-cover-three-vtx}.
    \end{proof}
    It follows from Claim~\ref{CLAIM-3-vtx-A-minus-U} and Theorem~\ref{THM:ABHP14-exact} that 
    \begin{align*}
        |H|
        \le \binom{h-3(\alpha-3)}{2}+ (\alpha-3)\cdot(2h-3(\alpha-3))
        = P(h,\alpha), 
    \end{align*}
    as desired. This completes the proof of Proposition~\ref{PROP:K4-tiling-cover-three-vtx}. 
\end{proof}

\subsection{Preparations for Lemma~\ref{LEMMA:A6BCD-upper-bound-a} \RomanNumeralCaps{2}: connecting}\label{SUBSEC:A6BCD-connecting}
In this subsection, we establish several local properties of $\mathcal{A}_6$, which will be useful for the proofs in the next subsection.
The following definition of \textbf{connection} is inspired by that in~\cite{ABHP15}.
\begin{definition}\label{DEF:connects-Q1-Q2-Q3}
    Let $Q_1, Q_2, Q_3\subseteq  V(G)$ be three pairwise disjoint sets of size $4$, with each inducing a copy of $K_4$ in $G$. We say that $Q_2$ \textbf{connects} $Q_1$ and $Q_3$ (or that there exists a connection from $Q_1$ to $Q_3$ via $Q_2$), denoted by $Q_1\rightsquigarrow Q_2 \rightsquigarrow  Q_3$, if one of the following conditions holds$\colon$
    \begin{enumerate}[label=(\roman*)]
        \item\label{DEF:connects-Q1-Q2-Q3-1} $\min\left\{|G[Q_1, Q_2]|,~|G[Q_2, Q_3]|\right\} \ge 15$. 
        \item\label{DEF:connects-Q1-Q2-Q3-2} $|G[Q_1, Q_2]| = 16$ and $|G[Q_2, Q_3]| \ge 14$.
    \end{enumerate}
\end{definition}

\begin{proposition}\label{PROP:connection-A6-b}
    Suppose that $Q_1\rightsquigarrow Q_2 \rightsquigarrow  Q_3$. Then the following statements hold. 
    \begin{enumerate}[label=(\roman*)]
        \item\label{PROP:connection-A6-b-1} For every pair of vertices $\{p_1, q_1\} \subseteq  Q_1$, there exists a pair of vertices $\{p_2, q_2\} \subseteq  Q_2$ that are adjacent to both $p_1$ and $q_1$, such that the remaining two vertices in $Q_2 \setminus\{p_2, q_2\}$ have at least three common neighbors in $Q_{3}$ (see Figure~\ref{Fig:connection-property}~(a)). 
        \item\label{PROP:connection-A6-b-2} For every pair of vertices $\{p_3, q_3\} \subseteq  Q_3$, there exists a pair of vertices $\{p_2, q_2\} \subseteq  Q_2$ that are adjacent to both $p_3$ and $q_3$,  such that the remaining  two vertices in $Q_2 \setminus\{p_2, q_2\}$ have at least three common neighbors in $Q_{1}$ (see Figure~\ref{Fig:connection-property}~(b)).
        \item\label{PROP:connection-A6-b-3} For every vertex $p_1 \in Q_1$, there exists three vertices $\{p_2, q_2, s_2\} \subseteq  Q_2$ that are all adjacent to $p_1$, such that the remaining vertex in $Q_2 \setminus \{p_2, q_2, s_2\}$ has at least three neighbors in $Q_{3}$ (see Figure~\ref{Fig:connection-property}~(c)).
        \item\label{PROP:connection-A6-b-4} For every triple of vertices $\{p_3, q_3,s_3\} \subseteq  Q_3$, there exists a vertex $p_2 \in Q_2$ that is adjacent to all three vertices in $\{p_3, q_3, s_3\}$, such that the remaining  three vertices in $Q_2 \setminus\{p_2\}$ have at least three common neighbors in $Q_{1}$ (see Figure~\ref{Fig:connection-property}~(d)). 
    \end{enumerate} 
\end{proposition}

\medskip 

\begin{figure}[H]
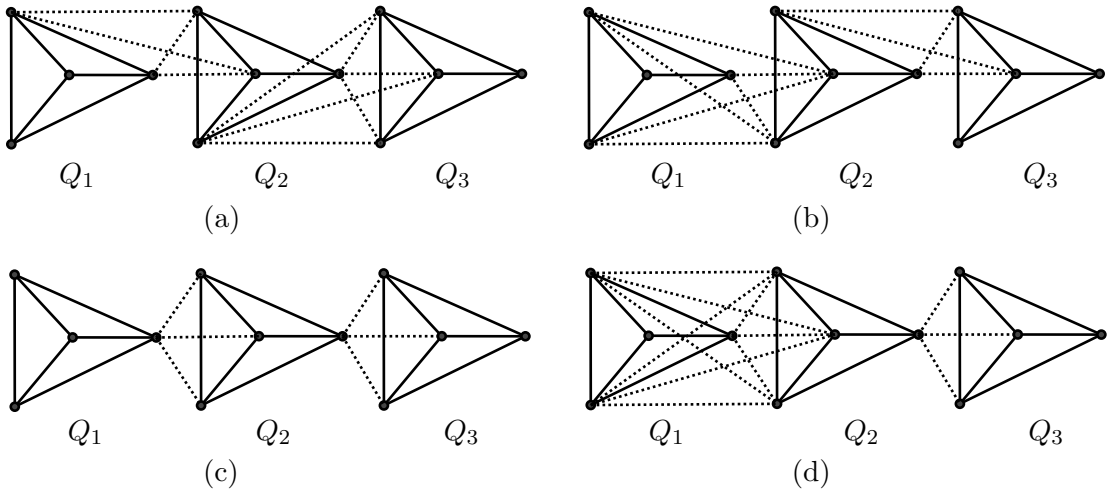

\centering
\tikzset{every picture/.style={line width=1pt}} 

\caption{Auxiliary figure for Proposition~\ref{PROP:connection-A6-b}.} 
\label{Fig:connection-property}
\end{figure}

%

\begin{proof}
    If $\min\{e(Q_1,Q_2),e(Q_2,Q_3)\}\ge15$, then each of $G[Q_1,Q_2]$ and $G[Q_2,Q_3]$ has at most one missing edge, and all four statements follow.
    So, by Definition~\ref{DEF:connects-Q1-Q2-Q3}, we may assume that $e(Q_1,Q_2)=16$ and $e(Q_2,Q_3)\ge14$. 

    For Proposition~\ref{PROP:connection-A6-b}~\ref{PROP:connection-A6-b-1}, since $e(Q_2,Q_3)\ge14$, at least two vertices in $Q_2$ are adjacent to all four vertices in $Q_3$.
    Thus, we can choose $\{p_2,q_2\}\subseteq Q_2$ such that $e(Q_2\setminus\{p_2,q_2\},Q_3)=8$.
    
    For Proposition~\ref{PROP:connection-A6-b}~\ref{PROP:connection-A6-b-2}, fix $\{p_3,q_3\}\subseteq Q_3$. Since $e(\{p_3,q_3\},Q_2)\ge14-2\cdot4=6$,
    Fact~\ref{FACT:3-by-4-bipartite-graph}~\ref{FACT:3-by-4-bipartite-graph--1} gives a pair $\{p_2,q_2\}\subseteq Q_2$ adjacent to both $p_3$ and $q_3$.
    
    For Proposition~\ref{PROP:connection-A6-b}~\ref{PROP:connection-A6-b-3}, since $e(Q_2,Q_3)\ge14>3\cdot4$, there exists a vertex $t_2\in Q_2$ adjacent to all four vertices in $Q_3$. Let $\{p_2,q_2,s_2\}=Q_2\setminus\{t_2\}$.
    
    For Proposition~\ref{PROP:connection-A6-b}~\ref{PROP:connection-A6-b-4}, fix $\{p_3,q_3,s_3\}\subseteq Q_3$. Since $e(Q_2,\{p_3,q_3,s_3\})\ge14-4=10\ge9$,
    Fact~\ref{FACT:3-by-4-bipartite-graph}~\ref{FACT:3-by-4-bipartite-graph-3} gives a vertex $p_2\in Q_2$ adjacent to all three vertices in $\{p_3,q_3,s_3\}$.
    
    In each case, the required adjacencies between $Q_1$ and $Q_2$ follow from $e(Q_1,Q_2)=16$. This completes the proof.
\end{proof}

For convenience, for every pair $Q_1, Q_3$ of copies of $K_4$ in $\mathcal{A}_6$, let 
\begin{align*}
    N_{\mathcal{A}_6}(Q_1, Q_3)
    \coloneqq 
    \left\{Q\in \mathcal{A}_6 \setminus \{Q_1, Q_3\} \colon Q_1\rightsquigarrow Q \rightsquigarrow  Q_3\right\}. 
\end{align*}
For every $Q\in \mathcal{A}_6$, let 
    \begin{align*}
        N_{\mathcal{A}_6}^{s}(Q)
        \coloneqq \left\{Q' \in \mathcal{A}_6 \colon e(Q, Q') \le 14\right\}
        \quad\text{and}\quad 
        N_{\mathcal{A}_6}^{\ell}(Q)
        \coloneqq \left\{Q' \in \mathcal{A}_6 \colon e(Q, Q') \ge 15\right\}. 
    \end{align*}
Note that $\mathcal{A}_6\setminus \{Q\} = N_{\mathcal{A}_6}^{s}(Q) \cup N_{\mathcal{A}_6}^{\ell}(Q)$ is a partition. 

\begin{lemma}\label{LEMMA:A6-min-degree}
    For every member $Q\in \mathcal{A}_6$, we have 
    \begin{align*}
        |N_{\mathcal{A}_6}^{s}(Q)| 
        \le \frac{a_6 -2}{2}. 
    \end{align*}
\end{lemma}
\begin{proof}[Proof of Lemma~\ref{LEMMA:A6-min-degree}]
    Suppose to the contrary that $|N_{\mathcal{A}_6}^{s}(Q)| > \frac{a_6 -2}{2}$. 
    Then 
     \begin{align*}
        e(Q, \mathcal{A}_6 \setminus \{Q\})
         \le 14 \cdot |N_{\mathcal{A}_6}^{s}(Q)| + 16 \left(a_6 - 1 - |N_{\mathcal{A}_6}^{s}(Q)|\right) 
        & = 16(a_6 - 1) - 2 |N_{\mathcal{A}_6}^{s}(Q)| \\[0.5em]
        & < 15 (a_6 - 1)+1, 
    \end{align*}
    Since the number of edges is an integer, we have $e(Q,\mathcal{A}_6\setminus\{Q\})\le 15(a_6-1)$, contradicting the definition of $\mathcal{A}_6$ (see Section~\ref{SUBSEC:setup}). 
\end{proof}
\begin{lemma}\label{LEMMA:many-connection-A6}
    For every pair $Q_1, Q_3$ of members in $\mathcal{A}_6$, we have
    \begin{align*}
        |N_{\mathcal{A}_6}(Q_1, Q_3)|
        \ge \frac{a_6 - 2}{12}. 
    \end{align*}
    In particular, for every subset $B \subseteq V(\mathcal{A}_6)$ of size less than $\frac{a_6 - 2}{12}$, there exists a member $Q \in \mathcal{A}_6 - B$ such that $Q_1\rightsquigarrow Q \rightsquigarrow  Q_3$. 
\end{lemma}
\begin{proof}[Proof of Lemma~\ref{LEMMA:many-connection-A6}]
    Fix two distinct members $Q_1, Q_3 \in \mathcal{A}_6$. Observe that, by definition, we have  
     \begin{align*}
        N_{\mathcal{A}_6}^{\ell}(Q_1) \cap N_{\mathcal{A}_6}^{\ell}(Q_3) \subseteq 
        N_{\mathcal{A}_6}(Q_1, Q_3) . 
    \end{align*}
    By the definition of $\mathcal{A}_6$ and the integrality of edge counts,
    If $|N_{\mathcal{A}_6}^{\ell}(Q_1)| \ge \frac{7(a_6 - 2)}{12}+1$, then it follows from Lemma~\ref{LEMMA:A6-min-degree} that 
    \begin{align*}
        |N_{\mathcal{A}_6}(Q_1, Q_3)|
        \ge |N_{\mathcal{A}_6}^{\ell}(Q_1)| - |N_{\mathcal{A}_6}^{s}(Q_3)|-1
        \ge \frac{a_6-2}{12}, 
    \end{align*}
    as desired. 
    So we may assume that $|N_{\mathcal{A}_6}^{\ell}(Q_1)| <  \frac{7(a_6 - 2)}{12}+1$. This implies that 
    \begin{align}\label{equ:LEMMA:many-connection-A6-Q1-s}
        |N_{\mathcal{A}_6}^{s}(Q_1)| 
        = (a_6 - 1) - |N_{\mathcal{A}_6}^{\ell}(Q_1)| 
        >  \frac{5(a_6 - 2)}{12}.
    \end{align}
    %
    %
    Let $x \coloneqq |\left\{Q' \in \mathcal{A}_6 \colon e(Q_{1}, Q') = 16\right\}|$. 
    It follows from the definition of $\mathcal{A}_6$ that 
    \begin{align*}
        15 (a_6-1)
        < e(Q_1, \mathcal{A}_6\setminus \{Q_1\})
        & \le 16 x + 15 \left((a_6-1)-x-|N_{\mathcal{A}_6}^{s}(Q_1)|\right) + 14 \cdot |N_{\mathcal{A}_6}^{s}(Q_1)| \\[0.5em]
        & = x + 15 (a_6-1) - |N_{\mathcal{A}_6}^{s}(Q_1)|, 
    \end{align*}
    which, by~\eqref{equ:LEMMA:many-connection-A6-Q1-s}, implies that 
    \begin{align*}
        x 
        \ge |N_{\mathcal{A}_6}^{s}(Q_1)|+1
        > \frac{5(a_6 - 2)}{12}+1. 
    \end{align*}
    %
    %
    Let $y \coloneqq |\left\{Q' \in \mathcal{A}_6 \colon e(Q_{3}, Q') \ge 14 \right\}|$. 
    Similarly, it follows from the definition of $\mathcal{A}_6$ that 
    \begin{align*}
        15 (a_6-1)
        < e(Q_3, \mathcal{A}_6\setminus \{Q_3\})
        \le 16 y + 13\left((a_6-1) - y\right)  
         = 15 (a_6-1) + 3y - 2(a_6-1), 
    \end{align*}
    which implies that 
    \begin{align*}
        y \ge \frac{2(a_6-2)}{3}+1. 
    \end{align*}
    Therefore, by the Inclusion-Exclusion Principle, the set 
    \begin{align*}
        \left\{Q' \in \mathcal{A}_6 \colon e(Q_1, Q') =16~\text{and}~e(Q_3, Q')\ge 14\right\}, 
    \end{align*}
    which is contained in $N_{\mathcal{A}_6}(Q_1, Q_3)$,  has size at least 
    \begin{align*}
        x + y  - a_6
        > \frac{5(a_6-2)}{12} +1+ \frac{2(a_6-2)}{3}+1 - a_6
        = \frac{a_6-2}{12}, 
    \end{align*}
    as desired. 
    This completes the proof of Lemma~\ref{LEMMA:many-connection-A6}. 
\end{proof}

Given a vertex set $Z \subseteq V(\mathcal{A}_6)$, we denote by $\mathcal{A}_6 - Z$ the collection of members in $\mathcal{A}_6$ that have an empty intersection with $Z$.
\begin{lemma}\label{LEMMA:find-K8}
    Let $\mu \ge 0$ be an integer. 
    Suppose that 
    \begin{align*}
        e(\mathcal{A}_{6})
        \ge \frac{15 a_6^2}{2} +  \mu a_6. 
    \end{align*}
    Then for every subset $Z\subseteq  V(\mathcal{A}_{6})$ of size at most $\mu -3$, there exist six distinct members 
    \begin{align*}
        Q_1, Q_1',Q_2, Q_2',Q_3, Q_3'  \in \mathcal{A}_6 - Z 
    \end{align*}
    such that, for each $i \in [3]$,  $Q_{i} \cup Q_{i}'$ induces a copy of $K_8$ in $G$.
\end{lemma}
\begin{proof}[Proof of Lemma~\ref{LEMMA:find-K8}]
    Let us define an auxiliary graph $H$ whose vertex set is $\mathcal{A}_{6}$ and two elements $Q,Q' \in \mathcal{A}_{6}$ are adjacent in $H$ iff $e(Q,Q') = 16$. 
    Observe that it suffices to show that $\nu(H) \ge \mu \ge |Z| + 3$. 
    By Theorem~\ref{THM:Erdos-Gallai-matching}, this reduces to showing that 
    \begin{align*}
        |H|
        > \mu a_6. 
    \end{align*}
    Suppose to the contrary that $|H| \le \mu a_6$. Then 
    \begin{align*}
        e(\mathcal{A}_6)
         \le 16 \cdot |H| + 15 \cdot \left(\binom{a_6}{2} - |H|\right) + 6a_6
        <  \frac{15 a_6^2}{2} + |H| \le \frac{15 a_6^2}{2} + \mu a_6, 
    \end{align*}
    contradicting the assumption that $e(\mathcal{A}_6) \ge \frac{15 a_6^2}{2} + \mu a_6$. 
\end{proof}

\subsection{Preparations for Lemma~\ref{LEMMA:A6BCD-upper-bound-a} \RomanNumeralCaps{3}: six operations}\label{SUBSEC:6-Operations}
In the proof of Lemma~\ref{LEMMA:A6BCD-upper-bound-a}, we will employ six operations (rotations).  
In this subsection, we define these six operations, establish their existence in $\mathcal{A}_6$, and state the consequences of each operation.  

Throughout this subsection, we assume the following condition holds:
\begin{align}\label{equ:assump-A6-edges-six-rotations}
    e(\mathcal{A}_6)
    \ge \frac{15 a_6^2}{2} + \mu  a_6, 
\end{align}
where $\mu > 0$ is a fixed (large) integer. 
Since $e(\mathcal{A}_6) \le \binom{4a_6}{2} \le 8 a_6^2$ holds trivially, it follows from~\eqref{equ:assump-A6-edges-six-rotations} that $a_6 \ge 2\mu$. 

The following lemma establishes the existence of Operations 1 and 2 (as shown in Figures~\ref{Fig:Rotation-1} and~\ref{Fig:Rotation-2}).
\begin{lemma}\label{LEMMA:operation-1-2-exists}
    Let $Q_1, Q_{3} \in \mathcal{A}_{6}$ be two distinct members, with $\{p_1, q_1\} \subseteq  Q_1$ as two fixed vertices in $Q_{1}$ and $p_3 \in Q_{3}$ as a fixed vertex in $Q_{3}$.
    For every subset $B \subseteq  V(\mathcal{A}_{6}) \setminus (Q_1 \cup Q_3)$ of size less than $\frac{a_{6}-2}{12}$, there exists a member $Q_{2} \in \mathcal{A}_6 - (Q_1 \cup Q_3 \cup B)$ such that$\colon$
    \begin{enumerate}[label=(\roman*)]
        \item there exists a vertex $p_2 \in Q_2$ that is adjacent to all three vertices in $Q_{3}\setminus \{p_3\}$, and
        \item the remaining  three vertices in $Q_2 \setminus\{p_2\}$ have a common neighbor in $Q_{1}\setminus \{p_1, q_1\}$. 
    \end{enumerate}
\end{lemma}
\textbf{Remark.} We call $Q_1 Q_2 Q_3$ a \textbf{switching} path (\textbf{avoiding} $B$ and \textbf{connecting} $Q_1$ and $Q_3$) and refer to $Q_1$ and $Q_3$ as the \textbf{ends} of this switching path. 
\begin{proof}[Proof of Lemma~\ref{LEMMA:operation-1-2-exists}]
    This lemma is a direct consequence of Lemma~\ref{LEMMA:many-connection-A6} and Propositions~\ref{PROP:connection-A6-b}~\ref{PROP:connection-A6-b-4}$\colon$ 
    Given that $\frac{a_6 - 2}{12} > |B|$,  Lemma~\ref{LEMMA:many-connection-A6} guarantees the existence of a member $Q_{2} \in \mathcal{A}_{6} - B$ such that $Q_1\rightsquigarrow Q_2 \rightsquigarrow  Q_3$. Consequently, the assertions in Lemma~\ref{LEMMA:operation-1-2-exists} follows from Propositions~\ref{PROP:connection-A6-b}~\ref{PROP:connection-A6-b-4}. 
\end{proof}

\medskip 

%
\begin{figure}[H]
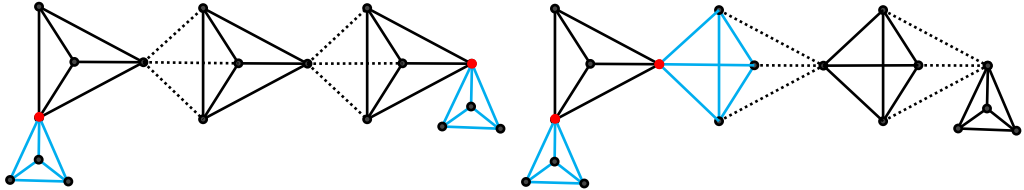

\centering
\tikzset{every picture/.style={line width=1pt}} 


\caption{Operation 1 moves one marked vertex (intersection of the cyan $K_4$'s and the black $K_4$'s) into a copy of $K_4$ that already contains a marked vertex. This process is represented in shorthand as $1+1 \to 2$.}
\label{Fig:Rotation-1}
\end{figure}

\medskip 

%
\begin{figure}[H]
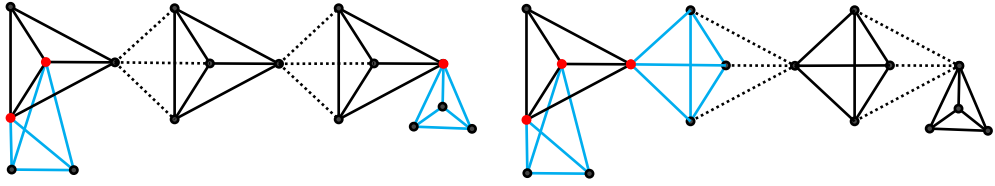

\centering
\tikzset{every picture/.style={line width=1pt}} 
%

\caption{Operation 2 moves one  marked vertex (intersection of the cyan $K_4$'s and the black $K_4$'s) into a copy of $K_4$ that already contains two marked vertices. This process is represented in shorthand as $2+1 \to 3$.}
\label{Fig:Rotation-2}
\end{figure}

The following lemma establishes the existence of Operation 3 (as shown in  Figure~\ref{Fig:Rotation-3}).
\begin{lemma}\label{LEMMA:operation-3-exists}
    Let $Q_1, Q_{6} \in \mathcal{A}_{6}$ be two distinct members, with $\{p_1, q_1\} \subseteq  Q_1$ as two fixed vertices in $Q_{1}$ and $\{p_6,q_6 \}\subseteq  Q_{6}$ as two fixed vertices in $Q_{6}$.
    For every subset $B \subseteq  V(\mathcal{A}_{6}) \setminus (Q_1 \cup Q_6)$ of size at most $\min\left\{\mu - 11, \frac{a_{6}-2}{12} - 6\right\}$, there exist four distinct members $Q_{2},Q_{3},Q_{4},Q_{5} \in \mathcal{A}_6 - (Q_1 \cup Q_6 \cup B)$ such that$\colon$
    \begin{enumerate}[label=(\roman*)]
        \item\label{LEMMA:operation-3-exists-1} there exists a pair of vertices $\{p_2, q_2\} \subseteq  Q_{2}$ such that both are adjacent to vertices in $Q_{1}\setminus \{p_1, q_1\}$, and 
        \item\label{LEMMA:operation-3-exists-2} the remaining two vertices in $Q_{2}\setminus \{p_2, q_2\}$ have at least two common neighbors, say $\{p_3, q_3\}$, in $Q_{3}$; 
        \item\label{LEMMA:operation-3-exists-3} there exists a pair of vertices $\{p_5, q_5\} \subseteq  Q_{5}$ such that both are adjacent to vertices in $Q_{6}\setminus \{p_6, q_6\}$, and 
        \item\label{LEMMA:operation-3-exists-4} the remaining two vertices in $Q_{5}\setminus \{p_5, q_5\}$ have two common neighbors, say $\{p_4, q_4\}$, in $Q_{4}$;
        \item\label{LEMMA:operation-3-exists-5} the remaining two vertices in $Q_{3}\setminus \{p_3, q_3\}$ are both adjacent to vertices in $Q_{4}\setminus \{p_4, q_4\}$.
    \end{enumerate}
\end{lemma}
\textbf{Remark.} We call $Q_1 \cdots Q_6$ a \textbf{switching} path (\textbf{avoiding} $B$ and \textbf{connecting} $Q_1$ and $Q_6$) and refer to $Q_1$ and $Q_6$ as the \textbf{ends} of this switching path. 

\begin{proof}[Proof of Lemma~\ref{LEMMA:operation-3-exists}]
    Let $Z \coloneqq V(Q_{1}\cup Q_{6})\cup B$, noting that $|Z|=8+|B| \le \mu-3$. 
    It follows from Lemma~\ref{LEMMA:find-K8} that there exists a pair of distinct members $Q_{3},Q_{4} \in \mathcal{A}_{6} - Z$ such that $V(Q_3 \cup Q_4)$ induces a copy of $K_8$ in $G$.

    Applying Lemma~\ref{LEMMA:many-connection-A6} to $(Q_{1},Q_{3})$, we find a member $Q_2 \in \mathcal{A}_{6} - (Q_{1}\cup Q_{3} \cup Q_{4} \cup Q_{6}\cup B)$ (this is possible since $|B| \le \frac{a_6 -2}{12} - 6 < \frac{a_6 -2}{12} - 4$) such that $Q_1 \rightsquigarrow Q_{2} \rightsquigarrow Q_{3}$. 
    Similarly, applying Lemma~\ref{LEMMA:many-connection-A6} to $(Q_{4},Q_{6})$, we find a member $Q_5 \in \mathcal{A}_{6} - (Q_{1}\cup Q_{2}\cup Q_{3} \cup Q_{4}  \cup Q_{6} \cup B)$ such that $Q_4 \rightsquigarrow Q_{5} \rightsquigarrow Q_{6}$. 

    Applying Proposition~\ref{PROP:connection-A6-b}~\ref{PROP:connection-A6-b-1} to $(Q_1, Q_2, Q_3)$, we obtain Lemma~\ref{LEMMA:operation-3-exists}~\ref{LEMMA:operation-3-exists-1} and~\ref{LEMMA:operation-3-exists-2}. 
    Similarly, applying Proposition~\ref{PROP:connection-A6-b}~\ref{PROP:connection-A6-b-2} to $(Q_4, Q_5, Q_6)$, we obtain Lemma~\ref{LEMMA:operation-3-exists}~\ref{LEMMA:operation-3-exists-3} and~\ref{LEMMA:operation-3-exists-4}. 
    Finally, Lemma~\ref{LEMMA:operation-3-exists}~\ref{LEMMA:operation-3-exists-5} follows from the fact that $V(Q_3 \cup Q_4)$ induces a copy of $K_8$ in $G$. 
\end{proof}

%
\begin{figure}[H]
\centering
\begin{minipage}{\textwidth}
\centering
\tikzset{every picture/.style={line width=1pt}} 


\end{minipage}
\caption{Operation 3 moves two  marked vertices (intersection of the cyan $K_4$'s and the black $K_4$'s) into a copy of $K_4$ that already contains two marked vertices. This process is represented in shorthand as $2+2 \to 4$.}
\label{Fig:Rotation-3}
\end{figure}

The following lemma establishes the existence of Operation 4 (as shown in Figure~\ref{Fig:Rotation-4}).
\begin{lemma}\label{LEMMA:operation-4-exists}
    Let $Q_1, Q_{6} \in \mathcal{A}_{6}$ be two distinct members, with $\{p_1, q_1, s_1\} \subseteq  Q_1$ as three fixed vertices in $Q_{1}$ and $p_6 \in Q_{6}$ as a fixed vertex in $Q_{6}$.
    For every subset $B \subseteq  V(\mathcal{A}_{6}) \setminus (Q_1 \cup Q_6)$ of size at most $\min\left\{\mu - 11, \frac{a_{6}-2}{12} - 6\right\}$, there exist four distinct members $Q_{2},Q_{3},Q_{4},Q_{5} \in \mathcal{A}_6 - (Q_1 \cup Q_6 \cup B)$ such that$\colon$
    \begin{enumerate}[label=(\roman*)]
        \item\label{LEMMA:operation-4-exists-1} there exist three vertices $\{p_2, q_2,s_2\} \subseteq  Q_{2}$ such that all are adjacent to the vertex in $Q_{1}\setminus \{p_1, q_1,s_1\}$, and 
        \item\label{LEMMA:operation-4-exists-2} the remaining vertex in $Q_{2}\setminus \{p_2, q_2,s_2\}$ is adjacent to at least three vertices, say $\{p_3, q_3, s_3\}$, in $Q_{3}$; 
        \item\label{LEMMA:operation-4-exists-3} there exist a vertex $p_5 \in Q_{5}$ that is adjacent to the remaining three vertices in $Q_{6}\setminus \{p_6\}$, and 
        \item\label{LEMMA:operation-4-exists-4} the remaining three vertices in $Q_{5}\setminus \{p_5\}$ have a common neighbor, say $p_4$, in $Q_{4}$; 
        \item\label{LEMMA:operation-4-exists-5} the remaining vertex in $Q_{3}\setminus \{p_3, q_3, s_3\}$ is adjacent to all three vertices in $Q_{4}\setminus \{p_4\}$.
    \end{enumerate}
\end{lemma}
\textbf{Remark.} We call $Q_1 \cdots Q_6$ a \textbf{switching} path (\textbf{avoiding} $B$ and \textbf{connecting} $Q_1$ and $Q_6$) and refer to $Q_1$ and $Q_6$ as the \textbf{ends} of this switching path. 
\begin{proof}[Proof of Lemma~\ref{LEMMA:operation-4-exists}]
    Choose $Q_2,Q_3,Q_4,Q_5$ as in the proof of Lemma~\ref{LEMMA:operation-3-exists}. 
    Applying Proposition~\ref{PROP:connection-A6-b}~\mbox{\ref{PROP:connection-A6-b-3}} to $(Q_1,Q_2,Q_3)$ with the vertex in $Q_1\setminus\{p_1,q_1,s_1\}$, we obtain triples $\{p_2,q_2,s_2\}\subseteq Q_2$ and $\{p_3,q_3,s_3\}\subseteq Q_3$ such that
    \begin{align*}
        e(Q_1\setminus\{p_1,q_1,s_1\},\{p_2,q_2,s_2\})=3,~ \quad ~ e(Q_2\setminus\{p_2,q_2,s_2\},\{p_3,q_3,s_3\})=3.
    \end{align*}
    Similarly, applying Proposition~\ref{PROP:connection-A6-b}~\ref{PROP:connection-A6-b-4} to $(Q_4,Q_5,Q_6)$ with the triple $Q_6\setminus\{p_6\}$, we obtain $p_5\in Q_5$ and $p_4\in Q_4$ such that
    \[
    e(\{p_5\},Q_6\setminus\{p_6\})=e(\{p_4\},Q_5\setminus\{p_5\})=3.
    \]
    Finally, since $e(Q_3,Q_4)=16$, we have
    \[ 
    e(Q_3\setminus\{p_3,q_3,s_3\},Q_4\setminus\{p_4\})=3.
    \]
    This completes the proof.
\end{proof}

\medskip 

%
\begin{figure}[H]
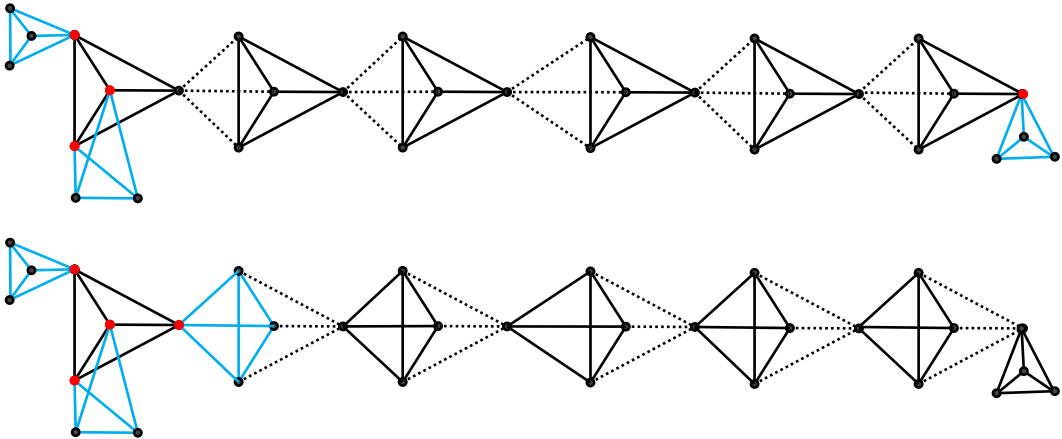

\centering
\begin{minipage}{\textwidth}
\centering
\tikzset{every picture/.style={line width=1pt}} 


\end{minipage}
\caption{Operation 4 moves one  marked vertex (intersection of the cyan $K_4$'s and the black $K_4$'s) into a copy of $K_4$ that already contains three marked vertices. This process is represented in shorthand as $3+1 \to 4$.}
\label{Fig:Rotation-4}
\end{figure}
%

%
The following lemma establishes the existence of Operation 5 (as shown in Figure~\ref{Fig:Rotation-5}).

\begin{lemma}\label{LEMMA:operation-5-exists}
    Let $Q_1, Q_{2},Q_{3} \in \mathcal{A}_{6}$ be three distinct members, with $\{p_1, q_1, s_1\} \subseteq  Q_1$ as three fixed vertices in $Q_{1}$, $\{p_2, q_2, s_2\} \subseteq  Q_{2}$ as three fixed vertices in $Q_{2}$, and $\{p_3, q_3\} \in Q_{3}$ as two fixed vertices in $Q_{3}$. 
    For every subset $B \subseteq  V(\mathcal{A}_{6}) \setminus (Q_1 \cup Q_2 \cup Q_3)$ of size at most $\min\left\{\mu - 15, \frac{a_{6}-2}{12} - 12\right\}$, there exist nine distinct members $Q_{4}, \ldots, Q_{12} \in \mathcal{A}_6 - (Q_1 \cup Q_2 \cup Q_3 \cup B)$ such that$\colon$
    \begin{enumerate}[label=(\roman*)]
        \item\label{LEMMA:operation-5-exists-1} there exist three vertices $\{p_4, q_4,s_4\} \subseteq  Q_{4}$ such that all are adjacent to the vertex in $Q_{1}\setminus \{p_1, q_1,s_1\}$, and the remaining vertex in $Q_{4}\setminus \{p_4, q_4,s_4\}$ is adjacent to at least three vertices, say $\{p_5, q_5, s_5\}$, in $Q_{5}$; 
        \item\label{LEMMA:operation-5-exists-2} there exist three vertices $\{p_6, q_6, s_6\} \subseteq  Q_{6}$ such that all are adjacent to the vertex in $Q_{2}\setminus \{p_2, q_2,s_2\}$, and  the remaining vertex in $Q_{6}\setminus \{p_6, q_6,s_6\}$ is adjacent to at least three vertices, say $\{p_7, q_7, s_7\}$, in $Q_{7}$; 
        \item\label{LEMMA:operation-5-exists-3} there exist two vertices $\{p_8, q_8\} \subseteq  Q_{8}$ such that both are adjacent to vertices in $Q_{3}\setminus \{p_3, q_3\}$, and the remaining two vertices in $Q_{8}\setminus \{p_8, q_8\}$ are both adjacent to at least two vertices, say $\{p_9, q_9\}$, in $Q_{9}$;
        \item\label{LEMMA:operation-5-exists-5} there exist two vertices $\{p_{12}, q_{12}\} \subseteq  Q_{12}$ such that both are adjacent to the vertex in  $Q_{5}\setminus \{p_5, q_5, s_5\}$ and the vertex in $Q_{7}\setminus \{p_7, q_7, s_7\}$; 
        \item\label{LEMMA:operation-5-exists-5b} the vertex in $Q_{5}\setminus \{p_5, q_5, s_5\}$ is adjacent to the vertex in $Q_{7}\setminus \{p_7, q_7, s_7\}$; 
        \item\label{LEMMA:operation-5-exists-6} there exist two vertices $\{p_{11}, q_{11}\} \subseteq  Q_{11}$ such that both are adjacent to vertices in $Q_{12}\setminus \{p_{12}, q_{12}\}$, and the remaining two vertices in $Q_{11}\setminus \{p_{11}, q_{11}\}$ are adjacent to at least two vertices, say $\{p_{10}, q_{10}\}$, in $Q_{10}$;
        \item\label{LEMMA:operation-5-exists-7} the remaining two vertices in $Q_{10}\setminus \{p_{10}, q_{10}\}$ are adjacent to the remaining two vertices in $Q_{9}\setminus \{p_{9}, q_{9}\}$.
    \end{enumerate}
\end{lemma}
\textbf{Remark.} We call $Q_1 \cdots Q_{12}$ a \textbf{switching} path\footnote{Strictly speaking, it is a tree.} (\textbf{avoiding} $B$ and \textbf{connecting} $Q_1$, $Q_2$, and $Q_3$) and refer to $Q_1$, $Q_2$, and $Q_3$ as the \textbf{ends} of this switching path. 
\medskip 
\begin{figure}[H]
\centering 
\begin{minipage}{\textwidth}
\centering
\tikzset{every picture/.style={line width=1pt}} 



\end{minipage}
\caption{Operation 5 moves two  marked vertices into two copies of $K_4$, each of which already contains three marked vertices. This process is represented in shorthand as $3+3+2 \to 4+4$. Due to space limitations, some cyan $K_4$'s have been omitted.}
\label{Fig:Rotation-5}
\end{figure}

\begin{proof}[Proof of Lemma~\ref{LEMMA:operation-5-exists}]
    Let $Z \coloneqq Q_1 \cup Q_2 \cup Q_3 \cup B$, noting that $|Z| \le \mu-3$.  
    It follows from Lemma~\ref{LEMMA:find-K8} that there exist four distinct members $Q_{5},Q_{7}, Q_{9}, Q_{10} \in \mathcal{A}_{6} - Z$ such that $V(Q_5 \cup Q_7)$ and $V(Q_9 \cup Q_{10})$ both induce a copy of $K_8$ in $G$. 
    
    Applying Lemma~\ref{LEMMA:many-connection-A6} to $(Q_{1},Q_{5})$, we find a member 
    \begin{align*}
        Q_4 \in \mathcal{A}_{6} - (Q_1 \cup Q_2 \cup Q_3 \cup Q_5 \cup Q_7 \cup Q_9 \cup Q_{10}\cup B)
    \end{align*}
    such that $Q_1 \rightsquigarrow Q_{4} \rightsquigarrow Q_{5}$. 
    
    Applying Lemma~\ref{LEMMA:many-connection-A6} to $(Q_{2},Q_{7})$, we find a member 
    \begin{align*}
        Q_6 
        \in \mathcal{A}_6 - (Q_1 \cup Q_2 \cup Q_3 \cup Q_4 \cup Q_5 \cup Q_7 \cup Q_9 \cup Q_{10}\cup B)
    \end{align*}
    such that $Q_2 \rightsquigarrow Q_{6} \rightsquigarrow Q_{7}$. 
    
    Applying Lemma~\ref{LEMMA:many-connection-A6} to $(Q_{3},Q_{9})$, we find a member 
    \begin{align*}
        Q_8 
        \in \mathcal{A}_6 - (Q_1 \cup Q_2 \cup Q_3 \cup Q_4 \cup Q_5 \cup Q_6 \cup Q_7 \cup Q_9 \cup Q_{10}\cup B)
    \end{align*}
    such that $Q_3 \rightsquigarrow Q_{8} \rightsquigarrow Q_{9}$. 
    
    Applying Lemma~\ref{LEMMA:many-connection-A6} to $(Q_{5},Q_{7})$, we find a member
    \begin{align*}
        Q_{12}
        \in \mathcal{A}_6 - (Q_1 \cup Q_2 \cup Q_3 \cup Q_4 \cup Q_5 \cup Q_6 \cup Q_7 \cup Q_8 \cup Q_9 \cup Q_{10}\cup B)
    \end{align*}
    such that $Q_5 \rightsquigarrow Q_{12} \rightsquigarrow Q_{7}$.
    
    Finally, applying Lemma~\ref{LEMMA:many-connection-A6} to $(Q_{10},Q_{12})$, we find a member 
    \begin{align*}
        Q_{11}
        \in \mathcal{A}_6 - (Q_1 \cup Q_2 \cup Q_3 \cup Q_4 \cup Q_5 \cup Q_6 \cup Q_7 \cup Q_8 \cup Q_9 \cup Q_{10} \cup Q_{12}\cup B)
    \end{align*}
    such that $Q_{10} \rightsquigarrow Q_{11} \rightsquigarrow Q_{12}$. 
    Applying Proposition~\ref{PROP:connection-A6-b}~\ref{PROP:connection-A6-b-3} to $(Q_1,Q_4,Q_5)$ and $(Q_2,Q_6,Q_7)$, with the vertices in $Q_1\setminus\{p_1,q_1,s_1\}$ and $Q_2\setminus\{p_2,q_2,s_2\}$, respectively, we obtain the triples in Lemma~\ref{LEMMA:operation-5-exists}~\ref{LEMMA:operation-5-exists-1} and Lemma~\ref{LEMMA:operation-5-exists}~\ref{LEMMA:operation-5-exists-2}.
    Applying Proposition~\ref{PROP:connection-A6-b}~\ref{PROP:connection-A6-b-1} to $(Q_3,Q_8,Q_9)$ with the pair $Q_3\setminus\{p_3,q_3\}$ gives Lemma~\ref{LEMMA:operation-5-exists}~\ref{LEMMA:operation-5-exists-3}.
    
    Let $\{t_5\}:=Q_5\setminus\{p_5,q_5,s_5\}$ and $\{t_7\}:=Q_7\setminus\{p_7,q_7,s_7\}$. By Definition~\ref{DEF:connects-Q1-Q2-Q3}, 
    \[
    e(\{t_5,t_7\},Q_{12}) \ge e(Q_5,Q_{12})+e(Q_{12},Q_7)-6\cdot4 \ge30-24=6.\]
    Thus, Fact~\ref{FACT:3-by-4-bipartite-graph}~\ref{FACT:3-by-4-bipartite-graph--1} gives a pair $\{p_{12},q_{12}\}\subseteq Q_{12}$ adjacent to both $t_5$ and $t_7$.
    Applying Proposition~\ref{PROP:connection-A6-b}~\ref{PROP:connection-A6-b-2} to $(Q_{10},Q_{11},Q_{12})$ with the pair $Q_{12}\setminus\{p_{12},q_{12}\}$, we obtain the pairs in Lemma~\ref{LEMMA:operation-5-exists}~\ref{LEMMA:operation-5-exists-6}.
    Finally, since $e(Q_5,Q_7)=e(Q_9,Q_{10})=16$, we have $t_5t_7\in G$ and
    \[
    e(Q_{10}\setminus\{p_{10},q_{10}\}, Q_9\setminus\{p_9,q_9\})=4.
    \]
    This proves Lemma~\ref{LEMMA:operation-5-exists}~\ref{LEMMA:operation-5-exists-5b} and Lemma~\ref{LEMMA:operation-5-exists}~\ref{LEMMA:operation-5-exists-7}, completing the proof.
\end{proof}
The following lemma establishes the existence of Operation 6 (see Figure~\ref{Fig:Rotation-6}).
\begin{lemma}\label{LEMMA:operation-6-exists}
    Let $Q_1, Q_{2},Q_{3}, Q_{4} \in \mathcal{A}_{6}$ be four distinct members, with $\{p_i, q_i, s_i\} \subseteq  Q_i$ as three fixed vertices in $Q_{i}$ for $i \in [4]$. 
    For every subset $B \subseteq  V(\mathcal{A}_{6}) \setminus (Q_1 \cup Q_2 \cup Q_3 \cup Q_4)$ of size at most $\min\left\{\mu - 19, \frac{a_{6}-2}{12} - 18\right\}$, there exist $14$ distinct members $\tilde{Q}_{1}, \hat{Q}_{1}, \ldots, \tilde{Q}_{4}, \hat{Q}_{4}, Q_{5}, \ldots, Q_{10} \in \mathcal{A}_6 - (Q_{1}\cup Q_{2} \cup Q_{3} \cup Q_4 \cup B)$ such that$\colon$
    \begin{enumerate}[label=(\roman*)]
        \item\label{LEMMA:operation-6-exists-1} for each $i \in [4]$, there exist three vertices $\{\tilde{p}_i, \tilde{q}_i, \tilde{s}_i\} \subseteq  \tilde{Q}_{i}$ such that all are adjacent to the vertex in $Q_{i}\setminus \{p_i, q_i, s_i\}$, and the remaining vertex in $\tilde{Q}_{i}\setminus \{\tilde{p}_i, \tilde{q}_i, \tilde{s}_i\}$ is adjacent to at least three vertices, say $\{\hat{p}_i, \hat{q}_i, \hat{s}_i\}$, in $\hat{Q}_{i}$; 
        \item\label{LEMMA:operation-6-exists-2} there exist two vertices $\{p_5, q_5\} \subseteq  Q_{5}$ such that both are adjacent to the vertex in $\hat{Q}_{1}\setminus \{\hat{p}_1, \hat{q}_1, \hat{s}_1\}$ and the vertex in $\hat{Q}_{2}\setminus \{\hat{p}_2, \hat{q}_2, \hat{s}_2\}$;
        \item\label{LEMMA:operation-6-exists-2b} the vertex in $\hat{Q}_{1}\setminus \{\hat{p}_1, \hat{q}_1, \hat{s}_1\}$ is adjacent to the vertex in $\hat{Q}_{2}\setminus \{\hat{p}_2, \hat{q}_2, \hat{s}_2\}$;
        \item\label{LEMMA:operation-6-exists-3}  there exist two vertices $\{p_{10}, q_{10}\} \subseteq  Q_{10}$ such that both are adjacent to the vertex in $\hat{Q}_{3}\setminus \{\hat{p}_3, \hat{q}_3, \hat{s}_3\}$ and the vertex in $\hat{Q}_{4}\setminus \{\hat{p}_4, \hat{q}_4, \hat{s}_4\}$;
        \item\label{LEMMA:operation-6-exists-3b} the vertex in $\hat{Q}_{3}\setminus \{\hat{p}_3, \hat{q}_3, \hat{s}_3\}$ is adjacent to the vertex in $\hat{Q}_{4}\setminus \{\hat{p}_4, \hat{q}_4, \hat{s}_4\}$;
        \item\label{LEMMA:operation-6-exists-4} there exists two vertices $\{p_6, q_6\} \subseteq  Q_6$ such that both are adjacent to vertices in $Q_5 \setminus \{p_5, q_5\}$, and the remaining two vertices in $Q_6\setminus \{p_6, q_6\}$ have at least two common neighbors, say $\{p_7, q_7\}$, in $Q_7$; 
        \item\label{LEMMA:operation-6-exists-5} there exists two vertices $\{p_9, q_9\} \subseteq  Q_9$ such that both are adjacent to vertices in $Q_{10} \setminus \{p_{10}, q_{10}\}$, and the remaining two vertices in $Q_9\setminus \{p_9, q_9\}$ have at least two common neighbors, say $\{p_8, q_8\}$, in $Q_8$; 
        \item\label{LEMMA:operation-6-exists-6} the remaining two vertices in $Q_{7}\setminus \{p_{7}, q_{7}\}$ are both adjacent to the remaining two vertices in $Q_{8}\setminus \{p_{8}, q_{8}\}$.
    \end{enumerate}
\end{lemma}
\textbf{Remark.} We call the above configuration a \textbf{switching} path\footnote{Strictly speaking, it is a tree.} (\textbf{avoiding} $B$ and \textbf{connecting} $Q_1$, $Q_2$, $Q_3$, and $Q_4$) and refer to $Q_1$, $Q_2$, $Q_3$, and $Q_4$ as the \textbf{ends} of this switching path. 

\medskip 

%
\begin{figure}[H]
\centering
\begin{minipage}{\textwidth}
\centering
\tikzset{every picture/.style={line width=0.9pt}} 

\end{minipage}
\caption{Operation 6 moves three  marked vertices into three copies of $K_4$, each of which already contains three marked vertices. This process is represented in shorthand as $3+3+3+3 \to 4+4+4$. Due to space limitations, some cyan $K_4$'s have been omitted.}
\label{Fig:Rotation-6}
\end{figure}

\begin{proof}[Proof of Lemma~\ref{LEMMA:operation-6-exists}]
    Let $Z \coloneqq Q_1 \cup Q_2 \cup Q_3 \cup Q_4 \cup B$, noting that $|Z| \le \mu-3$. 
    It follows from Lemma~\ref{LEMMA:find-K8} that there exist six distinct members $\hat{Q}_{1},\hat{Q}_{2},\hat{Q}_{3}, \hat{Q}_{4},  Q_{7}, Q_{8} \in \mathcal{A}_{6}- Z$ such that $V(\hat{Q}_{1} \cup \hat{Q}_{2})$, $V(\hat{Q}_{3} \cup \hat{Q}_{4})$, and $V(Q_7 \cup Q_{8})$ all induce a copy of $K_8$ in $G$. 
    
    For each $i \in [4]$, applying Lemma~\ref{LEMMA:many-connection-A6} to $(Q_{i},\hat{Q}_{i})$, we find a member 
    \begin{align*}
        \tilde{Q}_i 
        \in \mathcal{A}_{6} - (\hat{Q}_{1} \cup \hat{Q}_{2} \cup \hat{Q}_{3} \cup \hat{Q}_{4} \cup Q_7 \cup Q_{8} \cup Q_{1}\cup Q_{2} \cup Q_{3} \cup Q_4 \cup B)
    \end{align*}
    such that $Q_{i} \rightsquigarrow \tilde{Q}_i \rightsquigarrow \hat{Q}_{i}$. 
    Moreover, we can ensure that $\tilde{Q}_1,\ldots, \tilde{Q}_4$ are pairwise disjoint. 
    
    Applying Lemma~\ref{LEMMA:many-connection-A6} to $(\hat{Q}_{1}, \hat{Q}_{2})$ and $(\hat{Q}_{3}, \hat{Q}_{4})$ respectively, we find distinct members 
    \begin{align*}
        Q_5, Q_{10} 
        \in \mathcal{A}_{6} - (Q_{1}\cup \cdots \cup Q_4 \cup \tilde{Q}_{1} \cup \cdots \cup \tilde{Q}_{4} \cup \hat{Q}_{1} \cup \cdots \cup \hat{Q}_{4} \cup Q_7 \cup Q_{8} \cup  B)
    \end{align*} 
    such that $\hat{Q}_{1} \rightsquigarrow Q_5 \rightsquigarrow \hat{Q}_{2}$ and $\hat{Q}_{3} \rightsquigarrow Q_{10} \rightsquigarrow \hat{Q}_{4}$. 
    
    Finally, applying Lemma~\ref{LEMMA:many-connection-A6} to $(Q_5, Q_7)$ and $(Q_{10}, Q_8)$ respectively, we find distinct members 
    \begin{align*}
        Q_6, Q_{9} 
        \in \mathcal{A}_{6} - (Q_{1}\cup \cdots \cup Q_4 \cup \tilde{Q}_{1} \cup \cdots \cup \tilde{Q}_{4} \cup \hat{Q}_{1} \cup \cdots \cup \hat{Q}_{4} \cup Q_5 \cup Q_7 \cup Q_{8} \cup Q_{10} \cup  B) 
    \end{align*}
    such that $Q_5 \rightsquigarrow Q_6 \rightsquigarrow Q_7$ and $Q_{10} \rightsquigarrow Q_{9} \rightsquigarrow Q_8$. 
    For each $i\in[4]$, applying Proposition~\ref{PROP:connection-A6-b}~\ref{PROP:connection-A6-b-3} to $(Q_i,\tilde{Q}_i,\hat{Q}_i)$ with the vertex in $Q_i\setminus\{p_i,q_i,s_i\}$, we obtain the triples in  Lemma~\ref{LEMMA:operation-6-exists}~\ref{LEMMA:operation-6-exists-1}.
    Let
    \[
    \{\hat{t}_i\} :=\hat{Q}_i\setminus \{\hat{p}_i,\hat{q}_i,\hat{s}_i\}\quad\text{for }i\in[4].
    \]
    By Definition~\ref{DEF:connects-Q1-Q2-Q3},
    \begin{align*}
        e(\{\hat{t}_1,\hat{t}_2\},Q_5)
        &\ge e(\hat{Q}_1,Q_5)+e(Q_5,\hat{Q}_2)-24\ge6,\\
        e(\{\hat{t}_3,\hat{t}_4\},Q_{10})
        &\ge e(\hat{Q}_3,Q_{10})+e(Q_{10},\hat{Q}_4)-24
        \ge6.
    \end{align*}
    Thus, Fact~\ref{FACT:3-by-4-bipartite-graph}~\ref{FACT:3-by-4-bipartite-graph--1} gives pairs $\{p_5,q_5\}\subseteq Q_5$ and $\{p_{10},q_{10}\}\subseteq Q_{10}$ satisfying Lemma~\ref{LEMMA:operation-6-exists}~\ref{LEMMA:operation-6-exists-2} and Lemma~\ref{LEMMA:operation-6-exists}~\ref{LEMMA:operation-6-exists-3}, respectively.
    Applying Proposition~\ref{PROP:connection-A6-b}~\ref{PROP:connection-A6-b-1} to $(Q_5,Q_6,Q_7)$ and $(Q_{10},Q_9,Q_8)$, with the pairs $Q_5\setminus\{p_5,q_5\}$ and $Q_{10}\setminus\{p_{10},q_{10}\}$, respectively, we obtain Lemma~\ref{LEMMA:operation-6-exists}~\ref{LEMMA:operation-6-exists-4} and~\ref{LEMMA:operation-6-exists-5}.
    Finally, Since
    \[
    e(\hat{Q}_1,\hat{Q}_2)=e(\hat{Q}_3,\hat{Q}_4)=e(Q_7,Q_8)=16,
    \]
    we have $\hat{t}_1\hat{t}_2, \hat{t}_3\hat{t}_4\in G$ and
    \[
    e(Q_7\setminus\{p_7,q_7\},Q_8\setminus\{p_8,q_8\})=4.
    \]
    This proves Lemma~\ref{LEMMA:operation-6-exists}~\ref{LEMMA:operation-6-exists-2b},~\ref{LEMMA:operation-6-exists-3b}, and~\ref{LEMMA:operation-6-exists-6}, completing the proof.
\end{proof}

\subsection{Proof of Lemma~\ref{LEMMA:A6BCD-upper-bound-a}}\label{SUBSEC:proof-LEMMA:A6BCD-upper-bound-a}
In this subsection, we present the proof of Lemma~\ref{LEMMA:A6BCD-upper-bound-a}. 
Note that by Proposition~\ref{PROP:K4-tiling-cover-three-vtx} and~\eqref{equ:P-h-a-upper-bound}, it suffices to prove the following lemma. 
\begin{lemma}\label{LEMMA:A6BCD-upper-bound-a-reduced}
    Let $A \coloneqq V(\mathcal{B} \cup \mathcal{C} \cup \mathcal{D})$, $U \coloneqq V(\mathcal{A}_{6}) \cup A$, and $H \coloneqq G[U]$. 
    There exists an absolute constant $\mu > 0$ such that the following holds. 
    Suppose that 
    \begin{align*}
        e(\mathcal{A}_{6})
        \ge \frac{15 a_6^2}{2} + \mu a_6. 
    \end{align*}
    Then every $K_4$-tiling in $H$ contains at most three vertices in $A$. 
\end{lemma}
\textbf{Remark.} Applying Proposition~\ref{PROP:K4-tiling-cover-three-vtx} to $H$ with $h = 4a_6 + 3b+2c+d$ and $\alpha = 3b+2c+d$, we obtain Lemma~\ref{LEMMA:A6BCD-upper-bound-a}. 
\begin{proof}[Proof of Lemma~\ref{LEMMA:A6BCD-upper-bound-a-reduced}]
Let $\mu > 0$ be a sufficiently large integer. Recall from Section~\ref{SUBSEC:6-Operations} (see the argument after~\eqref{equ:assump-A6-edges-six-rotations}) that we have $a_6 \ge 2 \mu$ is also sufficiently large. 
Suppose to the contrary that there exists a $K_4$-tiling $\mathcal{S}$ in $H$ such that $|V(\mathcal{S}) \cap A| \ge 4$. Among all such $K_4$-tilings, select one with the smallest size. 
Recall from the maximality of $(|\mathcal{A}|, |\mathcal{B}|, |\mathcal{C}|, |\mathcal{D}|)$ that $H[A]$ is $K_4$-free. 

Observe that every member in $\mathcal{S}$ must contain at least one vertex from $V(\mathcal{A}_6)$, and therefore must be of Type \RomanNumeralCaps{1}, \RomanNumeralCaps{2}, or \RomanNumeralCaps{3}, as defined in Figure~\ref{Fig:three-types-K4} below. 

\medskip 

\begin{figure}[H]
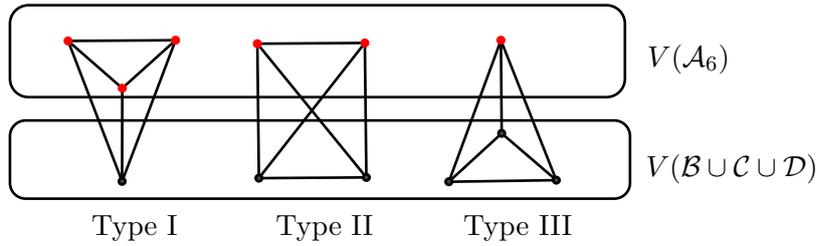

\centering

\tikzset{every picture/.style={line width=1pt}} 



\caption{Three types of $K_4$ crossing $V(\mathcal{A}_6)$ and $V(\mathcal{B}\cup \mathcal{C} \cup \mathcal{D})$. Red vertices are currently covered by two $K_4$'s (crossing $K_4$'s and $K_4$'s within $\mathcal{A}_6$). } 
\label{Fig:three-types-K4}
\end{figure}

By the minimality of $|\mathcal{S}|$, the number of different types of $K_4$ in $\mathcal{S}$ falls into one of the following six cases$\colon$
\begin{flushleft}
Case 1. $4 \times$ Type \RomanNumeralCaps{1}.\\
Case 2. $2 \times$ Type \RomanNumeralCaps{1}
      $+ 1 \times$ Type \RomanNumeralCaps{2}.\\
Case 3. $1 \times$ Type \RomanNumeralCaps{1}
      $+ 1 \times$ Type \RomanNumeralCaps{3}.\\
Case 4. $2 \times$ Type \RomanNumeralCaps{2}.\\
Case 5. $1 \times$ Type \RomanNumeralCaps{2}
      $+ 1 \times$ Type \RomanNumeralCaps{3}.\\
Case 6. $2 \times$ Type \RomanNumeralCaps{3}.
\end{flushleft}

\medskip 

\begin{figure}[H]
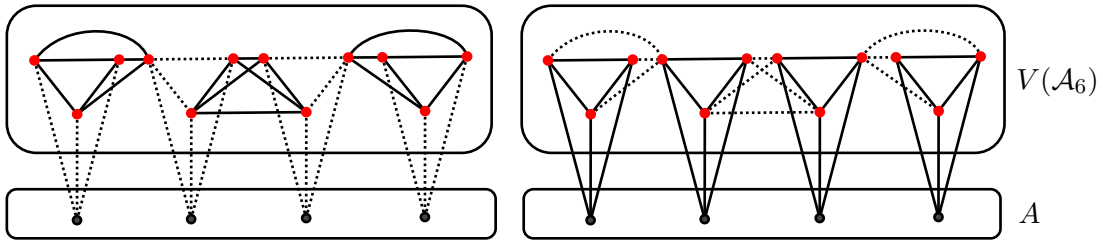

\centering
\tikzset{every picture/.style={line width=1pt}} 


\caption{Left: the set of marked vertices (red) are of type $4+4+4$. Right: after rotation, the number of vertex-disjoint copies of $K_4$ increases by one.} 
\label{Fig:Rotation-Final-case1}
\end{figure}

Let $Z \coloneqq V(\mathcal{S}) \cap V(\mathcal{A}_6)$. 
Let $\mathcal{Z}$ be the collection of members in $\mathcal{A}_6$ that have nonempty intersection with $Z$. 
Observe that we must have 
\begin{align}\label{equ:3-vtx-Z-S}
    |\mathcal{Z}| 
    \ge |\mathcal{S}|,
\end{align}
since otherwise, we could increase the size of $\mathcal{A}$ by updating it to $(\mathcal{A} \setminus \mathcal{Z}) \cup \mathcal{S}$ (see e.g. Figure~\ref{Fig:Rotation-Final-case1}), contradicting the maximality of $(|\mathcal{A}|, |\mathcal{B}|, |\mathcal{C}|, |\mathcal{D}|)$ (in particular, the maximality of $|\mathcal{A}|$). 

Let 
\begin{align*}
    Z_0 
    \coloneqq Z,
    \quad 
    \mathcal{Z}_{0}
    \coloneqq \mathcal{Z}, 
    \quad 
    \mathcal{S}_{0}
    \coloneqq \mathcal{S},
    \quad
    B_0 
    \coloneqq 
    V(\mathcal{Z}_0 \cup \mathcal{S}_0),
    \quad\text{and}\quad
    \mathcal{K}_0 
    \coloneqq \mathcal{A}_6 \cup \mathcal{S}.
\end{align*}
Suppose that we have defined $Z_{i}$, $\mathcal{Z}_i$, $\mathcal{S}_{i}$, $B_i$, $\mathcal{K}_i$ for some $i \ge 0$. 
For a sequence of descending positive integers $(\alpha_1, \ldots, \alpha_{m}) \in [4]^{m}$ that satisfies $\alpha_1 + \cdots + \alpha_m = |Z_i|$, we say that $Z_i$ is of \textbf{type $\alpha_1 \ldots \alpha_{m}$} if there exist $m$ distinct members $Q_1, \ldots, Q_m \in \mathcal{K}_i \setminus \mathcal{S}_i$ such that 
\begin{align*}
    \text{$|Q_j \cap Z_i| = \alpha_j$ for each $j \in [m]$.}
\end{align*}
The length $m$ of the tuple $(a_1, \ldots, a_m)$ is denoted by $\mathrm{length}(Z_i)$. 
For instance, the set $Z$ (of red vertices) in Figure~\ref{Fig:Rotation-Final-case1} is of type $444$ and $\mathrm{length}(Z) = 3$. 

\medskip 

\textbf{Case 1}$\colon$ $4 \times \text{Type \RomanNumeralCaps{1}}$.

For types not displayed in Figure~\ref{Fig:Operations-Case1}, we leave every $Q\in\mathcal{Z}_i$ with $|Q\cap Z_i|=4$ unchanged.
The remaining members contain eight or four marked vertices and are treated using the operations in Figure~\ref{Fig:Operation-Case2} and Figure~\ref{Fig:Rotations-Case3-6},respectively.
\begin{figure}[h]
\centering

\tikzset{every picture/.style={line width=0.75pt}} 



\caption{Operations for the displayed types of $Z_i$ in Case 1.}  
\label{Fig:Operations-Case1}
\end{figure}

\medskip 

In this case, $|Z_0|=|Z|=4\cdot3=12$. Each of the six operations preserves the number of marked vertices, so $|Z_i|=12$ throughout the process below. 
Define $Z_{i+1}, \mathcal{Z}_{i+1}$, $\mathcal{S}_{i+1}$, $B_{i+1}$, $\mathcal{K}_{i+1}$ using the following process:
\begin{enumerate}[label=(\roman*)]
    \item\label{process-1} Note that for every possible type of $Z_i$, there is a corresponding operation (with some cases having more than one operation) listed in Figure~\ref{Fig:Operations-Case1} and the preceding paragraph. For example, Operation 1 and Operation 2 are the corresponding operations for type $21111111111$. 
    
    Choose an operation according to the type of $Z_i$ and Figure~\ref{Fig:Operations-Case1} (if there are at least two options, select one arbitrarily). Denote the chosen operation by $\mathrm{Ope}_{Z_i}$. 
    \item\label{process-2} Given the operation $\mathrm{Ope}_{Z_i}$,  we need to construct a switching path. We begin by fixing the ends (which consist of at most four distinct members $Q_j \in \mathcal{Z}_i$) for this switching path, which are chosen  according $\mathrm{Ope}_{Z_i}$. 
    For example, if $\mathrm{Ope}_{Z_i}$ is Operation 2 (recall that the function of Operation 2 is $2+1 \to 3$), then we fix ends $Q_1, Q_2 \in \mathcal{Z}_i$ such that $|Q_1 \cap Z_i| = 2$ and $|Q_2 \cap Z_i| = 1$. 
    Denote the collection of these ends by $\mathcal{E}_i$.
    %
    \item\label{process-3} After fixing the ends, we apply the corresponding lemma from Section~\ref{SUBSEC:6-Operations} (i.e. Lemma~\ref{LEMMA:operation-1-2-exists}, \ref{LEMMA:operation-3-exists},~\ref{LEMMA:operation-4-exists},~\ref{LEMMA:operation-5-exists}, and~\ref{LEMMA:operation-6-exists}) with forbidden set $(B_i\cap V(\mathcal{A}_6))\setminus V(\mathcal{E}_i)$ to construct the desired switching path, denoted by $P_{Z_i}$, connecting these ends. Every member of $P_{Z_i}$ other than the prescribed ends is disjoint from $B_i$.
    %
    \item\label{process-4} After constructing the switching path, we perform the operation $\mathrm{Ope}_{Z_i}$ on $P_{Z_i}$. 
    Let $\hat{P}_{Z_i}$ denote the collection consisting of the new $K_4$'s formed by this operation and all ends of $P_{Z_i}$ except the one removed (refer to the configurations after the operation in Figures~\ref{Fig:Rotation-1},~\ref{Fig:Rotation-2},~\ref{Fig:Rotation-3},~\ref{Fig:Rotation-4},~\ref{Fig:Rotation-5}, and~\ref{Fig:Rotation-6}).

    %
    \item\label{process-5} The set $\mathcal{K}_{i+1}$ is obtained from $\mathcal{K}_{i}$ by removing all  members on $P_{Z_i}$ and adding all $K_4$'s on $\hat{P}_{Z_i}$. 
    For example, if $\mathrm{Ope}_{Z_i}$ is Operation 2, then we remove $\{Q_1, Q_2, Q_3\}$ and add $Q_1$ along with the two new $K_4$'s (see Figure~\ref{Fig:Rotation-2}).
    %

    The marked vertices in an end need not belong to the same member of $\mathcal{S}_i$. The new $K_4$'s contain no vertex of $Z_i$, and all members of $\mathcal{S}_i$ remain unchanged in $\mathcal{K}_{i+1}$.

    \item\label{process-6} Note that for all six possible operations, there is exactly one end $Q_j$ that is no longer contained in $\hat{P}_{Z_i}$ (and also $\mathcal{K}_{i+1}$). The set $\mathcal{Z}_{i+1}$ is obtained from $\mathcal{Z}_i$ by removing this $Q_j$. 
    The set $\mathcal{S}_{i+1}$ is obtained from $\mathcal{S}_i$ by removing all members that are disjoint from $V(\mathcal{Z}_{i+1})$ and adding all members of $\mathcal{K}_{i+1}\setminus\mathcal{Z}_{i+1}$ that have nonempty intersection with the other ends.
    For example, if $\mathrm{Ope}_{Z_i}$ is Operation 2, then we remove the cyan $K_4$ on the right from $\mathcal{S}_i$ and add the new cyan $K_4$ to $\mathcal{S}_{i+1}$ (see Figure~\ref{Fig:Rotation-2}).
    %
    \item\label{process-7} Let $B_{i+1} \coloneqq B_0\cup V(\mathcal{A}_6) \setminus V(\mathcal{A}_6 \cap \mathcal{K}_{i+1})$ and $Z_{i+1} \coloneqq V(\mathcal{Z}_{i+1}) \cap V(\mathcal{S}_{i+1})$.
\end{enumerate}
We repeat this process until $Z_i$ is of type $444$ and assume  that the process terminates after $t$ steps (so $t$ is the index at which the process terminates).
From the definition of the process above, it is clear that $\mathrm{length}(Z_{i+1}) = \mathrm{length}(Z_i) - 1$. Therefore, the process terminates after at most $12 - 3 = 9$ steps.

The following claim follows from the definition of the process. 

%
\begin{claim}\label{CLAIM:property-rotation-case-one}
    The following statements hold. 
    \begin{enumerate}[label=(\roman*)]
        \item\label{CLAIM:property-rotation-case-one-1} $|\mathcal{A}_6| + |\mathcal{S}| = |\mathcal{K}_{0}| = \cdots = |\mathcal{K}_t|$, 
        \item\label{CLAIM:property-rotation-case-one-2} $V(\mathcal{A}_6 \cup \mathcal{S}) = V(\mathcal{K}_{0}) = \cdots = V(\mathcal{K}_t)$,
        \item\label{CLAIM:property-rotation-case-one-3} $\mathcal{K}_{i}\setminus \mathcal{Z}_i$ is a $K_4$-tiling in $G$ for every $i \in [0,t]$,
        \item\label{CLAIM:property-rotation-case-one-4} $|\mathcal{Z}_i| = \mathrm{length}(Z_i)$ for every $i \in [0,t]$,
         \item\label{CLAIM:property-rotation-case-one-5} $|B_0| \le 4 + 12\cdot 4 = 52$ and, for every $0\le i <t $, since each switching path uses at most $18$ copies of $K_4$, we have
        \begin{align*}
            |B_{i+1}| 
            \le |B_{i}| + 18\cdot 4 = |B_{i}| + 72.
        \end{align*}
        Consequently, 
        \begin{align*}
            |B_{i}| 
            \le |B_0| + 72 t
            \le 52 + 72\cdot 9 = 700
            \quad\text{for every}\quad i \in [0,t]. 
        \end{align*}
    \end{enumerate}
\end{claim}

\begin{proof}[Proof of Claim~\ref{CLAIM:property-rotation-case-one}]
We prove Claim~\ref{CLAIM:property-rotation-case-one}~\ref{CLAIM:property-rotation-case-one-1}--\ref{CLAIM:property-rotation-case-one-4} by induction on $i$.
The assertions hold for $i=0$ by the definitions. Suppose that they hold for $i$.

For each operation in Figures~\ref{Fig:Rotation-1}--\ref{Fig:Rotation-6}, the members on $\mathcal{P}_{Z_i}$ are replaced by the same number of copies of $K_4$ in $\hat{\mathcal{P}}_{Z_i}$.

For Operations~5 and~6, the numbers of $K_4$ copies before and after the switching are $12=2+10$ and $18=3+15$, respectively.
This proves Claim~\ref{CLAIM:property-rotation-case-one}~\ref{CLAIM:property-rotation-case-one-1}.

Moreover, all vertices of the members on $\mathcal{P}_{Z_i}$ remain covered.
Indeed, the new copies of $K_4$ and all ends except the one removed cover these vertices except the marked vertices of the removed end, which are still covered by members of $\mathcal{S}_i$.
This proves Claim~\ref{CLAIM:property-rotation-case-one}~\ref{CLAIM:property-rotation-case-one-2}.

By the choice of the switching path and Figures~\ref{Fig:Rotation-1}--\ref{Fig:Rotation-6}, the new copies of $K_4$ are pairwise vertex-disjoint, and any other member of $\mathcal{K}_{i+1}$ intersecting one of these copies belongs to $\mathcal{Z}_{i+1}$.
Thus, $\mathcal{K}_{i+1}\setminus\mathcal{Z}_{i+1}$ is a $K_4$-tiling, proving Claim~\ref{CLAIM:property-rotation-case-one}~\ref{CLAIM:property-rotation-case-one-3}.

Finally, each operation removes one end from
$\mathcal{Z}_i$ and decreases $\operatorname{length}(Z_i)$
by one. This proves Claim~\ref{CLAIM:property-rotation-case-one}~\ref{CLAIM:property-rotation-case-one-4}.
\end{proof}


By Claim~\ref{CLAIM:property-rotation-case-one}~\ref{CLAIM:property-rotation-case-one-5}, we can choose $\mu$ to be sufficiently large at the beginning to ensure that in Step~\ref{process-3}, we can always apply Lemma~\ref{LEMMA:operation-1-2-exists},~\ref{LEMMA:operation-3-exists},~\ref{LEMMA:operation-4-exists},~\ref{LEMMA:operation-5-exists}, and~\ref{LEMMA:operation-6-exists}. 

By Claim~\ref{CLAIM:property-rotation-case-one}~\ref{CLAIM:property-rotation-case-one-1}--~\ref{CLAIM:property-rotation-case-one-4}, $\mathcal{A}_1 \cup \cdots \cup \mathcal{A}_5 \cup (\mathcal{K}_{t}\setminus \mathcal{Z}_t)$ is a $K_4$-tiling in $G$ of size 
\begin{align*}
    \sum_{i\in [5]}|\mathcal{A}_i| 
    + |\mathcal{A}_6| + |\mathcal{S}| -|\mathcal{Z}_t|
    = |\mathcal{A}| + 4-3
    = |\mathcal{A}| + 1,
\end{align*}
contradicting the maximality of $(|\mathcal{A}|, |\mathcal{B}|, |\mathcal{C}|, |\mathcal{D}|)$ (in particular, the maximality of $|\mathcal{A}|$). Therefore, Case 1 cannot happen. 


For Cases~2--6, we use the same process, with the
operations in Figure~\ref{Fig:Operation-Case2} and Figure~\ref{Fig:Rotations-Case3-6}.
In these cases, $|Z|=8,4,4,3,2$, respectively,
and the process terminates at types $44,4,4,3,2$,
respectively.
Since $|\mathcal{Z}_0|\le8$ and
$|V(\mathcal{S})\cap A|\le6$, we have
$|B_0|\le4\cdot8+6=38$.
As each operation decreases the length by one, $t\le6$.
Thus,
\[
|B_i|\le38+72\cdot6<700,
\]
so the corresponding lemmas from Section~6.4
can be applied as in Case~1.

The types $41111$, $4211$, and $431$, which are omitted
from Figure~\ref{Fig:Operation-Case2}, are handled by
\[
41111\xrightarrow{ope.1}4211
\xrightarrow{ope.1}422
\xrightarrow{ope.3}44,
\qquad
431\xrightarrow{ope.4}44.
\]
In these operations, every member $Q\in\mathcal{Z}_i$ with $|Q\cap Z_i|=4$ remains unchanged.
By the same argument as in Claim~\ref{CLAIM:property-rotation-case-one}~\ref{CLAIM:property-rotation-case-one-1}--~\ref{CLAIM:property-rotation-case-one-4}, $\mathcal{A}_1\cup\cdots\cup\mathcal{A}_5
\cup(\mathcal{K}_t\setminus\mathcal{Z}_t)$
is a $K_4$-tiling in $G$ of size
\[
|\mathcal{A}|+|\mathcal{S}|-|\mathcal{Z}_t|
=|\mathcal{A}|+1,
\]
since $(|\mathcal{S}|,|\mathcal{Z}_t|)=(3,2)$
in Case~2 and $(2,1)$ in Cases~3--6.
This contradicts the maximality of $|\mathcal{A}|$.
\begin{figure}[H]

\centering

\tikzset{every picture/.style={line width=0.75pt}} 


\caption{Operations for the displayed types of $Z_i$ in Case~2.}  
\label{Fig:Operation-Case2}
\end{figure}

\medskip 

\begin{figure}[H]
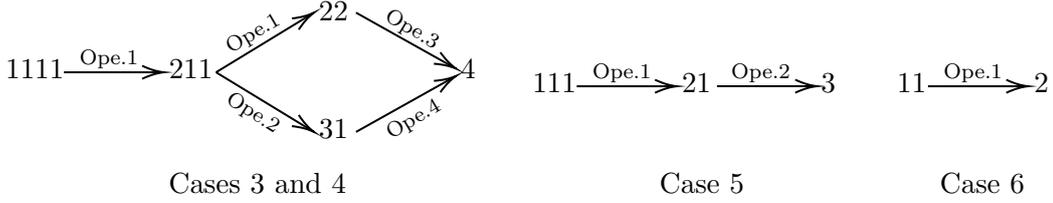

\centering
\tikzset{every picture/.style={line width=0.75pt}} 

%

\caption{Operations for all possible types of $Z_i$ in Cases 3 to 6.} 
\label{Fig:Rotations-Case3-6}
\end{figure}

In conclusion, we have shown that none of the six cases listed at the start of the proof can occur, thereby completing the proof of Lemma~\ref{LEMMA:A6BCD-upper-bound-a-reduced}. 
\end{proof}

\section{Global estimation}\label{SEC:Global-estimate}
To compare our upper bounds with the piecewise function $\Xi(n,k)$, we further subdivide the parameter range at the breakpoints appearing in Theorem~\ref{THM:Mian-HS-density-K4}.
The resulting quadratic inequalities can be verified by direct calculation or with the aid of Mathematica.
The nontrivial calculations can be checked using the Mathematica notebook available at \url{https://github.com/xliu2022/xliu2022.github.io/blob/main/Mathematica_Calculations/Section7-of-DHSK4.nb} and also as an ancillary file in the arXiv version of this paper.
In Sections~\ref{SUBSEC:global-definitions},~\ref{SUBSEC:inequality-Phi1},~\ref{SUBSEC:inequality-Phi2}, and~\ref{SUBSEC:inequality-Phi3}, we define and solve several quadratic optimization problems.
In Section~\ref{SUBSEC:proof-main}, we present the proof of Theorem~\ref{THM:Mian-HS-density-K4}. 

For the moment, let us set aside the combinatorial meaning of $(a_1, a_2, a_3, a_4, a_5, a_6, b, c, d)$ and $(n,k)$ in Sections~\ref{SUBSEC:global-definitions},~\ref{SUBSEC:inequality-Phi1},~\ref{SUBSEC:inequality-Phi2}, and~\ref{SUBSEC:inequality-Phi3}, and consider them simply as undetermined variables (or real numbers). 
\subsection{Definitions and preliminaries}\label{SUBSEC:global-definitions}
Given real numbers $n, k$ satisfying $n \ge 4 k \ge 0$, define the region
\begin{align*}
    \Omega_{n,k}
    \coloneqq 
    & \Bigl\{(a_1, \ldots, a_6, b,c,d)\in \mathbb{R}_{\ge 0}^9 \colon  \\[0.5em]
    & \qquad a_1+a_2+a_3+a_4+a_5+a_6 = k \quad\text{and}\quad 4k+ 3b+2c+d = n \Bigr\}.
\end{align*}

Let $(a_1, \ldots, a_6, b,c,d) \in \mathbb{R}_{\ge 0}^{9}$.
Recall that $\Psi(b,c,d) = 3 b^2  +4 b c + 2 b d + c^2 + c d$ was defined in~\eqref{equ:def-Psi-bcd}.
Define 
\begin{align*}
    \Phi(a_1, \ldots, a_6, b,c,d)
    & \coloneqq \frac{13 a_1^2}{2} + 13 a_1 (a_2 + \cdots + a_6) + 7 a_2^2 + 14 a_2(a_3 + \cdots + a_6) \\[0.5em]
    & \qquad   + 7 a_3^2 + 15a_3 (a_4+ a_5+a_6)  + \frac{15 a_4^2}{2} + 15a_4 (a_5 + a_6)  \\[0.5em]
    & \qquad + \frac{15 a_5^2}{2} + 15 a_5 a_6 + 8 a_6^2 \\[0.5em]
    & \qquad  + (9b + 6c + 3d)a_1 + (8b + 6c+ 3d)a_2 + (8b + 5c + 3d)a_3  \\[0.5em]
    & \qquad  + (7b+5c+3d)a_4 + (7b+5c+2d)a_5 + (7b+5c+2d)a_6, \\[0.5em]
    \Phi_{1}(a_1, \ldots, a_6, b,c,d)
    & \coloneqq \Phi(a_1, \ldots, a_6, b,c,d) + \Psi(b,c,d), \\[0.5em]
    \Phi_{2,1}(a_1, \ldots, a_6, b,c,d)
    & \coloneqq \Phi_{1}(a_1, \ldots, a_6, b,c,d) - \frac{a_6^2}{2}, \\[0.5em]
    \Phi_{2,2}(a_1, \ldots, a_6, b,c,d)
    & \coloneqq \Phi_{1}(a_1, \ldots, a_6, b,c,d) - \frac{(b+c)a_6}{2}, \\[0.5em]
    \Phi_{2,3}(a_1, \ldots, a_6, b,c,d)
    & \coloneqq \Phi(a_1, \ldots, a_6, b,c,d) 
            - (7 b+5 c+2 d)a_6 + (3b+2c+d)^2. 
\end{align*}
Define 
\begin{align*}
        \hat{\varphi}(a_4, b,c,d)
        & \coloneqq 
        \begin{cases}
            g_{1/2}(a_4, b) + h_{1/2}(a_4, c), &\quad\text{if}\quad d < \frac{b+c}{2}, \\[0.5em]
            \hat{g}(a_4, b) + h_{1/4}(a_4, c), &\quad\text{if}\quad d \ge \frac{b+c}{2},
        \end{cases} \quad\text{and} \\[0.5em]
        \varphi(a_4, b, c, d)
        & \coloneqq \hat{\varphi}(a_4, b,c,d) - \left(\frac{15 a_4^2}{2} + (7b+5c)a_4\right). 
\end{align*}
Define, for each $i \in [3]$, 
\begin{align*}
    \Phi_{3,i}(a_1,\ldots,a_6,b,c,d) 
    & \coloneqq \Phi_{2,i}(a_1,\ldots,a_6,b,c,d) + \varphi(a_4, b, c, d).\\[0.5em]
    %
\end{align*}  
Define 
\begin{align*}
    \Phi_{6}(a_1,\ldots,a_6,b,c,d)
    & \coloneqq \Phi_{1}(a_1,\ldots,a_6,b,c,d) + \varphi(a_4, b, c, d)\\[0.5em]
    \Phi_{6,1}(a_1,\ldots,a_6,b,c,d)
    & \coloneqq \Phi_{1}(a_1,\ldots,a_6,b,c,d) - \frac{(b+c)a_4}{2} + \frac{b^2 + c^2}{4}, \\[0.5em]
    \Phi_{6,2}(a_1,\ldots,a_6,b,c,d)
    & \coloneqq \Phi_{1}(a_1,\ldots,a_6,b,c,d) -\frac{9 b a_4}{10} - \frac{c a_4}{2} + \frac{151 b^2}{150} + \frac{c^2}{2}.
\end{align*}
%
%

The following facts follow directly from the definition (recall the definition of $\Xi(n,k)$ from Theorem~\ref{THM:Mian-HS-density-K4}). 
\begin{fact}\label{FACT:Extremal-function}
    For every $(n,k) \in \mathbb{R}_{\ge 0}^{2}$ satisfying $n \ge 4k$, we have 
    \begin{align*}
        \Xi(n,k)
        \ge \max\left\{\frac{n^2}{3} + \frac{k n}{3} - \frac{k^2}{6},~\frac{n^2}{4} + k n - k^2,~3kn - \frac{9k^2}{2}\right\}. 
    \end{align*}
\end{fact}

\begin{fact}\label{FACT:BCD-upper-bound}
    For every $(b,c,d) \in \mathbb{R}_{\ge 0}^{3}$, we have 
    \begin{align*}
        \Psi(b,c,d)
        = \frac{(3b+2c+d)^2}{3} - \frac{c^2+c d+d^2}{3}
        \le \frac{(3b+2c+d)^2}{3}.
    \end{align*} 
\end{fact} 

\begin{fact}\label{FACT:Phi31-Phi32-vs-Phi6}
    For every $(a_1,\ldots,a_6,b,c,d) \in \mathbb{R}_{\ge 0}^{9}$, we have 
    \begin{align*}
        \max\left\{\Phi_{2,1}(a_1,\ldots,a_6,b,c,d),~\Phi_{2,2}(a_1,\ldots,a_6,b,c,d)\right\}
        \le \Phi_{1}(a_1,\ldots,a_6,b,c,d). 
    \end{align*}
    Consequently, 
    \begin{align*}
        \max\left\{\Phi_{3,1}(a_1,\ldots,a_6,b,c,d),~\Phi_{3,2}(a_1,\ldots,a_6,b,c,d)\right\}
        \le \Phi_{6}(a_1,\ldots,a_6,b,c,d). 
    \end{align*}
\end{fact}

Recall the functions $g_\alpha$, $\hat g$, and $h_\alpha$ defined at the beginning of Section~\ref{SEC:A4BC}.

\begin{proposition}\label{PROP:g-alpha-h-alpha-upper-bounds}
    For every $\alpha \in (0,1)$ and for every $(a_4, b, c) \in \mathbb{R}_{\ge 0}^{3}$, we have 
    \begin{align}
        \hat{g}(a_4, b)
        & \le \min\left\{ \frac{45 a_4^2}{8} + \frac{61 b a_4}{10} + \frac{151 b^2}{150},~\frac{45 a_4^2}{8} + 7 b a_4  \right\}, \label{equ:hat-g-upper-bounds} \\[0.5em]
        g_{\alpha}(a_4, b)
        & \le \min\left\{ \frac{15 \alpha a_4^2}{2} + \frac{13 b a_4}{2} + \frac{b^2}{8\alpha},~\frac{15 \alpha a_4^2}{2} + 7 b a_4\right\},  \label{equ:g-upper-bounds} \\[0.5em]
        \quad\text{and}\quad 
        h_{\alpha}(a_4, c)
        & \le \min\left\{ \frac{15 \alpha a_4^2}{2} + \frac{9 c a_4}{2} + \frac{c^2}{8\alpha},~\frac{15 \alpha a_4^2}{2} + 5 c a_4\right\}. \label{equ:h-upper-bounds} 
    \end{align}
    Consequently, for every $i \in [3]$,
    \begin{align}
        \Phi_{3,i}(a_1,\ldots,a_6,b,c,d)
        & \le \Phi_{2,i}(a_1,\ldots,a_6,b,c,d), \label{equ:Phi3i-vs-Phi4i-Phi5i-2} 
        %
    \end{align}
    and 
    \begin{align}\label{equ:Phi6-vs-Phi61-Phi62}
        \Phi_{6}(a_1,\ldots,a_6,b,c,d)
        \le 
        \begin{cases}
            \Phi_{6,1}(a_1,\ldots,a_6,b,c,d), &\quad\text{if}\quad d \le \frac{b+c}{2}, \\[0.5em]
            \Phi_{6,2}(a_1,\ldots,a_6,b,c,d), &\quad\text{if}\quad d \ge \frac{b+c}{2}. 
        \end{cases} 
    \end{align}
\end{proposition}
\begin{proof}[Proof of Proposition~\ref{PROP:g-alpha-h-alpha-upper-bounds}]
    Inequalities \eqref{equ:g-upper-bounds} and~\eqref{equ:h-upper-bounds} follow by completing the square in their definitions. 
    For \eqref{equ:hat-g-upper-bounds}, the first, second and fifth expressions in the definition of $\hat{g}$ are immediate. 
    In the third and fourth cases, subtracting $\hat{g}(a_4,b)$ from this bound gives, respectively,
\[
\frac1{40}\left(a_4-\frac{86b}{15}\right)^2+\frac{7b^2}{75},
\qquad
\frac3{232}\left(a_4-\frac{38b}{5}\right)^2,
\]
both nonnegative.
The second bound follows directly for $a_4\le6b/5$,
and from the first otherwise, since $151b^2/150\le9ba_4/10$.

    It follows from~\eqref{equ:g-upper-bounds} and~\eqref{equ:h-upper-bounds} that 
    \begin{align}\label{equ:g-and-h-1}
        g_{1/2}(a_4, b) + h_{1/2}(a_4, c)
        & \le \frac{15 a_4^2}{4} + \frac{13 b a_4}{2} + \frac{b^2}{4} + \frac{15 a_4^2}{4} + \frac{9 c a_4}{2} + \frac{c^2}{4} \notag \\[0.5em]
        & = \frac{15 a_4^2}{2} + (7b + 5c) a_4 - \frac{(b+c)a_4}{2} + \frac{b^2 + c^2}{4},  
    \end{align}
    and 
    \begin{align}\label{equ:g-and-h-2}
        g_{1/2}(a_4, b) + h_{1/2}(a_4, c)
         \le \frac{15 a_4^2}{4} + 7 b a_4 + \frac{15 a_4^2}{4} + 5 c a_4 
        = \frac{15 a_4^2}{2} + (7b + 5c) a_4.  
    \end{align}
    It follows from~\eqref{equ:hat-g-upper-bounds} and~\eqref{equ:g-upper-bounds} that 
    \begin{align}\label{equ:hat-g-h-1}
        \hat{g}(a_4, b) + h_{1/4}(a_4, c)
        & \le \frac{45 a_4^2}{8} + \frac{61 b a_4}{10} + \frac{151 b^2}{150} + \frac{15 a_4^2}{8} + \frac{9 c a_4}{2} + \frac{c^2}{2} \notag \\[0.5em]
        & = \frac{15 a_4^2}{2} + (7b + 5c) a_4 - \frac{9 b a_4}{10} - \frac{c a_4}{2} + \frac{151 b^2}{150} + \frac{c^2}{2}, 
    \end{align}
    and 
    \begin{align}\label{equ:hat-g-h-2}
        \hat{g}(a_4, b) + h_{1/4}(a_4, c)
        & \le \frac{45 a_4^2}{8} + 7 b a_4 +  \frac{15 a_4^2}{8} + 5 c a_4 
        = \frac{15 a_4^2}{2} + (7b + 5c) a_4.
    \end{align}
    Now, Inequality~\eqref{equ:Phi6-vs-Phi61-Phi62} follow from~\eqref{equ:g-and-h-1},~\eqref{equ:hat-g-h-1}, and the definition.
    Inequality~\eqref{equ:Phi3i-vs-Phi4i-Phi5i-2} follows from~\eqref{equ:g-and-h-2},~\eqref{equ:hat-g-h-2}, and the definition. 
\end{proof}
\textbf{Remark.} 
Notice from Inequalities~\eqref{equ:g-and-h-2} and~\eqref{equ:hat-g-h-2} that 
\begin{align}\label{equ:varphi-less-zero}
    \varphi(a_4, b,c,d) 
    \le 0. 
\end{align}

Let $I \subseteq \mathbb{R}$ be an interval. 
Recall that a function $f(x) \colon I \to \mathbb{R}$ is \textbf{convex} if for every $\alpha \in [0,1]$ and for all $x, y \in I$, 
\begin{align*}
    f\left(\alpha x + (1-\alpha) y\right)
    \le \alpha \cdot f(x) + (1-\alpha) \cdot f(y).
\end{align*}
Note that, in particular, a quadratic polynomial $q(x)$ with a non-negative coefficient for $x^2$ is convex. 

\begin{proposition}\label{PROP:a4-convex}
    Let $(b,c) \in \mathbb{R}_{\ge 0}^{2}$. 
    Suppose that $\gamma_1 \ge - 7, \gamma_2 \in \mathbb{R}, \gamma_3 \in \mathbb{R}$ are real numbers. 
    Then both 
    \begin{align*}
        \rho_1(x)
        & \coloneqq 
        \gamma_1 x^2  + \gamma_2 x + \gamma_3 
        + g_{1/2}(x,b) + h_{1/2}(x,c) \quad\text{and} \\[0.5em]
        \rho_2(x)
        & \coloneqq 
        \gamma_1 x^2  + \gamma_2 x + \gamma_3 
        + \hat{g}(x,b) + h_{1/4}(x,c) 
    \end{align*}
    are convex functions on $\mathbb{R}_{\ge 0}$. 
\end{proposition}
\begin{proof}[Proof of Proposition~\ref{PROP:a4-convex}]
    First, let us consider $\rho_1(x)$.   
    We write $\rho_1(x)$ as the summation of three functions $f_1(x)$, $f_2(x)$, $f_{3}(x)$, where 
    \begin{align*}
        f_1(x)
        & \coloneqq \gamma_1 x^2  + \gamma_2 x + \gamma_3 + 7x^2,\\[0.5em]
        f_2(x)
        & \coloneqq g_{1/2}(x,b) - \frac{7 x^2}{2}, \\[0.5em]
        f_3(x)
        & \coloneqq h_{1/2}(x,c) - \frac{7 x^2}{2}. 
    \end{align*}
    Note that $f_1(x)$ is a quadratic polynomial in $x$ and the coefficient of $x^2$ is $\gamma_1 + 7  \ge 0$. So $f_{1}(x)$ is convex. 
    Observe that both $f_{2}(x)$ and $f_{3}(x)$ are continuous and straightforward calculations show that 
    \begin{align*}
        \frac{\mathrm{d} f_{2}(x)}{\mathrm{d} x}
        = 
        \begin{cases}
            7 b, &\quad\text{if}\quad x < b, \\[0.5em]
            \frac{x+13b}{2}, &\quad\text{if}\quad x \ge b,
        \end{cases}
        \quad\text{and}\quad 
        \frac{\mathrm{d} f_{3}(x)}{\mathrm{d} x}
        = 
        \begin{cases}
            5 c, &\quad\text{if}\quad x < c, \\[0.5em]
            \frac{x+9c}{2}, &\quad\text{if}\quad x \ge c.
        \end{cases}
    \end{align*}
    Observe that both $\frac{\mathrm{d} f_{2}(x)}{\mathrm{d} x}$ and $\frac{\mathrm{d} f_{3}(x)}{\mathrm{d} x}$ are non-decreasing. So both $f_{2}(x)$ and $f_{3}(x)$ are convex. 
    Since the summation of convex functions is also convex, $\rho_1(x)$ is convex. 

    Next, we consider $\rho_2(x)$. 
    We write $\rho_2(x)$ as the summation of three functions $f_1(x)$, $f_2(x)$, $f_{3}(x)$, where 
    \begin{align*}
        f_1(x)
        & \coloneqq \gamma_1 x^2  + \gamma_2 x + \gamma_3  + 7 x^2,\\[0.5em]
        f_2(x)
        & \coloneqq \hat{g}(x,b) - \frac{21 x^2}{4}, \\[0.5em]
        f_3(x)
        & \coloneqq h_{1/4}(x,c) - \frac{7 x^2}{4}. 
    \end{align*}
    The proof of Case 1 shows that $f_1(x)$ is convex. 
    Observe that both $f_{2}(x)$ and $f_{3}(x)$ are continuous and straightforward calculations show that 
    \begin{align*}
        \frac{\mathrm{d} f_{2}(x)}{\mathrm{d} x}
        & = 
        \begin{cases}
            7 b, &\quad\text{if}\quad x < \frac{6 b}{5}, \\[0.5em]
            \frac{3 x}{4} + \frac{61 b}{10}, &\quad\text{if}\quad x \in \left[\frac{6b}{5}, \frac{(86 - 4\sqrt{210})b}{15}\right], \\[0.5em]
            \frac{7 x}{10} + \frac{479 b}{75}, &\quad\text{if}\quad x \in \left[\frac{(86 - 4\sqrt{210})b}{15}, \frac{56 b}{15}\right], \\[0.5em]
            \frac{21x}{29} + \frac{913b}{145}, &\quad\text{if}\quad x \in \left[\frac{56b}{15}, \frac{38b}{5}\right], \\[0.5em]
            \frac{3 x}{4} + \frac{61 b}{10}, &\quad\text{if}\quad x \ge \frac{38b}{5}, 
        \end{cases}
    \end{align*}
    and 
    \begin{align*}
          \frac{\mathrm{d} f_{3}(x)}{\mathrm{d} x}
        & = 
        \begin{cases}
            5 c, &\quad\text{if}\quad x < 2c, \\[0.5em]
            \frac{x+18c}{4}, &\quad\text{if}\quad x \ge 2c.
        \end{cases}
    \end{align*}
    Observe that both $\frac{\mathrm{d} f_{2}(x)}{\mathrm{d} x}$ and $\frac{\mathrm{d} f_{3}(x)}{\mathrm{d} x}$ are non-decreasing. So both $f_{2}(x)$ and $f_{3}(x)$ are convex. 
    It follows that $\rho_2(x)$ is convex as well. 
    This completes the proof of Proposition~\ref{PROP:a4-convex}.
\end{proof}
%

The following simple but useful fact will be used repeatedly throughout this section.
\begin{fact}\label{FACT:convex-optimization}
    Let $[x_1, x_2] \subseteq \mathbb{R}$ be an interval. Suppose that $q(x) \colon [x_1, x_2] \to \mathbb{R}$ is convex. Then, for every $x \in [x_1, x_2]$, 
    \begin{align*}
        q(x) 
        \le \max\left\{q(x_1),~q(x_2)\right\}.
    \end{align*}
\end{fact}

In the remainder of this section, we assume that $n \ge 4k \ge 0$.
\subsection{Inequalities for $\Phi$ and $\Phi_1$}\label{SUBSEC:inequality-Phi1}
In this subsection, we consider the optimization problem of maximizing $\Phi_1$ over the region $\Omega_{n,k}$. 
%
\begin{proposition}\label{PROP:Phi-a5-a6-0-E1-a}
    For every $(a_1, \ldots, a_6, b,c,d) \in \Omega_{n,k}$, we have 
    \begin{align*}
        \Phi(a_1, \ldots, a_6, b,c,d)
        & \le \max\Big\{ \Phi(a_1+a_2+a_3+ a_4,0,0,0, a_5,a_6,b,c,d),  \\[0.5em]
        &  \qquad\qquad \Phi(0,a_1+a_2+a_3+ a_4,0,0, a_5, a_6, b,c,d), \\[0.5em]
        &\qquad\qquad  \Phi\left(0,0,\hat{a}_3, \hat{a}_4,a_5, a_6, b,c,d \right)  \Big\},  
    \end{align*}
    where 
    \begin{align*}
        (\hat{a}_3, \hat{a}_4)
        \coloneqq \left(\left(1+\frac{a_1+a_2}{a_3+a_4}\right)a_3, \left(1+\frac{a_1+a_2}{a_3+a_4}\right)a_4\right).
    \end{align*}
\end{proposition}
\textbf{Remark.} Proposition~\ref{PROP:Phi-a5-a6-0-E1-a} shows that when maximizing $\Phi$ (and hence, also $\Phi_1$) over the region $\Omega_{n,k}$, it suffices to consider the cases where $a_2 = a_3 = a_4 = 0$, $a_1 = a_3 = a_4 = 0$, and $a_1 = a_2 = 0$. 
\begin{proof}[Proof of Proposition~\ref{PROP:Phi-a5-a6-0-E1-a}]
    We may assume that $\Phi(a_1, \ldots, a_6, b,c,d)$ is maximized over $\Omega_{n,k}$. 
    \begin{claim}\label{CLAIM:Phi-a5-a6-0-E1-a1-a2a3a4}
        We may assume that $a_1 = 0$ or $a_2 + a_3 + a_4 = 0$. 
    \end{claim}
    \begin{proof}[Proof of Claim~\ref{CLAIM:Phi-a5-a6-0-E1-a1-a2a3a4}]
        Suppose that $a_1 > 0$ and $a_2 + a_3 + a_4 > 0$. Let 
        \begin{align*}
            q_{1}(x)
            \coloneqq \Phi(a_1+(a_2+a_3+a_4)x, (1-x)a_2, (1-x)a_3, (1-x)a_4, a_5, a_6, b,c,d).
        \end{align*}
        Note that $q_{1}(x)$ as a quadratic polynomial in $x$ over the interval $\left[\frac{-a_1}{a_2+a_3+a4}, 1\right]$, and straightforward calculations show that the coefficient of $x^2$ in $q_{1}(x)$ is 
        \begin{align*}
            \frac{a_2^2+a_3^2+2 a_4^2+2 a_2 \left(a_3+a_4\right) +4 a_3 a_4}{2} 
            \ge 0.
        \end{align*}
        So it follows from Fact~\ref{FACT:convex-optimization} that 
        \begin{align*}
            \Phi(a_1, \ldots, a_6, b,c,d) 
            = q_1(0)
            & \le \max\left\{q_1\left(\frac{-a_1}{a_2+a_3+a4}\right),~q_{1}(1)\right\}. 
        \end{align*}
        This implies that we can assume either $a_1 = 0$ or $a_2 + a_3 + a_4 = 0$. 
    \end{proof}
    %
    %
    \begin{claim}\label{CLAIM:Phi-a5-a6-0-E1-a2-a3a4}
        Suppose that $a_1 = 0$. Then we may assume that $a_2 = 0$ or $a_3 + a_4 = 0$. 
    \end{claim}
    \begin{proof}[Proof of Claim~\ref{CLAIM:Phi-a5-a6-0-E1-a2-a3a4}]
        Suppose that $a_2 > 0$ and $a_3 + a_4 > 0$. Let 
        \begin{align*}
            q_{2}(x)
            \coloneqq \Phi(0, a_2 + (a_3+a_4)x, (1-x)a_3, (1-x)a_4, a_5, a_6, b,c,d).
        \end{align*}
        Note that $q_{2}(x)$ as a quadratic polynomial in $x$ over the interval $\left[\frac{-a_2}{a_3+a4}, 1\right]$, and straightforward calculations show that the coefficient of $x^2$ in $q_{2}(x)$ is 
        \begin{align*}
            \frac{a_4^2 + 2 a_3 a_4}{2} 
            \ge 0.
        \end{align*}
        So it follows from Fact~\ref{FACT:convex-optimization} that 
        \begin{align*}
            \Phi(0, a_2, a_3, a_4, a_5, a_6, b,c,d) 
            = q_2(0)
            & \le \max\left\{q_{2}(1),~q_2\left(\frac{-a_2}{a_3+a4}\right)\right\}, 
        \end{align*}
        This implies that we can assume either $a_2 = 0$ or $a_3 + a_4 = 0$. 
    \end{proof}
    Proposition~\ref{PROP:Phi-a5-a6-0-E1-a} now follows directly from Claims~\ref{CLAIM:Phi-a5-a6-0-E1-a1-a2a3a4} and~\ref{CLAIM:Phi-a5-a6-0-E1-a2-a3a4}.
\end{proof}

\begin{proposition}\label{PROP:Phi-5-to-6-to-1}
    The following statements hold for $(a_1, \ldots, a_6, b,c,d) \in \Omega_{n,k}$. 
    \begin{enumerate}[label=(\roman*)]
        \item\label{PROP:Phi-5-to-6-to-1-1}
            $\Phi(a_1, \ldots, a_6, b,c,d) \le \Phi(a_1, a_2, a_3, a_4, 0, a_5+a_6, b,c,d)$. 
        \item\label{PROP:Phi-5-to-6-to-1-2}
            If $k \in \left[0, \frac{n}{8}\right]$, then $\Phi(a_1, \ldots, a_6, b,c,d) \le \Phi(a_1+a_6, a_2, a_3, a_4, a_5, 0, b,c,d)$.
    \end{enumerate}
    In particular, if $k \in \left[0, \frac{n}{8}\right]$, then 
    \begin{align}\label{equ:Phi-5-to-6-to-1}
        \Phi(a_1, \ldots, a_6, b,c,d)
        \le \Phi(a_1+a_5+a_6, a_2, a_3, a_4, 0, 0, b,c,d).
    \end{align}
\end{proposition}
\begin{proof}[Proof of Proposition~\ref{PROP:Phi-5-to-6-to-1}]
    Proposition~\ref{PROP:Phi-5-to-6-to-1}~\ref{PROP:Phi-5-to-6-to-1-1} follows directly from the fact that 
    \begin{align*}
        \Phi(a_1, a_2, a_3, a_4, 0, a_5+a_6, b,c,d) - \Phi(a_1, \ldots, a_6, b,c,d)
        = \frac{a_5^2}{2} + a_5 a_6 
        \ge 0.
    \end{align*}
    Proposition~\ref{PROP:Phi-5-to-6-to-1}~\ref{PROP:Phi-5-to-6-to-1-2} follows from the fact that 
    \begin{align*}
        & \Phi(a_1+a_6, a_2, a_3, a_4, a_5, 0, b,c,d) - \Phi(a_1, \ldots, a_6, b,c,d)  \\[0.5em]
        & =\frac{a_6}{2} \left(4 b+2 c+2 d -2 a_2-4 a_3-4 a_4-4 a_5-3 a_6\right) \\[0.5em]
        & \ge \frac{a_6}{2} \left(3 b+2 c+ d -4 (a_2+ a_3+ a_4+ a_5+a_6) \right) 
         \ge \frac{a_6}{2} \left(n-4k - 4k \right)
        \ge 0.
    \end{align*}
    This completes the proof of Proposition~\ref{PROP:Phi-5-to-6-to-1}. 
\end{proof}

\begin{proposition}\label{PROP:Phi1-inequality-A}
    The following inequalities hold for $(a_1, \ldots, a_6, b,c,d) \in \Omega_{n,k}$. 
    \begin{enumerate}[label=(\roman*)]
        \item\label{PROP:Phi1-inequality-A-1} 
        Suppose that $a_2 = \cdots = a_6 = 0$. Then $\Phi_{1}(a_1, \ldots, a_6, b,c,d) \le \frac{n^2}{3}+\frac{k n}{3}-\frac{k^2}{6}$. 
        \item\label{PROP:Phi1-inequality-A-2} 
        Suppose that $a_1 = a_3 = \cdots = a_6 = 0$. Then 
        \begin{align*}
            \Phi_{1}(a_1, \ldots, a_6, b,c,d)
            & \le 
            \begin{cases}
                \frac{n^2}{3}+2k^2, &\quad\text{if}\quad k \in\left[0,  \frac{n}{6}\right], \\[0.5em]
                \frac{n^2}{4}+k n-k^2, &\quad\text{if}\quad k \in\left[ \frac{n}{6}, \frac{n}{4}\right].
            \end{cases}
        \end{align*}
    \end{enumerate}
    In particular, 
    \begin{align*}
        \max\left\{\Phi_{1}(k,0,\ldots,0, b,c,d),~\Phi_{1}(0,k,0,\ldots,0, b,c,d)\right\}
        \le \Xi(n,k). 
    \end{align*}
\end{proposition}
\begin{proof}[Proof of Proposition~\ref{PROP:Phi1-inequality-A}]
    Recall from the definition of $\Omega_{n,k}$ that $d = n - 4k - 3b - 2c$. 
    
    Suppose that $a_2 = \cdots = a_6 = 0$. Then $a_1 = k$, and 
    \begin{align*}
        \Phi_{1}(a_1, 0,0,0, 0, 0, b,c,d)
        & = \frac{13 a_1^2}{2}+ 3  (3 b+2 c+d) a_1+ \Psi(b,c,d) \\[0.5em]
        & \le \frac{13 a_1^2}{2}+ 3  (3 b+2 c+d) a_1 + \frac{(3 b+2 c+d)^2}{3} \\[0.5em]
        & = \frac{13 a_1^2}{2}+ 3  (n - 4a_1) a_1 + \frac{(n-4a_1)^2}{3} \\
        &\overset{a_1=k}{=} \frac{n^2}{3}+\frac{k n}{3}-\frac{k^2}{6}.
    \end{align*}
    Suppose that $a_1 = a_3 = \cdots = a_6 = 0$. Then $a_2 = k$, and 
    \begin{align}\label{equ:Phi1-0a20000bcd}
        \Phi_{1}(0, a_2,0,0, 0, 0, b,c,d) 
        & = 7 a_2^2 +  (8 b+6 c+3 d)a_2+ \Psi(b,c,d) \notag \\[0.5em]
        & = -5 k^2 + 3 k n -3 b^2 -c^2 - 3bc + (n - 4k)c +(2n-9k)b  \notag \\[0.5em]
        & = \frac{n^2}{3} +2k^2 - \frac{3}{4}\left(x - \frac{n-6k}{3}\right)^2 -  \left(y - \frac{n-4k}{2}\right)^2, 
    \end{align}
    where $x \coloneqq b$ and $y \coloneqq \frac{3b}{2}  + c$. 

    Suppose that $k \in \left[0, \frac{n}{6}\right]$. Then it follows trivially from~\eqref{equ:Phi1-0a20000bcd} that 
    \begin{align*}
        \Phi_{1}(0, a_2,0,0, 0, 0, b,c,d) 
        \le \frac{n^2}{3} +2k^2. 
    \end{align*}
    Suppose that $k \in \left[\frac{n}{6}, \frac{n}{4}\right]$. Then $\frac{n-6k}{3} \le 0$ and it follows from~\eqref{equ:Phi1-0a20000bcd} that 
    \begin{align*}
        \Phi_{1}(0, a_2,0,0, 0, 0, b,c,d) 
        & \le \frac{n^2}{3} +2k^2 - \frac{3}{4}\left(0 - \frac{n-6k}{3}\right)^2 -  \left(y - \frac{n-4k}{2}\right)^2 \\[0.5em]
        & = \frac{n^2}{4}+k n-k^2-\left(y - \frac{n-4k}{2}\right)^2
        \le \frac{n^2}{4}+k n-k^2. 
    \end{align*}
    The ``In particular'' part follows from Fact~\ref{FACT:Extremal-function} and the fact that 
    \begin{align*}
        \Xi(n,k)
        \ge \begin{cases}
                \frac{n^2}{3}+2k^2, &\quad\text{if}\quad k \in\left[0,  \frac{n}{6}\right], \\[0.5em]
                \frac{n^2}{4}+k n-k^2, &\quad\text{if}\quad k \in\left[ \frac{n}{6}, \frac{n}{4}\right].
            \end{cases}
    \end{align*}
    This completes the proof of Proposition~\ref{PROP:Phi1-inequality-A}.
\end{proof}

\begin{proposition}\label{PROP:Phi-1-upper-bound}
    Suppose that $k \in \left[0, \frac{(20+\sqrt{10}) n}{130}\right] \cup \left[\frac{n}{5},\frac{n}{4}\right]$. Then  
    \begin{align*}
        \max\left\{\Phi_{1}(a_1, a_2, a_3, a_4, 0, 0, b,c,d) \colon (a_1, a_2, a_3, a_4, 0, 0, b,c,d) \in \Omega_{n,k} \right\}
        \le \Xi(n,k).
    \end{align*}
    In particular, by~\eqref{equ:Phi-5-to-6-to-1}, for $k \in \left[0, \frac{n}{8}\right]$, 
    \begin{align}\label{equ:Phi1-k-less-125n}
        \max\left\{\Phi_{1}(a_1, \ldots , a_6, b,c,d) \colon (a_1, \ldots , a_6, b,c,d) \in \Omega_{n,k} \right\}
        \le \Xi(n,k).
    \end{align}
\end{proposition}
\begin{proof}[Proof of Proposition~\ref{PROP:Phi-1-upper-bound}]
    We may assume that $\Phi_{1}(a_1, a_2, a_3, a_4, 0, 0, b,c,d)$ is maximized over the region $\left\{(w_1, \ldots, w_6, x,y,z) \in \Omega_{n,k} \colon w_5 = w_6 = 0\right\}$. 
    %
    By Proposition~\ref{PROP:Phi-a5-a6-0-E1-a}, it suffices to consider the cases where $a_2 = a_3 = a_4 = 0$, $a_1 = a_3 = a_4 = 0$, and $a_1 = a_2 = 0$. 
    The first two cases follow from Proposition~\ref{PROP:Phi1-inequality-A}~\ref{PROP:Phi1-inequality-A-1} and~\ref{PROP:Phi1-inequality-A-2}, respectively. So it suffices to consider the case where  $a_1 = a_2 = 0$. 

\medskip 

    \textbf{Case 1}: $b \ge k$ (Since $b \le \frac{n - 4k}{3}$ holds trivially, this case is possible only when $k \in \left[0, \frac{n}{7}\right]$).
 
    Note that 
    \begin{align}\label{equ:Phi-a3-a4-b-large}
        & \Phi_{1}(0, 0,a_3,a_4, 0, 0, b,c,d) \notag \\[0.5em]
        & = \Phi_{1}(0, 0,a_3,k-a_{3}, 0, 0, b,c,n - 4k -3b-2c) \notag \\[0.5em]
        & =  \frac{n^2+2 k n + 7 k^2}{4} + \frac{-b^2+2 b n-10 b k}{4} -\frac{\left(a_3-b\right)^2}{2} -\left(c - \frac{n - 5k -3b}{2}\right)^2  \notag \\[0.5em]
        & \le \frac{n^2+2 k n + 7 k^2}{4} + \frac{-b^2+2 b n-10 b k}{4} -\frac{\left(k-b\right)^2}{2}  -\left(c - \frac{n - 5k -3b}{2}\right)^2  \notag \\[0.5em]
        & = \frac{n^2 + 2kn + 5k^2 -3 b^2-6 b k+2 b n}{4} -\left(c - \frac{n - 5k -3b}{2}\right)^2.  
    \end{align}

    \medskip 

    \textbf{Case 1.1}: $b \in \left[\max\left\{k,\frac{n-5k}{3}\right\}, \frac{n-4k}{3}\right]$. 

    Since $b \ge \frac{n-5k}{3}$, we have $\frac{n - 5k -3b}{2} \le 0$. 
    It follows from~\eqref{equ:Phi-a3-a4-b-large} that 
    \begin{align*}
        \Phi_{1}(0, 0,a_3,a_4, 0, 0, b,c,d) 
        & \le \frac{n^2 + 2kn + 5k^2 -3 b^2-6 b k+2 b n}{4} -\left(0 - \frac{n - 5k -3b}{2}\right)^2 \\[0.5em]
        & = \frac{n^2}{3} +\frac{7 k^2}{4} -3 \left(b- \frac{2n-9k}{6}\right)^2.
    \end{align*}
    Suppose that $k \in \left[0, \frac{2n}{15}\right]$. Then $\frac{2n-9k}{6} \in \left[\max\left\{k,\frac{n-5k}{3}\right\}, \frac{n-4k}{3}\right]$, and hence, 
    \begin{align*}
        \Phi_{1}(0, 0,a_3,a_4, 0, 0, b,c,d) 
        \le \frac{n^2}{3} +\frac{7 k^2}{4}
        \le \Xi(n,k), 
    \end{align*}
    as desired. 

    Suppose that $k \in \left[\frac{2n}{15}, \frac{n}{7}\right]$. Then $\frac{2n-9k}{6} \le k$, and hence, 
    \begin{align*}
        \Phi_{1}(0, 0,a_3,a_4, 0, 0, b,c,d) 
        \le \frac{n^2}{3} +\frac{7 k^2}{4} -3 \left(k- \frac{2n-9k}{6}\right)^2
        = 5kn- 17k^2
        \le \Xi(n,k), 
    \end{align*}
    as desired. 

    \medskip 

    \textbf{Case 1.2}: $b \in \left[k,~\max\left\{k,\frac{n-5k}{3}\right\}\right]$. (The case $\max\{k, \frac{n-5k}{3}\}=k$ gives $b=k$ and is covered by Case~1.1, so we may assume that $\max\{k, \frac{n-5k}{3}\}=\frac{n-5k}{3}$.)
    
    It follows from~\eqref{equ:Phi-a3-a4-b-large} that 
    \begin{align*}
        \Phi_{1}(0, 0,a_3,a_4, 0, 0, b,c,d) 
         & \le \frac{n^2+2 k n+5 k^2 -3 b^2-6 b k+2 b n}{4} \\[0.5em]
        & = \frac{n^2}{3} + 2k^2 - \frac{3}{4} \left(b - \frac{n-3k}{3}\right)^2 \\[0.5em]
        & \le \frac{n^2}{3} + 2k^2 - \frac{3}{4} \left(\frac{n - 5k}{3} - \frac{n-3k}{3}\right)^2 
        = \frac{n^2}{3} + \frac{5k^2}{3} 
        \le \Xi(n,k),
    \end{align*}
    as desired. 

    \medskip 

    \textbf{Case 2}: $b \le k$.
    
    Similar to the proof of~\eqref{equ:Phi-a3-a4-b-large}, we have  
    \begin{align}\label{equ:Phi-a3-a4-b-small}
        & \Phi_{1}(0, 0,a_3,a_4, 0, 0, b,c,d) \notag \\[0.5em]
       & =  \frac{n^2+2 k n + 7 k^2}{4} + \frac{-b^2+2 b n-10 b k}{4} -\frac{\left(a_3-b\right)^2}{2} -\left(c - \frac{n - 5k -3b}{2}\right)^2  \notag \\[0.5em]
        & \le \frac{n^2+2 k n + 7 k^2}{4} + \frac{-b^2+2 b n-10 b k}{4} -\left(c - \frac{n - 5k -3b}{2}\right)^2.
    \end{align}

    \medskip 

    \textbf{Case 2.1}: $b \in \left[\frac{n-5k}{3},~\min\left\{k,\frac{n-4k}{3}\right\}\right]$ (this case is possible only when $k \in\left[\frac{n}{8}, \frac{n}{4}\right]$, since we need $\frac{n-5k}{3} \le k$). 

    Since $b \ge \frac{n-5k}{3}$, we have $\frac{n - 5k -3b}{2} \le 0$. 
    It follows from~\eqref{equ:Phi-a3-a4-b-small} that  
    \begin{align*}
         \Phi_{1}(0, 0,a_3,a_4, 0, 0, b,c,d) 
        & \le \frac{n^2+2 k n + 7 k^2}{4} + \frac{-b^2+2 b n-10 b k}{4} -\left(0 - \frac{n - 5k -3b}{2}\right)^2 \\[0.5em]
        & = \frac{2 n^2}{5}-k n+\frac{11 k^2}{2} -\frac{5}{2} \left(b-\frac{2(n-5 k)}{5}\right)^2. 
    \end{align*}

    Suppose that $k \in \left[\frac{n}{8},~\frac{(20+\sqrt{10}) n}{130}\right]$. Then 
    \begin{align*}
        \Phi_{1}(0, 0,a_3,a_4, 0, 0, b,c,d)
        \le \frac{2 n^2}{5}-k n+\frac{11 k^2}{2} 
        \le \Xi(n,k),
    \end{align*}
    as desired. 
    
    Suppose that $k \in \left[\frac{n}{5}, \frac{n}{4}\right]$. 
    Then $\frac{2(n-5 k)}{5} \le 0$, and hence, 
    \begin{align*}
        \Phi_{1}(0, 0,a_3,a_4, 0, 0, b,c,d)
        & \le \frac{2 n^2}{5}-k n+\frac{11 k^2}{2}
        -\frac{5}{2} \left(0-\frac{2(n-5 k)}{5}\right)^2 
        = 3 k n-\frac{9 k^2}{2}
        \le \Xi(n,k),
    \end{align*}
    as desired.


    \medskip 

    \textbf{Case 2.2}: $b \in \left[0,~\min\left\{k,\frac{n-5k}{3}\right\}\right]$ (this case is possible only when $k \in \left[0, \frac{n}{5}\right]$, since we need $0 \le \frac{n-5k}{3}$).

    By assumption, $k$ lies in the interval $\left[0,~\frac{(20+\sqrt{10}) n}{130}\right]$. 
    It follows from~\eqref{equ:Phi-a3-a4-b-small} that
    \begin{align*}
        \Phi_{1}(0, 0,a_3,a_4, 0, 0, b,c,d) 
        & \le \frac{n^2+2 k n + 7 k^2}{4} + \frac{-b^2+2 b n-10 b k}{4} \\[0.5em]
        & = \frac{n^2}{2} -2kn +8k^2 -\frac{1}{4} \left(b - (n-5k)\right)^2. 
    \end{align*}
    Suppose that $k \in \left[0, \frac{n}{8}\right]$. Then $k \le \frac{n-5k}{3} \le n-5k$, and hence, 
    \begin{align*}
        \Phi_{1}(0, 0,a_3,a_4, 0, 0, b,c,d) 
        & \le \frac{n^2}{2} -2kn +8k^2 -\frac{1}{4} \left(k - (n-5k)\right)^2  \\[0.5em]
        & = \frac{n^2}{4} + k n - k^2
        \le \Xi(n,k),
    \end{align*}
    as desired. 

    Suppose that $k \in \left[\frac{n}{8}, ~\frac{(20+\sqrt{10}) n}{130}\right]$. Then $\frac{n-5k}{3} \le k$, and hence, 
    \begin{align*}
        \Phi_{1}(0, 0,a_3,a_4, 0, 0, b,c,d) 
        & \le \frac{n^2}{2} -2kn +8k^2 -\frac{1}{4} \left(\frac{n-5k}{3} - (n-5k)\right)^2  \\[0.5em]
        & = \frac{1}{18} \left(7 n^2 -16 k n +94 k^2\right) 
        \le \Xi(n,k),
    \end{align*}
    as desired. 
    This completes the proof of Proposition~\ref{PROP:Phi-1-upper-bound}.
\end{proof}

\subsection{Inequalities for $\Phi_{2,i}$}\label{SUBSEC:inequality-Phi2}
In this subsection, we consider the optimization problem of maximizing $\Phi_{2,i}$ for $i \in [3]$ over the region $\Omega_{n,k}$. 

%
\begin{proposition}\label{PROP:Phi21-a6-sparse}
    Suppose that $k \in \left[0, \frac{(20+\sqrt{10}) n}{130}\right] \cup \left[\frac{n}{5},\frac{n}{4}\right]$. Then
    \begin{align*}
        \max\left\{ \Phi_{2,1}(a_1, \ldots, a_6, b,c,d) \colon (a_1, \ldots, a_6, b,c,d) \in \Omega_{n,k} \right\}
        \le \Xi(n,k). 
    \end{align*}
\end{proposition}
\begin{proof}[Proof of Proposition~\ref{PROP:Phi21-a6-sparse}]
    Straightforward calculations show that  
    \begin{align*}
        \Phi_{2,1}(a_1, \ldots, a_6, b,c,d)
        & = \Phi_{2,1}(a_1, a_2, a_3, a_4+a_5+a_6, 0, 0, b,c,d) - (a_5+a_6)d \\[0.5em]
        & \le \Phi_{2,1}(a_1, a_2, a_3, a_4+a_5+a_6, 0, 0, b,c,d) \\[0.5em]
        & = \Phi_{1}(a_1, a_2, a_3, a_4+a_5+a_6, 0, 0, b,c,d) 
        \le \Xi(n,k),
    \end{align*}
    where the last inequality follows from Proposition~\ref{PROP:Phi-1-upper-bound}. 
\end{proof}

\begin{proposition}\label{PROP:Phi22-a6-small}
    Let $(a_1, \ldots, a_6, b,c,d) \in \Omega_{n,k}$. The following statements hold. 
    \begin{enumerate}[label=(\roman*)]
        \item\label{PROP:Phi22-a6-small-1} Suppose that $k \in \left[0, \frac{n}{6} \right]$ and $a_6 = k$. Then 
            \begin{align*}
                \Phi_{2,2}(a_1, \ldots, a_6, b,c,d)
                \le \Xi(n,k).
            \end{align*}
        \item\label{PROP:Phi22-a6-small-2} Suppose that $k \in \left[\frac{n}{6}, \frac{n}{4}\right]$ and $a_6 = \frac{n-4k}{2}$. Then 
            \begin{align*}
                \Phi_{2,2}\left(a_1, \ldots, a_6, b,c,d\right)
                \le \Xi(n,k).
            \end{align*}
    \end{enumerate}
\end{proposition}
\begin{proof}[Proof of Proposition~\ref{PROP:Phi22-a6-small}]
    Recall that $\frac{n - 4k}{2} = \frac{3b+2c+d}{2}$, and 
    \begin{align*}
        \Phi_{2,2}(a_1, \ldots, a_6, b,c,d)
        = \Phi_{1}(a_1, \ldots, a_6, b,c,d) - \frac{(b+c)a_6}{2}.
    \end{align*}
    Suppose that $k \in \left[0, \frac{n}{6} \right]$ and $a_6 = k$. Then 
    \begin{align*}
        \Phi_{2,2}(a_1, \ldots, a_6, b,c,d)
        & = \Phi_{2,2}(0,\ldots,0,k, b,c,n-4k-3b-2c) \\[0.5em]
        & = \frac{4 n^2-6 k n + 57 k^2}{12} -\frac{3}{4}\left(b - \frac{2n-9k}{6}\right)^2 - \left(c- \frac{2n-7k-6b}{4}\right)^2 \\[0.5em]
        & \le \frac{4 n^2-6 k n + 57 k^2}{12}
        \le \Xi(n,k),
    \end{align*}
    which proves Proposition~\ref{PROP:Phi22-a6-small}~\ref{PROP:Phi22-a6-small-1}. 

    Next, we prove Proposition~\ref{PROP:Phi22-a6-small}~\ref{PROP:Phi22-a6-small-2}. 
    Suppose that $k \in \left[\frac{n}{6}, \frac{n}{4}\right]$ and $a_6 = \frac{n-4k}{2}$. 
    
\begin{claim}\label{CLAIM:Phi22-shift-b}
    We may assume that $a_1 = 0$ or $a_1 = k- a_6 = 3k - \frac{n}{2}$. 
\end{claim}
\begin{proof}[Proof of Claim~\ref{CLAIM:Phi22-shift-b}]
    Suppose that $0 < a_1 < k- a_6$. Then $a_2 + a_3 + a_4 + a_5 > 0$. 
    Let 
    \begin{align*}
        q_1(x)
        \coloneqq 
        \Phi_{2,2}\left(a_1+(a_2 + \cdots+a_5)x,(1-x)a_2,  \ldots, (1-x)a_5,a_6, b,c,d\right).
    \end{align*}
    Note that $q_{1}(x)$ is a quadratic polynomial in $x$, and straightforward calculations show that the coefficient of $x^2$ in $q_{1}(x)$ is 
    \begin{align*}
        \frac{a_2^2+a_3^2+2 \left(a_4+a_5\right)^2 + 2\left(a_3+a_4+a_5\right) a_2+4 a_3 \left(a_4+a_5\right)}{2}
        \ge 0. 
    \end{align*}
    Therefore, 
    \begin{align*}
        \Phi_{2,2}(a_1, \ldots, a_6, b,c,d)
        = q_{1}(0)
        \le \max\left\{q_{1}\left(\frac{- a_1}{a_2 + \cdots+a_5}\right),~q_{1}\left(1\right)\right\},
    \end{align*}
    which means that we may assume either $a_1 = 0$ or $a_1 = k- a_6 = 3k - \frac{n}{2}$.
\end{proof}

Suppose that $a_1 = k - a_6 = 3k-\frac{n}{2}$. 
Then $a_2 + \cdots+a_5 = k-(a_1+a_6) = 0$ and 
\begin{align*}
    \Phi_{2,2}(a_1, \ldots, a_6, b,c,d)
    & = \Phi_{2,2}\left(3k-\frac{n}{2},0, \ldots,0, \frac{n-4k}{2}, b,c,n-4k-3b-2c\right) \\[0.5em]
    & = \Phi_{2,2}\left(k,0, \ldots,0, b,c,n-4k-3b-2c\right) - \frac{(n-4k)(b+d)}{8} \\[0.5em]
    & \le \Phi_{2,2}\left(k,0, \ldots,0, b,c,n-4k-3b-2c\right) \\[0.5em]
    & = \Phi_{1}\left(k,0, \ldots,0, b,c,n-4k-3b-2c\right) 
    \le \Xi(n,k), 
\end{align*}
where the last inequality follows from Proposition~\ref{PROP:Phi1-inequality-A}~\ref{PROP:Phi1-inequality-A-1}. 

Therefore, it suffices to consider the case $a_1 =0$. 
\begin{claim}\label{CLAIM:Phi22-shift-c}
    We may assume that $a_2 = 0$ or $a_2 = k- a_6 = 3k - \frac{n}{2}$. 
\end{claim}
\begin{proof}[Proof of Claim~\ref{CLAIM:Phi22-shift-c}]
    The proof is similar to that of Claim~\ref{CLAIM:Phi22-shift-b}.
    Let 
    \begin{align*}
        q_2(x)
        \coloneqq 
        \Phi_{2,2}(a_1, a_2 + (a_3 + a_4 +a_5)x,(1-x)a_3,  (1-x)a_4, (1-x)a_5,a_6, b,c,d). 
    \end{align*}
    Note that $q_{2}(x)$ is a quadratic polynomial in $x$, and straightforward calculations show that the coefficient of $x^2$ in $q_{2}(x)$ is
    \begin{align*}
        \frac{\left(a_4+a_5\right) \left(2 a_3+a_4+a_5\right)}{2} \ge 0. 
    \end{align*}
    Therefore, 
    \begin{align*}
        \Phi_{2}(a_1, \ldots, a_6, b,c,d)
        = q_{2}(0)
        \le \max\left\{q_{2}\left(\frac{- a_2}{a_3 + a_4 +a_5}\right),~q_{2}\left(1\right)\right\},
    \end{align*}
    which means that we may assume either $a_2 = 0$ or $a_2 = k- a_6 = 3k - \frac{n}{2}$. 
    %
\end{proof}

Suppose that $a_2 = k-a_6 = 3k-\frac{n}{2}$. 
Then $a_1 + a_3 + a_4 + a_5 = k-(a_2+a_6) = 0$, and 
\begin{align*}
    \Phi_{2,2}(a_1, \ldots, a_6, b,c,d)
    & = \Phi_{2,2}\left(0, 3k-\frac{n}{2},0, 0,0, \frac{n-4k}{2}, b,c,n-4k-3b-2c\right) \\[0.5em]
    & = \Phi_{2,2}\left(0,k,0 \ldots,0, b,c,n-4k-3b-2c\right) - \frac{(n-4k)(c+d)}{4} \\[0.5em]
    & \le \Phi_{2,2}\left(0,k,0 \ldots,0, b,c,n-4k-3b-2c\right) \\[0.5em]
    & = \Phi_{1}\left(0,k,0 \ldots,0, b,c,n-4k-3b-2c\right)
    \le \Xi(n,k),
\end{align*}
where the last inequality follows from  Proposition~\ref{PROP:Phi1-inequality-A}~\ref{PROP:Phi1-inequality-A-2}. 

So it suffices to consider the case $a_2 = 0$ (recall that we are also assuming $a_1 = 0$). 
Note that, in this case, we have  
\begin{align}\label{equ:Phi2-a1-a2-0}
    \Phi_{2,2}(a_1, \ldots, a_6, b,c,d)
    & = \Phi_{2,2}\left(0,0,a_3, a_4, a_5, a_6, b,c,d\right) \notag \\[0.5em]
    & = \Phi_{2,2}\left(0,0,a_3, a_4+ a_5,0, a_6, b,c,d\right) -d a_5 \notag \\[0.5em]
    & \le  \Phi_{2,2}\left(0,0,a_3, a_4+ a_5,0, a_6, b,c,d\right) \notag \\[0.5em]
    & =  \Phi_{2,2}\left(0,0,a_3, 3k-\frac{n}{2}-a_3,0, \frac{n-4k}{2}, b,c,d\right) \notag \\[0.5em]
    & \le 
    \begin{cases}
        \Phi_{2,2}\left(0,0, 3k-\frac{n}{2}, 0,0, \frac{n-4k}{2}, b,c,d\right), &\quad\text{if}\quad b \ge 3k-\frac{n}{2}, \\[0.5em]
        \Phi_{2,2}\left(0,0,b, 3k-\frac{n}{2}-b,0, \frac{n-4k}{2}, b,c,d\right), &\quad\text{if}\quad b \le 3k-\frac{n}{2},
    \end{cases} 
\end{align}
where the last inequality follows from the fact that $a_3 \in \left[0, 3k -\frac{n}{2}\right]$, and 
\begin{align*}
    & \Phi_{2,2}\left(0,0,a_3, 3k-\frac{n}{2}-a_3,0, \frac{n-4k}{2}, b,c,d\right) \\[0.5em]
    & = -\frac{\left(a_3-b\right)^2}{2} + \frac{-20 b^2-24 b c-120 b k+26 b n-8 c^2-64 c k+14 c n-84 k^2+48 k n-3 n^2}{8}.
\end{align*}
By~\eqref{equ:Phi2-a1-a2-0}, it suffices to show that 
\begin{align*}
    \Xi(n,k)
    \ge 
    \begin{cases}
        \Phi_{2,2}\left(0,0, 3k-\frac{n}{2}, 0,0, \frac{n-4k}{2}, b,c,n-4k-3b-2c\right) =: \Phi_{2,2}^{\ast}, &\quad\text{if}\quad b \ge 3k-\frac{n}{2}, \\[0.5em]
        \Phi_{2,2}\left(0,0,b, 3k-\frac{n}{2}-b,0, \frac{n-4k}{2}, b,c,n-4k-3b-2c\right) =: \Phi_{2,2}^{\blacktriangle}, &\quad\text{if}\quad b \le 3k-\frac{n}{2}.
    \end{cases} 
\end{align*}
Note that $k\in[n/6,n/4]$ implies that $3k-n/2$ and $(n-4k)/2$ are nonnegative and sum to $k$. Since $b,c\ge0$ and $n-4k-3b-2c=d\ge0$, the tuple defining $\Phi_{2,2}^{*}$ belongs to $\Omega_{n,k}$.

\medskip

\textbf{Case 1}: $b \ge 3k-\frac{n}{2}$ (this case is possible only when $k \in \left[0, \frac{5n}{26}\right]$, since $b \le \frac{n-4k}{3}$ holds trivially).

It follows from the assumption that $k$ lies in the interval $\left[\frac{n}{6}, \frac{5n}{26}\right]$. 
Straightforward calculations show that 
\begin{align}\label{equ:Phi22-a}
    \Phi_{2,2}^{\ast} 
    & = \frac{-48 b^2+8 b n+64 k^2+32 k n+17 n^2}{64} - \left(c- \frac{7n -32 k -12b}{8}\right)^2. 
\end{align}  
Recall that $c$ lies in the interval $\left[0, \frac{n - 4k - 3b}{2}\right]$. 
Since $\frac{n - 4k - 3b}{2} \le \frac{7n -32 k -12b}{8}$ iff $k \le \frac{3 n}{16}$, it follows from~\eqref{equ:Phi22-a} that   
\begin{align*}
    \Phi_{2,2}^{\ast} 
    &\le 
    \begin{cases}
        \Phi_{2,2}\left(0,0, 3k-\frac{n}{2}, 0,0, \frac{n-4k}{2}, b,\frac{n-4k-3b}{2},d\right) =: \Phi_{2,2}^{\ast\ast}, &\quad\text{if}\quad k \in \left[\frac{n}{6}, \frac{3 n}{16}\right], \\[0.5em]
        \Phi_{2,2}\left(0,0, 3k-\frac{n}{2}, 0,0, \frac{n-4k}{2}, b,\frac{7n -32 k -12b}{8},d\right) =: \Phi_{2,2}^{\ast\blacktriangle}, &\quad\text{if}\quad k \in \left[\frac{3 n}{16}, \frac{5n}{26} \right].
    \end{cases}
\end{align*}
Suppose that $k \in \left[\frac{n}{6}, \frac{3 n}{16}\right]$, then 
\begin{align*}
    \Phi_{2,2}^{\ast}
    \le \Phi_{2,2}^{\ast\ast}
    & = \frac{25 n^2}{192} +2 k n -3 k^2 -\frac{3}{4} \left(b-\frac{n}{12}\right)^2  
    \le \frac{25 n^2}{192} +2 k n -3 k^2
    \le \Xi(n,k),
\end{align*}
as desired.

Suppose that $k \in \left[\frac{3 n}{16}, \frac{5n}{26} \right]$. Then 
\begin{align*}
    \Phi_{2,2}^{\ast}
    \le \Phi_{2,2}^{\ast\blacktriangle}
     = \frac{13 n^2}{48} +\frac{k n}{2} +k^2 -\frac{3}{4} \left(b-\frac{n}{12}\right)^2 
    \le  \frac{13 n^2}{48} +\frac{k n}{2} +k^2 
    \le \Xi(n,k),
\end{align*}
as desired.

\medskip 

\textbf{Case 2}: $b \le 3k-\frac{n}{2}$.

Straightforward calculations show that 
\begin{align*}
    \Phi_{2,2}^{\blacktriangle}
    & = \frac{-16 b^2-192 b k+40 b n+352 k^2-64 k n+25 n^2}{64} - \left(c- \frac{7n-32k-12b}{8}\right)^2.
\end{align*}
Note that $\frac{7n-32k-12b}{8} \le 0$ if $k \ge \frac{7 n}{32}$, and $c  \le \frac{n-4k-3b}{2} \le \frac{7n-32k-12b}{8}$ if $k \le \frac{3n}{16}$, so we have 
\begin{align*}
    \Phi_{2,2}^{\blacktriangle} 
    & \le 
    \begin{cases}
        \Phi_{2,2}\left(0,0,b, 3k-\frac{n}{2}-b,0, \frac{n-4k}{2}, b,c,d\right) =: \Phi_{2,2}^{\blacktriangle \ast}, &\quad\text{if}\quad k \in \left[\frac{7n}{32}, \frac{n}{4}\right], \\[0.5em]
        \Phi_{2,2}\left(0,0,b, 3k-\frac{n}{2}-b,0, \frac{n-4k}{2}, b,\frac{7n-32k-12b}{8},d\right) =: \Phi_{2,2}^{\blacktriangle \blacktriangle}, &\quad\text{if}\quad k \in \left[\frac{3n}{16},\frac{7n}{32}\right], \\[0.5em]
        \Phi_{2,2}\left(0,0,b, 3k-\frac{n}{2}-b,0, \frac{n-4k}{2}, b,\frac{n-4k-3b}{2},d\right) =: \Phi_{2,2}^{\blacktriangle \blacktriangledown}, &\quad\text{if}\quad k \in \left[\frac{n}{6}, \frac{3n}{16}\right].
    \end{cases}
\end{align*}
Suppose that $k \in \left[\frac{7n}{32}, \frac{n}{4}\right]$. 
Then $\frac{13n-60k}{20} \le 0$, and hence, 
\begin{align*}
    \Phi_{2,2}^{\blacktriangle \ast} 
    & = \frac{109 n^2}{160} - \frac{15 k n}{4} + 12 k^2 - \frac{5}{2}\left(b - \frac{13n-60k}{20}\right)^2 \\[0.5em]
    & \le \frac{109 n^2}{160} - \frac{15 k n}{4} + 12 k^2 - \frac{5}{2}\left(0 - \frac{13n-60k}{20}\right)^2 
    =  \frac{-3(n^2 - 16 kn + 28k^2)}{8}
    \le \Xi(n,k),
\end{align*}
as desired. 

Suppose that $k \in \left[\frac{3n}{16}, \frac{7n}{32}\right]$. 
Then 
\begin{align*}
    \Phi_{2,2}^{\blacktriangle \blacktriangle}  
    & = \frac{25 n^2-152 k n+464 k^2}{32}  - \frac{1}{4}\left(b - \frac{5n-24k}{4}\right)^2 \\[0.5em]
    & \le 
    \begin{cases}
        \frac{25 n^2}{64}-k n+\frac{11 k^2}{2}, &\quad\text{if}\quad k \in \left[\frac{5n}{24},\frac{7n}{32}\right], \\[0.5em]
        \frac{25 n^2-152 k n+464 k^2}{32}, &\quad\text{if}\quad k \in \left[\frac{3n}{16}, \frac{5n}{24}\right],
    \end{cases}
\end{align*}
which is smaller than $\Xi(n,k)$ for $k \in \left[\frac{3n}{16},\frac{7n}{32}\right]$.

Suppose that $k \in \left[\frac{n}{6}, \frac{3n}{16}\right]$. 
Then $b \le 3k - \frac{n}{2} \le \frac{5n - 24k}{4}$, and hence, 
\begin{align*}
    \Phi_{2,2}^{\blacktriangle \blacktriangledown} 
    & = \frac{41 n^2 - 208 kn + 672 k^2}{64}  - \frac{1}{4}\left(b - \frac{5n-24k}{4}\right)^2 \\[0.5em]
    & \le \frac{41 n^2 - 208 kn + 672 k^2}{64}  - \frac{1}{4}\left(3k-\frac{n}{2} - \frac{5n-24k}{4}\right)^2  \\[0.5em]
    & = \frac{-n^2+37 k n-78 k^2}{8}
    \le \Xi(n,k),  
\end{align*}
as desired. 
This completes the proof of Proposition~\ref{PROP:Phi22-a6-small}. 
\end{proof}
\begin{proposition}\label{PROP:Phi22-upper-bound}
    Let $(a_1, \ldots, a_6, b,c,d) \in \Omega_{n,k}$. 
    Suppose that $k \in \left[0, \frac{(20+\sqrt{10})n}{130}\right] \cup \left[\frac{n}{5}, \frac{n}{4}\right]$ and $a_6 \le \frac{n-4k}{2}$. Then 
    \begin{align*}
        \Phi_{2,2}(a_1, \ldots, a_6, b,c,d)
        \le \Xi(n,k). 
    \end{align*}
\end{proposition}
\begin{proof}[Proof of Proposition~\ref{PROP:Phi22-upper-bound}]
    Let $a_6^{\ast} \coloneqq \min\left\{\frac{n-4k}{2}, k\right\}$. Note the $a_6$ lies in the interval $[0, a_{6}^{\ast}]$. 
    
    \begin{claim}\label{CLAIM:Phi2-shift-a}
        We may assume that $a_6 = 0$ or $a_6 = a_6^{\ast}$. 
    \end{claim}
    \begin{proof}[Proof of Claim~\ref{CLAIM:Phi2-shift-a}]
        Note that 
        \begin{align*}
            q_1(x) 
            \coloneqq \Phi_{2,2}((1-x)a_1, \ldots, (1-x)a_5, a_6 + (a_1 + \cdots + a_5)x, b,c,d).
        \end{align*}
        is quadratic in $x$, and straightforward calculations show that the coefficient of $x^2$ in $q_1(x)$ is 
        \begin{align*}
            \frac{3 a_1^2}{2}+\left(2 a_2+a_3+a_4+a_5\right) a_1+a_2^2+a_2 \left(a_3+a_4+a_5\right)+\frac{\left(a_4+a_5\right) \left(2 a_3+a_4+a_5\right)}{2} 
            \ge 0.
        \end{align*}
        So we have 
        \begin{align*}
            \Phi_{2,2}(a_1, \ldots, a_6, b,c,d)
            = q_1(0)
            \le \max\left\{q_1\left(\frac{- a_6}{a_1 + \cdots + a_5}\right),~q_1\left(\frac{a_6^{\ast} - a_6}{a_1 + \cdots + a_5}\right)\right\}, 
        \end{align*}
        meaning that we can assume either $a_6 = 0$ or $a_6 = a_{6}^{\ast}$. 
    \end{proof}

    Suppose that $a_6 = 0$. 
    Then straightforward calculations show that 
    \begin{align*}
        \Phi_{2,2}(a_1, \ldots, a_5,0, b,c,d)
        & = \Phi_{2,2}(a_1, \ldots,a_4+ a_5,0,0, b,c,d) - d a_5 \\[0.5em]
        & \le \Phi_{2,2}(a_1, \ldots,a_4+ a_5,0,0, b,c,d) \\[0.5em]
        & = \Phi_{1}(a_1, \ldots,a_4+ a_5,0,0, b,c,d)
        \le \Xi(n,k), 
    \end{align*}
    where the last inequality follows from Proposition~\ref{PROP:Phi-1-upper-bound}. 

    Suppose that $k \le \frac{n-4k}{2}$ (equivalently, $k \in \left[0, \frac{n}{6} \right]$) and $a_6 = k$.
    Then it follows from Proposition~\ref{PROP:Phi22-a6-small}~\ref{PROP:Phi22-a6-small-1} that 
    \begin{align*}
        \Phi_{2,2}(a_1, \ldots, a_6, b,c,d) 
        \le \Xi(n,k). 
    \end{align*}
    
    Suppose that $\frac{n-4k}{2} \le k$ (equivalently, $k \in\left[\frac{n}{6}, \frac{n}{4}\right]$) and $a_6 = \frac{n-4k}{2}$. 
    Then it follows from Proposition~\ref{PROP:Phi22-a6-small}~\ref{PROP:Phi22-a6-small-2} that 
    \begin{align*}
        \Phi_{2,2}(a_1, \ldots, a_6, b,c,d) 
        \le \Xi(n,k). 
    \end{align*}
    This completes the proof of Proposition~\ref{PROP:Phi22-upper-bound}.
\end{proof}

\begin{proposition}\label{PROP:Phi23-a6-large}
    Let $(a_1, \ldots, a_6, b,c,d) \in \Omega_{n,k}$. 
    Suppose that $k \in \left[\frac{n}{6}, \frac{n}{4} \right]$ and $a_6 \in \left[\frac{n-4k}{2}, k\right]$. 
    Then 
    \begin{align*}
        \Phi_{2,3}(a_1, \ldots, a_6, b,c,d) 
        \le \Xi(n,k). 
    \end{align*}
\end{proposition}
\begin{proof}[Proof of Proposition~\ref{PROP:Phi23-a6-large}]
Recall that 
\begin{align*}
    \Phi_{2,3}(a_1, \ldots, a_6, b,c,d)
    = \Phi(a_1, \ldots, a_6, b,c,d) 
        - (7 b+5 c+2 d)a_6 + (3b+2c+d)^2. 
\end{align*}

%
\begin{claim}\label{CLAIM:Phi23-a6}
    We may assume that $a_6 = \frac{n-4k}{2}$ or $a_6 = k$.
\end{claim}
\begin{proof}[Proof of Claim~\ref{CLAIM:Phi23-a6}]
    Suppose that $\frac{n-4k}{2} < a_6 < k$. 
    Note that  
    \begin{align*}
        q_1(x)
        \coloneqq \Phi_{2,3}((1-x)a_1, \ldots, (1-x)a_5, a_6+(a_1 + \cdots+ a_5)x, b,c,d)
    \end{align*}
    is quadratic in $x$ and straightforward calculations show that the coefficient of $x^2$ is 
    \begin{align*}
        \frac{3 a_1^2}{2}+\left(2 a_2+a_3+a_4+a_5\right) a_1+a_2^2+a_2 \left(a_3+a_4+a_5\right)+\frac{\left(a_4+a_5\right) \left(2 a_3+a_4+a_5\right)}{2} 
        \ge 0.
    \end{align*}
    So 
    \begin{align*}
        \Phi_{2,3}(a_1, \ldots, a_6, b,c,d)
        = q_1(0)
        \le \max\left\{q_1\left(\frac{(n-4k)/2 - a_6}{a_1 + \cdots + a_5}\right),~q_1\left(\frac{k-a_6}{a_1 + \cdots + a_5}\right)\right\}, 
    \end{align*}
    meaning that we can assume either $a_6 = \frac{n-4k}{2}$ and $a_6 = k$.
\end{proof}

\medskip

\textbf{Case 1}: $a_6 = k$. 

We have $a_1 = \cdots = a_5 = k-a_6 = 0$ and 
\begin{align*}
    \Phi_{2,3}(a_1, \ldots, a_6, b,c,d)
     = \Phi_{2,2}(0, \ldots, 0, k, b,c,n-4k-3b-2c) 
     = n^2 - 8kn +24k^2,
\end{align*}
which is less or equal to $\Xi(n,k)$ for all $k \in \left[\frac{n}{6}, \frac{n}{4}\right]$. 

\medskip

\textbf{Case 2}:  $a_6 = \frac{n-4k}{2}$. 

Similar to Claim~\ref{CLAIM:Phi22-shift-b}, since the coefficient of $x^2$ in 
\begin{align*}
    \Phi_{2,3}\left(a_1 + (a_2 + a_3 + a_4 + a_5)x, (1-x)a_2, \ldots , (1-x)a_5, a_6, b,c,d\right)
\end{align*}
is 
\begin{align*}
    \frac{1}{2}\left(a_2^2+2 \left(a_3+a_4+a_5\right) a_2+a_3^2+2 \left(a_4+a_5\right){}^2+4 a_3 \left(a_4+a_5\right)\right)
    \ge 0, 
\end{align*}
we may assume either $a_1 = 0$ or $a_1 = k - a_6 = 3k-\frac{n}{2}$. 

Suppose that $a_1 = 3k-\frac{n}{2}$. Then $a_2 = \cdots = a_5 = k-(a_1+a_6) = 0$, and 
\begin{align*}
    \Phi_{2,3}(a_1, \ldots, a_6, b,c,d)
    & = \Phi_{2,3}\left(3k-\frac{n}{2},0, \ldots, 0, \frac{n-4k}{2}, b,c,n-4k-3b-2c\right) \\[0.5em]
    & = - \frac{n^2}{8} +4kn -\frac{15 k^2}{2},
\end{align*}
which is smaller than $\Xi(n,k)$ for $k \in \left[\frac{n}{6}, \frac{n}{4}\right]$. 

So we may assume that $a_1 = 0$. Similar to Claim~\ref{CLAIM:Phi22-shift-c},
since the coefficient of $x^2$ in 
\begin{align*}
    \Phi_{2,3}\left(a_1, a_2 + (a_3 + a_4 + a_5)x, (1-x)a_3, \ldots , (1-x)a_5, a_6, b,c,d\right)
\end{align*}
is 
\begin{align*}
    \frac{1}{2} \left(a_4+a_5\right) \left(2 a_3+a_4+a_5\right) 
    \ge 0, 
\end{align*}
we may assume that $a_2 = 0$ or $a_2 = k - a_6 = 3k-\frac{n}{2}$.

Suppose that $a_2 = 3k-\frac{n}{2}$. Then $a_1 = a_3 = a_4 = a_5 = k-(a_2 + a_6) = 0$, and 
\begin{align*}
    \Phi_{2,3}(a_1, \ldots, a_6, b,c,d)
    & = \Phi_{2,3}\left(0,3k-\frac{n}{2},0, 0, 0, \frac{n-4k}{2}, b,c,n-4k-3b-2c\right) \\[0.5em]
    & =  -\frac{n^2}{4}+5 k n -9 k^2 - \frac{b (6 k - n)}{2} 
    \le -\frac{n^2}{4}+5 k n -9 k^2, 
\end{align*}
which is smaller than $\Xi(n,k)$ for $k \in \left[\frac{n}{6}, \frac{n}{4}\right]$. 

So we may assume that $a_2 = 0$ (recall that we are also assuming $a_1 = 0$). Straightforward calculations show that 
\begin{align*}
    & \Phi_{2,3}\left(0,0,a_3,3k-\frac{n}{2} - a_3,0, \frac{n-4k}{2}, b,c,n-4k-3b-2c\right) \\[0.5em]
    & = -\frac{\left(a_3-b\right)^2}{2} + \frac{4 b^2-48 b k+8 b n-24 c k+4 c n-84 k^2+48 k n-3 n^2}{8}. 
\end{align*}
Since $a_3 \le k-a_6 = 3k-\frac{n}{2}$, it follows from the inequality above that  
\begin{align*}
    & \Phi_{2,3}\left(0,0,a_3, a_4, a_5, a_6, b,c,d\right) \\[0.5em]
    & = \Phi_{2,3}\left(0,0,a_3, a_4+a_5,0, a_6, b,c,d\right) -d a_5 \\[0.5em]
    & \le \Phi_{2,3}\left(0,0,a_3, a_4+a_5,0, a_6, b,c,d\right) \\[0.5em]
    & = \Phi_{2,3}\left(0,0,a_3,3k-\frac{n}{2} - a_3,0, \frac{n-4k}{2}, b,c,n-4k-3b-2c\right) \\[0.5em]
    & \le 
    \begin{cases}
        \Phi_{2,3}\left(0,0,3k-\frac{n}{2},0,0, \frac{n-4k}{2}, b,c,n-4k-3b-2c\right) =: \Phi_{2,3}^{\ast}, &\quad\text{if}\quad b \ge 3k- \frac{n}{2}, \\[0.5em]
        \Phi_{2,3}\left(0,0,b,3k-\frac{n}{2}-b,0, \frac{n-4k}{2}, b,c,n-4k-3b-2c\right) =: \Phi_{2,3}^{\blacktriangle}, &\quad\text{if}\quad b \le 3k- \frac{n}{2}. 
    \end{cases}
\end{align*}

Suppose that $b \ge 3k- \frac{n}{2}$. 
Then 
\begin{align*}
    \Phi_{2,3}^{\ast} 
    & = \frac{1}{2} \left(-n^2+15 k n -30 k^2- (b+c)(6k-n)\right) \\[0.5em]
    & \le \frac{1}{2} \left(-n^2+15 k n -30 k^2- \left(3k- \frac{n}{2}\right)(6k-n)\right) 
    = \frac{3}{4} \left(-n^2 + 14 k n- 32 k^2\right),
\end{align*}
which is smaller than $\Xi(n,k)$ for $k \in \left[\frac{n}{6}, \frac{n}{4}\right]$. 

Suppose that $b  \le 3k- \frac{n}{2} = \frac{6k-n}{2} \le 6k-n$. 
Then 
\begin{align*}
    \Phi_{2,3}^{\blacktriangle}
    & = \frac{-7n^2+96kn-228k^2}{8} +\frac{\left(b-(6k-n)\right)^2}{2} - \frac{c(6k-n)}{2} \\[0.5em]
    & \le \frac{-7n^2+96kn-228k^2}{8} +\frac{\left(0-(6k-n)\right)^2}{2}  
     = \frac{3}{8} \left(-n^2+16 k n-28 k^2\right),
\end{align*}
which is smaller than $\Xi(n,k)$ for $k \in \left[\frac{n}{6}, \frac{n}{4}\right]$. 

This completes the proof of Proposition~\ref{PROP:Phi23-a6-large}.
\end{proof}
%

\subsection{Inequalities for $\Phi_{3,i}$ and others}\label{SUBSEC:inequality-Phi3}
In this subsection, we consider the optimization problem for $\Phi_{3,i}$ and other functions defined in Section~\ref{SUBSEC:global-definitions}. 
\begin{proposition}\label{PROP:Phi61-a1a2a4a6-zero-upper-bound}
    Let $(a_1, \ldots, a_6, b,c,d) \in \Omega_{n,k}$. 
    Suppose that $k \in \left[\frac{n}{6}, \frac{n}{4}\right]$ and $a_1 = a_2 = a_4 = a_6 = 0$. 
    Then 
    \begin{align*}
        \Phi_{6,1}(a_1, \ldots, a_6, b,c,d)
        \le \Xi(n,k). 
    \end{align*}
\end{proposition}
\begin{proof}[Proof of Proposition~\ref{PROP:Phi61-a1a2a4a6-zero-upper-bound}]
    Let $\Phi_{6,1}^{\ast} \coloneqq \Phi_{6,1}(0,0,a_3,0,k-a_3,0,b,c,n-4k-3b-2c)$. 
    Straightforward calculations show that 
    \begin{align*}
        \Phi_{6,1}^{\ast} 
        & =  -\frac{3 b^2}{4}+\frac{5 c^2}{4}+b c+b k+5 c k-c n + \frac{n^2}{2}-2 k n +\frac{15 k^2}{2}- \frac{\left(a_3+2 b+2 c+4 k-n\right)^2}{2} \\[0.5em]
        & \le -\frac{3 b^2}{4}+\frac{5 c^2}{4}+b c+b k+5 c k-c n + \frac{n^2}{2}-2 k n +\frac{15 k^2}{2} 
        =: q(b,c). 
    \end{align*}
    Let us consider $q_{1}(x) \coloneqq q\left(b+\frac{x}{3},c-\frac{x}{2}\right)$, where we use $\frac{x}{3}$ and $-\frac{x}{2}$ because we need $3\left(b+\frac{x}{3}\right) + 2 \left(c-\frac{x}{2}\right)$ to be a constant. Note that $q_1(x)$ is quadratic in $x$ and  simple calculations show that the coefficient of $x^2$ is $\frac{1}{16} > 0$. Thus we have 
    \begin{align*}
        q(b,c)
        = q_1(0)
        \le \max\left\{q_1\left(-3b\right),~q_1\left(2c\right)\right\},
    \end{align*}
    meaning that we can assume either $b = 0$ or $c=0$. 
    
    Suppose that $b = 0$. Then 
    \begin{align*}
        \Phi_{6,1}^{\ast}
        & = \frac{5}{4} \left(c-\frac{2 (n-5 k)}{5} \right)^2+\frac{5 k^2}{2}+\frac{3 n^2}{10} - \frac{\left(a_3+2 c+4 k-n\right)^2}{2}  \\[0.5em]
        & \le \frac{5}{4} \left(c-\frac{2 (n-5 k)}{5} \right)^2+\frac{5 k^2}{2}+\frac{3 n^2}{10}. 
    \end{align*}
    Notice that $c$ lies in the interval $\left[0, \frac{n-4k}{2}\right]$. 
    Since $k \ge \frac{n}{6}$, we have $\frac{n-4k}{2} \ge 2 \cdot \frac{2 (n-5 k)}{5}$. Thus, the inequality above continues as 
    \begin{align*}
        \Phi_{6,1}^{\ast}
        & \le \frac{5}{4} \left(\frac{n-4k}{2}-\frac{2 (n-5 k)}{5} \right)^2+\frac{5 k^2}{2}+\frac{3 n^2}{10}  
        = \frac{5 n^2}{16} + \frac{5 k^2}{2}, 
    \end{align*}
    which is smaller that $\Xi(n,k)$ for $k \in \left[\frac{n}{6}, \frac{n}{4}\right]$. 

    Suppose that $c = 0$. Then 
    \begin{align*}
        \Phi_{6,1}^{\ast}
        & = - \frac{3}{4}\left(b-\frac{2 k}{3}\right)^2+\frac{47 k^2}{6}-2 k n+\frac{n^2}{2} -\frac{\left(a_3+2 b+4 k-n\right)^2}{2} \\[0.5em]
        & \le - \frac{3}{4}\left(b-\frac{2 k}{3}\right)^2+\frac{47 k^2}{6}-2 k n+\frac{n^2}{2}. 
    \end{align*}
    Notice that $b$ lies in the interval $\left[0, \frac{n-4k}{3}\right]$. Since $k \ge \frac{n}{6}$, we have $\frac{n-4k}{3} \le \frac{2 k}{3}$. Thus, the inequality above continues as 
    \begin{align*}
        \Phi_{6,1}^{\ast}
        & \le - \frac{3}{4}\left(\frac{n-4k}{3}-\frac{2 k}{3}\right)^2+\frac{47 k^2}{6}-2 k n+\frac{n^2}{2} 
        = \frac{29 k^2}{6}-k n+\frac{5 n^2}{12}, 
    \end{align*}
    which is smaller that $\Xi(n,k)$ for $k \in \left[\frac{n}{6}, \frac{n}{4}\right]$. 

    This completes the proof of Proposition~\ref{PROP:Phi61-a1a2a4a6-zero-upper-bound}.
\end{proof}

\begin{proposition}\label{PROP:Phi62-a1a2a4a6-upper-bound}
    Let $(a_1, \ldots, a_6, b,c,d) \in \Omega_{n,k}$. 
    Suppose that $k \in \left[\frac{(19982-35 \sqrt{402})n}{108278}, \frac{n}{4}\right]$, $d \ge \frac{b+c}{2}$, and $a_1 = a_2 = a_4 = a_6 = 0$. 
    Then 
    \begin{align*}
        \Phi_{6,2}(a_1, \ldots , a_6, b,c,d)
        \le \Xi(n,k). 
    \end{align*}
\end{proposition}
\begin{proof}[Proof of Proposition~\ref{PROP:Phi62-a1a2a4a6-upper-bound}]
    Combining $d = n-4k - 3b-2c$ with $d \ge \frac{b+c}{2}$, we obtain $7b+5c \le 2n-8k$, which implies that $b \le \frac{2n-8k}{7}$ and $c \le \frac{2n-8k}{5}$. 

    Let $\Phi_{6,2}^{\ast} \coloneqq \Phi_{6,2}(0,0,a_3,0,k-a_3,0,b,c,n-4k-3b-2c)$. 
    Straightforward calculations show that  
    \begin{align*}
        \Phi_{6,2}^{\ast} 
        & =  \frac{b^2}{150}+b c+b k+\frac{3 c^2}{2}+5 c k-c n+\frac{15 k^2}{2}-2 k n+\frac{n^2}{2} -\frac{\left(a_3+2 b+2 c+4 k-n\right)^2}{2}  \\[0.5em]
        & \le \frac{b^2}{150}+b c+b k+\frac{3 c^2}{2}+5 c k-c n+\frac{15 k^2}{2}-2 k n+\frac{n^2}{2}
        =: q(b,c). 
    \end{align*}
    Let us consider $q_1(x) \coloneqq q\left(b+\frac{x}{3},c-\frac{x}{2}\right)$. Note that $q_1(x)$ is quadratic in $x$ and  simple calculations show that the coefficient of $x^2$ is $\frac{1129}{5400} > 0$. Thus we have 
    \begin{align*}
        q(b,c)
        = q_1(0)
        \le \max\left\{q_1\left(-3b\right), q_1\left(2c\right)\right\},
    \end{align*}
    meaning that we can assume either $b = 0$ or $c=0$. 
    
    Suppose that $b = 0$. Then 
    \begin{align*}
        \Phi_{6,2}^{\ast}
        & = \frac{3}{2} \left(c-\frac{n-5 k}{3} \right)^2+\frac{\left(n^2-k n+10 k^2\right)}{3} - \frac{(a_3+2c+4k-n)^2}{2} \\[0.5em]
        & \le \frac{3}{2} \left(c-\frac{n-5 k}{3} \right)^2+\frac{\left(n^2-k n+10 k^2\right)}{3}.
    \end{align*}
    Recall that $c$ is contained in the interval $\left[0, \frac{2n-8k}{5}\right]$. 
    Since $k \ge \frac{(19982-35 \sqrt{402})n}{108278} \ge \frac{2n}{13}$, we have $\frac{2n-8k}{5}  \ge 2 \cdot \frac{n-5 k}{3}$. Thus, the inequality above continues as 
    \begin{align*}
        \Phi_{6,2}^{\ast}
        & \le \frac{3}{2} \left(\frac{2n-8k}{5}-\frac{n-5 k}{3} \right)^2+\frac{1}{3} \left(10 k^2-k n+n^2\right)   
        = \frac{1}{50} \left(167 k^2-16 k n+17 n^2\right), 
    \end{align*}
    which is no more than that $\Xi(n,k)$ for $k \in \left[\frac{(19982-35 \sqrt{402})n}{108278}, \frac{n}{4}\right]$. 

    Suppose that $c = 0$. Then 
    \begin{align*}
        \Phi_{6,2}^{\ast}
        & = \frac{n^2}{2} -2 k n -30 k^2 + \frac{(b+75 k)^2}{150} - \frac{(a_3+2b+4k-n)^2}{2} \\[0.5em]
         & \le  \frac{n^2}{2} -2 k n -30 k^2 + \frac{(b+75 k)^2}{150}.  
    \end{align*}
    Recall that $b$ lies in the interval $\left[0, \frac{2n-8k}{7}\right]$.  Thus, the inequality above continues as 
    \begin{align*}
        \Phi_{6,2}^{\ast}
        & \le \frac{1}{150} \left(\frac{2n-8k}{7}+75 k\right)^2-30 k^2-2 k n+\frac{n^2}{2}  
        = \frac{46789 k^2-12632 k n+3679 n^2}{7350}, 
    \end{align*}
    which is smaller that $\Xi(n,k)$ for $k \in \left[\frac{(19982-35 \sqrt{402})n}{108278}, \frac{n}{4}\right]$. 
\end{proof}

\begin{proposition}\label{PROP:Phi62-a1a2a5a6-zero-upper-bound}
    Let $(a_1, \ldots, a_6, b,c,d) \in \Omega_{n,k}$. 
    Suppose that $k \in \left[\frac{(24086-35 \sqrt{3282})n}{128054}, \frac{n}{4}\right]$,  $d \ge \frac{b+c}{2}$, and $a_1 = a_2 = a_5 = a_6 = 0$. 
    Then 
    \begin{align*}
        \Phi_{6,2}(a_1, \ldots , a_6, b,c,d)
        \le \Xi(n,k). 
    \end{align*}
\end{proposition}
\begin{proof}[Proof of Proposition~\ref{PROP:Phi62-a1a2a5a6-zero-upper-bound}]
    Combining $d = n-4k - 3b-2c$ with $d \ge \frac{b+c}{2}$, we obtain $7b+5c \le 2n-8k$, which implies that $b \le \frac{2n-8k}{7}$ and $c \le \frac{2n-8k}{5}$. 
    
    Let $\Phi_{6,2}^{\ast} \coloneqq \Phi_{6,2}(0,0,a_3,k-a_3,0,0,b,c,n-4k-3b-2c)$. Straightforward calculations show that 
    \begin{align*}
        \Phi_{6,2}^{\ast} 
        & =  3 k n-\frac{9 k^2}{2}-\frac{113 b^2}{600}-\frac{41 b c}{20}-\frac{109 b k}{10}+2 b n-\frac{3 c^2}{8}-\frac{11 c k}{2}+c n -\frac{1}{2} \left(a_3-\frac{19 b+5 c}{10}\right)^2 \\[0.5em]
        & \le 3 k n-\frac{9 k^2}{2} -\frac{113 b^2}{600}-\frac{41 b c}{20}-\frac{109 b k}{10}+2 b n-\frac{3 c^2}{8}-\frac{11 c k}{2}+c n
        =: q(b,c).
    \end{align*}
    Let us consider $q_1(x) \coloneqq q\left(b+\frac{x}{3},c-\frac{x}{2}\right)$. Note that $q_1(x)$ is quadratic in $x$ and  simple calculations show that the coefficient of $x^2$ is $\frac{4903}{21600} > 0$. Thus we have 
    \begin{align*}
        q(b,c)
        = q_1(0)
        \le \max\left\{q_1(-3b),~q_1(2c)\right\},
    \end{align*}
    meaning that we can assume either $b=0$ or $c=0$.

    Suppose that $b = 0$. Then 
    \begin{align*}
        \Phi_{6,2}^{\ast} 
        & = \frac{\left(2 n^2-13 k n+47 k^2\right)}{3} -\frac{3}{8} \left(c - \frac{2(2n - 11 k)}{3} \right)^2 - \frac{1}{2}\left(a_3-\frac{c}{2}\right)^2 \\[0.5em]
        & \le \frac{\left(47 k^2-13 k n+2 n^2\right)}{3} -\frac{3}{8} \left(c - \frac{2(2n - 11 k)}{3} \right)^2.
    \end{align*}

    If $k \in \left[\frac{(24086-35 \sqrt{3282})n}{128054}, \frac{2n}{11}\right]$, then we have 
    \begin{align*}
        \Phi_{6,2}^{\ast} 
        \le \frac{\left(47 k^2-13 k n+2 n^2\right)}{3} 
        \le \Xi(n,k), 
    \end{align*}
    as desired. 

    If $k \in  \left[\frac{2n}{11}, \frac{n}{4}\right]$, then $\frac{2(2n - 11 k)}{3} \le 0$, and hence,  
    \begin{align*}
        \Phi_{6,2}^{\ast} 
        \le \frac{\left(47 k^2-13 k n+2 n^2\right)}{3} -\frac{3}{8} \left(0 - \frac{2(2n - 11 k)}{3} \right)^2 
        =3 k n-\frac{9 k^2}{2} 
        \le \Xi(n,k), 
    \end{align*}
    as desired. 

    Suppose that $c = 0$. Then 
    \begin{align*}
        \Phi_{6,2}^{\ast} 
        & = \frac{3\left(200 n^2-2067 k n+5771 k^2\right)}{113}-\frac{113}{600} \left(b - \frac{30(20n-109k)}{113}\right)^2 - \frac{1}{2} \left(a_3-\frac{19 b}{10}\right)^2 \\[0.5em]
         & \le  \frac{3\left(200 n^2-2067 k n+5771 k^2\right)}{113}-\frac{113}{600} \left(b - \frac{30(20n-109k)}{113}\right)^2.
    \end{align*}
    Recall that $b$ lies in the interval $\left[0, \frac{2n-8k}{7}\right]$. 

    If $k \in \left[\frac{20n}{109}, \frac{n}{4}\right]$, then $\frac{30(20n-109k)}{113} \le 0$, and hence, 
    \begin{align*}
        \Phi_{6,2}^{\ast}
        & \le  \frac{3\left(200 n^2-2067 k n+5771 k^2\right)}{113}-\frac{113}{600} \left(0 - \frac{30(20n-109k)}{113}\right)^2 \\[0.5em]
        & = 3kn- \frac{9k^2}{2} 
        \le \Xi(n,k),
    \end{align*}
    as desired. 

    If $k \in \left[\frac{1987 n}{10993}, \frac{20n}{109}\right]$, then trivially, 
    \begin{align*}
        \Phi_{6,2}^{\ast}
        & \le  \frac{3\left(200 n^2-2067 k n+5771 k^2\right)}{113} 
        \le \Xi(n,k),
    \end{align*}
    as desired. 

    If $k \in \left[\frac{(24086-35 \sqrt{3282})n}{128054}, \frac{1987 n}{10993}\right]$, then $\frac{2n-8k}{7} \le \frac{30(20n-109k)}{113}$, and hence, 
    \begin{align*}
        \Phi_{6,2}^{\ast}
        & \le  \frac{3\left(200 n^2-2067 k n+5771 k^2\right)}{113}-\frac{113}{600} \left(\frac{2n-8k}{7} - \frac{30(20n-109k)}{113}\right)^2 \\[0.5em]
        & = \frac{4087 n^2 -16736 k n+ 56677 k^2}{7350}
        \le \Xi(n,k),
    \end{align*}
    as desired. 
    This completes the proof of Proposition~\ref{PROP:Phi62-a1a2a5a6-zero-upper-bound}.
\end{proof}

\begin{proposition}\label{PROP:Phi3i-a6-zero}
    Let $(a_1,\ldots,a_6,b,c,d) \in \Omega_{n,k}$. 
    Suppose that $k \in \left[\frac{(19982-35 \sqrt{402})n}{108278}, \frac{n}{4}\right]$   and $a_6 = 0$. Then 
    \begin{align*}
        \max\left\{\Phi_{3,1}(a_1,\ldots,a_6,b,c,d),~\Phi_{3,2}(a_1,\ldots,a_6,b,c,d)\right\}
        \le \Xi(n,k). 
    \end{align*}
\end{proposition}
\begin{proof}[Proof of Proposition~\ref{PROP:Phi3i-a6-zero}]
    By Fact~\ref{FACT:Phi31-Phi32-vs-Phi6}, it suffices to show that 
    \begin{align*}
        \Phi_{6}(a_1,\ldots,a_6,b,c,d) 
        \le \Xi(n,k). 
    \end{align*}
    Since $a_6 = 0$, all terms involving $a_6$ in $\Phi_{6}(a_1,\ldots,a_6,b,c,d)$ vanish. 

    \begin{claim}\label{CLAIM:Phi3i-a6-zero-a1-a2a3a5}
        We may assume that $a_1 = 0$ or $a_2 = a_3 = a_5 = 0$.
    \end{claim}
    \begin{proof}[Proof of Claim~\ref{CLAIM:Phi3i-a6-zero-a1-a2a3a5}]
    Suppose that $a_1 > 0$ and $a_2 + a_3 + a_5 > 0$. 
    Notice that 
    \begin{align*}
        q_1(x)
        \coloneqq \Phi_{6}\left(a_1+(a_2 + a_3 + a_5)x, (1-x)a_2, (1-x)a_3, a_4, (1-x)a_5,a_6,b,c,d \right)
    \end{align*}
    is a quadratic polynomial in $x$, and straightforward calculations show that the coefficient of $x^2$ is 
    \begin{align*}
        \frac{1}{2}\left(a_2^2+a_3^2+2 a_5^2+2 \left(a_3+a_5\right) a_2+4 a_3 a_5\right) > 0.
    \end{align*}
    So it follows from Fact~\ref{FACT:convex-optimization} that 
    \begin{align*}
        \Phi_{6}(a_1,\ldots,a_6,b,c,d)
        = q_1(0)
        \le \max\left\{q_1\left(\frac{- a_1}{a_2+a_3+a_5}\right),~q_1(1)\right\},  
    \end{align*}
    meaning that we can assume either $a_1 = 0$ or $a_2 = a_3 = a_5 = 0$. 
    \end{proof}

\begin{claim}\label{CLAIM:Phi3i-a6-zero-a1-a4}
    Suppose that $a_2 = a_3 = a_5 = 0$.  Then we may assume that $a_1 = 0$ or $a_4 = 0$.
\end{claim}
\begin{proof}[Proof of Claim~\ref{CLAIM:Phi3i-a6-zero-a1-a4}]
    Let $\Sigma_{1,4} \coloneqq a_1 + a_4$ and consider 
    \begin{align*}
        q_{2}(x)
        & \coloneqq  \Phi_{6}(\Sigma_{1,4} - x,0,0,x,0,a_6,b,c,d) \\[0.5em]
        & = -\frac{13 x^2}{2} + (2a_6 - 9 b - 6 c)x + \hat{\varphi}(x, b,c,d) + A, 
    \end{align*}
    where $A$ is independent of $x$. 
    It follows from Proposition~\ref{PROP:a4-convex} that $q_2(x)$ is a convex function in $x$. So, by Fact~\ref{FACT:convex-optimization}, we have 
    \begin{align*}
        \Phi_{6}(a_1,0,0,a_4,0,a_6,b,c,d)
         = q_2(a_4)
        \le \max\left\{q_2(0),~q_{2}(\Sigma_{1,4})\right\}, 
    \end{align*}
    meaning that we can assume either $a_1 = 0$ or $a_4 = 0$. 
    \end{proof}
    
    Suppose that $a_2 = a_3 = a_5 = 0$ and $a_4 = 0$ (recall that we are also assuming $a_6 = 0$). Then 
    \begin{align*}
        \Phi_{6}(a_1,\ldots,a_6,b,c,d)
        & = \Phi_{6}(k,0,\ldots ,0,b,c,d) = \Phi_{1}(k,0,\ldots ,0,b,c,d)
        \le \Xi(n,k), 
    \end{align*}
    where the last inequality follows from Proposition~\ref{PROP:Phi1-inequality-A}.

    So it suffices to consider the case $a_1 = 0$. 

    \begin{claim}\label{CLAIM:Phi6-a1-zero-a2-a3a5}
        Suppose that $a_1 = 0$. Then we may assume that $a_2 = 0$ or $a_3 = a_5= 0$.
    \end{claim}
    \begin{proof}[Proof of Claim~\ref{CLAIM:Phi6-a1-zero-a2-a3a5}]
    Suppose that $a_2 > 0$ and $a_3 + a_5 > 0$. 
    Note that 
    \begin{align*}
        q_3(x)
        \coloneqq \Phi_{6}\left(0, a_2+(a_3 + a_5)x, (1-x)a_3, a_4, (1-x)a_5,a_6,b,c,d \right)
    \end{align*}
    is a quadratic polynomial in $x$, and straightforward calculations show that the coefficient of $x^2$ is 
    \begin{align*}
        \frac{1}{2}a_5 \left(2 a_3+a_5\right) \ge 0.
    \end{align*}
    So, it follows from Fact~\ref{FACT:convex-optimization} that 
    \begin{align*}
        \Phi_{6}(0,a_2,\ldots,a_6,b,c,d)
        = q_{3}(0)
        \le \max\left\{q_3\left(\frac{-a_2}{a_3+a_5}\right),~q_{3}(1)\right\},
    \end{align*}
    meaning that we can assume either $a_2 = 0$ or $a_3 = a_5 = 0$. 
    \end{proof}

    \begin{claim}\label{CLAIM:Phi6-a3a5-zero-a2-a4}
        Suppose that $a_3 = a_5 = 0$. Then we may assume that $a_2 = 0$ or $a_4 = 0$. 
    \end{claim}
    \begin{proof}[Proof of Claim~\ref{CLAIM:Phi6-a3a5-zero-a2-a4}]
    Suppose that $a_3 = a_5 = 0$ (recall that we are also assuming $a_1 = 0$). Then let $\Sigma_{2,4} \coloneqq a_2 + a_4$ and consider  
    \begin{align*}
        q_{4}(x)
        & \coloneqq  \Phi_{6}(0,\Sigma_{2,4} - x,0,x,0,a_6,b,c,d) \\[0.5em]
        & = - 7 x^2 + (a_6-8 b-6 c)x + \hat{\varphi}(x, b,c,d) + B, 
    \end{align*}
    where $B$ is independent of $x$. 
    It follows from Proposition~\ref{PROP:a4-convex} that $q_2(x)$ is a convex function in $x$. So, 
    \begin{align*}
        \Phi_{6}(0,a_2,0,a_4,0,a_6,b,c,d)
        = q_4(a_4)
        \le \max\left\{q_4(0),~q_4(\Sigma_{2,4})\right\}, 
    \end{align*}
     meaning that we can assume either $a_2 = 0$ or $a_4 = 0$.    
    \end{proof}
    
    Suppose that $a_1 = a_3 = a_4 = a_5 = 0$ (recall that we are also assuming $a_6 = 0$). Then 
    \begin{align*}
        \Phi_{6}(a_1,\ldots,a_6,b,c,d)
        & = \Phi_{6}(0,k,0,0,0,0,b,c,d) 
        = \Phi_{1}(0,k,0,0,0,0,b,c,d)
        \le \Xi(n,k), 
    \end{align*}
    where the last inequality follows from Proposition~\ref{PROP:Phi1-inequality-A}.

    So it remains to consider the case $a_1 = a_2 = 0$ (recall that we are also assuming $a_6 = 0$). 

    \medskip 

    \textbf{Case 1}: $d \le \frac{b+c}{2}$. 
    
    It follows from~\eqref{equ:Phi6-vs-Phi61-Phi62} and the assumption $d  \le \frac{b+c}{2}$ that  
    \begin{align*}
        \Phi_{6}(0,0,a_3,a_4,a_5,0,b,c,d) 
        & \le \Phi_{6,1}(0,0,a_3,a_4,a_5,0,b,c,d)\\[0.5em]
        & = \Phi_{6,1}(0,0,a_3,0,a_4+a_5,0,b,c,d) - \frac{a_4 (b+c-2 d)}{2} \\[0.5em]
        & \le \Phi_{6,1}(0,0,a_3,0,a_4+a_5,0,b,c,d)
        \le \Xi(n,k)
    \end{align*}
    where the last inequality follows from Proposition~\ref{PROP:Phi61-a1a2a4a6-zero-upper-bound}. 

    \medskip 

    \textbf{Case 2}: $d \in \left[\frac{b+c}{2}, \frac{9b+5c}{10}\right]$.

    It follows from~\eqref{equ:Phi6-vs-Phi61-Phi62} and the assumption $d \in \left[\frac{b+c}{2}, \frac{9b+5c}{10}\right]$ that 
    \begin{align*}
        \Phi_{6}(0,0,a_3,a_4,a_5,0,b,c,d) 
        & \le \Phi_{6,2}(0,0,a_3,a_4,a_5,0,b,c,d)\\[0.5em]
        & = \Phi_{6,2}(0,0,a_3,0,a_4+a_5,0,b,c,d) - \frac{a_4 (9b+5c-10d)}{10} \\[0.5em]
        & \le \Phi_{6,2}(0,0,a_3,0,a_4+a_5,0,b,c,d)
        \le \Xi(n,k)
    \end{align*}
    where the last inequality follows from Proposition~\ref{PROP:Phi62-a1a2a4a6-upper-bound}.

    \medskip 

    \textbf{Case 3}: $d \ge  \frac{9b+5c}{10}$.

    It follows from~\eqref{equ:Phi6-vs-Phi61-Phi62} and the assumption $d \ge \frac{9b+5c}{10} \ge \frac{b+c}{2}$ that 
    \begin{align*}
        \Phi_{6}(0,0,a_3,a_4,a_5,0,b,c,d)
        & \le \Phi_{6,2}(0,0,a_3,a_4,a_5,0,b,c,d)\\[0.5em]
        & = \Phi_{6,2}(0,0,a_3,a_4+a_5,0,0,b,c,d) - \frac{a_5 (10d - 9b-5c)}{10} \\[0.5em]
        & \le \Phi_{6,2}(0,0,a_3,a_4+a_5,0,0,b,c,d)
        \le \Xi(n,k)
    \end{align*}
    where the last inequality follows from Proposition~\ref{PROP:Phi62-a1a2a5a6-zero-upper-bound}.

    This completes the proof of Proposition~\ref{PROP:Phi3i-a6-zero}. 
\end{proof}

\begin{proposition}\label{PROP:Phi31-upper-bound}
    Suppose that $k \in \left[\frac{(19982-35 \sqrt{402})n}{108278}, \frac{n}{4}\right]$. 
    Then 
    \begin{align*}
        \max\left\{\Phi_{3,1}(a_1,\ldots,a_6,b,c,d) \colon (a_1,\ldots,a_6,b,c,d) \in \Omega_{n,k} \right\} 
        \le \Xi(n,k). 
    \end{align*}
\end{proposition}
\begin{proof}[Proof of Proposition~\ref{PROP:Phi31-upper-bound}]
    Fix $(a_1,\ldots,a_6,b,c,d) \in \Omega_{n,k}$. 
    Recall that 
    \begin{align*}
         \Phi_{3,1}(a_1,\ldots,a_6,b,c,d) 
        & = \Phi_{2,1}(a_1,\ldots,a_6,b,c,d)  + \varphi(a_4, b, c, d) \\[0.5em]
        & = \Phi_{1}(a_1,\ldots,a_6,b,c,d) - \frac{a_6^2}{2} +\varphi(a_4, b, c, d).
    \end{align*}
    Straightforward calculations show that 
    \begin{align*}
        \Phi_{3,1}(a_1,\ldots,a_6,b,c,d)
        = \Phi_{3,1}(a_1,a_2,a_3,a_4,a_5+a_6,0,b,c,d).
    \end{align*}
    So it follows from Proposition~\ref{PROP:Phi3i-a6-zero} that 
    \begin{align*}
        \Phi_{3,1}(a_1,\ldots,a_6,b,c,d) 
        \le \Xi(n,k), 
    \end{align*}
    completing the proof of Proposition~\ref{PROP:Phi31-upper-bound}. 
\end{proof}
\begin{proposition}\label{PROP:Phi32-upper-bound}
    Let $(a_1, \ldots, a_6, b,c,d) \in \Omega_{n,k}$. 
    Suppose $k \in \left[\frac{(19982-35 \sqrt{402})n}{108278}, \frac{n}{4}\right]$ and $a_6 \le \frac{n-4k}{2}$. 
    Then 
    \begin{align*}
        \Phi_{3,2}(a_1, \ldots, a_6, b,c,d) 
        \le \Xi(n,k). 
    \end{align*}
\end{proposition}
\begin{proof}[Proof of Proposition~\ref{PROP:Phi32-upper-bound}]
    Recall that 
    \begin{align*}
        \Phi_{3,2}(a_1,\ldots,a_6,b,c,d)
        & = \Phi_{2,2}(a_1,\ldots,a_6,b,c,d)  + \varphi(a_4, b, c, d) \\[0.5em]
        & = \Phi_{1}(a_1,\ldots,a_6,b,c,d) - \frac{(b+c)a_6}{2} + \varphi(a_4, b, c, d).
    \end{align*}
    Let $a_6^{\ast} \coloneqq \min\left\{\frac{n-4k}{2}, k-a_4\right\}$, noting that $a_6$ lies in the interval $[0, a_{6}^{\ast}]$.

    \begin{claim}\label{CLAIM:Phi32-a6}
        We may assume that $a_6 = 0$ or $a_6 = a_6^{\ast}$.
    \end{claim}
    \begin{proof}[Proof of Claim~\ref{CLAIM:Phi32-a6}]
        Suppose that $0 < a_6 < a_6^{\ast}$. 
        Then $a_1+a_2+a_3+a_5 >0$. 
        Notice that 
        \begin{align*}
            q_1(x)
            \coloneqq 
            \Phi_{3,2}((1-x)a_1,(1-x)a_2,(1-x)a_3,a_4,(1-x)a_5,a_6+(a_1 + a_2 + a_3 + a_5)x,b,c,d)
        \end{align*}
        is a quadratic polynomial in $x$, and straightforward calculations show that the coefficient of $x^2$ is 
        \begin{align*}
            \frac{3 a_1^2}{2}+\left(2 a_2+a_3+a_5\right) a_1+a_2^2+a_2 \left(a_3+a_5\right)+\frac{1}{2} a_5 \left(2 a_3+a_5\right) \ge 0. 
        \end{align*}
        Thus,
        \begin{align*}
            \Phi_{3,2}(a_1,\ldots,a_6,b,c,d)
            = q_1(0)
            \le \max\left\{q_1\left(\frac{-a_6}{a_1+a_2+a_3+a_5}\right),~q_1\left(\frac{a_6^{\ast}-a_6}{a_1+a_2+a_3+a_5}\right)\right\}, 
        \end{align*}
        meaning that we can assume either $a_6 = 0$ or $a_6 = a_6^{\ast}$. 
    \end{proof}
    
    By Proposition~\ref{PROP:Phi3i-a6-zero}, we are done if $a_6 = 0$. So it suffices to consider the case $a_6 = a_6^{\ast}$.

    \begin{claim}\label{CLAIM:Phi32-a6-2}
        We may assume that $a_6 = \frac{n-4k}{2}$. 
    \end{claim}
    \begin{proof}[Proof of Claim~\ref{CLAIM:Phi32-a6-2}]
        Suppose that $k-a_4 < \frac{n-4k}{2}$ (equivalently, $a_4 > k - \frac{n-4k}{2}$). Then $a_6 = a_6^{\ast} = \min\left\{\frac{n-4k}{2}, k-a_4\right\} = a_6^{\ast} = k-a_4$, meaning that $a_1 + a_2 + a_3 + a_5 = k-(a_4+a_6) = 0$. 
        Since 
        \begin{align*}
            q_2(x)
            \coloneqq 
            \Phi_{3,2}(0,0,0,x,0,k-x,b,c,d)
            = -7x^2 - \left(k+\frac{13b+9c}{2} - d\right)x + \hat{\varphi}(x,b,c,d) + A,  
        \end{align*}
        where $A$ is independent of $x$, it follows from Proposition~\ref{PROP:a4-convex} that $q_2(x)$ is a convex function in $x$. 
        Thus, by Fact~\ref{FACT:convex-optimization}, we have  
        \begin{align*}
            \Phi_{3,2}(0,0,0,a_4,0,a_6,b,c,d)
            = q_2(a_4)
            \le \max\left\{q_2(k),~q_2\left(k-\frac{n-4k}{2}\right)\right\},
        \end{align*}
        meaning that we can assume either $a_6 = 0$ or $a_6 = \frac{n-4k}{2}$. 
        
        By Proposition~\ref{PROP:Phi3i-a6-zero}, we are done if $a_6 = 0$. So we may assume that $a_6 = \frac{n-4k}{2}$. 
    \end{proof}
    Now assume that $a_6 = \frac{n-4k}{2}$. 
    It follows from~\eqref{equ:varphi-less-zero} that $\varphi \le 0$. Therefore, 
    \begin{align*}
        \Phi_{3,2}\left(a_1, \ldots, a_5, \frac{n-4k}{2},b,c,d\right)
        & \le \Phi_{2,2}\left(a_1, \ldots, a_5, \frac{n-4k}{2},b,c,d\right) 
        \le \Xi(n,k). 
    \end{align*}
    where the last inequality follows from Proposition~\ref{PROP:Phi22-a6-small}~\ref{PROP:Phi22-a6-small-2} and the assumption $k \ge \frac{(19982-35 \sqrt{402})n}{108278} > \frac{n}{6}$.  
    This completes the proof of Proposition~\ref{PROP:Phi32-upper-bound}.
\end{proof}

\begin{proposition}\label{PROP:Phi33-upper-bound}
    Let $(a_1, \ldots, a_6, b,c,d) \in \Omega_{n,k}$. 
    Suppose $k \in \left[\frac{n}{6}, \frac{n}{4}\right]$ and $a_6 \ge \frac{n-4k}{2}$. 
    Then 
    \begin{align*}
        \Phi_{3,3}(a_1, \ldots, a_6, b,c,d) 
        \le \Xi(n,k). 
    \end{align*}
\end{proposition}
\begin{proof}[Proof of Proposition~\ref{PROP:Phi33-upper-bound}]
    Since 
    \begin{align*}
        \Phi_{3,3}(a_1,\ldots,a_6,b,c,d)  
        & = \Phi_{2,3}(a_1,\ldots,a_6,b,c,d) + \varphi(a_4, b, c, d) \\[0.5em]
        & \le \Phi_{2,3}(a_1,\ldots,a_6,b,c,d), 
    \end{align*}
    Proposition~\ref{PROP:Phi33-upper-bound} is a direct consequence of Proposition~\ref{PROP:Phi23-a6-large}.
\end{proof}

\subsection{Proof for the main result}\label{SUBSEC:proof-main}
We present the proof of Theorem~\ref{THM:Mian-HS-density-K4} in this section. 
Recall that 
\begin{align*}
    (a_1, \ldots , a_6, b,c,d) 
    = (|\mathcal{A}_1|, \ldots , |\mathcal{A}_6|,|\mathcal{B}|,|\mathcal{C}|,|\mathcal{D}|). 
\end{align*}

Fix a sufficiently large integer $\mu > 0$ such that, in particular, Lemma~\ref{LEMMA:A6BCD-upper-bound-a} holds .  
Let 
\begin{align*}
    & \Phi_{2}(a_1, \ldots, a_6, b,c,d)  \\[0.5em]
    & \quad \coloneqq 
    \begin{cases}
        \Phi_{2,1}(a_1, \ldots, a_6, b,c,d), &\quad\text{if}\quad e(\mathcal{A}_6) \le \frac{15 a_6^2}{2} + \mu a_6 \quad\text{or}\quad a_6 \le b+c, \\[0.5em]
        \Phi_{2,2}(a_1, \ldots, a_6, b,c,d), &\quad\text{if}\quad e(\mathcal{A}_6) \ge \frac{15 a_6^2}{2}+ \mu a_6 \quad\text{and}\quad a_6 \in \left[b+c,\frac{n-4k}{2}\right], \\[0.5em] 
        \Phi_{2,3}(a_1, \ldots, a_6, b,c,d),
         &\quad\text{if}\quad e(\mathcal{A}_6) \ge \frac{15 a_6^2}{2}+ \mu a_6 \quad\text{and}\quad a_6 \ge \frac{n-4k}{2}, 
    \end{cases}
\end{align*}
and 
\begin{align*}
        \Phi_{3}(a_1, \ldots, a_6, b,c,d)
    & \coloneqq \Phi_{2}(a_1, \ldots, a_6, b,c,d) + \varphi(a_4, b, c, d) \\[0.5em]
    & =  \Phi_{2}(a_1, \ldots, a_6, b,c,d)  - \left(\frac{15 a_4^2}{2} + (7b+5c)a_4\right) \\[0.5em]
    & \quad + 
    \begin{cases}
        g_{1/2}(a_4, b) + h_{1/2}(a_4, c), &\quad\text{if}\quad d \le \frac{b+c}{2}, \\[0.5em]
        \hat{g}(a_4, b) + h_{1/4}(a_4, c), &\quad\text{if}\quad d \ge \frac{b+c}{2}.
    \end{cases}
\end{align*}
The following upper bounds for $|G|$ follow from the inequalities established in Sections~\ref{SEC:Local-estimate},~\ref{SEC:A3},~\ref{SEC:A4BC}, and~\ref{SEC:A6BCD}.
\begin{proposition}\label{PROP:G-upper-bound-Phi1-Phi2-Phi3}
    There exists a constant $C> 0$ such that the following inequalities hold. 
    \begin{enumerate}[label=(\roman*)]
        \item\label{PROP:G-upper-bound-Phi1-Phi2-Phi3-1} $|G| \le \Phi_{1}(a_1, \ldots, a_6, b,c,d) + Cn$.
        \item\label{PROP:G-upper-bound-Phi1-Phi2-Phi3-2} $|G| \le \Phi_{2}(a_1, \ldots, a_6, b,c,d) + Cn$.
        \item\label{PROP:G-upper-bound-Phi1-Phi2-Phi3-3} $|G| \le \Phi_{3}(a_1, \ldots, a_6, b,c,d) + Cn$.
    \end{enumerate}
\end{proposition}
\begin{proof}[Proof of Proposition~\ref{PROP:G-upper-bound-Phi1-Phi2-Phi3}]
It follows from Fact~\ref{FACT:edges-ABCD} that 
\begin{align}\label{equ:Psi-bcd-final}
    e(\mathcal{B} \cup \mathcal{C} \cup \mathcal{D})
    & \le 3b^2 + c^2 + 4bc + 2bd + cd
    = \Psi(b,c,d). 
\end{align}
It follows from Lemmas~\ref{LEMMA:Q1Qi},~\ref{LEMMA:Q2Qi},~\ref{LEMMA:Q3Qi}, and~\ref{LEMMA:A3-upper-bound} that 
\begin{align*}
    \sum_{i\in [3]}\sum_{j \in [i,6]}e(\mathcal{A}_i, \mathcal{A}_j)
    & \le \frac{13 a_1^2}{2} + 13 a_1 (a_2 + \cdots + a_6)  + 7 a_2^2 + 14 a_2(a_3 + \cdots + a_6)\\[0.5em]   
    &\quad+ 7 a_3^2 + 15a_3 (a_4+ a_5+a_6)+3a_2+5a_3+6(a_4+a_5+a_6).
\end{align*}
It follows from Lemmas~\ref{LEMMA:Q1BCD},~\ref{LEMMA:Q2BCD},~\ref{LEMMA:Q3BCD}, and~\ref{LEMMA:Q5Q6BCD} that 
\begin{align*}
    e(\mathcal{A}_1 \cup \mathcal{A}_2 \cup \mathcal{A}_3 \cup \mathcal{A}_5, \mathcal{B} \cup \mathcal{C} \cup \mathcal{D}) 
    & \le (9b + 6c + 3d)a_1 + (8b + 6c+ 3d)a_2 + 4a_2 \\[0.5em]
    & \quad + (8b + 5c + 3d)a_3 + 4a_3 + (7b+5c+2d)a_5 + 15 a_5.
\end{align*}
It follows from Lemmas~\ref{LEMMA:Q4Qi}~\ref{LEMMA:Q4Qi-2},~\ref{LEMMA:A5A6-edges}, and~\ref{LEMMA:Q4BCD}~\ref{LEMMA:Q4BCD-3}
\begin{align*}
    e(\mathcal{A}_4, \mathcal{A}_5 \cup \mathcal{A}_6) + e(\mathcal{A}_5) + e(\mathcal{A}_5, \mathcal{A}_6)  + e(\mathcal{A}_4, \mathcal{D})
    \le 15 a_4 (a_5+a_6) + \frac{15 a_5^2}{2} + 15a_5a_6 + 3a_4 d+a_5+a_6.
\end{align*}
It follows from Lemmas~\ref{LEMMA:Q4Qi}~\ref{LEMMA:Q4Qi-1} and~\ref{LEMMA:Q4BCD}~\ref{LEMMA:Q4BCD-1}~\ref{LEMMA:Q4BCD-2} that 
\begin{align}\label{equ:A4BC-triviall-upper-bound}
    e(\mathcal{A}_4) + e(\mathcal{A}_4, \mathcal{B} \cup \mathcal{C}) 
    \le \frac{15 a_4^2}{2} + (7b+5c)a_4 + 13 a_4.
\end{align}
It follows from the trivial upper bound $e(\mathcal{A}_6) \le \binom{4a_6}{2} \le 8 a_6^2$ and Lemma~\ref{LEMMA:Q5Q6BCD} that 
\begin{align}\label{equ:A6BCD-trivial-upper-bound}
    e(\mathcal{A}_6) + e(\mathcal{A}_6, \mathcal{B} \cup \mathcal{C} \cup \mathcal{D})
    \le 8 a_6^2 + (7b+5c+2d)a_6 + 15 a_6.
\end{align}
Summing these inequalities, we obtain 
\begin{align*}
    |G|
    & \le \Phi_{1}(a_1, \ldots, a_6, b,c,d) + 7a_2+9a_3+19a_4+22a_5+22a_6 
    \le \Phi_{1}(a_1, \ldots, a_6, b,c,d) + 20n,
\end{align*}
which proves  Proposition~\ref{PROP:G-upper-bound-Phi1-Phi2-Phi3}~\ref{PROP:G-upper-bound-Phi1-Phi2-Phi3-1}. 

Proposition~\ref{PROP:G-upper-bound-Phi1-Phi2-Phi3}~\ref{PROP:G-upper-bound-Phi1-Phi2-Phi3-2} follows by substituting Inequalities~\eqref{equ:Psi-bcd-final} and~\eqref{equ:A6BCD-trivial-upper-bound} with the corresponding upper bounds from Lemmas~\ref{LEMMA:A6BCD-upper-bound-b} and~\ref{LEMMA:A6BCD-upper-bound-a}. 

Proposition~\ref{PROP:G-upper-bound-Phi1-Phi2-Phi3}~\ref{PROP:G-upper-bound-Phi1-Phi2-Phi3-3} follows by substituting Inequalities~\eqref{equ:Psi-bcd-final} and~\eqref{equ:A6BCD-trivial-upper-bound} with the corresponding upper bounds from Lemmas~\ref{LEMMA:A6BCD-upper-bound-b} and~\ref{LEMMA:A6BCD-upper-bound-a}, and by substituting Inequality~\eqref{equ:A4BC-triviall-upper-bound} with the corresponding upper bounds from Lemma~\ref{LEMMA:A4BC-together}. 
\end{proof}

We are now ready to present the proof of Theorem~\ref{THM:Mian-HS-density-K4}. 
Whenever $a_6\ge(n-4k)/2$, the bound $a_6\le k$ implies that $k\ge n/6$.

\begin{proof}[Proof of Theorem~\ref{THM:Mian-HS-density-K4}]
    Let us consider the following three cases. 

    \medskip

    \textbf{Case 1}: $k \in \left[0, \frac{n}{8}\right]$.

    This case follows directly from Proposition~\ref{PROP:G-upper-bound-Phi1-Phi2-Phi3}~\ref{PROP:G-upper-bound-Phi1-Phi2-Phi3-1} and~\eqref{equ:Phi1-k-less-125n}.

    \medskip

    \textbf{Case 2}: $k \in \left[\frac{n}{8}, \frac{(20+\sqrt{10}) n}{130}\right] \cup \left[\frac{n}{5}, \frac{n}{4}\right]$. 

    By Proposition~\ref{PROP:G-upper-bound-Phi1-Phi2-Phi3}~\ref{PROP:G-upper-bound-Phi1-Phi2-Phi3-2}, it suffices to show that 
    \begin{align*}
        \Phi_{2}(a_1, \ldots, a_6, b,c,d)
        \le \Xi(n,k).
    \end{align*}

    Suppose that $e(\mathcal{A}_6) \le \frac{15 a_6^2}{2} + \mu a_6$ or $a_6 \le b+c$.
    Then, it follows from the definition of $\Phi_{2}$ and Proposition~\ref{PROP:Phi21-a6-sparse} that 
    \begin{align*}
        \Phi_{2}(a_1, \ldots, a_6, b,c,d)
        & = \Phi_{2,1}(a_1, \ldots, a_6, b,c,d) 
        \le \Xi(n,k). 
    \end{align*}

    Suppose that $e(\mathcal{A}_6) \ge \frac{15 a_6^2}{2} + \mu a_6$ and $a_6 \in \left[b+c,\frac{3b+2c+d}{2}\right] = \left[b+c,\frac{n-4k}{2}\right]$.
    Then, it follows from the definition of $\Phi_{2}$ and Proposition~\ref{PROP:Phi22-upper-bound} that 
    \begin{align*}
        \Phi_{2}(a_1, \ldots, a_6, b,c,d)
        & = \Phi_{2,2}(a_1, \ldots, a_6, b,c,d) 
        \le \Xi(n,k). 
    \end{align*}

    Suppose that $e(\mathcal{A}_6) \ge \frac{15 a_6^2}{2} + \mu a_6$ and $a_6 \ge \frac{3b+2c+d}{2} = \frac{n-4k}{2}$.
    Then, it follows from the definition of $\Phi_{2}$ and Proposition~\ref{PROP:Phi23-a6-large} that 
    \begin{align*}
        \Phi_{2}(a_1, \ldots, a_6, b,c,d)
        & = \Phi_{2,3}(a_1, \ldots, a_6, b,c,d) 
        \le \Xi(n,k). 
    \end{align*}

    \medskip

    \textbf{Case 3}: $k \in \left[\frac{(20+\sqrt{10}) n}{130}, \frac{n}{5}\right]$. 

    By Proposition~\ref{PROP:G-upper-bound-Phi1-Phi2-Phi3}~\ref{PROP:G-upper-bound-Phi1-Phi2-Phi3-3}, it suffices to show that 
    \begin{align*}
        \Phi_{3}(a_1, \ldots, a_6, b,c,d)
        \le \Xi(n,k).
    \end{align*}
    
    Suppose that $e(\mathcal{A}_6) \le \frac{15 a_6^2}{2} + \mu a_6$ or $a_6 \le b+c$.
    Then, it follows from the definition of $\Phi_{3}$ and Proposition~\ref{PROP:Phi31-upper-bound} (note that $\frac{(19982-35 \sqrt{402})n}{108278} \le \frac{(20+\sqrt{10}) n}{130}$) that
    \begin{align*}
        \Phi_{3}(a_1, \ldots, a_6, b,c,d)
        & = \Phi_{3,1}(a_1, \ldots, a_6, b,c,d) 
        \le \Xi(n,k). 
    \end{align*}

    Suppose that $e(\mathcal{A}_6) \ge \frac{15 a_6^2}{2} + \mu a_6$ and $a_6 \in \left[b+c,\frac{3b+2c+d}{2}\right] = \left[b+c,\frac{n-4k}{2}\right]$.
    Then, it follows from the definition of $\Phi_{3}$ and Proposition~\ref{PROP:Phi32-upper-bound} that 
    \begin{align*}
        \Phi_{3}(a_1, \ldots, a_6, b,c,d)
        & = \Phi_{3,2}(a_1, \ldots, a_6, b,c,d) 
        \le \Xi(n,k). 
    \end{align*}

    Suppose that $e(\mathcal{A}_6) \ge \frac{15 a_6^2}{2} + \mu a_6$ and $a_6 \ge \frac{3b+2c+d}{2} = \frac{n-4k}{2}$.
    Then, it follows from the definition of $\Phi_{3}$ and Proposition~\ref{PROP:Phi33-upper-bound} that 
    \begin{align*}
        \Phi_{3}(a_1, \ldots, a_6, b,c,d)
        & = \Phi_{3,3}(a_1, \ldots, a_6, b,c,d) 
        \le \Xi(n,k). 
    \end{align*}

    This completes the proof of Theorem~\ref{THM:Mian-HS-density-K4}. 
\end{proof}

\section{Concluding remarks}\label{SEC:Remarks}
$\bullet$ Below, we describe the candidate set for the extremal constructions when $r \ge 5$. For convenience, we slightly abuse notation by referring to a family, such as $E_{i}(n,k,r)$, as a typical member in this class.
All constructions defined below are assumed, by default, to have exactly $n$ vertices.

\begin{enumerate}[label=(\roman*)]
    \item The vertex set of $E_{1}(n,k,r)$ consists of $r$ parts $X,Y_1,\ldots,Y_{r-1}$, with sizes satisfying 
    \begin{align*}
        |X| = k 
        \quad\text{and}\quad 
        |Y_1| \le \cdots \le |Y_{r-1}| \le |Y_1| + 1.
    \end{align*}
    The edge set of $E_{1}(n,k,r)$ is defined as 
    \begin{align*}
        E_1(n,k) 
        \coloneqq \binom{X}{2} \cup K[X,Y_1,\ldots,Y_{r-1}].
    \end{align*}
    \item Suppose that $i \in [2,r-1]$. 
    The vertex set of $E_{i}^{a}(n,k,r)$ consists of $r+2-i$ parts $X,Y_1,\ldots,Y_{r+1-i}$ with sizes satisfying 
    \begin{align*}
        |X| = ik +i -1
        \quad\text{and}\quad 
        |Y_1| \le \cdots \le |Y_{r-i}| \le |X\cup Y_{r-i+1}| \le |Y_1|+1.
    \end{align*}
    The edge set of $E_{i}^{a}(n,k,r)$ is defined as 
    \begin{align*}
        E_{i}^{a}(n,k,r) 
        = \binom{X}{2} \cup K[Y_1, \ldots, Y_{r-i}, X \cup Y_{r-i+1}].
    \end{align*}
    Note that $E_{i}^{a}(n,k,r)$ is defined only for $k \leq \frac{1}{i}\left(\left\lceil \frac{n}{r+1-i} \right\rceil -i +1\right)$.
    
    The vertex set of $E_{i}^{b}(n,k,r)$ consists of $r+1-i$ parts $X, Y_1,\ldots,Y_{r-i}$ with sizes satisfying
    \begin{align*}
        |X| = ik+i-1 
        \quad\text{and}\quad
        |Y_1| \le \cdots \le |Y_{r-i}| \le |Y_1| + 1.
    \end{align*}  
    The edge set of $E_{i}^{b}(n,k,r)$ is defined as 
    \begin{align*}
        E_{i}^{b}(n,k,r) 
        = \binom{X}{2} \cup K[X,Y_1,\ldots,Y_{r-i}].
    \end{align*}
    \item The vertex set of $E_{r}(n,k,r)$ consists of $2r-1$ parts $X,Y_1,\ldots,Y_{r-1},Z_1,\ldots,Z_{r-1}$ with sizes satisfying
    \begin{align*}
        & |X| = rk+r-1-(r-2)(n-rk-r+1)-j, \\[0.5em]
        & \sum_{i=1}^{r-1}|Y_i| = (r-2)(n-rk-r+1)+j, \quad\text{and}\quad \\[0.5em]
        & |Y_1 \cup Z_1| \le \cdots \le |Y_{r-1} \cup Z_{r-1}| \le |Y_1 \cup Z_1| + 1, 
    \end{align*}
    where $j \in [0, r-1]$ is an arbitrary integer. 
    The edge set of $E_{r}(n,k,r)$ is defined as 
    \begin{align*}
        E_{r}(n,k,r) 
        = \binom{X}{2} \cup K[X,Y_1\cup \cdots \cup Y_{r-1}] \cup K[Y_1 \cup Z_1, \ldots ,Y_{r-1} \cup Z_{r-1}].
    \end{align*}
    Note that $E_{r}(n,k,r)$ is defined only for $k \ge \frac{(r-2)n}{r(r-1)} + \frac{1-r}{r}$.
\end{enumerate}
It is easy to verify that all the graphs defined above have a $K_r$-matching number of $k$. We conjecture that $\mathrm{ex}(n,(k+1)K_{r})$ is asymptotically attained by these constructions. 
\begin{conjecture}\label{CONJ:HS-Kr}
    Let $r \ge 5$ and $0\le k \le \lfloor\frac{n}{r} \rfloor -1$ be integers. 
    Then 
    \begin{align*}
        \mathrm{ex}(n,(k+1)K_{r})
        & = \max\big\{|E_{1}(n,k,r)|,~|E_{2}^{a}(n,k,r)|,~|E_{2}^{b}(n,k,r)|, \cdots, \\[0.5em]
        & \qquad\qquad |E_{r-1}^{a}(n,k,r)|,~|E_{r-1}^{b}(n,k,r)|,~|E_{r}(n,k,r)| \big\} + O(n).
    \end{align*}
\end{conjecture}
\textbf{Remark}. 
Elementary calculations show that $|E_{i}^{a}(n,k,r)| \le |E_{2}^{a}(n,k,r)|$ holds for all $i \in [3,r-1]$ and for all $(n,k,r)$ for which $E_{i}^{a}(n,k,r)$ is defined. 
Indeed, put $q=k+1$. The two largest parts of the complete multipartite graph in $E_i^a(n,k,r)$ have sizes at least $iq-1$ and $iq-2$.
Repartitioning their union into $i$ balanced parts and then balancing all $r-1$ parts gives
\begin{align*}
|E_2^a(n,k,r)|-|E_i^a(n,k,r)|
&\ge\binom{iq-2}{2}-2\binom{2q-1}{2}
    -(i-3)\binom{2q}{2}\\
&=\binom{(i-2)q-1}{2}\ge0.
\end{align*}
Consequently, there are only $r+1$ classes of extremal constructions in the conjecture above. Nevertheless, we include $E_{i}^{a}(n,k,r)$ in case they may assist in the proof.

\medskip 
For $r=5$, extending our approach would require a suitable partition of a $K_5$-tiling and bounds on the numbers of edges within and between the parts.
The rotations in Sections~\ref{SEC:Local-estimate}--\ref{SEC:A6BCD} would need to be adapted to copies of $K_5$, and the resulting bounds would need to be combined and optimized as in Section~\ref{SEC:Global-estimate}.

$\bullet$ Recently Lo--Williams~\cite{LW24} considered the following extension of the Corr\'{a}di--Hajnal Theorem in edge-colored graphs:  
Given an edge-colored graph $G$ on $n$ vertices, let $\delta^{c}(G)$ denote the largest integer $m$ such that, for every vertex $v \in V (G)$, there are at least $m$ distinct colors on edges incident to $v$. 
Lo--Williams~\cite{LW24} obtained several bounds on  $\delta^{c}(G)$ to ensure that $G$ contains a perfect rainbow-$K_r$-tiling (we refer the reader to~\cite{LW24} for the definition). In particular, they proved that for $r \ge 5$, 
\begin{align*}
    \delta^{c}(G) 
    \geq \left(1-\frac{1}{(2r-3)r} +o(1)\right)n 
\end{align*}
ensures the existence of a perfect rainbow-$K_r$-tiling in $G$. 
Unfortunately, this bound is somewhat distant from their conjectured bound $\left(1-\frac{r-1}{r^2-2} \right)n$, and they remarked:
\begin{displayquote}
    For $r \geq 5$, we are unable to obtain any ``reasonable'' upper bound for the existence of a rainbow-$K_r$-tiling. One reason is that we are unaware of a density version of the Hajnal--Szemer\'{e}di theorem, that is, the Tur\'{a}n number for $t$ vertex-disjoint $K_{r-1}$.
\end{displayquote}
Using the density version of the Corr\'{a}di--Hajnal Theorem (i.e. Theorem~\ref{THM:ABHP-density-CH}), Lo--Williams improved\footnote{Their proof seems to give the better bound $\delta^{c}(G) \geq \left(\frac{3 + \sqrt{6}}{6} + o(1)\right)n \approx (0.908248+o(1)) n$.} $\delta^{c}(G)$ to $\left(\frac{4+\sqrt{26}}{10} +o(1)\right)n \approx (0.909902 + o(1))n$ for $r =4$. 
Similarly, using Theorem~\ref{THM:Mian-HS-density-K4}, we can improve the bound $\delta^{c}(G) \geq \left(\frac{34}{35} + o(1)\right)n \approx (0.971428 + o(1))n$ of Lo--Williams to $\delta^{c}(G) \geq \left(\frac{8+\sqrt{37}}{15} +o(1)\right) n \approx \left(0.93885 + o(1)\right)n$ for $r = 5$.

\section*{Acknowledgments}
We would like to thank the referees for their valuable suggestions and insightful comments.

J.H. was supported by National Key R\&D Program of China (Grant No. 2023YFA1010202), National Natural Science Foundation of China (Grant No. 12071077), the Central Guidance on Local Science and Technology Development Fund of Fujian Province (Grant No. 2023L3003).
X.L. was supported by the Excellent Young Talents Program (Overseas) of the National Natural Science Foundation of China and the ERC Advanced Grant 101020255.

\bibliographystyle{abbrv}
\bibliography{CorradiHajnal}
\newpage
\begin{appendix}
\section{Calculations}
For every real number $\gamma \ge 0$ and integer $d \ge 0$, let 
\begin{align*}
    \mathbb{S}^{d}(\gamma)
    \coloneqq \left\{(x_1, \ldots, x_{d+1})\in \mathbb{R}^{d+1} \colon \sum_{i\in [d+1]}x_i \le \gamma~\mathrm{and}~x_i \ge 0~\mathrm{for}~i \in [d+1]\right\}. 
\end{align*}
For real numbers $b,x_2, \ldots, x_{10}$, let 
\begin{align*}
    \eta(b,x_2,\ldots,x_{10})
    \coloneqq 
    & - \frac{3 x_2^2}{112} - \frac{3 x_3^2}{56} - \frac{3 x_4^2}{32} - \frac{x_5^2}{8} - \frac{3}{8} \left(x_6 + \cdots + x_{10}\right)^2  \\[0.5em]
    & - \frac{3}{4} \left((x_2 + \cdots + x_5)x_{10} + (x_3 + x_4 + x_5)x_{9} + (x_4 + x_5)x_8 + x_5 x_7\right) \\[0.5em]
    & + \sum_{i=2}^{10} \frac{(i-1) b x_i}{10}.
\end{align*}
We prove the following upper bound for $\eta(b,x_2,\ldots,x_{10})$. All Mathematica calculations for this proposition are available at \url{https://xliu2022.github.io/DensityHSK4.nb} and also as an ancillary file in the arXiv version of this paper.
\begin{proposition}\label{PROP:inequality-psi-upper-bound}
    Let $\gamma, b \ge 0$ be real numbers. 
    We have 
    \begin{align*}
        \max_{(x_2,\ldots,x_{10}) \in \mathbb{S}^{8}(\gamma)} \eta(b,x_2,\ldots,x_{10})
        \le 
        \begin{cases}
          -\frac{3 \gamma^2}{8}+\frac{9 b \gamma}{10}, &\quad\text{if}\quad \gamma \in \left[0, \frac{6 b}{5}\right],  \\[0.5em]
          \frac{27 b^2}{50}, &\quad\text{if}\quad \gamma \in \left[\frac{6 b}{5}, \frac{(86-4 \sqrt{210})b}{15} \right], \\[0.5em]
          -\frac{\gamma^2}{40}+\frac{43 b \gamma}{150}+\frac{103 b^2}{1125}, &\quad\text{if}\quad \gamma \in \left[ \frac{(86-4 \sqrt{210})b}{15}, \frac{56 b}{15} \right], \\[0.5em]
          -\frac{3 \gamma^2}{232}+ \frac{57 b \gamma}{290}+\frac{113 b^2}{435}, &\quad\text{if}\quad \gamma \in \left[ \frac{56 b}{15}, \frac{38 b}{5} \right], \\[0.5em]
            \frac{151 b^2}{150}, &\quad\text{if}\quad \gamma \geq \frac{38 b}{5}.
        \end{cases}
    \end{align*}
\end{proposition}
\begin{proof}[Proof of Proposition~\ref{PROP:inequality-psi-upper-bound}]
    Fix $\gamma, b \ge 0$.
    Choose $(x_2,\ldots,x_{10}) \in \mathbb{S}^{8}(\gamma)$ such that 
    \begin{align*}
        \eta(b,x_2,\ldots,x_{10})
        = \eta_{\ast}
        \coloneqq \max_{(z_2,\ldots,z_{10}) \in \mathbb{S}^{8}(\gamma)} \eta(b,z_2,\ldots,z_{10}).
    \end{align*}
    \begin{claim}\label{CLAIM:psi-shift-10}
        We may assume that $x_2+ \cdots + x_9 = 0$ or $x_{10} = 0$. 
    \end{claim}
    \begin{proof}[Proof of Claim~\ref{CLAIM:psi-shift-10}]
        Suppose that $x_2+ \cdots + x_9 \neq 0$ and $x_{10} \neq 0$. 
        Let 
        \begin{align*}
            q_{1}(y) 
            \coloneqq \eta\left(b, (1-y)x_2,\ldots, (1-y)x_{9},x_{10}+(x_2+ \cdots + x_{9})y \right).
        \end{align*}
        Notice that $q_{1}(y)$ is a quadratic polynomial in $y$, and straightforward calculations show that the coefficient of  $y^2$ in $q_{1}(y)$ is 
        \begin{align*}
            & \frac{39 x_2^2}{112}+\frac{9 x_3^2}{28}+\frac{9 x_4^2}{32}+\frac{x_5^2}{4} \\[0.5em]
            & + \frac{3}{4} \left(x_2(x_3 + \cdots + x_9) 
            + x_3(x_4+ \cdots + x_8)
            + x_4(x_5 + x_6 + x_7)
            + x_5 x_6 \right) \ge 0.
        \end{align*}
        So it follows from Fact~\ref{FACT:convex-optimization} that 
        \begin{align*}
            \eta(b,x_2,\ldots,x_{10})
            = q_{1}(0)
            \le \max\left\{q_{1}\left(-\frac{x_{10}}{x_2 + \cdots + x_9}\right),q_{1}(1)\right\}, 
        \end{align*}
        which means that we may assume that $x_2+ \cdots + x_9 = 0$ or $x_{10} = 0$. 
    \end{proof}
    \begin{claim}\label{CLAIM:psi-shift-2-9}
        Suppose that $x_{10} = 0$. Then we may assume that $x_3+ \cdots + x_8 = 0$ or $x_{9} = 0$. 
    \end{claim}
    \begin{proof}[Proof of Claim~\ref{CLAIM:psi-shift-2-9}]
        Suppose that $x_3+ \cdots + x_8 \neq 0$ and $x_{9} \neq 0$. 
        Similar to the proof of Claim~\ref{CLAIM:psi-shift-10}, let 
        \begin{align*}
            q_{2}(y)
            \coloneqq \eta\left(b, x_2,(1-y)x_{3},  \ldots, (1-y)x_{8}, x_9+ (x_3+ \cdots + x_{8})y, 0\right).
        \end{align*}
        Notice that $q_{2}(y)$ is a quadratic polynomial in $y$, and straightforward calculations show that the coefficient of $y^2$ in $q_{2}(y)$ is
        \begin{align*}
            \frac{9 x_3^2}{28} + \frac{9 x_4^2}{32} + \frac{x_5^2}{4} + \frac{3}{4}\left(x_3 (x_4 + \cdots + x_8) + x_4(x_5+x_6+x_7) + x_5 x_6\right)
            \ge 0. 
        \end{align*}
        Therefore, we have 
        \begin{align*}
            \eta(b,x_2,\ldots,x_{10})
            = q_{2}(0)
            \le \max\left\{q_{2}\left(-\frac{x_{9}}{x_3 + \cdots + x_8}\right),q_{2}(1)\right\}, 
        \end{align*}
        meaning that we can assume that $x_3+ \cdots + x_8 = 0$ or $x_{9} = 0$. 
    \end{proof}
    Similar proofs to those in Claims~\ref{CLAIM:psi-shift-10} and~\ref{CLAIM:psi-shift-2-9} yield the following results. 
    \begin{claim}\label{CLAIM:psi-shift-2-3-8}
        Suppose that $x_9 = x_{10} = 0$. Then we may assume that $x_4+ x_5 + x_6 + x_7 = 0$ or $x_{8} = 0$. 
    \end{claim}
    \begin{claim}\label{CLAIM:psi-shift-2-3-4-7}
        Suppose that $x_8 = x_9 = x_{10} = 0$. Then we may assume that $x_5 + x_6 = 0$ or $x_{7} = 0$. 
    \end{claim}
    It follows from Claims~\ref{CLAIM:psi-shift-10},~\ref{CLAIM:psi-shift-2-9},~\ref{CLAIM:psi-shift-2-3-8}, and~\ref{CLAIM:psi-shift-2-3-4-7} that 
    \begin{align}\label{equ:psi-max-five-vectors}
        \eta_{\ast}
        & =\max\left\{\eta(b,0,\ldots, 0, x_{10}),~\eta(b,x_2, 0,\ldots, 0, x_{9},0),~\eta(b,x_2,x_3, 0,\ldots, 0, x_{8},0,0), \right. \notag \\[0.5em]
        & \qquad\qquad \left. \eta(b,x_2,x_3,x_4, 0, 0, x_{7},0,0,0),~\eta(b,x_2,x_3,x_4,x_5,x_6, 0,\ldots, 0)\right\}. 
    \end{align}
    Next, we consider respectively the upper bounds for $\eta(b,0,\ldots, 0, x_{10})$, $\eta(b,x_2, 0,\ldots, 0, x_{9},0)$, $\eta(b,x_2,x_3, 0,\ldots, 0, x_{8},0,0)$, $\eta(b,x_2,x_3,x_4, 0, 0, x_{7},0,0,0)$, and $\eta(b,x_2,x_3,x_4,x_5,x_6, 0,\ldots, 0)$.
    
    For $\eta(b,0,\ldots, 0, x_{10})$, we have 
    \begin{align}\label{equ:psi-max-five-vectors-x10}
        \eta(b,0,\ldots, 0, x_{10})
        = -\frac{3}{8}\left(x_{10} - \frac{6b}{5}\right)^2 + \frac{27 b^2}{50}
        \le 
        \begin{cases}
            - \frac{3 \gamma^2}{8} + \frac{9 b \gamma}{10}, &\quad\text{if}\quad \gamma \in \left[0, \frac{6b}{5}\right], \\[0.5em]
            \frac{27 b^2}{50}, &\quad\text{if}\quad \gamma > \frac{6b}{5}. 
        \end{cases}
    \end{align}
    For $\eta(b,x_2, 0,\ldots, 0, x_{9},0)$, we have 
    \begin{align}\label{equ:psi-max-five-vectors-x2-x9}
        \eta(b,x_2, 0,\ldots, 0, x_{9},0)
        &= -\frac{3}{112} \left(x_2 - \frac{28 b}{15} \right)^2 -\frac{3}{8}\left(x_9-\frac{16 b}{15}\right)^2 +\frac{13 b^2}{25} \notag \\[0.5em]
        &\leq 
        \begin{cases}
            -\frac{3 \gamma^2}{8} + \frac{4 b \gamma}{5}, &\quad\text{if}\quad \gamma \in \left[0, \frac{6b}{5}\right], \\[0.5em]
            -\frac{\gamma^2}{40}+\frac{11 b \gamma}{75}+ \frac{343 b^2}{1125}, &\quad\text{if}\quad \gamma \in \left[\frac{14 b}{15}, \frac{44 b}{15} \right], \\[0.5em]
            \frac{13 b^2}{25}, &\quad\text{if}\quad \gamma \geq \frac{44 b}{15}.
        \end{cases}
    \end{align}
    Here, the last inequality is derived using Mathematica but can also be proven without a computer, as demonstrated in Proposition~\ref{PROP:max-two-var-poiynomial}. 

    For $\eta(b,x_2,x_3, 0,\ldots, 0, x_{8},0,0)$, we have 
    \begin{align}\label{equ:psi-max-five-vectors-x2-x3-x8}
        & \eta(b,x_2,x_3, 0,\ldots, 0, x_{8},0,0) \notag \\[0.5em]
        & \quad = -\frac{3}{112} \left(x_2 - \frac{28}{15}b \right)^2 - \frac{3}{56}\left(x_3 - \frac{28}{15}b \right)^2 -\frac{3}{8}\left(x_8-\frac{14}{15}b \right)^2   + \frac{91 b^2}{150} \notag \\[0.5em]
        & \quad \le 
        \begin{cases}
            -\frac{3 \gamma^2}{8}+\frac{7 b \gamma}{10}, &\quad\text{if}\quad \gamma \in \left[0, \frac{6b}{5}\right], \\[0.5em]
            -\frac{3 \gamma^2}{64}+\frac{21 b \gamma}{80} + \frac{7 b^2}{48}, &\quad\text{if}\quad \gamma \in \left[\frac{2 b}{3}, \frac{26 b}{15} \right], \\[0.5em]
            -\frac{3 \gamma^2}{176}+\frac{7 b \gamma}{44} +\frac{259 b^2}{1100}, &\quad\text{if}\quad \gamma \in \left[\frac{26 b}{15}, \frac{14 b}{3} \right], \\[0.5em]
            \frac{91 b^2}{150}, &\quad\text{if}\quad \gamma \geq \frac{14 b}{3}.
        \end{cases}
    \end{align}
    Similarly, the inequality above is derived using Mathematica; but can also be proven using a similar argument as in Proposition~\ref{PROP:max-two-var-poiynomial}. We omit the tedious calculations here. 

    For $\eta(b,x_2,x_3,x_4, 0, 0, x_{7},0,0,0)$, we have 
    \begin{align}\label{equ:psi-max-five-vectors-x2-x3--x4-x7}
        & \eta(b,x_2,x_3,x_4, 0, 0, x_{7},0,0,0) \notag \\[0.5em]
        &\quad = -\frac{3}{112}\left(x_2 - \frac{28 b}{15}\right)^2 - \frac{3}{56}\left(x_3 - \frac{28 b}{15}\right)^2 - \frac{3}{32}\left(x_4 - \frac{8 b}{5}\right)^2 - \frac{3}{8}\left(x_7 - \frac{4 b}{5}\right)^2 + \frac{19 b^2}{25} \notag \\[0.5em]
        &\quad \leq 
        \begin{cases}
           -\frac{3 \gamma^2}{8} + \frac{3 b \gamma}{5}, &\quad\text{if}\quad \gamma \in \left[0, \frac{6b}{5}\right], \\[0.5em]
           -\frac{3 \gamma^2}{40} +\frac{9 b \gamma}{25}+ \frac{6 b^2}{125}, &\quad\text{if}\quad  \gamma \in \left[\frac{2 b}{5}, \frac{16 b}{15} \right], \\[0.5em]
           -\frac{\gamma^2}{32} + \frac{4 b \gamma}{15} + \frac{22 b^2}{225}, &\quad\text{if}\quad \gamma \in \left[\frac{16 b}{15}, \frac{8 b}{3}b \right],  \\[0.5em]
           -\frac{3 \gamma^2}{208} + \frac{23 b \gamma}{130} +\frac{212 b^2}{975}, &\quad\text{if}\quad \gamma \in \left[\frac{8 b}{3}, \frac{92 b}{15} \right], \\[0.5em]
           \frac{19 b^2}{25}, &\quad\text{if}\quad \gamma \geq \frac{92 b}{15}. 
        \end{cases}
    \end{align}
    Similarly, the inequality above is derived using Mathematica; but can also be proven using a similar argument as in Proposition~\ref{PROP:max-two-var-poiynomial}. We omit the tedious calculations here.

    For $\eta(b,x_2,x_3,x_4,x_5,x_6, 0,\ldots, 0)$, we have 
    \begin{align}\label{equ:psi-max-five-vectors-x2-x3--x4-x5-x6}
         \eta(b,x_2,x_3,x_4,x_5,x_6, 0,\ldots, 0)
        &  = -\frac{3}{112}\left(x_2 - \frac{28 b}{15}\right)^2 - \frac{3}{56}\left(x_3 - \frac{28 b}{15}\right)^2 \notag \\[0.5em]
        & \qquad - \frac{3}{32}\left(x_4 - \frac{8 b}{5}\right)^2  - \frac{1}{8}\left(x_5 - \frac{8 b}{5}\right)^2
        - \frac{3}{8}\left(x_6 - \frac{2 b}{3}\right)^2 + \frac{151 b^2}{150} \notag \\[0.5em]
        &\quad \le 
        \begin{cases}
            -\frac{3 \gamma^2}{8} + \frac{b \gamma}{2}, &\quad\text{if}\quad \gamma \in \left[0, \frac{6b}{5}\right], \\[0.5em]
            -\frac{3 \gamma^2}{32}+\frac{17 b \gamma}{40}+\frac{b^2}{200}, &\quad\text{if}\quad \gamma \in \left[\frac{2b}{15}, \frac{2 b}{3}\right], \\[0.5em]
            -\frac{3 \gamma^2}{64}+\frac{29 b \gamma}{80}+\frac{31 b^2}{1200}, &\quad\text{if}\quad \gamma \in \left[\frac{2 b}{3}, \frac{26 b}{15}\right], \\[0.5em]
            -\frac{\gamma^2}{40}+\frac{43 b \gamma}{150}+\frac{103 b^2}{1125}, &\quad\text{if}\quad \gamma \in \left[\frac{26 b}{15}, \frac{56 b}{15}\right], \\[0.5em]
            -\frac{3 \gamma^2}{232}+ \frac{57 b \gamma}{290}+\frac{113 b^2}{435}, &\quad\text{if}\quad \gamma \in \left[\frac{56 b}{15}, \frac{38 b}{5}\right], \\[0.5em]
            \frac{151 b^2}{150}, &\quad\text{if}\quad \gamma \geq \frac{38 b}{5}.
        \end{cases}
    \end{align}
    Similarly, the inequality above is derived using Mathematica; but can also be proven using a similar argument as in Proposition~\ref{PROP:max-two-var-poiynomial}. We omit the tedious calculations here.

    Merging (this is done using Mathematica) the piecewise functions given by~\eqref{equ:psi-max-five-vectors-x10},~\eqref{equ:psi-max-five-vectors-x2-x9},~\eqref{equ:psi-max-five-vectors-x2-x3-x8},~\eqref{equ:psi-max-five-vectors-x2-x3--x4-x7}, and~\eqref{equ:psi-max-five-vectors-x2-x3--x4-x5-x6}, and combining them with~\eqref{equ:psi-max-five-vectors}, we obtain  
    \begin{align*}
        \eta_{\ast}
        \le 
        \begin{cases}
          -\frac{3 \gamma^2}{8}+\frac{9 b \gamma}{10}, &\quad\text{if}\quad \gamma \in \left[0, \frac{6b}{5}\right], \\[0.5em]
          \frac{27 b^2}{50}, &\quad\text{if}\quad \gamma \in \left[\frac{6 b}{5}, \frac{(86-4 \sqrt{210})b}{15} \right], \\[0.5em]
          -\frac{\gamma^2}{40}+\frac{43 b \gamma}{150}+\frac{103 b^2}{1125}, &\quad\text{if}\quad \gamma \in \left[ \frac{(86-4 \sqrt{210})b}{15}, \frac{56 b}{15} \right], \\[0.5em]
          -\frac{3 \gamma^2}{232}+ \frac{57 b \gamma}{290}+\frac{113 b^2}{435}, &\quad\text{if}\quad \gamma \in \left[ \frac{56 b}{15}, \frac{38 b}{5} \right], \\[0.5em]
            \frac{151 b^2}{150}, &\quad\text{if}\quad \gamma \geq \frac{38 b}{5}.
        \end{cases}
    \end{align*}
    This completes the proof of Proposition~\ref{PROP:inequality-psi-upper-bound}.
\end{proof}
We present the calculations for~\eqref{equ:psi-max-five-vectors-x2-x9} in the proposition below. 
\begin{proposition}\label{PROP:max-two-var-poiynomial}
    Let $\gamma, b \ge 0$ be real numbers and $\Phi(x,y) \coloneqq -\frac{3}{112} \left(x - \frac{28 b}{15} \right)^2 -\frac{3}{8}\left(y-\frac{16 b}{15}\right)^2$. 
    We have 
    \begin{align*}
        \max_{(x,y) \in \mathbb{S}^{1}(\gamma)}\Phi(x,y)
        =
        \begin{cases}
            -\frac{3 \gamma^2}{8} + \frac{4 b \gamma}{5} - \frac{13b^2}{25}, &\quad\text{if}\quad \gamma \in \left[0, \frac{14 b}{15} \right], \\[0.5em]
            -\frac{\gamma^2}{40}+\frac{11 b \gamma}{75} -\frac{242 b^2}{1125}, &\quad\text{if}\quad \gamma \in \left[\frac{14 b}{15}, \frac{44 b}{15} \right], \\[0.5em]
            0, &\quad\text{if}\quad \gamma \geq \frac{44 b}{15}.
        \end{cases}
    \end{align*}
\end{proposition}
\begin{proof}[Proof of Proposition~\ref{PROP:max-two-var-poiynomial}]
Fix $\gamma, b \ge 0$. 
Fix $(x,y) \in \mathbb{S}^{1}(\gamma)$ such that $\Phi(x,y)$ is maximized over $\mathbb{S}^{1}(\gamma)$.  
The case $\gamma \ge \frac{44b}{15} = \frac{28b}{15} + \frac{16 b}{15}$ is clear, so we may assume that $0 \le \gamma < \frac{44b}{15}$.

\textbf{Case 1}: $x = 0$. 

We now have a one-variable optimization problem:
\begin{align*}
    \text{maximize}\quad & \Phi\left(0,y\right) = -\frac{7 b^2}{75}  -\frac{3}{8}\left(y-\frac{16 b}{15}\right)^2, \\[0.5em]
    \text{subject to}\quad & y \in [0, \gamma].
\end{align*}
It follows from Fact~\ref{FACT:convex-optimization} that 
\begin{align}\label{equ:quad-poly-1}
    \max_{y \in [0, \gamma]}\Phi\left(0,y\right)
    & = 
    \begin{cases}
        -\frac{7 b^2}{75}, &\quad\text{if}\quad \gamma \ge \frac{16b}{15}, \\[0.5em]
        -\frac{3 \gamma^2}{8} +\frac{4 b \gamma}{5} -\frac{13 b^2}{25}, &\quad\text{if}\quad \gamma < \frac{16b}{15}. 
    \end{cases}
\end{align}

\medskip

\textbf{Case 2}: $y = 0$.

Note that 
\begin{align*}
    \Phi(x,0)
    = -\frac{3}{112} \left(x - \frac{28 b}{15} \right)^2 -\frac{32 b^2}{75}. 
\end{align*}
Similar to Case 1, we have 
\begin{align}\label{equ:quad-poly-2}
    \max_{x \in [0, \gamma]}\Phi(x,0)
    & = 
    \begin{cases}
        -\frac{32 b^2}{75}, &\quad\text{if}\quad \gamma \ge \frac{28 b}{15}, \\[0.5em]
        -\frac{3 \gamma^2}{112} +\frac{b \gamma}{10} -\frac{13 b^2}{25}, &\quad\text{if}\quad \gamma < \frac{28 b}{15}. 
    \end{cases}
\end{align}

\medskip

\textbf{Case 3}: $x + y = \gamma$.

Note that 
\begin{align*}
    \Phi(x,\gamma - x)
    = -\frac{45}{112}\left(x - \frac{14(15 \gamma - 14 b)}{225}\right)^2 -\frac{\gamma^2}{40} +\frac{11 b \gamma}{75} -\frac{242 b^2}{1125}.
\end{align*}
Since $\gamma \ge \frac{114(15 \gamma - 14 b)}{225}$ for all $\gamma \ge 0$, we have 
\begin{align}\label{equ:quad-poly-3}
    \max_{x\in [0, \gamma]}\Phi(x,\gamma - x)
    & = 
    \begin{cases}
        -\frac{\gamma^2}{40} +\frac{11 b \gamma}{75} -\frac{242 b^2}{1125}, &\quad\text{if}\quad \gamma \ge \frac{14 b}{15}, \\[0.5em]
        -\frac{3 \gamma^2}{8} +\frac{4 b \gamma}{5} -\frac{13 b^2}{25}, &\quad\text{if}\quad \gamma < \frac{14 b}{15}. 
    \end{cases}
\end{align}

\medskip

\textbf{Case 4}: $x > 0$, $y > 0$, and $x+y < \gamma$. 

It follows from the maximality of $\Phi(x,y)$ that 
\begin{align*}
    \frac{\partial \Phi(x,y)}{\partial x} 
     = -\frac{6}{112}\left(x - \frac{28 b}{15}\right) = 0, \quad\text{and}\quad 
    \frac{\partial \Phi(x,y)}{\partial y} 
     = -\frac{3}{4}\left(y - \frac{16 b}{15}\right) = 0.
\end{align*}
This implies that $(x, y) = \left(\frac{28 b}{15},\frac{16 b}{15}\right)$ and thus $\gamma > \frac{28 b}{15} + \frac{16 b}{15} = \frac{44b}{15}$, contradicting the assumption that $\gamma < \frac{44b}{15}$.

\bigskip

Next, we compare the lower bounds~\eqref{equ:quad-poly-1},~\eqref{equ:quad-poly-2}, and~\eqref{equ:quad-poly-3} obtained above. 

Since 
\begin{align*}
    \left(-\frac{3 \gamma^2}{8} +\frac{4 b \gamma}{5} -\frac{13 b^2}{25}\right)
    - \left(-\frac{3 \gamma^2}{112} +\frac{b \gamma}{10} -\frac{13 b^2}{25}\right)
    = \frac{\gamma (392 b-195 \gamma)}{560} 
    \ge 0
\end{align*}
for $\gamma \in \left[0, \frac{14 b}{15}\right]$, 
we have 
\begin{align*}
    \max_{(x,y) \in \mathbb{S}^{1}(\gamma)}\Phi(x,y)
    = -\frac{3 \gamma^2}{8} +\frac{4 b \gamma}{5} -\frac{13 b^2}{25} 
    \quad\text{for}\quad \gamma \in \left[0, \frac{14 b}{15}\right]. 
\end{align*}

Since 
\begin{align*}
    \left(-\frac{\gamma^2}{40} +\frac{11 b \gamma}{75} -\frac{242 b^2}{1125}\right)
    - \left(-\frac{7b^2}{75}\right)
    = \frac{238 b^2}{1125}+\frac{11 b \gamma}{75}-\frac{\gamma^2}{40}
    \ge 0, \quad\text{for}\quad \gamma \in \left[\frac{16b}{15}, \frac{44b}{15}\right],
\end{align*}
\begin{align*}
    \left(-\frac{\gamma^2}{40} +\frac{11 b \gamma}{75} -\frac{242 b^2}{1125}\right)
    - \left(-\frac{3 \gamma^2}{8} +\frac{4 b \gamma}{5} -\frac{13 b^2}{25}\right)
    = \frac{7 (14 b-15 \gamma)^2}{4500}
    \ge 0, 
\end{align*}
\begin{align*}
    \left(-\frac{\gamma^2}{40} +\frac{11 b \gamma}{75} -\frac{242 b^2}{1125}\right)
    - \left(-\frac{32b^2}{75}\right)
    = \frac{238 b^2}{1125}+\frac{11 b \gamma}{75}-\frac{\gamma^2}{40}
    \ge 0, \quad\text{for}\quad \gamma \in \left[\frac{18b}{15}, \frac{44b}{15}\right],
\end{align*}

and
\begin{align*}
    \left(-\frac{\gamma^2}{40} +\frac{11 b \gamma}{75} -\frac{242 b^2}{1125}\right)
    - \left(-\frac{3 \gamma^2}{112} +\frac{b \gamma}{10} -\frac{13 b^2}{25}\right)
    = \frac{(196 b+15 \gamma)^2}{126000}
    \ge 0, 
\end{align*}
we have 
\begin{align*}
    \max_{(x,y) \in \mathbb{S}^{1}(\gamma)}\Phi(x,y)
    = -\frac{\gamma^2}{40} +\frac{11 b \gamma}{75} -\frac{242 b^2}{1125}
    \quad\text{for}\quad \gamma \in \left[\frac{14 b}{15}, \frac{44b}{15}\right]. 
\end{align*}
This completes the proof of Proposition~\ref{PROP:max-two-var-poiynomial}.
\end{proof}

\section{Proof of Proposition~\ref{PROP:T3-h-a-subgraph}}\label{APPEN:PROP:T3-h-a-subgraph}
\begin{proof}[Proof of Proposition~\ref{PROP:T3-h-a-subgraph}]
    Let $S \subseteq V(H)$ be a set of size $s \ge \max\{\alpha,h - 12 \alpha - 42\}$. 
    Let $S_i \coloneqq S \cap (Y_i \cup Z_i)$ and $x_i \coloneqq  |S_i|$ for $i \in [3]$, noting that $x_i \le \alpha+1$ for every $i \in [3]$. Let $x \coloneqq x_1 + x_2 + x_3$. 
    
    Since $H - X$ is complete $3$-partite, simple calculations show that  
    \begin{align*}
        |H[S_1 \cup S_2 \cup S_3]|
        & = x_1 x_2 + x_2 x_3 + x_3 x_1 \\[0.5em]
        & \ge 
        \begin{cases}
            0, &\quad\text{if}\quad x \le \alpha + 1, \\[0.5em]
            (\alpha + 1) (x- \alpha -1), &\quad\text{if}\quad \alpha + 1 \le x \le 2 (\alpha + 1), \\[0.5em]
            (\alpha+1)^2 + (2\alpha+2)(x-2\alpha - 2), & \quad\text{if}\quad x \ge 2 (\alpha + 1).
        \end{cases}
    \end{align*}

    \medskip 

    \textbf{Case 1}: $x \le \alpha$

    We have 
    \begin{align*}
        |H[S]|
        & = |H[S \cap X]| + |H[S \cap X, S\cap (Y_1 \cup Y_2 \cup Y_3)]| + |H[S_1 \cup S_2 \cup S_3]| \\[0.5em] 
        & \ge \binom{s-x}{2} 
        \ge \binom{s - \alpha}{2}. 
    \end{align*}

    \medskip 

    \textbf{Case 2}: $\alpha + 1 \le x \le 2 (\alpha + 1)$. 

    We have 
    \begin{align*}
        |H[S]|
        & = |H[S \cap X]| + |H[S \cap X, S\cap (Y_1 \cup Y_2 \cup Y_3)]| + |H[S_1 \cup S_2 \cup S_3]| \\[0.5em]
        & \ge \binom{s-x}{2} + (s-x)(x-\alpha) + (\alpha + 1) (x- \alpha -1) \\[0.5em] 
        & = - \frac{1}{2} \left(x - 2\alpha - \frac{3}{2}\right)^2 + \frac{s^2}{2} - \alpha s + \alpha^2 - \frac{s}{2} + \alpha + \frac{1}{8} \\[0.5em]
        & \ge - \frac{1}{2} \left(\alpha + 1 - 2\alpha - \frac{3}{2}\right)^2 + \frac{s^2}{2} - \alpha s + \alpha^2 - \frac{s}{2} + \alpha + \frac{1}{8} \\[0.5em]
        & = \frac{s^2}{2} - \alpha s + \frac{\alpha^2}{2} - \frac{s-\alpha}{2}
        = \binom{s-\alpha}{2}. 
    \end{align*}

    \medskip

    \textbf{Case 3}: $x \ge 2(\alpha+1)$. 

    We have 
    \begin{align*}
        |H[S]|
        & = |H[S \cap X]| + |H[S \cap X, S\cap (Y_1 \cup Y_2 \cup Y_3)]| + |H[S_1 \cup S_2 \cup S_3]| \\[0.5em]
        & \ge \binom{s-x}{2} + (s-x)(x-\alpha) + (\alpha+1)^2 + (2\alpha+2)(x-2\alpha - 2) \\[0.5em] 
        & = -\frac{1}{2}\left(x - 3\alpha - \frac{5}{2}\right)^2 + \frac{s^2}{2} - \alpha s + \frac{3\alpha^2}{2} - \frac{s}{2} + \frac{3\alpha}{2} + \frac{1}{8} \\[0.5em]
        & \ge -\frac{1}{2}\left(2\alpha+2 - 3\alpha - \frac{5}{2}\right)^2 + \frac{s^2}{2} - \alpha s + \frac{3\alpha^2}{2} - \frac{s}{2} + \frac{3\alpha}{2} + \frac{1}{8} \\[0.5em]
        & = \frac{s^2}{2} - \alpha s + \alpha^2 - \frac{s}{2} + \alpha
        = \binom{s-\alpha}{2} + \frac{\alpha^2 + \alpha}{2}
        \ge \binom{s-\alpha}{2}. 
    \end{align*}
    This completes the proof of Proposition~\ref{PROP:T3-h-a-subgraph}. 
\end{proof}
\end{appendix}
\end{document}